\documentclass[10pt, a4paper]{amsart}

\usepackage[T1]{fontenc}	
\usepackage{totpages}
\usepackage{graphicx}
\usepackage{amsmath, amsthm, amssymb}
\usepackage{upgreek}
\usepackage{mathrsfs}
\usepackage{pdfsync}
\usepackage[a4paper,top=3cm,bottom=3cm,left=2.5cm,right=2.5cm, bindingoffset=0mm]{geometry}
\usepackage[usenames,dvipsnames]{xcolor}
\usepackage[inline,shortlabels]{enumitem}
\usepackage[hyphens]{url}									
\usepackage{verbatim}								
\usepackage{dsfont}									
\usepackage[all]{xy}
\usepackage{mathtools}
\usepackage{microtype}
\usepackage{cite}
\usepackage{subcaption}
\usepackage{caption}
\usepackage{adjustbox}
\usepackage[pdftex,bookmarks,
						colorlinks,				
						breaklinks,
						citecolor=OliveGreen,
						linkcolor=Maroon
						]{hyperref} 

\usepackage{tikz}
\usepackage{tikz-cd}
\usetikzlibrary{cd}
\usetikzlibrary{decorations}

\usepackage[normalem]{ulem}

\allowdisplaybreaks

\usepackage{xparse}

\ExplSyntaxOn
\NewDocumentCommand{\makeabbrev}{mmm}
 {
  \yoruk_makeabbrev:nnn { #1 } { #2 } { #3 }
 }

\cs_new_protected:Npn \yoruk_makeabbrev:nnn #1 #2 #3
 {
  \clist_map_inline:nn { #3 }
   {
    \cs_new_protected:cpn { #2 } { #1 { ##1 } }
   }
 }
 \ExplSyntaxOff

\makeabbrev{\textbf}{tbf#1}{a,b,c,d,e,f,g,h,i,j,k,l,m,n,o,p,q,r,s,t,u,v,w,x,y,z,A,B,C,D,E,F,G,H,I,J,K,L,M,N,O,P,Q,R,S,T,U,V,W,X,Y,Z}

\makeabbrev{\textbf}{bf#1}{a,b,c,d,e,f,g,h,i,j,k,l,m,n,o,p,q,r,s,t,u,v,w,x,y,z,A,B,C,D,E,F,G,H,I,J,K,L,M,N,O,P,Q,R,S,T,U,V,W,X,Y,Z}

\makeabbrev{\textsf}{tsf#1}{a,b,c,d,e,f,g,h,i,j,k,l,m,n,o,p,q,r,s,t,u,v,w,x,y,z,A,B,C,D,E,F,G,H,I,J,K,L,M,N,O,P,Q,R,S,T,U,V,W,X,Y,Z}

\makeabbrev{\mathsf}{mss#1}{a,b,c,d,e,f,g,h,i,j,k,l,m,n,o,p,q,r,s,t,u,v,w,x,y,z,A,B,C,D,E,F,G,H,I,J,K,L,M,N,O,P,Q,R,S,T,U,V,W,X,Y,Z}

\makeabbrev{\mathfrak}{mf#1}{a,b,c,d,e,f,g,h,i,j,k,l,m,n,o,p,q,r,s,t,u,v,w,x,y,z,A,B,C,D,E,F,G,H,I,J,K,L,M,N,O,P,Q,R,S,T,U,V,W,X,Y,Z}

\makeabbrev{\mathrm}{mrm#1}{a,b,c,d,e,f,g,h,i,j,k,l,m,n,o,p,q,r,s,t,u,v,w,x,y,z,A,B,C,D,E,F,G,H,I,J,K,L,M,N,O,P,Q,R,S,T,U,V,W,X,Y,Z}

\makeabbrev{\mathbf}{mbf#1}{a,b,c,d,e,f,g,h,i,j,k,l,m,n,o,p,q,r,s,t,u,v,w,x,y,z,A,B,C,D,E,F,G,H,I,J,K,L,M,N,O,P,Q,R,S,T,U,V,W,X,Y,Z}

\makeabbrev{\mathcal}{mc#1}{A,B,C,D,E,F,G,H,I,J,K,L,M,N,O,P,Q,R,S,T,U,V,W,X,Y,Z}

\makeabbrev{\mathbb}{mbb#1}{A,B,C,D,E,F,G,H,I,J,K,L,M,N,O,P,Q,R,S,T,U,V,W,X,Y,Z}

\makeabbrev{\mathscr}{ms#1}{A,B,C,D,E,F,G,H,I,J,K,L,M,N,O,P,Q,R,S,T,U,V,W,X,Y,Z}

\makeabbrev{\mathrm}{#1}{
Id,id,ran,rk,diag,stab,ann,conv,pr,ev,tr,End,Hom,sgn,im,op,can,fin,ext,red,tot,
rot,usc,lsc,Lip,lip,bSymLip,osc,AC,loc,coz,z,Tr,
supp,Opt,Adm,Cpl,Geo,GeoOpt,GeoAdm,GeoCpl,reg,res,
bd,co,Ric,Exp,dExp,dist,seg,Seg,cut,fcut,Cut,SDiff,Iso,Isom,diam,cl,Homeo,Diff,Der,vol,dvol,inj,relint, Gph, sub,
var,law,Var,Poi,Gam,pa,so,iso,fs,inv,dinv,pqi,mix,erg,cov,Ber,mme,Mosco,
TestF,
}

\makeabbrev{\mathsf}{#1}{CD,BE,MCP,Ent,wMTW,MTW,Ch,RCD,EVI,Rad,dRad,SL,cSL,dSL,ScL,Irr,SC,wFe,VA}

\makeabbrev{\mathsc}{#1}{mmaf,cg,fv,df,g}

\DeclareMathOperator{\cenv}{c-env}

\newcommand{\Leb}{\msL}

\newcommand{\eps}{\varepsilon}

\newcommand{\defeq}{\eqqcolon}
\renewcommand{\complement}{\mathrm{c}}

\newcommand{\mathsc}[1]{\text{\textsc{#1}}}
\newcommand{\emparg}{{\,\cdot\,}}

\newcommand{\forallae}[1]{{\textrm{\,for ${#1}$-a.e.~}}}
\newcommand{\as}[1]{\quad #1\text{-a.e.}}

\newcommand{\dom}[1]{D(#1)}

\DeclareMathOperator{\eqdef}{\coloneqq}

\DeclareMathOperator*{\argmin}{argmin}
\DeclareMathOperator*{\Limsup}{Limsup}
\DeclareMathOperator*{\Liminf}{Liminf}
\DeclareMathOperator*{\Lim}{Lim}

\let\epsilon\varepsilon

\newcommand{\longrar}{\longrightarrow}

\newcommand{\nlim}{\lim_{n}}								

\newcommand{\diff}{\mathop{}\!\mathrm{d}}						

\newcommand{\tabs}[1]{\big\lvert#1\big\rvert}	
\newcommand{\abs}[1]{\left\lvert#1\right\rvert}						
\newcommand{\tnorm}[1]{\big\lVert#1\big\rVert}					
\newcommand{\norm}[1]{\left\lVert#1\right\rVert}					
\newcommand{\set}[1]{\left\{#1\right\}}							
\newcommand{\tset}[1]{\big\{#1\big\}}							
\newcommand{\paren}[1]{\left(#1\right)}							
\newcommand{\tparen}[1]{\big({#1}\big)}

\newcommand{\braket}[1]{\left[#1\right]}							

\newcommand{\sym}[1]{{\scriptscriptstyle{(#1)}}}

\newcommand{\reptwo}[1]{\tilde{#1}}							
\newcommand{\tscalar}[2]{\big\langle #1 \, \big |\, #2\big\rangle}			
\newcommand{\ttscalar}[2]{\langle #1 \, |\, #2\rangle}			
\newcommand{\scalar}[2]{\left\langle #1 \,\middle |\, #2\right\rangle}		
\newcommand{\seq}[1]{\paren{#1}}								
\newcommand{\tseq}[1]{{\big(#1\big)}}

\newcommand{\pfwd}{\sharp}
\DeclareMathOperator*{\esssup}{esssup}

\newcommand{\tperp}{{\scriptscriptstyle{\perp}}}
\newcommand{\otym}[1]{{\scriptscriptstyle{\otimes #1}}}
\newcommand{\tym}[1]{{\scriptscriptstyle{\times #1}}}
\newcommand{\opl}[1]{{\scriptscriptstyle{\oplus #1}}}

\DeclareMathOperator{\car}{\mathds{1}}
\newcommand{\cproj}{P}
\newcommand{\aproj}{Q}

\DeclareMathOperator{\emp}{\varnothing}
\newcommand{\N}{{\mathbb N}}
\newcommand{\R}{{\mathbb R}}

\newcommand{\tleq}{{\scriptscriptstyle{\leq}}}

\newcommand{\comma}{\,\mathrm{,}\;\,}
\newcommand{\semicolon}{\,\mathrm{;}\;\,}
\newcommand{\fstop}{\,\mathrm{.}}

\newcommand{\ground}[1]{#1_\circ}

\usepackage{scrextend}						

\let\temp\phi
\let\phi\varphi
\let\varphi\temp

\numberwithin{equation}{section}
\theoremstyle{plain}
\newtheorem{theorem}{Theorem}[section]
\newtheorem*{thm*}{Theorem}
\newtheorem*{mthm*}{Main Theorem}
\newtheorem{proposition}[theorem]{Proposition}
\newtheorem{lemma}[theorem]{Lemma}
\newtheorem{corollary}[theorem]{Corollary}

\theoremstyle{definition}
\newtheorem{definition}[theorem]{Definition}
\newtheorem*{defs*}{Definition}

\theoremstyle{remark}
\newtheorem{remark}[theorem]{Remark}
\newtheorem{example}[theorem]{Example}
\newtheorem*{ass*}{Assumption}

\renewcommand{\paragraph}[1]{\medskip \emph{#1.} \ }

\begin{document}

\title[Non-linear Dirichlet Functionals I]{Non-linear Dirichlet Functionals:\\Markovianity, locality, invariance}

\author{Giovanni Brigati}
\address{IRIMAS, University of Upper Alsace, Department of Mathematics, FST, 18 Rue des Frères Lumière, 68093 Mulhouse}

\author{Lorenzo Dello Schiavo}
\address{Dipartimento di Matematica -- Universit\`a degli Studi di Roma ``Tor Vergata''\\ Via della Ricerca Scientifica~1, 00133 Rome, Italy}

\begin{abstract}
We present in a unified setting the foundations for a theory of non-linear Dirichlet functionals on Hilbert spaces. 
We prove known and new equivalences between non-linear semigroups, non-linear resolvents, non-linear generators, and their energy functionals, including a complete characterization of Markovianity.
We introduce and characterize strong notions of invariance for general lower semicontinuous convex functionals, and notions of locality and strong locality for non-linear Dirichlet functionals.
Contrary to many partial results in the literature, these characterizations are complete under minimal natural assumptions, 
and reduce to the classical ones in the bilinear case.
Our definitions and proofs simplify and clarify the analogous facts in the linear theory, showing in particular that (the equality defining) locality is a saturation of (the inequality defining) Markovianity.
\end{abstract}


\maketitle

\vspace{.5cm}
\begin{center}
\today
\end{center}
\vspace{.5cm}

\setcounter{tocdepth}{2}
\tableofcontents

\section{Introduction}
Let~$(X,\mfA,\mssm)$ be a $\sigma$-finite measure space.
A closed, symmetric, and densely defined real-valued bilinear form~$(E,\dom{E})$ on~$L^2_\mssm$ is a (\emph{symmetric}) \emph{Dirichlet form} if the associated quadratic functional~$E(u)\eqdef E(u,u)$ satisfies
\[
u_+ \wedge 1\in \dom{E} \quad \text{and} \quad E(u_+ \wedge 1) \leq E(u)\comma \qquad u\in\dom{E}\fstop
\]
Every such form is uniquely associated with a strongly continuous contraction semigroup~$P_\bullet\eqdef \seq{P_t}_{t\geq 0}$ of bounded (non-negative self-adjoint) operators~$P_t\colon L^2_\mssm\to L^2_\mssm$ for~$t\geq 0$, satisfying the (\emph{sub-})\emph{Markov property}
\[
0\leq f \leq 1 \implies 0\leq P_t f \leq 1 \comma \qquad f\in L^2_\mssm\comma \qquad t\geq 0\fstop
\]
Under additional assumptions of a topological nature, $P_\bullet$~is in turn the transition semigroup of an $X$-valued $\mssm$-symmetric Markov process.

This correspondence provides a powerful unifying analytic framework for the study of Markov processes, potential theory, and the probabilistic representation of solutions to linear partial differential equations; see, e.g., the monographs~\cite{MaRoe92,BouHir91,FukOshTak11,CheFuk11}.

\medskip

In recent years, interest has grown in \emph{non-linear} versions of the theory, with a view to providing a similar suite of analytical tools for the study of solutions to certain non-linear stochastic differential equations.
On the analytic side:
\emph{non-quadratic} (also: \emph{non-(bi)linear}) \emph{Dirichlet functionals} were originally considered by Bénilan and Picard~\cite{BenPic79}, and much later by Claus~\cite{Cla21}, by Claus and Chill~\cite{Cla21b,ChiCla24}, by Brigati and Hartarsky~\cite{BriHar22}, by Puchert~\cite{Puc25} and by Schmidt and Zimmermann~\cite{SchZim25};
\emph{non-quadratic Dirichlet forms} by van Beusekom~\cite{vBe94}, and more recently by Beznea, Beznea, and~R\"ockner~\cite{BezBezRoe24}; \emph{$p$-energy forms} by Kigami~\cite{Kig23}, Kajino and Shimizu~\cite{KajShi24}, Sasaya~\cite{Sas26}; \emph{non-quadratic resistance forms} by Puchert and Schmidt~\cite{PucSch25}.
On the stochastic side, some classes of \emph{non-linear Markov processes} have been considered by Rehmeier and Röckner~\cite{RehRoe25}. 
However, a systematic link between analytic and stochastic aspects has not yet been established.
In this work, we are concerned with the analytic side of the theory.

\paragraph{Convex lower semicontinuous functionals}
Since the seminal work of K\={o}mura~\cite{Kom67}, Crandall and Pazy~\cite{CraPaz69}, Crandall and Liggett~\cite{CraLig71}, and Br\'ezis~\cite{Bre71}, an equivalence has been established among:
\begin{itemize}
\item[$(T)$] strongly continuous non-expansive non-linear semigroups (see Dfn.~\ref{d:Semigroup}), 
\item[$(A)$] maximal monotone generators (see Dfn.~\ref{d:Generator}),
\item[$(J)$] strongly continuous non-expansive non-linear resolvents (see Dfn.~\ref{d:Resolvent});
\end{itemize}
and, when the generator is additionally cyclically monotone, 
\begin{itemize}
\item[$(E)$] proper convex lower semicontinuous functionals.
\end{itemize}
We call semigroups, generators, resolvents, and functionals collectively \emph{TAJE-operators}.

\paragraph{Markovianity and Dirichlet functionals}
When TAJE-operators are defined on a Lebesgue space, its natural Banach-lattice structure may be used to introduce further properties of the operators.
The main property of our interest is Markovianity.
For non-linear operators, it was originally considered by Bénilan and Picard in~\cite{BenPic79}, who related a bivariate contraction property of the functional (see~\eqref{eq:Claus} below) to properties of its generator.
The notion was however understood more deeply by Cipriani and Grillo in~\cite{CipGri03}, who put it forward as a fundamental concept in this generality, and recognized that the Markovianity of a non-linear semigroup is nothing but the invariance under its double of certain closed convex subsets of~$L^2_\mssm\oplus L^2_\mssm$.
This makes the whole theory of invariant convex sets for non-linear semigroups~\cite{Bre73} available for its study.
(For a comparison of this notion with the weaker one considered by van~Beusekom in~\cite{vBe94}, see Remark~\ref{r:SubMarkovUnivariate}.)

Let~$C$ be a closed convex subset of~$L^2_\mssm$.
A (non-linear) operator~$B\colon C\to L^2_\mssm$ is called \emph{sub-Markovian} if it is both:
\emph{order-preserving}, i.e.\ for every~$u,v\in C$ 
\[
u\leq v \implies B(u)\leq B(v)\comma
\]
and
\emph{$L^\infty$-non-expansive}, i.e.\ for every~$u,v\in C$ (the right-hand side below being possibly infinite)
\[
\norm{B(u)-B(v)}_{L^\infty_\mssm}\leq \norm{u-v}_{L^\infty_\mssm}  \fstop
\]

Following~\cite{CipGri03}, consider the closed convex subsets of~$L^2_\mssm\oplus L^2_\mssm$
\[
C_\tleq\eqdef\set{(u,v) : u\leq v} \qquad\text{and}\qquad C_\alpha\eqdef\set{(u,v) : \abs{u-v}\leq\alpha}\comma \quad \alpha\geq 0\comma
\]
and denote by~$\cproj_\tleq$, resp.~$\cproj_\alpha$, the projection onto~$C_\tleq$, resp.~$C_\alpha$.
We say that a functional~$E\colon L^2_\mssm \to (-\infty,\infty]$ is \emph{Markovian} if its double~$E^\opl{2}(u,v)\eqdef E(u)+E(v)$ satisfies
\[
E^\opl{2}(\cproj_\tleq w)\leq E^\opl{2}(w) \qquad\text{and}\qquad E^\opl{2}(\cproj_\alpha w)\leq E^\opl{2}(w) \comma \qquad w\in L^2_\mssm\oplus L^2_\mssm\comma \quad \alpha\geq 0 \comma
\]
and that~$E$ is a \emph{Dirichlet functional} if it is additionally proper convex lower semicontinuous; see Definition~\ref{d:Dirichlet} and~\ref{i:d:Form:10}.
By Theorem~\ref{t:Main}, a proper convex lower semicontinuous functional~$E$ is a Dirichlet functional if and only if the associated non-expansive non-linear resolvent~$J_\bullet$ is sub-Markovian, i.e.\ if~$J_\lambda$ is sub-Markovian for every $\lambda>0$.
If~$E$ is densely defined, as in~\cite{CipGri03}, or more generally if the convex projection onto~$\overline{\dom{E}}$ is sub-Markovian, this is further equivalent to the sub-Markovianity of the associated semigroup~$T_\bullet$, which~\cite{CipGri03} used as a definition for the sub-Markovianity of~$E$.
In general, our definition requires fewer assumptions than that of~\cite{CipGri03}; see Remark~\ref{r:Main0}.
Throughout, we reserve the adjective `sub-Markovian' for operators (e.g., semigroups and resolvents), `Markovian' for functionals, and `Dirichlet' for proper convex lower semicontinuous Markovian functionals and for their analogues among maximal monotone operators (see~\S\ref{s:SubMarkov}).

Non-linear Dirichlet functionals have been systematically studied in~\cite{Cla21}, and, in connection with normal contractions, in~\cite{BriHar22}.
The fundamental concepts of \emph{transience} and \emph{recurrence} have been recently explored by M.~Schmidt and I.~Zimmermann in~\cite{SchZim25}.
Here ---in addition to an extensive characterization of \emph{Markovianity}--- we address \emph{locality} and \emph{invariance}.
Along the way, we characterize the non-linear semigroups generated by convex energies as those which are \emph{cyclically non-expansive}, and we construct the energy directly from the semigroup.

We briefly describe the content of the paper, divided into a first part, mostly of survey nature, and a second part, containing our main results.

\subsection{Survey part and cyclical non-expansiveness}
In the first part of this work (\S\ref{s:TAJE} and~\S\ref{s:Equivalence}), we survey the extensive theory of TAJE-operators, while in the following sections we present a framework for the study of non-quadratic Dirichlet functionals on \emph{Hilbert} spaces.

In contrast to much work on the subject, our treatment is 
\begin{itemize}[]
\item \emph{systematic}, as we simultaneously consider: semigroups, generators, resolvents, and energy functionals;
\item \emph{complete}, for we consider relations between \emph{all possible pairs} among semigroups, generators, resolvents, and energy functionals;
\item \emph{general}, since we work under minimal assumptions, dispensing whenever possible with unnecessary assumptions like positive-definiteness, evenness, homogeneity, dense-definiteness, linearity of the domain, etc.
\end{itemize}

In the setting of Hilbert spaces, we collect the main known relations among TAJE-operators.
These results are conveniently summarized in the diagram in Figure~\ref{fig:Intro}, whose arrangement mirrors that of the analogous diagram(s) presented in the linear case in the seminal textbook~\cite{MaRoe92}.

Beyond surveying and systematizing known results, we consider the regularization for maximal monotone operators via their semigroups (see~Prop.~\ref{p:At}; as opposed to the standard Yosida regularization via resolvents, cf.~Dfn.~\ref{d:YosidaRegGenerator}). 
We further identify the notion of \emph{cyclical non-expansiveness} (Dfn.~\ref{d:CyclicalNonExp}) as the equivalent notion for semigroups of the cyclical monotonicity of maximal monotone operators: a non-linear semigroup is generated by the subdifferential of a proper convex lower semicontinuous functional if and only if it is cyclically non-expansive (Theorem~\ref{t:CyclicalTAJ}).
This notion appears to be novel, and is naturally understood as a cyclical counterpart of the celebrated \emph{evolution variational inequality} (EVI)~\cite{AmbGigSav08,MurSav20}: a single map is cyclically non-expansive if and only if it satisfies a fixed-time \emph{variational inequality}~$(\mathrm{VI}_t)$ with respect to some convex lower semicontinuous potential (Lemma~\ref{l:OneStepEVI}), and a semigroup satisfies~$(\mathrm{VI}_t)$ for every~$t>0$ with respect to the same potential if and only if it satisfies the integral form~$(\mathrm{EVI}''_0)$ of the EVI in~\cite{MurSav20} (Proposition~\ref{p:EVIMurSav}).
As a consequence, we construct the energy directly from the semigroup, as the limit of \emph{cyclical envelopes} at dyadic times, monotone up to explicit summable corrections (Theorem~\ref{t:FormSemigroup}).
This is the non-linear counterpart of the classical formula~$2E(u)=\lim_{t\downarrow 0}\tfrac{1}{t}\scalar{u-T_tu}{u}$ for quadratic forms (here,~$T_t=P_{2t}$ with~$P_\bullet$ as above, since~$\partial E$ is twice the generator of~$P_\bullet$), and it uses neither resolvents nor minimal (principal) sections of the generator.

\subsection{Main results}
In the second part, we present new results on Markovianity, locality, and invariance.

\subsubsection{Markovianity}
In Theorem~\ref{t:Main} we completely characterize Markovianity (the Dirichlet property) across relations among TAJE-operators: under a natural assumption on the domain, a semigroup, its generator, its resolvent, and its energy are sub-Markovian, resp.\ Dirichlet, all at once.
Again, this is best pictured by the diagram in Fig.~\ref{fig:Intro}.
The characterization is in fact finer: order preservation and $L^\infty$-non-expansiveness are each transferred separately among all four objects.

As in~\cite{CipGri03}, the key is to \emph{double} the space.
An operator is order-preserving, resp.\ $L^\infty$-non-expansive, if and only if its double leaves invariant the closed convex set~$C_\tleq$, resp.\ every strip~$C_\alpha$, $\alpha\geq 0$ (Proposition~\ref{p:SubMarkovianConvex}).
Markovianity thus becomes a question of invariance of convex sets, on which semigroups, generators, resolvents, and functionals are compared in Proposition~\ref{p:EquivConvex}.
The resolvent and the functional always go together, and they imply the corresponding properties of the semigroup and the generator.
The converse holds whenever the projection onto the closure of the domain is itself sub-Markovian (e.g.\ if~$E$ is densely defined), and it may fail otherwise (Remark~\ref{r:Main0}).

In particular, we introduce the Dirichlet property for generators.
A linear Dirichlet operator is singled out by the univariate condition~$\scalar{Au}{(u-\car)_+}\geq 0$.
Its non-linear counterpart is genuinely \emph{bi}variate: a maximal monotone operator~$A$ is a Dirichlet operator if and only if its minimal section satisfies
\[
\scalar{A^\circ u-A^\circ v}{(u-v)_+}\geq 0 \qquad\text{and}\qquad \scalar{A^\circ u-A^\circ v}{(u-v-\alpha)_+-(v-u-\alpha)_+}\geq 0
\]
for all~$u,v\in\dom{A}$ and all~$\alpha\geq 0$, see~\ref{i:d:Generator:7} and~\ref{i:d:Generator:8}.

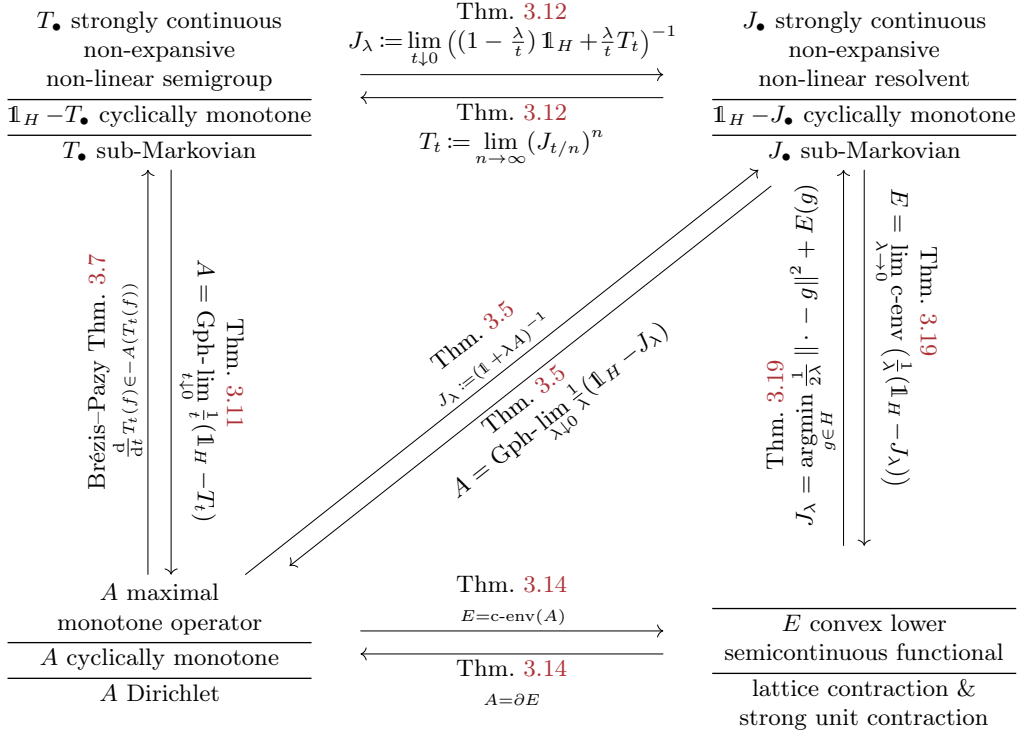
\begin{figure}[htb!]
\begin{adjustbox}{center}
{\small
\begin{tikzcd}[row sep=5cm, column sep=5cm]
\parbox{5cm}{\centering $(T_\bullet, C)$ strongly continuous\\non-expansive\\non-linear semigroup\\ \rule[1.5ex]{4cm}{.5pt}\vspace{-.3cm}\\ $T_\bullet$ cyclically non-expansive\\ \rule[1.5ex]{4cm}{.5pt}\vspace{-.3cm}\\ $T_\bullet$ sub-Markovian}
\arrow[r, shift left=2, "\parbox{5cm}{\centering Thm.~\ref{t:SemigroupResolvent}\\$\displaystyle{J_\lambda\eqdef \lim_{t\downarrow 0} \tparen{(1+\tfrac{\lambda}{t})\car_H - \tfrac{\lambda}{t} T_t\circ\cproj_C}^{-1}}$}"]
\arrow[d, shift left=2, "\parbox{3cm}{\centering Thm.~\ref{t:SemigroupGenerator}\\$\displaystyle{A=\Gph\text{-}\lim_{t \downarrow 0} \tfrac{1}{t}(\car_H-T_t\circ\cproj_C)}$}",sloped]
& \parbox{5cm}{\centering $J_\bullet$ strongly continuous\\non-expansive\\non-linear resolvent\\ \rule[1.5ex]{4cm}{.5pt}\vspace{-.3cm}\\ $\car_H-J_\bullet$ cyclically monotone\\ \rule[1.5ex]{4cm}{.5pt}\vspace{-.3cm}\\ $J_\bullet$ sub-Markovian}
\arrow[l, shift left=2, "\parbox{5cm}{\centering Thm.~\ref{t:SemigroupResolvent}\\$\displaystyle{T_t\eqdef \lim_{n\to\infty} (J_{t/n})^n}$}"] \arrow[dl,shorten >=15, sloped, yshift=-2ex, labels=below, "\parbox{3cm}{\centering Thm.~\ref{t:ResolventGenerator}\\$\displaystyle{A=\Gph\text{-}\lim_{\lambda\downarrow 0} \tfrac{1}{\lambda}(\car_H-J_\lambda)}$}"]
\arrow{d}[anchor=center, rotate=-90, yshift=4ex, pos=.35]{\parbox{3cm}{\centering Thm.~\ref{t:FormResolvent}\\ $\displaystyle{E=\lim_{\lambda\to 0} \cenv\tparen{\tfrac{1}{\lambda}(\car_H-J_\lambda)}}$}}
\\
\parbox{5cm}{\centering $A$ maximal\\monotone operator\\ \rule[1.5ex]{4cm}{.5pt}\vspace{-.3cm}\\$A$ cyclically monotone\\ \rule[1.5ex]{4cm}{.5pt}\vspace{-.3cm}\\$A$ Dirichlet} 
\arrow[u,shift left=2,"\parbox{5cm}{\centering Br\'ezis--Pazy Thm.~\ref{t:BrezisPazy}\\$\tfrac{\diff}{\diff t}T_t(f)\in -A(T_t(f))$}",sloped]
\arrow{ur}[shift left=2, sloped]{\parbox{5cm}{\centering Thm.~\ref{t:ResolventGenerator}\\$J_\lambda\eqdef (\car +\lambda A)^{-1}$}}
\arrow[r, shift left=2, "\parbox{3cm}{\centering Thm.~\ref{t:MaxMonSubDiff}\\$E=\cenv(A)$}"]
& \parbox{5cm}{\centering  \vspace{.4cm} -\\  \rule[1.5ex]{4cm}{.5pt}\vspace{-.3cm}\\$E$ convex lower\\semicontinuous functional\\ \rule[1.5ex]{4cm}{.5pt}\vspace{-.3cm}\\ lattice contraction \& \\ strong unit contraction}
\arrow[l, shift left=2, "\parbox{3cm}{\centering Thm.~\ref{t:MaxMonSubDiff}\\$A=\partial E$}"]
\arrow[u,anchor=center, yshift=2ex, pos=.35, "\parbox{3cm}{\centering Thm.~\ref{t:FormResolvent}\\ $\displaystyle{J_\lambda=\argmin_{g\in H} \tfrac{1}{2\lambda}\norm{\emparg-g}^2+E(g)}$}", sloped]
\end{tikzcd}
}
\end{adjustbox}
\caption{Equivalences between \emph{sub-Markovian} semigroups/resolvents and \emph{Dirichlet} functionals/operators. The sub-Markov/Dirichlet properties of~$J_\bullet$ and~$E$ are always equivalent and imply those of~$T_\bullet$ and~$A$; the converse implications require~\eqref{eq:t:Main:0}, i.e.\ that~$\cproj_C$ be sub-Markovian (Theorem~\ref{t:Main}, Remark~\ref{r:Main0}).}\label{fig:Intro}
\end{figure}

\subsubsection{Dirichlet functionals}
We further consider several characterizations of Markovianity for proper convex lower semicontinuous functionals, including a novel one in terms of the one-sided contractions
\begin{align*}
h_\alpha^+(s,t)&\eqdef s \wedge (t+\alpha)\comma & k_\alpha^+(s,t)&\eqdef t \vee (s-\alpha)\comma
\\
h_\alpha^-(s,t)&\eqdef s \vee (t-\alpha) \comma & k_\alpha^-(s,t)&\eqdef t\wedge (s+\alpha)\fstop
\end{align*}
We say that a functional~$E\colon L^2_\mssm\to (-\infty,\infty]$ has
\begin{itemize}
\item the \emph{strong upper unit-contraction property} if, for~$u,v\in \dom{E}$ and~$\alpha\geq 0$,
\[
h_\alpha^+(u,v)\comma k_\alpha^+(u,v)\in\dom{E} \qquad \text{and}\qquad E\tparen{h_\alpha^+(u,v)}+E\tparen{k_\alpha^+(u,v)} \leq E(u)+E(v) \semicolon
\]
\item the \emph{strong lower unit-contraction property} if, for~$u,v\in \dom{E}$ and~$\alpha\geq 0$,
\[
h_\alpha^-(u,v)\comma k_\alpha^-(u,v)\in\dom{E} \qquad \text{and}\qquad E\tparen{h_\alpha^-(u,v)}+E\tparen{k_\alpha^-(u,v)} \leq E(u)+E(v) \fstop
\]
\end{itemize}

\begin{proposition}[see Proposition~\ref{p:HKpm}]\label{p:MarkovianityIntro} 
A proper convex lower semicontinuous functional is Markovian if and only if it has the strong upper unit-contraction property, equivalently the strong lower unit-contraction property. 
\end{proposition}

See Figure~\ref{fig:DiagramMarkovFunctionals} for several characterizations of Markovianity.

\begin{figure}[hbt!]
\begin{adjustbox}{center}
{\small
\begin{tikzcd}[row sep=3cm, column sep=2cm]
\parbox{3.6cm}{\centering $(T_\bullet,C)$ strongly cont.\\non-expansive non-lin.\\ $T_\bullet$ cycl.\ non-expansive \\  \rule[1.5ex]{3.6cm}{.5pt}\vspace{-.3cm}\\ $T_\bullet$ order-preserving \rule[1.5ex]{3.6cm}{.5pt}\vspace{-.3cm}\\ $T_\bullet$ $L^\infty$-non-expansive \rule[1.5ex]{3.6cm}{.5pt}\vspace{-.3cm}\\ $T_\bullet$ sub-Markovian}
\arrow[r, leftrightarrow, "\parbox{5cm}{\centering Thm.~\ref{t:Main}}"]
& \parbox{3cm}{\centering $E$ proper convex\\lower semicontinuous\\ \rule[1.5ex]{3cm}{.5pt}\vspace{-.3cm}\\ $\cproj_\tleq$-non-expansive \\ \rule[1.5ex]{3cm}{.5pt}\vspace{-.3cm} $\cproj_\alpha$-non-expansive\\ \rule[1.5ex]{3cm}{.5pt}\vspace{-.3cm} Markovian}
\arrow[dl, shorten >=15, leftrightarrow, sloped, yshift=-2ex, labels=below, "\parbox{4cm}{\centering \cite{Cla21,Puc25}, see~Prop.~\ref{p:CP}}"]
\arrow[d, leftrightarrow, sloped, anchor=center, rotate=-180, "\parbox{3cm}{\centering \cite{Bar96}, see~Prop.~\ref{p:Barthelemy}}"]
&
\parbox{3cm}{\centering $E$ proper convex\\lower semicontinuous\\ \rule[1.5ex]{3cm}{.5pt}\vspace{-.3cm}\\ - \\ \rule[1.5ex]{3cm}{.5pt}\vspace{-.3cm}\\ $\cproj_\alpha^+$- or $\cproj_\alpha^-$-\\ non-expansive\\ \rule[1.5ex]{3cm}{.5pt}\vspace{-.3cm} Markovian} 
\arrow[l, leftrightarrow, "\parbox{3cm}{\centering Prop.~\ref{p:HKpm}}"]
\\
\parbox{3cm}{\centering $E$ proper convex\\lower semicontinuous\\ \rule[1.5ex]{3cm}{.5pt}\vspace{-.3cm}\\ - \\ \rule[1.5ex]{3cm}{.5pt}\vspace{-.3cm}\\ - \\ \rule[1.5ex]{3cm}{.5pt}\vspace{-.3cm} Bénilan--Picard~\eqref{eq:Claus} or Puchert~\eqref{eq:Puchert}} 
& \parbox{3cm}{\centering $E$ proper convex\\ lower semicontinuous\\ \rule[1.5ex]{3cm}{.5pt}\vspace{-.3cm}\\ lattice contraction\\ \rule[1.5ex]{3cm}{.5pt}\vspace{-.3cm}\\strong unit contraction\\ \rule[1.5ex]{3cm}{.5pt}\vspace{-.3cm}\\ Dirichlet}
& \parbox{3cm}{\centering $E$ proper convex\\ lower semicontinuous\\ \rule[1.5ex]{3cm}{.5pt}\vspace{-.3cm}\\ - \\  \rule[1.5ex]{3cm}{.5pt}\vspace{-.3cm}\\- \\ \rule[1.5ex]{3cm}{.5pt}\vspace{-.3cm}\\ strong upper or\\ lower unit\\ contraction}
\arrow[l, leftrightarrow, "\parbox{3cm}{\centering Prop.~\ref{p:HKpm}}"]
\end{tikzcd}
}
\end{adjustbox}
\caption{Characterization of Markovianity for functionals.
In the right column, $\cproj_\tleq$-non-expansiveness and the lattice contraction property are contained in either one-sided condition, and the upper and lower conditions are equivalent for every functional (Proposition~\ref{p:HKpm}). The equivalences with the properties of~$T_\bullet$ (Theorem~\ref{t:Main}) require~\eqref{eq:t:Main:0}; without it, only the implications from~$E$ to~$T_\bullet$ hold (Remark~\ref{r:Main0}).}\label{fig:DiagramMarkovFunctionals}
\end{figure}
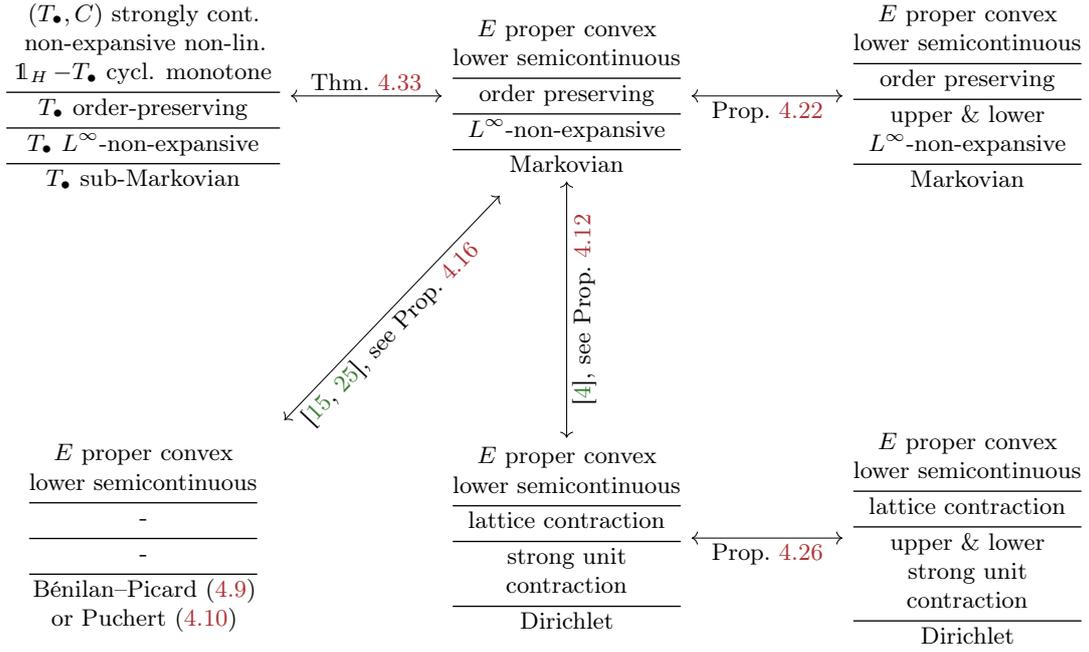

\subsubsection{Locality}
We give a complete characterization of (\emph{weak} and \emph{strong}) \emph{locality} for non-linear Dirichlet functionals.
For (weak) locality, see Proposition~\ref{p:WLocality} and Corollary~\ref{c:WLocality}.

As for strong locality, let~$\Phi$ be the space of real-valued non-expansive maps on the line, viz.
\[
\Phi\eqdef\set{\varphi\colon \R\to\R : \abs{\varphi(s)-\varphi(t)}\leq \abs{s-t}\comma \quad s,t\in\R} \comma
\]
and, for~$\varphi\in\Phi$, set~$\ground{\varphi}\eqdef \varphi-\varphi(0)$.
We prove the following.

\begin{theorem}[Strong locality, see Theorem~\ref{t:Locality} and Corollary~\ref{c:SLocality}]\label{t:LocalityIntro}
Let~$E\colon L^2_\mssm\to [0,\infty]$ be a non-negative \emph{even} Dirichlet functional with~$E(0)=0$. Then, the following are equivalent:
\begin{enumerate}[$({L}_1)$]\setcounter{enumi}{-1}
\item\label{i:t:LocalityIntro:0} for every~$u,v\in\dom{E}$ with~$uv=0$,
\begin{equation*}
E\tparen{\tabs{\abs{u+\alpha}-\alpha}+\tabs{\abs{v+\alpha}-\alpha}}=E(u)+E(v) \comma \qquad \alpha\geq 0 \semicolon
\end{equation*}
\item\label{i:t:LocalityIntro:1} for every~$u,v\in\dom{E}$ with~$(u-c)v=0$ for some~$c\in\R$ and~$u+v\in\dom{E}$,
\begin{equation*}
E(u+v)=E(u)+E(v) \semicolon
\end{equation*}
\item\label{i:t:LocalityIntro:2}
for every~$u,v\in\dom{E}$, with~$(u-c)v=0$ for some~$c\in\R$,
\begin{equation*}
E\tparen{h_{\abs{c}}^\pm(u,v)} + E\tparen{k_{\abs{c}}^\pm(u,v)} = E(u)+E(v) \semicolon
\end{equation*}

\item\label{i:t:LocalityIntro:3} for every~$u,v\in\dom{E}$ with~$uv=0$, for every~$\varphi,\psi\in\Phi$ with~$\varphi \psi=0$,
\begin{equation*}
\begin{aligned}
E\tparen{\abs{(\ground{\varphi}+\ground{\psi})\circ u}+\abs{(\ground{\varphi}+\ground{\psi})\circ v}} &= E(\ground{\varphi}\circ u) + E(\ground{\psi}\circ u)
+ E(\ground{\varphi}\circ v) + E(\ground{\psi}\circ v) \fstop
\end{aligned}
\end{equation*}
\end{enumerate}
If additionally~$\dom{E}$ is a \emph{zero-product cone}, i.e.\ if~$u+v\in\dom{E}$ whenever~$u,v\in\dom{E}$ satisfy~$uv=0$ (e.g., if~$E$ is linearly defined), then~\ref{i:t:LocalityIntro:0}--\ref{i:t:LocalityIntro:3} are further equivalent to each of the following univariate conditions:
\begin{enumerate}[$(\dot{L}_1)$]\setcounter{enumi}{-1}
\item\label{i:t:LocalityIntro:0d} for every~$u\in\dom{E}$,
\begin{equation*}
E\tparen{\abs{u+\alpha}-\alpha}=E(u) \comma \qquad \alpha\geq 0 \semicolon
\end{equation*}
\setcounter{enumi}{2}
\item\label{i:t:LocalityIntro:3d} for every~$u\in\dom{E}$, for every~$\varphi,\psi\in\Phi$ with~$\varphi \psi=0$, we have~$(\ground{\varphi}+\ground{\psi})\circ u\in\dom{E}$ and
\begin{equation*}
E\tparen{(\ground{\varphi}+\ground{\psi})\circ u} = E(\ground{\varphi}\circ u) + E(\ground{\psi}\circ u) \fstop
\end{equation*}
\end{enumerate}
\end{theorem}

The above characterization extends to our general non-linear setting the standard characterizations in~\cite[\S{I.5}]{BouHir91}, which are expressed by the univariate conditions~\ref{i:t:LocalityIntro:0d} and~\ref{i:t:LocalityIntro:3d}.
Without the zero-product-cone property, the characterization holds with the bivariate conditions~\ref{i:t:LocalityIntro:0} and~\ref{i:t:LocalityIntro:3}.
For weak locality, the analogous univariate conditions are in general strictly weaker than their bivariate counterparts, see Example~\ref{e:WLocality}.
Combined with Proposition~\ref{p:MarkovianityIntro}, Theorem~\ref{t:LocalityIntro} shows that strong locality coincides with the \emph{saturation} of the inequality(/ies) defining the Markov property on the (non-convex) sets $\set{(u,v)\in\dom{E}^\opl{2} : (u-c)v=0}$, $c\in\R$. 
A similar fact is true for weak locality, see Proposition~\ref{p:WLocality}.

\subsubsection{Invariance}
We give the following complete characterization of \emph{invariance} of a set w.r.t.\ the action of TAJE-operators.

\begin{theorem}[Invariance, see Thm.~\ref{t:Invariance} and Cor.~\ref{c:DoubleInvariance}]\label{t:InvarianceIntro}
Let~$E\colon L^2_\mssm\to [0,\infty]$ be a non-negative convex lower semicontinuous functional with~$E(0)=0$, with cyclically monotone maximal monotone generator~$A$, non-linear semigroup~$T_\bullet$ and non-linear resolvent~$J_\bullet$, and set~$C\eqdef\overline{\dom{E}}$.
For every~$Y\in\mfA$, the following are equivalent:
\begin{enumerate}[$(i)$]
\item\label{i:t:InvarianceIntro:1} $\car_Y\cproj_C(f)=\cproj_C(\car_Y f)$ and~$\car_{Y^\complement}\cproj_C(f)=\cproj_C(\car_{Y^\complement} f)$ for every~$f\in L^2_\mssm$, and, for every~$t\geq 0$ and every~$f\in C$,
\[
\car_Y T_t(f)=T_t(\car_Y f)\qquad \text{and} \qquad \car_{Y^\complement} T_t(f)=T_t(\car_{Y^\complement} f) \semicolon
\]

\item\label{i:t:InvarianceIntro:2} for every~$f\in\dom{A}$ we have~$\car_Y f, \car_{Y^\complement} f\in\dom{A}$ and
\[
\car_Y A(f)\subset A(\car_Y f) \qquad \text{and} \qquad \car_{Y^\complement} A(f)\subset A(\car_{Y^\complement} f) \semicolon
\]

\item\label{i:t:InvarianceIntro:3} for every~$\lambda>0$ and every~$f\in L^2_\mssm$,
\[
\car_Y J_\lambda(f) = J_\lambda(\car_Y f)\qquad \text{and}\qquad \car_{Y^\complement}  J_\lambda(f) = J_\lambda(\car_{Y^\complement} f) \semicolon
\]

\item\label{i:t:InvarianceIntro:4} for every~$f\in L^2_\mssm$,
\[
E(f)=E(\car_Y f) + E(\car_{Y^\complement} f) \fstop
\]
\end{enumerate}
\end{theorem}

The above characterization extends to our general non-linear setting the standard characterizations in~\cite[\S{I.1.6}]{FukOshTak11}.
In particular, it provides the right definition of invariance for this extension to hold, which we call \emph{double invariance}.
This notion is strictly stronger than the more standard invariance for non-linear semigroups in~\cite{Bre73}, recently considered for convex functionals in~\cite[\S3.4]{SchZim25}, yet the two notions are equivalent for linear self-adjoint semigroups.
See~\S\ref{s:Invariance} for a thorough comparison.

Finally, denoting by~$\mfB_\dinv$ the family of all 
doubly invariant sets (see Definition~\ref{d:InvariantSet}), we prove the following.

\begin{theorem}[See Thm.~\ref{t:BDinvSAlg}]
Let~$E\colon L^2_\mssm \to [0,\infty]$ be a non-negative proper convex lower semicontinuous functional.
Then,~$\mfB_\dinv$ is non-empty if and only if~$0$ is a global minimizer of~$E$.
When this is the case,~$\mfB_\dinv$ is a $\sigma$-algebra, and~$Y\mapsto E(\car_Y f)-E(0)$ is a measure on~$\mfB_\dinv$ for every~$f\in L^2_\mssm$.
\end{theorem}

\subsubsection{Beyond Hilbert spaces}
Finally, let us point out that we confine our treatment to non-quadratic functionals on \emph{Hilbert} spaces, but this is mostly for simplicity.
The Hilbert structure of~$L^2_\mssm\oplus L^2_\mssm$ enters our approach to Markovianity in particular through the projections~$\cproj_\tleq$ and~$\cproj_\alpha$, and through the characterization of invariant convex sets in Proposition~\ref{p:EquivConvex}.
Extensions of the results presented here for non-quadratic Dirichlet functionals on $L^2$-spaces to the setting of more general Banach lattices will be addressed in the future.

\subsection*{Acknowledgments}
GB gratefully acknowledges funding from the European Union’s Horizon 2020 research and innovation programme under the Marie Sk{\l}odowska-Curie grant agreement No.~101034413.

LDS gratefully acknowledges funding from the Austrian Science Fund (FWF) project \href{https://doi.org/10.55776/ESP208}{{10.55776/ESP208}}; and from PRIN Department of Excellence MatMod@TOV (CUP: E83C23000330006).

The authors are grateful to Michael Röckner, Daniel Lenz, Simon Puchert, and Marcel Schmidt for useful conversations on the subject of the manuscript.
GB is grateful to Ralph Chill for pointing out to him Example~\ref{e:Chill}.
Part of this research has been carried out while GB was a guest of the Mathematics Departments of the University of Bielefeld and of the Friedrich-Schiller-Universit\"at Jena.
He is grateful to his hosts Michael Röckner and Daniel Lenz and to the Mathematics Departments for their very kind hospitality.
Part of this research has been carried out while LDS was a guest of the IRIMAS of the University of Upper Alsace (Mulhouse).
He is grateful to his host Nicolas Juillet and to the Mathematics Department for their very kind hospitality.

\section{Non-linear semigroups, resolvents, monotone operators, and convex energies}\label{s:TAJE}
For ease of reference, we collect here all the necessary definitions.
Let~$H$ be a Hilbert space with norm~$\norm{\emparg}$ and scalar product~$\scalar{\emparg}{\emparg}$.
We denote by~$\car_H$ the identity operator on~$H$.

Everywhere in the following,~$C\subset H$ is a non-empty closed convex set.
We denote by~$\cproj_C\colon H\to H$ the (\emph{convex}) \emph{projection} operator
\[
\cproj_C(f)\eqdef \argmin_{g\in C} \norm{f-g}^2 \comma
\]
mapping a point in~$H$ to the (unique) closest point in~$C$. For the properties of projections, see e.g.~\cite[\S4.3]{BauCom17}.

Let~$X$ be any set. For any subset~$Y\subset X$, we denote
\begin{itemize}
\item by~$\mbfI_Y$ the \emph{indicator function} of~$Y$, i.e.\ the function
\[
\mbfI_Y(x)\eqdef \begin{cases} 0 &\text{if } x\in Y\comma \\ \infty & \text{if } x\notin Y\comma \end{cases} \qquad x\in X\semicolon
\]

\item by~$\car_Y$ the \emph{characteristic function} of~$Y$, i.e.\ the function
\[
\car_Y(x)\eqdef \begin{cases} 1 &\text{if } x\in Y\comma \\ 0 & \text{if } x\notin Y\comma \end{cases} \qquad x\in X\fstop
\]
\end{itemize}

\subsection{Semigroups}
We recall the notion of `semigroup' in the non-linear setting.

\begin{definition}[Non-linear semigroup]\label{d:Semigroup}
Let~$C$ be a non-empty closed convex subset of~$H$.
A \emph{strongly continuous non-expansive non-linear semigroup~$T_\bullet\eqdef\seq{T_t}_{t\geq 0}$ on~$C$} is a family of maps~$T_t\colon C\to C$ satisfying
\begin{enumerate}[$(T_1)$]
\item (\emph{identity property}) $T_0=\id_C$;
\item\label{i:d:Semigroup:2} (\emph{semigroup property}) $T_{t+s}=T_t\circ T_s$ for every~$s,t\geq 0$;
\item (\emph{strong continuity}) $\lim_{t\downarrow 0} \norm{T_t(f)-f}=0$ for every~$f\in C$;
\item\label{i:d:Semigroup:4} (\emph{non-expansiveness}) $\norm{T_t(f)-T_t(g)}\leq \norm{f-g}$ for every~$f,g\in C$ and every~$t\geq 0$.
\end{enumerate}
\end{definition}

We note that, by~\ref{i:d:Semigroup:2} and~\ref{i:d:Semigroup:4},
\[
\norm{T_t(f)-T_s(f)}=\tnorm{T_s\tparen{T_{t-s}(f)}-T_s(f)}\leq\norm{T_{t-s}(f)-f}\comma \qquad 0\leq s\leq t\comma \quad f\in C\comma
\]
and therefore, by strong continuity,
\begin{equation}\label{eq:OrbitContinuity}
\text{the curve~$t\mapsto T_t(f)$ is continuous on~$[0,\infty)$ for every~$f\in C$}\fstop
\end{equation}

\subsection{Generators}
A (\emph{multi-valued}) \emph{operator}~$A\colon H\rightrightarrows H$ is a map~$A\colon H\to 2^H$.

\begin{definition}[Generators]\label{d:Generator}
The domain of an operator~$A$ on~$H$ is the set
\[
\dom{A}\eqdef \set{f\in H : A(f)\neq \emp} \fstop
\]
We write~$(A,\dom{A})$ when an explicit indication of the domain of the operator~$A$ is necessary or convenient.

An operator~$A$ is
\begin{enumerate}[$({A}_1)$]
\item \emph{densely defined} if~$\dom{A}$ is dense in~$H$;
\item \emph{monotone} if for every~$f_1,f_2\in \dom{A}$, and for every~$g_i\in A(f_i)$ with~$i=1,2$,
\[
\scalar{g_1-g_2}{f_1-f_2}\geq 0 \semicolon
\]
\item \emph{maximal monotone} if it is monotone and maximal w.r.t.\ the inclusion in the set of monotone graphs;
\item\label{i:d:Generator:4} \emph{cyclically monotone} if for every~$n\in\N$, for every~$f_i\in \dom{A}$, and for every~$g_i\in A(f_i)$ with~$1\leq i\leq n$,
\[
\sum_{i=1}^n \scalar{f_i-f_{i-1}}{g_i} \geq 0 \comma \qquad f_0\eqdef f_n\semicolon
\]
\item \emph{odd} if $-f\in\dom{A}$ and~$A(-f)=-A(f)$ for every~$f\in\dom{A}$.
\end{enumerate}
\end{definition}

Everywhere in the following, for a Hilbert space~$H$ set~$H^\opl{2}\eqdef H\oplus H$.
\begin{definition}[Double operator]
The \emph{double~$A^\opl{2}$} of an operator~$A\colon H\rightrightarrows H$ is the operator
\[
A^\opl{2}\colon H^\opl{2}\rightrightarrows H^\opl{2}\comma \qquad A^\opl{2}\colon (f,g) \mapsto A(f) \times A(g) \fstop
\]
\end{definition}

Since~$\emp\times C=C\times \emp=\emp$ for every~$C\subset H$, we have~$\dom{A^\opl{2}}=\dom{A}^\opl{2}$.

\paragraph{Minimal sections and cyclical envelopes}
Every maximal monotone operator~$A$ has closed convex values, thus the convex projection~$\cproj_{A(f)}$ onto~$A(f)$ is well-defined and single-valued, cf.~\cite[p.~28]{Bre73}.

\begin{definition}[Minimal section]
Let~$A$ be a maximal monotone operator.
The \emph{minimal section}~$A^\circ$ of~$A$ is the operator
\[
\dom{A^\circ}\eqdef \dom{A} \comma \qquad A^\circ (f) \eqdef \cproj_{A(f)}(0) \fstop
\]
\end{definition}

Minimal sections are \emph{principal sections} in the sense of~\cite[Dfn.~2.3, Prop.~2.7, p.~29]{Bre73}.
It turns out that the minimal section of a maximal monotone operator is a monotone operator, and that there is a bijective correspondence between maximal monotone operators and their minimal sections, see~\cite[Cor.~2.2, p.~29]{Bre73}.

\subsection{Resolvents}
We recall the notion of \emph{resolvent} in the non-linear setting.

\begin{definition}[Non-linear resolvent]\label{d:Resolvent}
Let~$C$ be a non-empty closed convex subset of~$H$.
A \emph{strongly continuous non-expansive non-linear resolvent~$J_\bullet\eqdef\seq{J_\lambda}_{\lambda\geq 0}$ on~$C$} is a family of maps~$J_\lambda\colon H\to H$ satisfying
\begin{enumerate}[$(J_1)$]
\item\label{i:d:Resolvent:1} (\emph{identity property}) $J_0=\cproj_C$;
\item\label{i:d:Resolvent:2} (\emph{first resolvent identity}) $J_\lambda=J_\mu\tparen{\tfrac{\mu}{\lambda}\car_H + (1-\tfrac{\mu}{\lambda})J_\lambda}$ for every~$\lambda> 0$ and every~$\mu\geq 0$;
\item\label{i:d:Resolvent:3} (\emph{strong continuity}) $\lim_{\lambda\downarrow 0} \norm{J_\lambda(f)-\cproj_C(f)}=0$ for every~$f\in H$;
\item\label{i:d:Resolvent:4} (\emph{non-expansiveness}) $\norm{J_\lambda(f)-J_\lambda(g)}\leq \norm{f-g}$ for every~$f,g\in H$ and every~$\lambda\geq 0$.
\end{enumerate}
Note that, in particular, $J_\lambda (H)\subset C$ for every~$\lambda\geq 0$ choosing~$\mu=0$ in~\ref{i:d:Resolvent:2}.
\end{definition}

\subsection{Energies}
Let us now turn to functionals on~$H$.

\begin{definition}[Properties of functionals]
The \emph{domain} of a functional~$E\colon H\to (-\infty,\infty]$ is the set
\[
\dom{E}\eqdef \set{f\in H : E(f)<\infty} \fstop
\]

A functional~$E\colon H\to (-\infty,\infty]$ is called
\begin{enumerate}[$(E_1)$]
\item\label{i:d:Form:1} \emph{proper} if~$\dom{E}\neq \emp$;
\item \emph{densely defined} if~$\dom{E}$ is dense in~$H$;
\item \emph{linearly defined} if~$\dom{E}$ is a linear subspace of~$H$;
\item\label{i:d:Form:4} \emph{even} if $-f\in\dom{E}$ and~$E(-f)\leq E(f)$ for every~$f\in\dom{E}$; (This forces~$E(-f)=E(f)$.)
\item \emph{positively $p$-homogeneous} if $a f\in\dom{E}$ and~$E(a f)=a^p E(f)$ for every~$f\in\dom{E}$ and every~$a>0$ for some fixed~$p\geq 1$.
\end{enumerate}

A convex functional~$E\colon H\to [0,\infty]$ satisfies 
\begin{enumerate}[$(E_1)$, resume]
\item\label{i:d:Form:6} the \emph{parallelogram identity} if
\[
2E(f)+2E(g)= E(f+g)+E(f-g)\comma \qquad f,g\in\dom{E} \fstop
\]
\end{enumerate}
\end{definition}

\begin{remark}[On terminology]
\begin{enumerate*}[$(a)$]
\item We prefer the terminology of \emph{even} functional over that of \emph{symmetric} functional, used in, e.g.,~\cite{Cla21,Puc25, SchZim25}, since \emph{evenness} more properly denotes \emph{symmetry around the origin}.

\item We prefer the terminology of \emph{linearly defined} functionals over that of \emph{quasilinear} functional, used in~\cite{Cla21}; cf.~\cite[Dfn.~3.1, p.~30]{Cla21}, since we find it more explicit.
\end{enumerate*}
\end{remark}

Let us collect here the following observation about even functionals, which will allow us to restrict our attention to \emph{non-negative} even functionals.

\begin{remark}[On even functionals]\label{r:Even}
(Cf., e.g.,~\cite[Rmk~3.2, p.~30]{Cla21}.) For a proper \emph{convex} functional to be \emph{even} in the sense of~\ref{i:d:Form:4} imposes strong constraints on both the domain and the range of~$E$.
Indeed, it is readily seen from the convexity of~$E$ that~$0\in \dom{E}$ and that~$0$ is a global minimizer of~$E$, hence in particular that~$E$ is \emph{bounded below}.
Furthermore,~$\dom{E}$ is symmetric around the origin of~$H$.

The existence of a minimum of~$E$ further has important consequences on the long-time behaviour of the gradient flow of~$E$; see, e.g.~\cite{SieWoj24}.
\end{remark}

\begin{definition}[Reflection]\label{d:Reflection}
The \emph{reflection} of a functional~$E\colon H\to(-\infty,\infty]$ is the functional~$\hat E\colon H\to(-\infty,\infty]$ defined by
\[
\hat E(f)\eqdef E(-f)\comma \qquad f\in H\fstop
\]
Note that~$\dom{\hat E}=-\dom{E}$, and that~$E$ is even~\ref{i:d:Form:4} if and only if~$\hat E=E$.
\end{definition}

In the following, it will be often convenient to consider function(al)s up to an additive constant.
This observation motivates the next definition.

\begin{definition}[Grounding]
An extended-real valued function(al)~$\varphi$ is called \emph{grounded} if~$\varphi(0)=0$.
To each $\varphi$ as above additionally finite at~$0$ we may associate its \emph{grounding}~$\ground{\varphi}\eqdef \varphi-\varphi(0)$.
\end{definition}

When the zero set of a convex functional contains a line, the domain splits along it.

\begin{lemma}[Invariance along zero lines]\label{l:ZeroLines}
Let~$E\colon H \to (-\infty,\infty]$ be a proper convex lower semicontinuous functional.
Further assume that~$\alpha g\in\dom{E}$ and~$E(\alpha g)=0$ for some~$g\in H$ and every~$\alpha\in\R$.
Then,
\begin{equation}\label{eq:l:ZeroLines:0}
f+\alpha g\in\dom{E} \qquad \text{and}\qquad E(f+\alpha g)= E(f)\comma \qquad f \in\dom{E}\comma \alpha\in\R\fstop
\end{equation}
\end{lemma}

\begin{proof}
Fix~$\alpha\in\R$. For every~$\alpha'\in\R$ and every~$\lambda\in (0,1)$,
\[
E\tparen{\lambda f+ (1-\lambda)\alpha' g} \leq \lambda E(f)+ (1-\lambda) E(\alpha' g) = \lambda E(f) \fstop
\]
Choosing~$\alpha'\eqdef (1-\lambda)^{-1}\lambda\alpha$, we conclude that, for every~$\lambda\in (0,1)$,
\begin{equation}\label{eq:l:ZeroLines:1}
E\tparen{\lambda (f+\alpha g)}\leq \lambda E(f) \fstop
\end{equation}
By lower semicontinuity of~$E$, letting~$\lambda\to 1$ above we conclude that
\begin{equation}\label{eq:l:ZeroLines:2}
f+\alpha g\in\dom{E} \qquad \text{and} \qquad E(f+\alpha g)\leq E(f)\comma \qquad f\in\dom{E}\comma \alpha\in\R\fstop
\end{equation}
Applying~\eqref{eq:l:ZeroLines:2} with~$f-\alpha g$ in place of~$f$ gives
\begin{equation}\label{eq:l:ZeroLines:3}
E(f)\leq E(f-\alpha g)\comma \qquad f\in\dom{E}\comma \alpha\in\R\comma
\end{equation}
and combining~\eqref{eq:l:ZeroLines:2} and~\eqref{eq:l:ZeroLines:3} (with~$-\alpha$ in place of~$\alpha$) finally yields~\eqref{eq:l:ZeroLines:0}.
\end{proof}

\begin{definition}[Doubling]
The \emph{double}~$E^\opl{2}$ of a functional~$E\colon H\to (-\infty,\infty]$ is the functional
\[
E^\opl{2}\colon H^\opl{2} \to (-\infty,\infty] \comma \qquad E^\opl{2}\colon (f,g) \mapsto E(f) + E(g) \fstop
\]
\end{definition}

We note that if a functional~$E\colon H\to (-\infty,\infty]$ satisfies any of the properties~\ref{i:d:Form:1}--\ref{i:d:Form:6}, then so does its double.

\begin{definition}[Fréchet subdifferential]
For~$f\in H$, the (\emph{Fr\'echet}) \emph{subdifferential at~$f$} of a functional~$E\colon H \to (-\infty,\infty]$ is the set
\[
\partial E(f)\eqdef \set{v \in H : E(g)-E(f) \geq \scalar{v}{g-f} + o(\norm{g-f}) \text{ as } \dom{E}\ni g\to f}\comma \qquad f\in\dom{E} \fstop
\]
We further let
\[
\dom{\partial E}\eqdef \set{f\in\dom{E}: \partial E(f)\neq \emp}
\]
and denote by~$\partial E \colon \dom{\partial E} \rightrightarrows H$ the resulting correspondence.
\end{definition}

\begin{remark}[Rockafellar subdifferential]
If~$E\colon H\to (-\infty,\infty]$ is convex, then
\begin{equation}\label{eq:ConvexSubDiff}
\partial E(f)=\set{v \in H : E(g)-E(f) \geq \scalar{v}{g-f}\comma g\in \dom{E}}
\end{equation}
and~$\partial E(f)$ is a closed convex subset of~$H$ for every~$f\in \dom{E}$.
\end{remark}

We conclude with a natural construction of functionals from operators.

\begin{definition}[Cyclical envelope]\label{d:CEnv}
Let~$A$ be an operator on~$H$, and fix~$f_0\in\dom{A}$ and~$g_0\in A(f_0)$.
We define the \emph{cyclical envelope of~$A$} (\emph{based at~$(f_0,g_0)$}) (also: \emph{Rockafellar function}, cf.~\cite{BarBauBorReiWan07}) by
\[
\cenv(A)(f)\eqdef \sup \set{\sum_{i=0}^n \scalar{f_{i+1}-f_i}{g_i} : \begin{aligned} &n\in\N_0\comma \quad f_1,\dotsc, f_n\in \dom{A}\comma\\ &g_i\in A(f_i) \text{ for } 1\leq i\leq n\comma \quad f_{n+1}\eqdef f \end{aligned}} \comma \qquad f\in H\fstop
\]
\end{definition}

\begin{lemma}\label{l:CEnv}
Let~$A$ be an operator on~$H$, fix~$f_0\in\dom{A}$ and~$g_0\in A(f_0)$, and let~$\cenv(A)$ be based at~$(f_0,g_0)$.
Then,~$\cenv(A)\colon H\to(-\infty,\infty]$ is convex and lower semicontinuous. Furthermore,
\begin{enumerate}[$(i)$]
\item\label{i:l:CEnv:1} if~$A$ is cyclically monotone, then~$\cenv(A)(f_0)=0$; in particular,~$\cenv(A)$ is proper;
\item\label{i:l:CEnv:2} if~$A$ is not cyclically monotone, then~$\cenv(A)\equiv +\infty$ on~$H$.
\end{enumerate}
\end{lemma}

\begin{proof}
For fixed~$n\in\N_0$, $f_i\in\dom{A}$, and~$g_i\in A(f_i)$, $1\leq i\leq n$, the sum in Definition~\ref{d:CEnv} depends on~$f$ only through its last summand~$\scalar{f-f_n}{g_n}$, which is affine and continuous in~$f$.
Hence,~$\cenv(A)$ is the supremum of a non-empty family of affine continuous functionals, and it is therefore convex and lower semicontinuous with values in~$(-\infty,\infty]$.

Note that the sum in Definition~\ref{d:CEnv} pairs each~$g_i$ with the increment~$f_{i+1}-f_i$ \emph{leaving}~$f_i$, whereas the sum in~\ref{i:d:Generator:4} pairs each~$g_i$ with the increment~$f_i-f_{i-1}$ \emph{entering}~$f_i$.
The two conventions are exchanged by reversing the orientation of cycles.

\ref{i:l:CEnv:1} (Cf.~\cite[Thm.~2.5, p.~38]{Bre73}.)
Choosing~$n=0$ and~$f=f_0$ shows that~$\cenv(A)(f_0)\geq 0$.
Conversely, fix~$n\in\N_0$, $f_i\in\dom{A}$, and~$g_i\in A(f_i)$ for~$1\leq i\leq n$, and set~$f_{n+1}\eqdef f_0$ and~$m\eqdef n+1$.
Further set~$e_j\eqdef f_{m-j}$ for~$0\leq j\leq m$, and~$b_j\eqdef g_{m-j}\in A(e_j)$ for~$1\leq j\leq m$, so that~$e_0=e_m=f_0$.
Changing the summation index to~$j\eqdef m-i$,
\[
\sum_{i=0}^n \scalar{f_{i+1}-f_i}{g_i} = -\sum_{j=1}^m \scalar{e_j-e_{j-1}}{b_j} \leq 0
\]
by cyclical monotonicity of~$A$.
Taking the supremum over all such chains yields~$\cenv(A)(f_0)\leq 0$.

\ref{i:l:CEnv:2}
Since~$A$ is not cyclically monotone, there exist~$m\in\N$, $h_j\in\dom{A}$, and~$a_j\in A(h_j)$ for~$1\leq j\leq m$, such that, with~$h_0\eqdef h_m$,
\[
\delta\eqdef \sum_{j=1}^m\scalar{h_{j-1}-h_j}{a_j}=-\sum_{j=1}^m\scalar{h_j-h_{j-1}}{a_j}>0\fstop
\]
Fix~$f\in H$ and~$k\in\N$, set~$n\eqdef km+1$, and let~$(f_1,g_1),\dotsc,(f_n,g_n)$ be the finite sequence
\[
(h_m,a_m), (h_{m-1},a_{m-1}),\dotsc,(h_1,a_1)
\]
repeated~$k$ times, followed by~$(h_m,a_m)$; that is, the cycle is run~$k$ times \emph{backwards}.
For every~$1\leq i\leq km$ there is~$j\in\set{1,\dotsc,m}$ with~$(f_i,g_i)=(h_j,a_j)$ and~$f_{i+1}=h_{j-1}$, and each~$j$ occurs exactly~$k$ times.
Hence,
\[
\sum_{i=0}^n \scalar{f_{i+1}-f_i}{g_i} = \scalar{h_m-f_0}{g_0} + k\,\delta + \scalar{f-h_m}{a_m} \comma
\]
and letting~$k\to\infty$ shows that~$\cenv(A)(f)=+\infty$.
\end{proof}

\section{Equivalence theorems}\label{s:Equivalence}
We collect here many equivalence results relating:
\begin{enumerate}
\item[$(T)$] strongly continuous non-expansive non-linear semigroups,
\item[$(A)$] maximal monotone operators,
\item[$(J)$] strongly continuous non-expansive non-linear resolvents.
\end{enumerate}
Under suitable additional cyclical monotonicity assumptions, these objects are further equivalent to
\begin{enumerate}
\item[$(E)$] proper convex lower semicontinuous functionals.
\end{enumerate}

These results include the Br\'ezis--Pazy Theorem, and Crandall--Liggett's non-linear version of Hille--Yosida Theorem.
The complete set of equivalences is summarized in Figure~\ref{f:Diagram1}.
Some of the implications ---for which we present complete proofs--- appear not to have been discussed elsewhere.
For the known implications we mostly refer to Br\'ezis' monograph~\cite{Bre73}.

\subsection{Equivalence of semigroups, generators, and resolvents}
The relations between generators and semigroups/resolvents may be expressed by taking the limits of the latter ones as the parameter vanishes.
Since generators are multivalued, the suitable topology in which these limits are taken is one capturing the inclusion of accumulation points in their graphs.

\paragraph{Painlevé--Kuratowski graph convergence of operators}
Let~$(X,\mssd)$ be a metric space, and~$\seq{C_n}_n\subset X$ be a sequence of subsets.
The \emph{outer} and~\emph{inner limits} of~$\seq{C_n}_n$ as~$n\to\infty$ are respectively defined as
\begin{align*}
\Limsup_n C_n \eqdef \set{x\in X : \liminf_n \mssd(x,C_n) =0 } 
\qquad
\text{and}
\qquad
\Liminf_n C_n \eqdef \set{x\in X : \limsup_n \mssd(x,C_n) =0 } \fstop
\end{align*}
The sequence~$\seq{C_n}_n$ \emph{converges in the sense of Painlevé--Kuratowski} (shortly: \emph{PK-converges}) if the outer and inner limit agree, in which case they are denoted by~$\Lim_n C_n$.

\begin{definition}[Graph convergence of operators]
Let~$H$ be a Hilbert space and~$\seq{A_n}_n$ be a sequence of operators~$A_n \colon H\rightrightarrows H$.
We say that the sequence~$\seq{A_n}_n$ \emph{graph converges} to~$A\colon H \rightrightarrows H$ if the corresponding sequence of graphs~$\Gph(A_n)$ PK-converges to~$\Gph(A)$ in~$H^\tym{2}$, in which case we write~$\Gph\text{-}\lim_n A_n = A$.
We say that a one-parameter family~$\seq{A_\lambda}_\lambda$, with~$A_\lambda \colon H\rightrightarrows H$, \emph{graph converges as~$\lambda\to \lambda_*$} to~$A_{\lambda_*}\colon H \rightrightarrows H$ if the corresponding sequence of graphs~$\Gph(A_{\lambda_n})$ PK-converges to~$\Gph(A_{\lambda_*})$ in~$H^\tym{2}$ for \emph{every} sequence~$\seq{\lambda_n}_n$ with~$\nlim \lambda_n=\lambda_*$, in which case we write~$\Gph\text{-}\lim_{\lambda\to\lambda_*} A_{\lambda} =A_{\lambda_*}$.
\end{definition}

We refer the reader to~\cite{AdlAttRoc23} for the account of some properties of the graph convergence of operators, especially in connection with maximal monotonicity.

\subsubsection{Equivalence of generators and resolvents}
Let~$A$ be a maximal monotone operator.
By~\cite[Prop.~2.2, p.~23]{Bre73}, for each~$g\in H$ and~$\lambda>0$ the problem
\begin{equation}\label{eq:YosidaRegularization}
f+\lambda A(f) \ni g
\end{equation}
has a unique solution~$f\in \dom{A}$. (In fact, this is a characterization of maximality for monotone operators by Minty's Theorem.)

\begin{definition}[Resolvent of an operator]
The \emph{non-linear resolvent} of a maximal monotone operator~$A$ is the operator~$J_\lambda\colon H \to \dom{A}$ defined by~$J_\lambda(g)=f$, where~$f$ is the unique solution to~\eqref{eq:YosidaRegularization}.
\end{definition}

\begin{definition}[Yosida regularization]\label{d:YosidaRegGenerator}
The \emph{Yosida} ($\lambda$-)\emph{regularization} of a maximal monotone operator~$A$ with non-linear resolvent~$J_\bullet$ is the operator~$A_\lambda\colon H \to H$ defined by
\begin{equation}\label{eq:Alambda}
A_\lambda\eqdef \tfrac{1}{\lambda}(\car_H-J_\lambda)\colon H\longrightarrow H\comma \qquad \lambda>0\fstop
\end{equation}
\end{definition}

\begin{lemma}\label{l:Alambda}
Let~$J_\bullet\colon H\to H$ be a strongly continuous non-expansive non-linear resolvent, and let~$A_\bullet$ be defined as in~\eqref{eq:Alambda}.
Then, the set
\[
\Gamma\eqdef\set{\tparen{J_\lambda(f),A_\lambda(f)} : f\in H}\subset H^\tym{2}
\]
does not depend on~$\lambda>0$, and it is the graph of a maximal monotone operator~$A$, satisfying~$J_\lambda=(\car_H+\lambda A)^{-1}$ for every~$\lambda>0$.
Furthermore,
\begin{enumerate}[$(i)$]
\item\label{i:l:Alambda:1} $(A_\lambda,H)$ is a single-valued maximal monotone operator, Lipschitz-continuous with Lipschitz constant~$1/\lambda$;
\item\label{i:l:Alambda:2} $A_{\lambda+\mu}=(A_\lambda)_\mu$ for all $\lambda, \mu>0$;
\item\label{i:l:Alambda:3} $J_{\lambda+\mu} (f) = J_\lambda \tparen{J^\lambda_\mu (f)}$ for all~$\lambda,\mu>0$ and~$f\in H$, where~$J^\lambda_\mu\eqdef(\car_H+\mu A_\lambda)^{-1}$ is the resolvent of~$A_\lambda$;
\item\label{i:l:Alambda:4} for all~$\lambda,\mu>0$,
\[
(\car_H+\mu A_\lambda)^{-1}=\tfrac{\lambda}{\lambda+\mu}\car_H+\tfrac{\mu}{\lambda+\mu}J_{\lambda+\mu}\semicolon
\]
furthermore, $\lambda\mapsto J_\lambda(f)$ is continuous on~$(0,\infty)$ for every~$f\in H$.
\end{enumerate}

\begin{proof}
For~$\lambda>0$ let~$\Gamma^{\sym{\lambda}}\eqdef\set{\tparen{J_\lambda(f),A_\lambda(f)} : f\in H}$.
Fix~$\lambda,\mu>0$ and~$f\in H$, and set~$h\eqdef J_\lambda(f)$, $a\eqdef A_\lambda(f)$, and~$g\eqdef h+\mu a=\tfrac{\mu}{\lambda}f+(1-\tfrac{\mu}{\lambda})J_\lambda(f)$.
By~\ref{i:d:Resolvent:2}, $J_\mu(g)=J_\lambda(f)=h$, whence~$A_\mu(g)=\tfrac{1}{\mu}(g-h)=a$, i.e.~$(h,a)\in \Gamma^{\sym{\mu}}$.
Thus~$\Gamma^{\sym{\lambda}} \subset \Gamma^{\sym{\mu}}$, and the reverse inclusion follows exchanging~$\lambda$ and~$\mu$.
We may thus set~$\Gamma\eqdef \Gamma^{\sym{\lambda}}$ for any~$\lambda>0$.

Now, define~$A\colon H \rightrightarrows H$ as the correspondence with graph~$\Gamma$. 
For every~$(h,a)\in \Gamma$ and every~$\lambda>0$ there exists~$f\in H$ with~$h=J_\lambda(f)$ and~$a=\tfrac{1}{\lambda}(f-h)$, i.e.\ with~$f=h+\lambda a$. Thus,
\begin{equation}\label{eq:l:Alambda:1}
h=J_\lambda(h+\lambda a)\comma \qquad (h,a)\in \Gamma\comma\quad \lambda>0\fstop
\end{equation}
As a consequence,~$\car_H+\lambda A$ is surjective, since~$f=J_\lambda(f)+\lambda A_\lambda(f)$ with~$\tparen{J_\lambda(f),A_\lambda(f)}\in \Gamma$, and $(\car_H+\lambda A)^{-1}=J_\lambda$ by~\eqref{eq:l:Alambda:1}.
Furthermore, for~$(h_i,a_i)\in A$, with~$i=1,2$, and~$\mu>0$, by~\eqref{eq:l:Alambda:1} with~$\mu$ in place of~$\lambda$ and~\ref{i:d:Resolvent:4}, we have~$\norm{h_1-h_2}^2\leq \norm{h_1-h_2+\mu(a_1-a_2)}^2$, i.e.
\[
0\leq 2\mu\scalar{h_1-h_2}{a_1-a_2}+\mu^2\norm{a_1-a_2}^2 \fstop
\]
Dividing by~$\mu$ and letting~$\mu\downarrow 0$ shows that~$A$ is monotone, hence maximal monotone by Minty's Theorem, since~$\car_H+\lambda A$ is surjective.

\paragraph{Proof of~\ref{i:l:Alambda:1}}
For~$f,g\in H$, since~$A_\lambda(f)\in A\tparen{J_\lambda(f)}$, $A_\lambda(g)\in A\tparen{J_\lambda(g)}$, and~$A$ is monotone,
\[
\scalar{A_\lambda(f)-A_\lambda(g)}{f-g} = \scalar{A_\lambda(f)-A_\lambda(g)}{J_\lambda(f)-J_\lambda(g)} + \lambda\norm{A_\lambda(f)-A_\lambda(g)}^2 \geq \lambda\norm{A_\lambda(f)-A_\lambda(g)}^2 \fstop
\]
Thus,~$A_\lambda$ is monotone and, by the Cauchy--Schwarz inequality, $1/\lambda$-Lipschitz; as an everywhere defined continuous monotone operator, it is maximal monotone by~\cite[Prop.~2.4, p.~26]{Bre73}.

\paragraph{Proof of~\ref{i:l:Alambda:2} and~\ref{i:l:Alambda:3}}
Fix~$\lambda,\mu>0$ and~$f\in H$, and let~$h\eqdef J^\lambda_\mu(f)$, i.e.~$f=h+\mu A_\lambda(h)$, which is well-defined by~\ref{i:l:Alambda:1}.
Set~$g\eqdef J_\lambda(h)$ and~$a\eqdef A_\lambda(h)$, so that~$(g,a)\in A$ and~$h=g+\lambda a$.
Then~$f=g+(\lambda+\mu)a$, and therefore~$g=J_{\lambda+\mu}(f)$ by the first assertion, which is~\ref{i:l:Alambda:3}.
Furthermore,~$A_{\lambda+\mu}(f)=\tfrac{1}{\lambda+\mu}(f-g)=a=A_\lambda(h)=\tfrac{1}{\mu}\tparen{f-J^\lambda_\mu(f)}=(A_\lambda)_\mu(f)$, which is~\ref{i:l:Alambda:2}.

\paragraph{Proof of~\ref{i:l:Alambda:4}}
By~\ref{i:l:Alambda:2}, $(\car_H+\mu A_\lambda)^{-1}=\car_H-\mu(A_\lambda)_\mu=\car_H-\mu A_{\lambda+\mu}$, which is the asserted formula.
For~$\lambda,\mu>0$ and~$f\in H$, by~\ref{i:d:Resolvent:2} and~\ref{i:d:Resolvent:4}, letting~$\delta\eqdef\abs{1-\mu/\lambda}$,
\begin{align*}
\norm{J_\lambda(f)-J_\mu(f)} &=\tnorm{J_\mu\tparen{\tfrac{\mu}{\lambda}f+(1-\tfrac{\mu}{\lambda})J_\lambda(f)}-J_\mu(f)}
\leq \delta\,\norm{f-J_\lambda(f)}
\\
&\leq \delta\tparen{\norm{f-J_\mu(f)}+\norm{J_\mu(f)-J_\lambda(f)}}\fstop
\end{align*}
For~$\lambda$ close to~$\mu$ we have~$\delta<1$, whence~$\norm{J_\lambda(f)-J_\mu(f)}\leq\tfrac{\delta}{1-\delta}\norm{f-J_\mu(f)}\to 0$ as~$\lambda\to\mu$.
\end{proof}
\end{lemma}

\begin{corollary}\label{c:AdlAttRoc}
Let~$A$ be a maximal monotone operator. Then,
\[
A=\Gph\text{-}\lim_{\lambda\to 0} A_\lambda \fstop
\]
\begin{proof}
By Lemma~\ref{l:Alambda}\ref{i:l:Alambda:4},~$(\car_H+\mu A_\lambda)^{-1}(f)$ converges to~$J_\mu(f)$ strongly in~$H$ as~$\lambda\downarrow 0$ for every~$\mu>0$ and every~$f\in H$.
Now, it suffices to apply~\cite[Prop.~2.3(ii)$\implies$(i)]{AdlAttRoc23}.
\end{proof}
\end{corollary}

We are now ready to prove the equivalence between maximal monotone operators and strongly continuous non-expansive non-linear resolvents.

\begin{theorem}[Generators and resolvents]\label{t:ResolventGenerator}
Let~$C$ be a non-empty closed convex subset of~$H$, and let $J_\bullet\eqdef\seq{J_\lambda}_{\lambda\geq 0}$ be a family of maps~$J_\lambda\colon H\to H$.
Then, the following are equivalent:
\begin{enumerate}[$(i)$]
\item\label{i:t:ResolventGenerator:1} $J_\bullet$ is a strongly continuous non-expansive non-linear resolvent on~$C$;
\item\label{i:t:ResolventGenerator:2} $J_0=\cproj_C$, and the graph limit
\[
A\eqdef \Gph\text{-}\lim_{\lambda\downarrow 0} \tfrac{1}{\lambda}(\car_H-J_{\lambda})
\]
exists and is a maximal monotone operator satisfying~$\overline{\dom{A}}=C$ and~$J_\lambda=(\car_H +\lambda A)^{-1}$ for every~$\lambda>0$.
\end{enumerate}
\end{theorem}

\begin{remark}[Comparison with~{\cite[Thm.~3.62, p.~365]{Att84}}]
Theorem~\ref{t:ResolventGenerator} ought to be understood as a \emph{pre-compactness} result for the family of operators~$\seq{A_\lambda}_{\lambda>0}$ w.r.t.\ the graph convergence.
As such, it does not immediately follow from the characterization of \emph{continuity} w.r.t.\ the graph convergence in~\cite[Thm.~3.62, p.~365]{Att84}, since we do not assume the operator~$A$ to be given \emph{a priori}.
\end{remark}

\begin{proof}[Proof of Theorem~\ref{t:ResolventGenerator}]
\paragraph{\ref{i:t:ResolventGenerator:1}$\implies$\ref{i:t:ResolventGenerator:2}}
By Lemma~\ref{l:Alambda} there exists a maximal monotone operator~$A$, with $J_\lambda=(\car_H+\lambda A)^{-1}$. By Corollary~\ref{c:AdlAttRoc}, $A=\Gph\text{-}\lim_{\lambda\downarrow 0} A_\lambda$.
Since~$J_0=\cproj_C$ by~\ref{i:d:Resolvent:1}, it remains to show that~$\overline{\dom{A}}=C$.
On the one hand,~$\lim_{\lambda\downarrow 0}J_\lambda(f)=\cproj_C(f)$ for every~$f\in H$ by~\ref{i:d:Resolvent:3}.
On the other hand, since~$A$ is maximal monotone,~$\overline{\dom{A}}$ is convex and~$\lim_{\lambda\downarrow 0}J_\lambda(f)=\cproj_{\overline{\dom{A}}}(f)$ for every~$f\in H$ by~\cite[Thm.~2.2, p.~27]{Bre73}.
Thus,~$\cproj_C=\cproj_{\overline{\dom{A}}}$ on~$H$, and therefore~$C=\cproj_C(H)=\cproj_{\overline{\dom{A}}}(H)=\overline{\dom{A}}$.

\paragraph{\ref{i:t:ResolventGenerator:2}$\implies$\ref{i:t:ResolventGenerator:1}}
The identity property~\ref{i:d:Resolvent:1} holds by assumption.
For~$\mu>0$, the first resolvent identity~\ref{i:d:Resolvent:2} is~\cite[\S2.4, p.~27]{Bre73}; for~$\mu=0$ it reads~$J_\lambda=\cproj_C\circ J_\lambda$, which holds since~$J_\lambda(H)\subset\dom{A}\subset\overline{\dom{A}}=C$.
Since~$A$ is maximal monotone, we have~$\lim_{\lambda\downarrow 0}J_\lambda(f)=\cproj_{\overline{\dom{A}}}(f)$ for every~$f\in H$ by~\cite[Thm.~2.2, p.~27]{Bre73}, and this is the strong continuity~\ref{i:d:Resolvent:3}, since~$\overline{\dom{A}}=C$.
Finally, the non-expansiveness~\ref{i:d:Resolvent:4} follows from~\cite[Prop.~2.2 i)$\implies$iii), p.~23]{Bre73} for~$\lambda>0$, and from the non-expansiveness of convex projections, see e.g.~\cite[\S4.3]{BauCom17}, for~$\lambda=0$.
\end{proof}

\subsubsection{Equivalence of semigroups and generators}
The semigroup~$T_\bullet$ of a maximal monotone operator~$A$ is given by the solution to a differential inclusion, in the following sense.

\begin{theorem}[Kato--K\=omura, Br\'ezis--Pazy, e.g.~{\cite[Thm.~3.1]{Bre73}}]\label{t:BrezisPazy}
Let~$A$ be a maximal monotone operator on~$H$. Then,
\begin{enumerate}[$(i)$]
\item for every~$f\in\dom{A}$ there exists a unique Lipschitz map~$u\colon [0,\infty)\to H$ with~$u(t)\in \dom{A}$ for~$t\geq 0$,~$u(0)=f$ and
\begin{equation}\label{eq:BrezisPazy}
\tfrac{\diff}{\diff t} u(t) \in - A\tparen{u(t)} \quad \forallae{\diff \Leb^1} t>0\semicolon
\end{equation}
\item for every~$t\in [0,\infty)$ the right derivative~$\tfrac{\diff^+}{\diff t} u(t)$ exists and equals~$-A^\circ\tparen{u(t)}$;
\item for every~$t\in (0,\infty)$, the assignment~$f\mapsto u(t)$ defines a non-expansive operator~$T_t\colon \dom{A}\to H$ which extends to the whole of~$\overline{\dom{A}}$ 
and satisfies~$T_t \overline{\dom{A}}\subset \overline{\dom{A}}$;
\item\label{i:t:BrezisPazy:4} the family~$T_\bullet\colon \overline{\dom{A}}\to \overline{\dom{A}}$ is a strongly continuous non-expansive non-linear semigroup.
\end{enumerate}
\end{theorem}

This justifies the following definition.

\begin{definition}[Generator of a semigroup]\label{d:GeneratorSemigroup}
Let~$C$ be a non-empty closed convex subset of~$H$ and~$T_\bullet\colon C\to C$ be a strongly continuous non-expansive non-linear semigroup.
We say that~$A$ is a \emph{generator} of~$T_\bullet$ if~$\overline{\dom{A}}=C$ and for every~$f\in\dom{A}$ the function~$t\mapsto T_t(f)$ is Lipschitz on~$[0,\infty)$ and satisfies~\eqref{eq:BrezisPazy}.
\end{definition}

It is well-known that there is a bijective correspondence between strongly continuous non-expansive non-linear semigroup and maximal monotone operators.
\begin{theorem}[Cf.~{\cite[Thm.~IV.1.2, p.~175]{Bar76}}]\label{t:Barbu}
Let~$C$ be a non-empty closed convex subset of~$H$ and~$T_\bullet\colon C\to C$ be a strongly continuous non-expansive non-linear semigroup.
Then, there exists a unique maximal monotone operator~$A$ such that~$A$ is the generator of~$T_\bullet$.

Vice versa, let~$A$ be a maximal monotone operator on~$H$.
Then, there exists a unique (non-expansive non-linear) semigroup~$T_\bullet$ on~$\overline{\dom{A}}$ such that~$A$ is the generator of~$T_\bullet$.
\end{theorem}

The above correspondence is however rather abstract.
Indeed, the forward implication of Theorem~\ref{t:Barbu} firstly constructs a monotone operator~$A^\circ$ satisfying~$\tfrac{\diff^+}{\diff t} T_t(u) = -A^\circ\tparen{T_t(u)}$ for every~$u$ in the domain of~$A^\circ$ and every~$t\geq 0$, and subsequently defines~$A$ as the maximal monotone extension of (its minimal section)~$A^\circ$. 
In the following, we will provide a constructive realization of the correspondence between semigroups and generators, relying on an approximation procedure in the sense of Painlevé--Kuratowski graph convergence.
Since~$T_t$ is only defined on~$C$, the approximating operators are defined after extending~$T_t$ to the whole of~$H$ by pre-composition with the projection~$\cproj_C$.
Without this extension, they would in general fail to be \emph{maximal} monotone, see Remark~\ref{r:At}.

\begin{proposition}\label{p:At}
Let~$C$ be a non-empty closed convex subset of~$H$, and~$T_\bullet\colon C\to C$ be a strongly continuous non-expansive non-linear semigroup.
For every~$t>0$, define the operator
\begin{equation}\label{eq:ApproxSemigroup}
A^\sym{t}\eqdef \tfrac{1}{t} \tparen{\car_H-T_t\circ\cproj_C} \colon H\longrar H\fstop
\end{equation}
Then, for every~$t>0$,
\begin{enumerate}[$(i)$]
\item\label{i:p:At:1} $(A^\sym{t},H)$ is a single-valued maximal monotone operator, Lipschitz-continuous with Lipschitz constant~$2/t$; 
\item\label{i:p:At:2} 
the resolvent of~$A^\sym{t}$ is the everywhere defined map
\begin{equation}\label{eq:ApproxResolvent}
J^\sym{t}_\lambda= \tparen{(1+\tfrac{\lambda}{t})\car_H-\tfrac{\lambda}{t}\,T_t\circ\cproj_C}^{-1}\colon H\longrar H\comma \qquad \lambda>0\comma
\end{equation}
and it satisfies~$J^\sym{t}_\lambda(C)\subset C$.
\end{enumerate}
If furthermore~$A$ is the maximal monotone generator of~$T_\bullet$ in the sense of Theorem~\ref{t:Barbu} and~$J_\bullet$ is its resolvent, then
\begin{enumerate}[$(i)$, resume]
\item\label{i:p:At:3} $\lim_{t\downarrow 0}A^\sym{t}(f)=A^\circ(f)$ strongly in~$H$ for every~$f\in\dom{A}$;
\item\label{i:p:At:4} $\lim_{t\downarrow 0}J^\sym{t}_\lambda(f)=J_\lambda(f)$ strongly in~$H$ for every~$f\in H$ and every~$\lambda>0$.
\end{enumerate}
\end{proposition}

\begin{proof}
\ref{i:p:At:1}
Since~$T_t$ maps~$C$ into~$C$ and is non-expansive on~$C$, and since~$\cproj_C$ is non-expansive on~$H$, the map~$T_t\circ\cproj_C\colon H\to H$ is non-expansive.
Thus, for all~$f,g \in H$,
\[
\begin{aligned}
\tscalar{A^\sym{t}(f) -  A^\sym{t}(g)}{f-g} &= \tfrac{1}{t} \tparen{\norm{f-g}^2 - \tscalar{T_t\tparen{\cproj_C(f)} - T_t\tparen{\cproj_C(g)}}{f-g}}
\\
&\geq  \tfrac{1}{t} \tparen{ \norm{f-g}^2 - \norm{f-g}\, \tnorm{T_t\tparen{\cproj_C(f)} - T_t\tparen{\cproj_C(g)}} } \geq 0\comma
\end{aligned}
\]
hence~$(A^\sym{t},H)$ is monotone.
Furthermore,
\[
\tnorm{A^\sym{t}(f) -  A^\sym{t}(g)} \leq \tfrac{1}{t} \tparen{ \norm{f -g} +\tnorm{T_t\tparen{\cproj_C(f)} - T_t\tparen{\cproj_C(g)}} } \leq \tfrac{2}{t} \norm{f-g}\comma
\]
hence~$(A^\sym{t}, H)$ is $\tfrac{2}{t}$-Lipschitz.
Since~$A^\sym{t}$ is everywhere defined, monotone, and continuous, it is maximal monotone by~\cite[Prop.~2.4, p.~26]{Bre73}.
Finally,~$A^\sym{t}(f)=\tfrac{1}{t}\tparen{f-T_t(f)}$ for every~$f\in C$, since~$\cproj_C(f)=f$.

\ref{i:p:At:2}
Since~$A^\sym{t}$ is maximal monotone by~\ref{i:p:At:1}, its resolvent~$J^\sym{t}_\lambda$, defined by~\eqref{eq:ApproxResolvent}, is well-defined on all of~$H$ by~\cite[Prop.~2.2, p.~23]{Bre73}.
Fix~$\lambda>0$.
By definition, for $g \in C$, we have that $f = J^\sym{t}_\lambda(g)$ if and only if
\[
\tparen{1+\tfrac{\lambda}{t}} f - \tfrac{\lambda}{t}\, T_t\tparen{\cproj_C(f)} = g \fstop
\]
Set~$h \eqdef T_t\tparen{\cproj_C(f)}$. Since~$T_\bullet\colon C\to C$ we have~$h \in C$, and~$f = \frac{1}{1+\lambda/ t}g + \tfrac{\lambda/t}{1+\lambda/t}\, h$.
Since~$g, h \in C$, and the latter set is convex, then~$f \in C$ which shows that~$J^\sym{t}_\lambda(C)\subset C$.

\ref{i:p:At:3}
By~\ref{i:p:At:1},~$A^\sym{t}(f)=\tfrac{1}{t}\tparen{f-T_t(f)}$ for~$f\in\dom{A}\subset C$, and the assertion is~\cite[Thm.~3.1(5), p.~54]{Bre73}.

\ref{i:p:At:4}
Fix~$g\in H$ and~$\lambda>0$, and set~$f\eqdef J_\lambda(g)\in\dom{A}$ and~$f_t\eqdef J^\sym{t}_\lambda(g)$ for each~$t>0$.
By definition of resolvent operator,
\begin{equation}\label{eq:p:At:5}
g-f_t=\lambda A^\sym{t}(f_t) \fstop
\end{equation}
Thus, by monotonicity~\ref{i:p:At:1} of~$A^\sym{t}$,
\begin{align}
\nonumber
0&\leq \lambda\scalar{A^\sym{t}(f_t) - A^\sym{t}(f)}{f_t-f}=\scalar{g-f_t-\lambda A^\sym{t}(f)}{f_t-f}
\\
\label{eq:p:At:1}
&=\scalar{g-(f_t-f)-f-\lambda A^\sym{t}(f)}{f_t-f} = \scalar{g-f-\lambda A^\sym{t}(f)}{f_t-f}-\norm{f_t-f}^2 \fstop
\end{align}
Applying the Cauchy--Schwarz inequality and reshuffling the terms, we have~$\norm{f_t-f}\leq \norm{g-f-\lambda A^\sym{t}(f)}$.
Since the right-hand side is bounded in~$t$ by~\ref{i:p:At:3}, then~$\seq{f_t}_t$ is bounded, and thus weakly relatively compact.
Let~$f'$ be any accumulation point, and~$\seq{f_t}_t$ be any (non-relabelled) subnet weakly converging to~$f'$.

Combining~\eqref{eq:p:At:5} with the definition of~$A^\sym{t}(f_t)$, we have
\begin{equation}\label{eq:p:At:6}
f_t-T_t\tparen{\cproj_C(f_t)}=\tfrac{t}{\lambda}(g-f_t)\fstop
\end{equation}
Further note that~$\lambda^{-1}(g-f_t)$ is bounded uniformly in~$t$, since so is~$\seq{f_t}_t$, thus~$\lim_{t\downarrow 0}\tfrac{t}{\lambda}\norm{g-f_t}=0$. 
From the identity~\eqref{eq:p:At:6} we conclude that~$\lim_{t\downarrow 0} \tnorm{f_t-T_t\tparen{\cproj_C(f_t)}}=0$, and thus that~$f'\in C$ since~$T_t\tparen{\cproj_C(f_t)}\in C$ and since~$C$ is both weakly and strongly closed.

\paragraph{Claim:~$\seq{f_t}_t$ weakly converges to~$f$ as~$t\to 0$}
Fix~$h\in\dom{A}$.
Again by~\eqref{eq:p:At:5} and by monotonicity~\ref{i:p:At:1} of~$A^\sym{t}$, for every~$t>0$,
\begin{align}
\nonumber
0&\leq \lambda\scalar{A^\sym{t}(f_t) - A^\sym{t}(h)}{f_t-h}=\scalar{g-f_t-\lambda A^\sym{t}(h)}{f_t-h}
\\
\label{eq:p:At:3}
&=\scalar{g-\lambda A^\sym{t}(h)}{f_t-h}-\norm{f_t}^2+\scalar{f_t}{h} \fstop
\end{align}
We pass to the limit along the subnet weakly converging to~$f'$.
By~\ref{i:p:At:3} we have~$\lim_{t\downarrow 0}\tnorm{A^\sym{t}(h)-A^\circ(h)}=0$.
Since the net~$\seq{f_t-h}_t$ is bounded and weakly converges to~$f'-h$, it follows that
\[
\lim_t \scalar{g-\lambda A^\sym{t}(h)}{f_t-h} = \scalar{g-\lambda A^\circ(h)}{f'-h} \qquad \text{and} \qquad \lim_t \scalar{f_t}{h}=\scalar{f'}{h} \fstop
\]
Furthermore,~$\norm{f'}^2\leq \liminf_t \norm{f_t}^2$ by weak lower semicontinuity of the norm.
Thus, taking the limit superior of the right-hand side of~\eqref{eq:p:At:3} along the subnet,
\[
0\leq \scalar{g-\lambda A^\circ(h)}{f'-h}-\norm{f'}^2+\scalar{f'}{h} = \scalar{g-f'-\lambda A^\circ(h)}{f'-h} \fstop
\]
Dividing by~$\lambda>0$ and since~$h\in\dom{A}$ was arbitrary,
\begin{equation}\label{eq:p:At:2}
\scalar{\tfrac{1}{\lambda}(g-f')-A^\circ(h)}{f'-h}\geq 0\comma \qquad h\in\dom{A} \fstop
\end{equation}
Since~$A$ is the generator of~$T_\bullet$, we have~$\overline{\dom{A}}=C$, which is convex by assumption.
Since~$f'\in C$ and~$A^\circ$ is a principal section of~$A$, 
it follows from~\eqref{eq:p:At:2} that~$\tfrac{1}{\lambda}(g-f')\in A(f')$.
That is,~$g\in f'+\lambda A(f')$, whence~$f'=J_\lambda(g)=f$ by uniqueness of the solution to~\eqref{eq:YosidaRegularization}.

We have thus shown that~$f$ is the only weak accumulation point of the net~$\seq{f_t}_t$ as~$t\downarrow 0$, so that the whole net converges to it weakly.

\paragraph{Conclusion}
Finally, set~$h\eqdef\tfrac{1}{\lambda}(g-f)$, and note that~$h\in A(f)$ by definition of~$f=J_\lambda(g)$.
Rearranging~\eqref{eq:p:At:1}, we have
\begin{equation}\label{eq:p:At:4}
\begin{aligned}
\norm{f_t-f}^2&\leq \scalar{g-f-\lambda A^\sym{t}(f)}{f_t-f} = \lambda\scalar{h-A^\sym{t}(f)}{f_t-f}
\\
&= \lambda\scalar{h-A^\circ(f)}{f_t-f}+\lambda\scalar{A^\circ(f)-A^\sym{t}(f)}{f_t-f} \fstop
\end{aligned}
\end{equation}
As~$t\downarrow 0$, the first term on the right-hand side of~\eqref{eq:p:At:4} vanishes since~$f_t$ converges weakly to~$f$.
The second term vanishes as well, since
\[
\abs{\scalar{A^\circ(f)-A^\sym{t}(f)}{f_t-f}}\leq \tnorm{A^\circ(f)-A^\sym{t}(f)}\,\norm{f_t-f} \comma
\]
where the first factor vanishes by~\ref{i:p:At:3} and the second one is bounded in~$t$.
Therefore,
\[
\lim_{t\downarrow 0}\norm{J^\sym{t}_\lambda(g)-J_\lambda(g)}=\lim_{t\downarrow 0}\norm{f_t-f}=0\fstop
\]
Since~$g\in H$ and~$\lambda>0$ were arbitrary, this proves~\ref{i:p:At:4}.
\end{proof}

\begin{remark}\label{r:At}
\begin{enumerate}[$(a)$, wide]
\item\label{i:r:At:1} The pre-composition with~$\cproj_C$ in~\eqref{eq:ApproxSemigroup} is essential whenever~$C\neq H$.
Indeed, the restriction of~$A^\sym{t}$ to~$C$, i.e.\ the operator~$\tfrac{1}{t}(\car_H-T_t)\colon C\to H$, is monotone but in general not maximal monotone.
For instance, let~$H\eqdef\R$,~$C\eqdef[0,1]$, and~$T_t\eqdef\id_C$ for every~$t\geq 0$, so that the generator of~$T_\bullet$ is~$A=\partial\mbfI_C$.
Then~$\tfrac{1}{t}(\car_H-T_t)\equiv 0$ on~$C$.
Its graph~$[0,1]\times\set{0}$, which is also its graph limit as~$t\downarrow 0$, is strictly contained in that of~$A$.
Furthermore,~$\tparen{(1+\tfrac{\lambda}{t})\car_H-\tfrac{\lambda}{t}T_t}^{-1}=\id_C$ is only defined on~$C$, while~$J_\lambda=\cproj_C$.
On the contrary,~$A^\sym{t}\eqdef \tfrac{1}{t}(\car_H-T_t\circ \cproj_C) =\tfrac{1}{t}(\car_H-\cproj_C)$ is the Yosida regularization~$A_t$ of~$A$, which graph-converges to~$A$ as~$t\downarrow 0$ by Corollary~\ref{c:AdlAttRoc}.

\item\label{i:r:At:2} By Proposition~\ref{p:At}\ref{i:p:At:2}, for~$f\in C$ the element~$J^\sym{t}_\lambda(f)$ is the unique~$h\in C$ satisfying~$(1+\tfrac{\lambda}{t})h-\tfrac{\lambda}{t}T_t(h)=f$.
Thus, the pre-composition with~$\cproj_C$ only affects the values of~$J^\sym{t}_\lambda$ outside of~$C$.
\end{enumerate}
\end{remark}

\begin{theorem}[Semigroups and generators]\label{t:SemigroupGenerator}
Let~$C$ be a non-empty closed convex subset of~$H$, and $T_\bullet\colon C\to C$ be a strongly continuous non-expansive non-linear semigroup with maximal monotone generator~$A$ in the sense of Theorem~\ref{t:Barbu}.
Then,
\begin{equation}\label{eq:t:SemigroupGenerator:0}
A= \Gph\text{-}\lim_{t\downarrow 0} \tfrac{1}{t}\tparen{\car_H-T_t\circ\cproj_C} \fstop
\end{equation}
In particular,~$A$ is constructively determined by~$T_\bullet$.
Vice versa, every maximal monotone operator~$A$ is the generator of the semigroup~$T_\bullet$ on~$C\eqdef\overline{\dom{A}}$ given by Theorem~\ref{t:BrezisPazy}, and~\eqref{eq:t:SemigroupGenerator:0} holds.

\begin{proof}
Let~$J_\bullet$ be the resolvent of~$A$, and~$A^\sym{t}$ be defined as in~\eqref{eq:ApproxSemigroup}, with resolvent~$J^\sym{t}_\bullet$.
By Proposition~\ref{p:At}\ref{i:p:At:1}, each~$A^\sym{t}$ is maximal monotone, and by Proposition~\ref{p:At}\ref{i:p:At:4},~$\lim_{t\downarrow 0}J^\sym{t}_\lambda(f)=J_\lambda(f)$ strongly in~$H$ for every~$f\in H$ and every~$\lambda>0$.
Fix a sequence~$t_n\downarrow 0$.
By~\cite[Prop.~2.3(ii)$\implies$(i)]{AdlAttRoc23} with~$A^n\eqdef A^\sym{t_n}$, the sequence~$\seq{A^\sym{t_n}}_n$ graph-converges to the maximal monotone operator with resolvent~$J_\bullet$, that is, to~$A$.
Since the sequence~$\seq{t_n}_n$ was arbitrary, this proves~\eqref{eq:t:SemigroupGenerator:0}.
The second assertion is Theorem~\ref{t:BrezisPazy} together with the first one.
\end{proof}
\end{theorem}

\subsubsection{Equivalence of semigroups and resolvents}
We conclude this section with the equivalence of semigroups and resolvents.

\begin{theorem}[Semigroups and resolvents]\label{t:SemigroupResolvent}
Let~$C$ be a non-empty closed convex subset of~$H$.
\begin{enumerate}[$(i)$, wide]
\item\label{i:t:SemigroupResolvent:1} Let~$T_\bullet\colon C\to C$ be a strongly continuous non-expansive non-linear semigroup. Then, setting
\begin{equation}\label{eq:t:SemigroupResolvent:1}
\begin{aligned}
J_\lambda(f)\eqdef& \lim_{t\downarrow 0} \tparen{(1+\tfrac{\lambda}{t})\car_H -\tfrac{\lambda}{t}\, T_t\circ\cproj_C }^{-1}(f) \quad \text{for all }\lambda> 0\comma
\\
J_0(f)\eqdef& \cproj_C(f)\comma
\end{aligned}
\qquad f\in H\comma
\end{equation}
defines~$J_\bullet\colon H\to H$ a strongly continuous non-expansive non-linear resolvent on~$C$, namely the resolvent of the generator of~$T_\bullet$.
For~$f\in C$, the inverse in~\eqref{eq:t:SemigroupResolvent:1} may be equivalently computed with~$T_t$ in place of~$T_t\circ\cproj_C$.

\item\label{i:t:SemigroupResolvent:2} Vice versa, let~$J_\bullet\colon H\to H$ be a strongly continuous non-expansive non-linear resolvent on~$C$. Then, setting
\begin{equation}\label{eq:p:SemigroupResolvent:2}
T_t(f)\eqdef \lim_{n\to\infty} (J_{t/n})^n(f) \comma \qquad f\in C\comma  t>0\comma \qquad \text{and} \qquad T_0\eqdef \id_C
\end{equation}
defines a strongly continuous non-expansive non-linear semigroup~$T_\bullet\colon C\to C$ satisfying~\eqref{eq:t:SemigroupResolvent:1}.
\end{enumerate}
\begin{proof}
\ref{i:t:SemigroupResolvent:1}
Let~$A$ be the generator of~$T_\bullet$ in the sense of Theorem~\ref{t:Barbu}, so that~$C=\overline{\dom{A}}$, and let~$A^\sym{t}$ be defined as in~\eqref{eq:ApproxSemigroup}.
By Proposition~\ref{p:At}\ref{i:p:At:2}, the inverse in~\eqref{eq:t:SemigroupResolvent:1} is the everywhere defined resolvent~$J^\sym{t}_\lambda$ of~$A^\sym{t}$.
By Proposition~\ref{p:At}\ref{i:p:At:4}, the limit in~\eqref{eq:t:SemigroupResolvent:1} exists for every~$f\in H$ and coincides with~$(\car_H+\lambda A)^{-1}(f)$.
Thus, when completed by~$J_0=\cproj_C=\cproj_{\overline{\dom{A}}}$, the family~$J_\bullet$ is the resolvent of the maximal monotone operator~$A$.
As such, it is a strongly continuous non-expansive non-linear resolvent on~$C$ by Theorem~\ref{t:ResolventGenerator}.
The last assertion follows from Proposition~\ref{p:At}\ref{i:p:At:2}, cf.\ Remark~\ref{r:At}\ref{i:r:At:2}, letting~$t\to 0$.

\ref{i:t:SemigroupResolvent:2} 
Let~$A$ be the generator of~$J_\bullet$, and~$T_\bullet$ be the semigroup generated by~$A$.
The conclusion follows from~\cite[Cor.~4.4, p.~126]{Bre73} since~$(J_{t/n})^n = \tparen{\car_H+\tfrac{t}{n}A}^{-n}$ on~$C=\overline{\dom{A}}$.
\end{proof}
\end{theorem}

\begin{remark}
The implication~\ref{i:t:SemigroupResolvent:1} appeared with a different proof in~\cite[Lem.~18.9, p.~242]{BenCraPaz91} (unpublished).
\end{remark}

\begin{figure}[htb!]
\begin{adjustbox}{center}
{\small
\begin{tikzcd}[row sep=5cm, column sep=5cm]
\parbox{5cm}{\centering $T_\bullet$ strongly continuous\\non-expansive\\non-linear semigroup}
\arrow[r, shift left=2, "\parbox{5cm}{\centering Thm.~\ref{t:SemigroupResolvent}\ref{i:t:SemigroupResolvent:1}\\$\displaystyle{J_\lambda\eqdef \lim_{t\downarrow 0} \tparen{(1+\tfrac{\lambda}{t})\car_H - \tfrac{\lambda}{t} T_t\circ\cproj_C}^{-1}}$}"]
\arrow[d, shift left=2, "\parbox{4cm}{\centering Thm.~\ref{t:SemigroupGenerator}\\$\displaystyle{A=\Gph\text{-}\lim_{t \downarrow 0} \tfrac{1}{t}(\car_H-T_t\circ\cproj_C)}$}",sloped]
& \parbox{5cm}{\centering $J_\bullet$ strongly continuous\\non-expansive\\non-linear resolvent}
\arrow[l, shift left=2, "\parbox{5cm}{\centering Thm.~\ref{t:SemigroupResolvent}\ref{i:t:SemigroupResolvent:2}\\$\displaystyle{T_t\eqdef \lim_{n\to\infty} (J_{t/n})^n}$}"] \arrow[dl,shorten >=15, sloped, yshift=-2ex, labels=below, "\parbox{3cm}{\centering Thm.~\ref{t:ResolventGenerator}\\$\displaystyle{A=\Gph\text{-}\lim_{\lambda\downarrow 0} \tfrac{1}{\lambda}(\car_H-J_\lambda)}$}"]
\\
\parbox{5cm}{\centering $A$ maximal\\monotone operator} 
\arrow[u,shift left=2,"\parbox{5cm}{\centering Br\'ezis--Pazy Thm.~\ref{t:BrezisPazy}\\$\tfrac{\diff}{\diff t}T_t (f)\in -A(T_t(f))$}",sloped]
\arrow{ur}[shift left=2, sloped]{\parbox{5cm}{\centering Thm.~\ref{t:ResolventGenerator}\\$J_\lambda\eqdef (\car +\lambda A)^{-1}$}}
\end{tikzcd}
}
\end{adjustbox}
\caption{Equivalences between semigroups, generators, and resolvents: \emph{maximal monotonicity}.}\label{fig:Diagram2}
\label{f:Diagram1}
\end{figure}
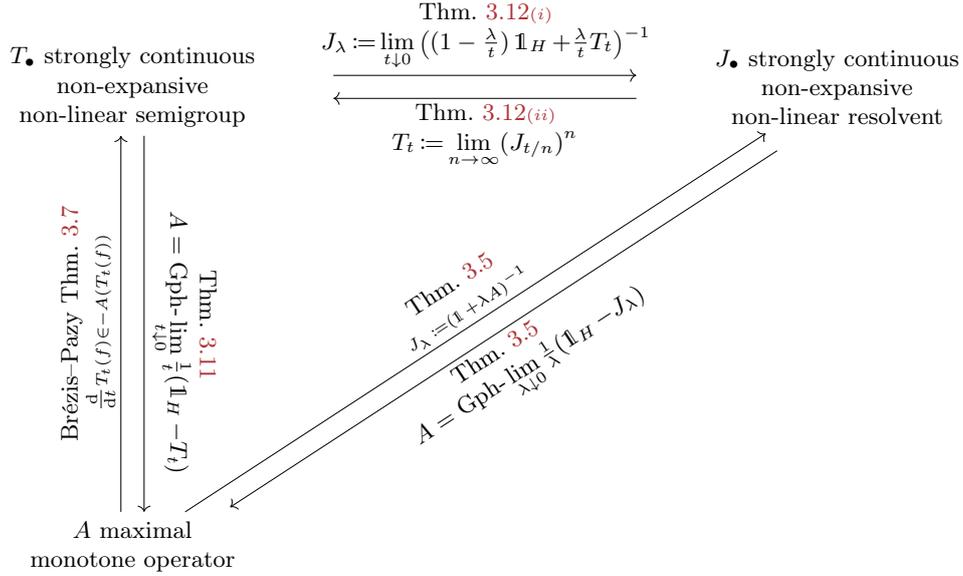

\subsection{Equivalences with energies: cyclical monotonicity}
The relation between generators and convex functionals relies on the additional property of \emph{cyclical monotonicity} of the generators.
Here, we recall the standard characterization of cyclically monotone maximal monotone operators as subdifferentials of convex lower semicontinuous functionals, and we further characterize equivalent notions for semigroups and resolvents.
In order to do so, we shall need the following definition.

\begin{definition}[Cyclical non-expansiveness]\label{d:CyclicalNonExp}
Let~$C$ be a non-empty closed convex subset of~$H$.
An operator~$B\colon C\to H$ is \emph{cyclically non-expansive} if for every~$n\in\N$ and every~$h_1,\dotsc, h_n\in C$,
\[
\sum_{i=1}^n \norm{B(h_i)-B(h_{i-1})}^2 \leq \sum_{i=1}^n \norm{h_i-B(h_{i-1})}^2\comma \qquad h_0\eqdef h_n \fstop
\]
\end{definition}

\subsubsection{Equivalence of generators and energies}
Let us now turn to the relation between energies and generators.

\begin{theorem}[Generators and energies, e.g.~\cite{Bre73}]\label{t:MaxMonSubDiff}
The following assertions hold:

\begin{enumerate}[$(i)$, wide]
\item\label{i:t:MaxMonSubDiff:1} If~$E\colon H\to (-\infty,\infty]$ is convex and proper, then~$\partial E$ is a monotone operator, e.g.~\cite[Ex.~2.1.4]{Bre73}, additionally maximal if~$E\colon H\to (-\infty,\infty]$ is additionally lower semicontinuous,~\cite[Ex.~2.3.4]{Bre73};

\item\label{i:t:MaxMonSubDiff:2} $A$ is cyclically monotone if and only if there exists a proper convex lower semicontinuous~$E\colon H\to (-\infty,\infty]$ such that~$A\subset \partial E$ with equality~$A=\partial E$ if and only if~$A$ is additionally maximal monotone,~\cite[Thm.~2.5]{Bre73}.
In both cases, we can choose~$E=\cenv(A)$;

\item\label{i:t:MaxMonSubDiff:3} a cyclically monotone, maximal monotone~$A$ as in~\ref{i:t:MaxMonSubDiff:2} is densely defined if and only if~$E$ in~\ref{i:t:MaxMonSubDiff:2} is additionally densely defined,~\cite[Prop.~2.11]{Bre73}.
\end{enumerate}
\end{theorem}

\begin{remark}\label{r:AdditiveConstant}
We note that, given a cyclically monotone maximal monotone operator~$A$, the proper convex lower semicontinuous functional~$E$ given by Theorem~\ref{t:MaxMonSubDiff}\ref{i:t:MaxMonSubDiff:2} is uniquely determined only up to an additive constant.
In particular, if~$E$ is bounded below (e.g., if it is even, cf.~Rmk~\ref{r:Even}), then there is no substantial difference in assuming that~$E$ is non-negative.
\end{remark}

\begin{remark}[Linear case]
If a linearly defined proper convex lower semicontinuous functional~$E\colon H\to [0,\infty]$ satisfies the parallelogram identity~\ref{i:d:Form:6}, then the bivariate functional~$\mcE\colon \dom{E}^\tym{2} \to \R$ defined as
\[
\dom{\mcE}\eqdef \dom{E}\comma \qquad \mcE(u,v)\eqdef \tfrac{1}{4}\tparen{E(u+v)-E(u-v)}\comma
\]
is a proper closed symmetric bilinear form on~$H$.
Vice versa, if~$\mcE$ is a proper closed symmetric bilinear form on~$H$, then the quadratic functional~$E\colon H\to [0,\infty]$ defined as
\[
\dom{E}\eqdef\dom{\mcE}\comma \qquad E(u)\eqdef \mcE(u,u)\comma
\]
is a proper linearly defined lower semicontinuous functional satisfying the parallelogram identity~\ref{i:d:Form:6}, i.e.\ it is \emph{quadratic} (in particular: convex and $2$-homogeneous).
\end{remark}

\begin{definition}[Yosida regularization]\label{d:YosidaRegForm}
Let~$E\colon H \to (-\infty,\infty]$ be convex and proper.
We call \emph{Yosida} ($\lambda$-)\emph{regularization} of~$E$ the functional~$E_\lambda\colon H\to [-\infty,\infty]$ defined as
\begin{equation}\label{eq:YosidaRegE}
E_\lambda \eqdef \inf_{g\in H} \braket{\tfrac{1}{2\lambda} \norm{\emparg-g}^2 + E(g)}\comma \qquad \lambda>0 \fstop
\end{equation}
\end{definition}

We note that
\begin{equation}\label{eq:MonotonicityYosidaRegE}
\text{$\lambda\mapsto E_\lambda(f)$ is monotone non-increasing for every~$f\in H$.}
\end{equation}
Indeed, if~$\lambda\leq\mu$, then~$\tfrac{1}{2\lambda}\norm{f-g}^2+ E(g)\geq \tfrac{1}{2\mu}\norm{f-g}^2+ E(g)$ for every~$g\in H$, and taking the infimum over~$g$ yields~$E_\lambda(f)\geq E_\mu(f)$.

Furthermore, when~$E$ is additionally lower semicontinuous, then the infimum in~\eqref{eq:YosidaRegE} is attained, and~$E_\lambda$ is $\R$-valued on the whole of~$H$.

\medskip

Let~$A$ be a cyclically monotone maximal monotone operator such that~$A=\partial E$. For every~$\lambda>0$, the Yosida regularization~$A_\lambda$ of~$A$ and the Yosida regularization~$E_\lambda$ of~$E$ are related by, see~\cite[Prop.~2.11]{Bre73},
\begin{equation}\label{eq:YosidaForm}
A_\lambda = \partial E_\lambda \comma \qquad \lambda>0 \fstop
\end{equation}
(In particular,~$E_\lambda$ is $C^{1,1}$-Fr\'echet differentiable and~$A_\lambda=\nabla E_\lambda$ is the Fr\'echet differential of $E_\lambda$.)

\subsubsection{The Evolution Variational Inequality}\label{sss:EVI}
Here, we recall the \emph{Evolution Variational Inequality} (EVI) of Ambrosio, Gigli, and Savar\'e~\cite[Thm.~4.0.4]{AmbGigSav08}, in the formulation of Muratori and Savar\'e~\cite[\S3]{MurSav20}, and we compare it with a fixed-time inequality for single maps, which will be instrumental in characterizing cyclical non-expansiveness.
Whenever referring to~\cite{MurSav20}, we consider the metric-functional system~$(X,\mssd,\phi)\eqdef\tparen{H,\dist_H,G}$ in the sense of~\cite[(3.1)]{MurSav20}, with~$G\colon H\to(-\infty,\infty]$ proper and lower semicontinuous, and we only consider the case~$\lambda=0$.

\paragraph{A fixed-time variational inequality}
For convenience of exposition, we introduce the following inequality, modelled after~\cite[$(\mathrm{EVI}''_0)$]{MurSav20}.

\begin{definition}[Variational inequality at time~$t$]\label{d:VI}
Let~$C$ be a non-empty closed convex subset of~$H$, let~$G\colon H\to (-\infty,\infty]$ be proper, let~$B\colon C\to H$, and fix~$t>0$.
We say that~$B$ satisfies the  \emph{variational inequality at time~$t$ with respect to~$G$}, in short that~$B$ satisfies~\emph{$(\mathrm{VI}_t)$ with respect to~$G$}, if~$B(C)\subset\dom{G}$ and
\begin{equation}\tag{$\mathrm{VI}_t$}\label{eq:VI}
\tfrac{1}{2}\norm{B(f)-g}^2-\tfrac{1}{2}\norm{f-g}^2 \leq t\,\tparen{G(g)-G\tparen{B(f)}} \comma \qquad f\in C\comma \quad g\in H \fstop
\end{equation}
\end{definition}

Since~$B(C)\subset\dom{G}$, the right-hand side of~\eqref{eq:VI} is well-defined in~$(-\infty,\infty]$, and~\eqref{eq:VI} is trivially satisfied if~$g\notin\dom{G}$.

\begin{lemma}[Properties of~$(\mathrm{VI}_t)$]\label{l:VI}
Let~$C$ be a non-empty closed convex subset of~$H$, and~$G\colon H\to(-\infty,\infty]$ be proper.
\begin{enumerate}[$(i)$]
\item\label{i:l:VI:1} If~$B\colon C\to H$ satisfies~$(\mathrm{VI}_t)$ with respect to~$G$ for some~$t>0$, then
\[
G\tparen{B(f)}\leq G(f)-\tfrac{1}{2t}\norm{B(f)-f}^2 \comma \qquad f\in C \fstop
\]
\item\label{i:l:VI:2} If~$B\colon C\to H$ satisfies~$(\mathrm{VI}_t)$ with respect to~$G$ for some~$t>0$, then~$B$ is cyclically non-expansive.
\item\label{i:l:VI:3} Let~$\seq{T_t}_{t\geq 0}$ be a family of maps~$T_t\colon C\to C$ satisfying the semigroup property~\ref{i:d:Semigroup:2}.
If~$T_\tau$ satisfies~$(\mathrm{VI}_\tau)$ with respect to~$G$ for some~$\tau>0$, then~$T_{m\tau}$ satisfies~$(\mathrm{VI}_{m\tau})$ with respect to~$G$ for every~$m\in\N$.
\end{enumerate}
\end{lemma}

\begin{proof}
\ref{i:l:VI:1} is~\eqref{eq:VI} with~$g=f$, and it is trivially satisfied if~$f\notin\dom{G}$.

\ref{i:l:VI:2} Let~$h_1,\dotsc,h_n\in C$, and set~$h_0\eqdef h_n$.
Since~$B(C)\subset\dom{G}$, the values~$G\tparen{B(h_i)}$ are finite for every~$i\leq n$.
Applying~\eqref{eq:VI} with~$h_i$ in place of~$f$ and~$B(h_{i-1})$ in place of~$g$,
\[
\tfrac{1}{2}\norm{B(h_i)-B(h_{i-1})}^2-\tfrac{1}{2}\norm{h_i-B(h_{i-1})}^2 \leq t\,\tparen{G\tparen{B(h_{i-1})}-G\tparen{B(h_i)}} \comma \qquad i\leq n \fstop
\]
Since~$h_0=h_n$, summing over~$i\leq n$ the right-hand side telescopes to~$0$, which shows that~$B$ is cyclically non-expansive in the sense of Definition~\ref{d:CyclicalNonExp}.

\ref{i:l:VI:3} Fix~$m\in\N$.
By~\ref{i:d:Semigroup:2},~$T_{m\tau}(C)=T_\tau\tparen{T_{(m-1)\tau}(C)}\subset T_\tau(C)\subset\dom{G}$.
Now, fix~$f\in C$ and~$g\in\dom{G}$, and set~$w_0\eqdef f$ and~$w_j\eqdef T_{j\tau}(f)$ for~$j\in\N$, so that, by~\ref{i:d:Semigroup:2},~$w_{j+1}=T_\tau(w_j)$ for every~$j\in\N_0$, and~$G(w_j)$ is finite for every~$j\geq 1$.
Applying~\eqref{eq:VI} with~$B=T_\tau$,~$t=\tau$, and~$w_j$ in place of~$f$, and summing over~$j<m$,
\[
\tfrac{1}{2}\norm{w_m-g}^2-\tfrac{1}{2}\norm{f-g}^2 \leq m\tau\, G(g)-\tau\sum_{j=1}^m G(w_j) \leq m\tau\,\tparen{G(g)-G(w_m)} \comma
\]
where the last inequality holds since~$G(w_j)\geq G(w_{j+1})$ for every~$j\geq 1$ by~\ref{i:l:VI:1}.
\end{proof}

\paragraph{The EVI after~\cite{AmbGigSav08,MurSav20}}
We recall the notion of EVI-gradient flow in~\cite{MurSav20}, which we only need for~$\lambda=0$.

\begin{definition}[{EVI-gradient flows, \cite[Dfn.~3.1]{MurSav20}}]\label{d:EVIMurSav}
Let~$G\colon H\to(-\infty,\infty]$ be proper and lower semicontinuous.
\begin{enumerate}[$(a)$]
\item\label{i:d:EVIMurSav:1} A \emph{solution of the Evolution Variational Inequality}~$\mathrm{EVI}_0\tparen{H,\dist_H,G}$ is a continuous curve~$u\colon t\in(0,\infty)\mapsto u_t\in\dom{G}$ such that
\begin{equation}\tag{$\mathrm{EVI}_0$}\label{eq:EVI0}
\tfrac{1}{2}\,\tfrac{\diff^+}{\diff t}\norm{u_t-g}^2\leq G(g)-G(u_t)\comma \qquad t\in(0,\infty)\comma \quad g\in\dom{G}\comma
\end{equation}
where~$\tfrac{\diff^+}{\diff t}\zeta(t)\eqdef \limsup_{h\downarrow 0}\tfrac{1}{h}\tparen{\zeta(t+h)-\zeta(t)}$ denotes the upper right Dini derivative.

\item\label{i:d:EVIMurSav:2} Let~$D\subset\overline{\dom{G}}$.
An \emph{$\mathrm{EVI}_0$-gradient flow of~$G$ in~$D$} is a family of continuous maps~$S_t\colon D\to D$,~$t\geq 0$, such that, for every~$u_0\in D$,
\begin{itemize}
\item $S_{t+h}(u_0)=S_h\tparen{S_t(u_0)}$ for every~$t,h\geq 0$, and~$\lim_{t\downarrow 0}S_t(u_0)=S_0(u_0)=u_0$, cf.~\cite[(3.2a)]{MurSav20};
\item the curve~$t\mapsto S_t(u_0)$ is a solution of~$\mathrm{EVI}_0\tparen{H,\dist_H,G}$, cf.~\cite[(3.2b)]{MurSav20}.
\end{itemize}
\end{enumerate}
\end{definition}

We shall make use of the following results in~\cite{MurSav20}, which we state in the case~$\lambda=0$, so that the function~$\mathsf{E}_\lambda$ in~\cite[(3.3)]{MurSav20} is~$\mathsf{E}_0(r)=r$, and all exponential factors in~\cite[Thms.~3.3 and~3.5]{MurSav20} are equal to~$1$.

\begin{theorem}[Muratori--Savar\'e]\label{t:MurSav}
Let~$G\colon H\to(-\infty,\infty]$ be proper and lower semicontinuous.
\begin{enumerate}[$(i)$, wide]
\item\label{i:t:MurSav:1} \textup{(}Integral characterizations $(\mathrm{IC})$ of the $\mathrm{EVI}$,~\cite[Thm.~3.3, Rmk.~3.4]{MurSav20}\textup{)} Let~$u\colon [0,\infty)\to\overline{\dom{G}}$ be continuous, with~$u_t\in\dom{G}$ for every~$t>0$.
Then, the following are equivalent:
\begin{enumerate}[$(a)$]
\item\label{i:t:MurSav:1:1} the restriction of~$u$ to~$(0,\infty)$ is a solution of~$\mathrm{EVI}_0\tparen{H,\dist_H,G}$;
\item\label{i:t:MurSav:1:2} $u$ satisfies~\textup{(IC1)}, also denoted by~$(\mathrm{EVI}'_0)$, in~\cite[Thm.~3.3]{MurSav20}, that is: for every~$g\in\dom{G}$, the maps~$r\mapsto G(u_r)$ and~$r\mapsto \norm{u_r-g}^2$ belong to~$L^1_\loc\tparen{(0,\infty)}$, and
\begin{equation}\label{eq:IC1}
\tfrac{1}{2}\norm{u_t-g}^2-\tfrac{1}{2}\norm{u_s-g}^2+\int_s^t G(u_r)\diff r\leq (t-s)\, G(g)\comma \qquad 0<s<t\semicolon
\end{equation}
\item\label{i:t:MurSav:1:3} $u$ satisfies~\textup{(IC2)}, also denoted by~$(\mathrm{EVI}''_0)$, in~\cite[Thm.~3.3]{MurSav20} down to~$s=0$, that is: for every~$g\in\dom{G}$,
\begin{equation}\label{eq:IC2}
\tfrac{1}{2}\norm{u_t-g}^2-\tfrac{1}{2}\norm{u_s-g}^2\leq (t-s)\,\tparen{G(g)-G(u_t)}\comma \qquad 0\leq s<t\fstop
\end{equation}
\end{enumerate}
If any of the above holds, then~$t\mapsto G(u_t)$ is non-increasing on~$[0,\infty)$.

\item\label{i:t:MurSav:2} \textup{(}Contraction and uniqueness,~\cite[Thm.~3.5, (3.10)]{MurSav20}\textup{)} Let~$u,v\colon [0,\infty)\to\overline{\dom{G}}$ be as in~\ref{i:t:MurSav:1}, and assume that both satisfy~\ref{i:t:MurSav:1}\ref{i:t:MurSav:1:1}.
Then,
\[
\norm{u_t-v_t}\leq \norm{u_s-v_s}\comma \qquad 0\leq s<t \fstop
\]
In particular, if~$u_0=v_0$, then~$u=v$.
\end{enumerate}
\end{theorem}

Following~\cite{MurSav20}, we name after~\textup{(IC2)} the corresponding property of non-linear semigroups.

\begin{definition}[{$(\mathrm{EVI}''_0)$ for semigroups, cf.~\cite[Thm.~3.3, (IC2)]{MurSav20}}]\label{d:EVI}
Let~$C$ be a non-empty closed convex subset of~$H$, let~$T_\bullet\colon C\to C$ be a strongly continuous non-expansive non-linear semigroup, and let~$G\colon H\to(-\infty,\infty]$ be proper and lower semicontinuous.
We say that~$T_\bullet$ \emph{satisfies~$(\mathrm{EVI}''_0)$ with respect to~$G$} if~$T_t(C)\subset\dom{G}$ for every~$t>0$ and
\begin{equation}\tag{$\mathrm{EVI}''_0$}\label{eq:EVI}
\tfrac{1}{2}\norm{T_t(f)-g}^2-\tfrac{1}{2}\norm{f-g}^2 \leq t\,\tparen{G(g)-G\tparen{T_t(f)}} \comma \qquad f\in C\comma  g\in H\comma \quad t>0 \fstop
\end{equation}
Equivalently,~$T_\bullet$ satisfies~$(\mathrm{EVI}''_0)$ with respect to~$G$ if and only if~$T_t$ satisfies~$(\mathrm{VI}_t)$ with respect to~$G$ for every~$t>0$.
\end{definition}

Specializing~\cite[Thm.~3.3]{MurSav20} to our setting we have the following result.

\begin{proposition}[Comparison with the EVI of~{\cite{AmbGigSav08,MurSav20}}]\label{p:EVIMurSav}
Let~$C$,~$T_\bullet$, and~$G$ be as in Definition~\ref{d:EVI}.
For~$f\in C$ set~$u(t)\eqdef T_t(f)$,~$t\geq 0$.
Then, the following are equivalent:
\begin{enumerate}[$(i)$]
\item\label{i:p:EVIMurSav:1} $T_\bullet$ satisfies~$(\mathrm{EVI}''_0)$ with respect to~$G$;
\item\label{i:p:EVIMurSav:2} for every~$f\in C$ we have~$u(t)\in\dom{G}$ for every~$t>0$, and the curve~$u$ satisfies~\eqref{eq:IC2} for every~$g\in\dom{G}$;
\item\label{i:p:EVIMurSav:3} $T_\bullet$ is an~$\mathrm{EVI}_0$-gradient flow of~$G$ in~$C$ in the sense of Definition~\ref{d:EVIMurSav}\ref{i:d:EVIMurSav:2}.
\end{enumerate}
\end{proposition}
%
%

The following fact too is well known, cf.~\cite[Thm.~4.0.4]{AmbGigSav08} in the (more general) metric setting.

\begin{lemma}[Gradient flows of convex functionals]\label{l:EVI}
Let~$E\colon H\to(-\infty,\infty]$ be proper convex lower semicontinuous, set~$C\eqdef\overline{\dom{E}}$, and let~$T_\bullet\colon C\to C$ be the non-linear semigroup generated by~$\partial E$.
Then,~$T_\bullet$ satisfies~$(\mathrm{EVI}''_0)$ with respect to~$E$; equivalently, by Proposition~\ref{p:EVIMurSav},~$T_\bullet$ is an~$\mathrm{EVI}_0$-gradient flow of~$E$ in~$C$.
\end{lemma}

\begin{lemma}[$(\mathrm{EVI}''_0)$ identifies the generator]\label{l:EVIGradientFlow}
Let~$C$ be a non-empty closed convex subset of~$H$, and $T_\bullet\colon C\to C$ be a strongly continuous non-expansive non-linear semigroup with maximal monotone generator~$A$.
Further let~$G\colon H\to(-\infty,\infty]$ be proper convex lower semicontinuous, and assume that~$T_\bullet$ satisfies~$(\mathrm{EVI}''_0)$ with respect to~$G$.
Then,~$E\eqdef G+\mbfI_C$ is proper convex lower semicontinuous,~$\overline{\dom{E}}=C$,~$T_\bullet$ satisfies~$(\mathrm{EVI}''_0)$ with respect to~$E$, and~$A=\partial E$.
In particular,~$A$ is cyclically monotone.
\end{lemma}

\begin{proof}
Since~$C$ is closed and convex,~$E$ is convex and lower semicontinuous.
For every~$t>0$ we have~$T_t(C)\subset\dom{G}\cap C=\dom{E}$, thus~$E$ is proper, and~$E\tparen{T_t(f)}=G\tparen{T_t(f)}$ for every~$f\in C$.
Thus,~$T_\bullet$ satisfies~$(\mathrm{EVI}''_0)$ with respect to~$E$, since~\eqref{eq:EVI} for~$E$ coincides with~\eqref{eq:EVI} for~$G$ if~$g\in C$, and it is trivially satisfied if~$g\notin C$.
Furthermore,~$\dom{E}\subset C$, and~$f=\lim_{t\downarrow 0}T_t(f)\in\overline{\dom{E}}$ for every~$f\in C$, thus~$\overline{\dom{E}}=C$.

Let~$S_\bullet\colon C\to C$ be the non-linear semigroup generated by~$\partial E$, which satisfies~$(\mathrm{EVI}''_0)$ with respect to~$E$ by Lemma~\ref{l:EVI}.
Fix~$f\in C$, and set~$u_t\eqdef T_t(f)$ and~$v_t\eqdef S_t(f)$ for~$t\geq 0$.
By Proposition~\ref{p:EVIMurSav} and~\eqref{eq:OrbitContinuity}, applied to both~$T_\bullet$ and~$S_\bullet$, the curves~$u$ and~$v$ satisfy the assumptions of Theorem~\ref{t:MurSav}\ref{i:t:MurSav:1} with~$E$ in place of~$G$, and they are both solutions of~$\mathrm{EVI}_0\tparen{H,\dist_H,E}$ on~$(0,\infty)$.
Since~$u_0=f=v_0$, we have~$u=v$ by Theorem~\ref{t:MurSav}\ref{i:t:MurSav:2}.
Since~$f\in C$ was arbitrary,~$T_\bullet=S_\bullet$.

As a consequence,~$\partial E$ is a maximal monotone generator of~$T_\bullet$ in the sense of Definition~\ref{d:GeneratorSemigroup}, and~$A=\partial E$ by the uniqueness of the generator in Theorem~\ref{t:Barbu}.
Finally,~$A$ is cyclically monotone by Theorem~\ref{t:MaxMonSubDiff}\ref{i:t:MaxMonSubDiff:2}.
\end{proof}

\subsubsection{Characterization of cyclical monotonicity}
We now show the main result of this section, a complete characterization of cyclical monotonicity for generators across~TAJE operators.

\begin{theorem}\label{t:CyclicalTAJ}
Let~$C$ be a non-empty closed convex subset of~$H$ and~$T_\bullet\colon C\to C$ be a strongly continuous non-expansive non-linear semigroup with maximal monotone generator~$A$ and strongly continuous non-expansive non-linear resolvent~$J_\bullet$.
Then, the following are equivalent:
\begin{enumerate}[$(i)$]
\item\label{i:t:CyclicalTAJ:4} $T_t$ is cyclically non-expansive for every~$t>0$;
\item\label{i:t:CyclicalTAJ:1} $A$ is cyclically monotone;
\item\label{i:t:CyclicalTAJ:2} $\car_H -J_\lambda$ is cyclically monotone for every~$\lambda>0$;
\item\label{i:t:CyclicalTAJ:3} $A=\partial E$ for some proper convex lower semicontinuous~$E\colon H\to(-\infty,\infty]$ with~$\overline{\dom{E}}=C$, unique up to an additive constant.
\end{enumerate}
\end{theorem}

\begin{proof}
We prove~\ref{i:t:CyclicalTAJ:1}$\iff$\ref{i:t:CyclicalTAJ:3} and~\ref{i:t:CyclicalTAJ:1}$\iff$\ref{i:t:CyclicalTAJ:2}. 
The equivalence~\ref{i:t:CyclicalTAJ:4}$\iff$\ref{i:t:CyclicalTAJ:3} will require some further work and is thus postponed to Corollary~\ref{c:CyclicalTAJ} below.
Note that~\ref{i:t:CyclicalTAJ:3}$\implies$\ref{i:t:CyclicalTAJ:1} is Theorem~\ref{t:MaxMonSubDiff}.

\paragraph{\ref{i:t:CyclicalTAJ:1}$\implies$\ref{i:t:CyclicalTAJ:3}} By the assumption and Theorem~\ref{t:MaxMonSubDiff}\ref{i:t:MaxMonSubDiff:2},~$A=\partial E$ for some proper convex lower semicontinuous functional~$E\colon H\to (-\infty,\infty]$, unique up to an additive constant by Remark~\ref{r:AdditiveConstant}.
Furthermore,~$C=\overline{\dom{A}}=\overline{\dom{E}}$, respectively by definition of generator and by~\cite[Prop.~2.11, p.~39]{Bre73}.
By the uniqueness of the semigroup with generator~$A$ in Theorem~\ref{t:Barbu},~$T_\bullet$ is the semigroup generated by~$\partial E$, and it satisfies~$(\mathrm{EVI}''_0)$ with respect to~$E$ by Lemma~\ref{l:EVI}.

\paragraph{\ref{i:t:CyclicalTAJ:1}$\implies$\ref{i:t:CyclicalTAJ:2}} Let~$E$ be as in the proof of~\ref{i:t:CyclicalTAJ:1}$\implies$\ref{i:t:CyclicalTAJ:3}.
By~\eqref{eq:YosidaForm} together with Theorem~\ref{t:MaxMonSubDiff}\ref{i:t:MaxMonSubDiff:2}, the Yosida regularization~$A_\lambda$ of~$A$ is cyclically monotone.
Thus,
\[
\sum_{i=1}^n \scalar{f_i-f_{i-1}}{f_i -J_\lambda(f_i)} = \lambda \sum_{i=1}^n \scalar{f_i-f_{i-1}}{ A_\lambda(f_i)} \geq 0 \comma \qquad f_1,\dotsc,f_n\in H\comma \quad f_0\eqdef f_n\comma
\]
which shows~\ref{i:t:CyclicalTAJ:2}.

\paragraph{\ref{i:t:CyclicalTAJ:2}$\implies$\ref{i:t:CyclicalTAJ:1}} Let~$(f_i,a_i)\in\Gph(A)$ for~$i\leq n$, and set~$(f_0,a_0)\eqdef(f_n,a_n)$.
For~$\lambda>0$ set~$g_i\eqdef f_i+\lambda a_i$ for~$i\leq n$.
Since~$J_\lambda=(\car_H+\lambda A)^{-1}$, we have~$J_\lambda(g_i)=f_i$ and thus~$g_i-J_\lambda(g_i)=\lambda a_i$.
Respectively by~\ref{i:t:CyclicalTAJ:2}, and since~$g_i-g_{i-1}=f_i-f_{i-1}+\lambda(a_i-a_{i-1})$,
\[
0\leq \frac{1}{\lambda}\sum_{i=1}^n \scalar{g_i-g_{i-1}}{g_i-J_\lambda(g_i)} = \sum_{i=1}^n \scalar{f_i-f_{i-1}}{a_i} + \lambda\sum_{i=1}^n\scalar{a_i-a_{i-1}}{a_i} \fstop
\]
Letting~$\lambda\downarrow 0$ shows that~$\sum_{i=1}^n \scalar{f_i-f_{i-1}}{a_i}\geq 0$, i.e.~\ref{i:t:CyclicalTAJ:1}.
(In fact, it suffices to assume~\ref{i:t:CyclicalTAJ:2} along a sequence~$\lambda_n\downarrow 0$.)
%
\end{proof}

\subsubsection{Equivalence of energies and resolvents}
Let us now turn to the relation between energies and resolvents.
Let~$J_\bullet\colon H\to H$ be a strongly continuous non-expansive non-linear resolvent.
Fix~$f_0\in H$ such that~$\limsup_{\lambda\downarrow 0} \norm{\tfrac{1}{\lambda}(\car_H-J_\lambda)(f_0)}<\infty$ (i.e.,~$f_0\in \dom{A}$, where~$A$ is the maximal monotone generator of~$J_\lambda$).
For every~$\lambda>0$, define now a functional~$E^\lambda \colon H\to (-\infty,\infty]$ by setting
\begin{equation}\label{eq:ELambda}
E^\lambda \eqdef \cenv\paren{\tfrac{1}{\lambda}(\car_H-J_\lambda)} \comma
\end{equation}
where the cyclical envelope is based at~$(f_0, A_\lambda(f_0))$ for every~$\lambda>0$.


\begin{theorem}\label{t:FormResolvent}
The following assertions hold.
\begin{enumerate}[$(i)$, wide]
\item\label{i:t:FormResolvent:1} Let~$E\colon H\to (-\infty,\infty]$ be convex proper and lower semicontinuous.
Then, setting
\[
J_\lambda(\emparg) \eqdef \argmin_{g\in H}  \braket{\tfrac{1}{2\lambda}\norm{\emparg-g}^2+E(g)} \quad \text{for all } \lambda>0 \quad \text{and} \quad J_0\eqdef \cproj_{\overline{\dom{E}}}\comma
\]
defines~$J_\bullet\colon H\to H$ a strongly continuous non-expansive non-linear resolvent on~$\overline{\dom{E}}$, generated by $\partial E$ in the sense of Theorem~\ref{t:ResolventGenerator} and such that~$\car_H-J_\bullet$ is cyclically monotone.

\item\label{i:t:FormResolvent:2} Vice versa, let~$C$ be a non-empty closed convex subset of~$H$, and~$J_\bullet\colon H\to H$ be a strongly continuous non-expansive non-linear resolvent on~$C$ such that~$\car_H-J_\lambda$ is cyclically monotone for every~$\lambda>0$.
Then, for every~$f\in H$ there exists the limit
\begin{equation}\label{eq:t:formResolvent:0}
E(f)\eqdef \lim_{\lambda\downarrow 0} E^\lambda(f) \in (-\infty,\infty] \comma
\end{equation}
with~$E^\lambda$ as in~\eqref{eq:ELambda}, and it defines a proper convex lower semicontinuous functional~$E\colon H\to (-\infty,\infty]$ with~$E(f_0)=0$ and~$\dom{E}\subset C$.
Furthermore,~$J_\bullet$ is generated by~$\partial E$ in the sense of Theorem~\ref{t:ResolventGenerator}.
\end{enumerate}
\begin{proof}
\ref{i:t:FormResolvent:1} is a consequence of~\cite[Lem.~2.1, p.~25]{Bre73}, also cf.~\cite[Prop.~2.11, p.~39]{Bre73}.
The operator~$\partial E$ is cyclically monotone by Theorem~\ref{t:MaxMonSubDiff}\ref{i:t:MaxMonSubDiff:2}, and the cyclical monotonicity of~$\car_H-J_\bullet$ follows from Theorem~\ref{t:CyclicalTAJ}\ref{i:t:CyclicalTAJ:1}$\implies$\ref{i:t:CyclicalTAJ:2}, applied to the semigroup generated by~$\partial E$, see Theorem~\ref{t:Barbu}.

\ref{i:t:FormResolvent:2}
Let~$A$ be the maximal monotone generator of~$J_\bullet$ given by Theorem~\ref{t:ResolventGenerator}, so that~$J_\lambda=(\car_H+\lambda A)^{-1}$, and~$\tfrac{1}{\lambda}(\car_H-J_\lambda)=A_\lambda$ is the Yosida regularization of~$A$ (Dfn.~\ref{d:YosidaRegGenerator}) for every~$\lambda>0$.
Since~$\car_H-J_\lambda$ is cyclically monotone for every~$\lambda>0$, the operator~$A$ is cyclically monotone by Theorem~\ref{t:CyclicalTAJ}\ref{i:t:CyclicalTAJ:2}$\implies$\ref{i:t:CyclicalTAJ:1}, applied to the semigroup generated by~$A$ in the sense of Theorem~\ref{t:Barbu}.
Thus, by Theorem~\ref{t:MaxMonSubDiff}\ref{i:t:MaxMonSubDiff:2}, there exists a proper convex lower semicontinuous functional~$\tilde E\colon H\to (-\infty,\infty]$ with~$A=\partial\tilde E$.
Denote by~$\tilde E_\lambda$ the Yosida regularization (Dfn.~\ref{d:YosidaRegForm}) of~$\tilde E$.
By~\cite[Prop.~2.11, p.~39]{Bre73}, for every~$\lambda>0$ the functional~$\tilde E_\lambda\colon H\to\R$ is convex and Fr\'echet differentiable with~$\nabla\tilde E_\lambda=A_\lambda$, cf.~\eqref{eq:YosidaForm}, and
\begin{equation}\label{eq:t:FormResolvent:1}
\lim_{\lambda\downarrow 0}\tilde E_\lambda(f)= \sup_{\lambda>0}\tilde E_\lambda(f)= \tilde E(f) \comma \qquad f\in H \comma
\end{equation}
by monotonicity~\eqref{eq:MonotonicityYosidaRegE} of the Yosida regularization.

\paragraph{Claim: $E^\lambda=\tilde E_\lambda-\tilde E_\lambda(f_0)$ for every~$\lambda>0$}
Fix~$\lambda>0$ and~$f\in H$, and recall that~$\dom{A_\lambda}=H$.
Further fix~$n\in\N_0$ and~$f_1,\dotsc,f_n\in H$, and set~$f_{n+1}\eqdef f$.
Respectively: since~$A_\lambda(f_i)=\partial\tilde E_\lambda(f_i)$ and by~\eqref{eq:ConvexSubDiff}, and since the last sum is telescopic,
\[
\sum_{i=0}^n \scalar{f_{i+1}-f_i}{A_\lambda(f_i)} \leq \sum_{i=0}^n \tparen{\tilde E_\lambda(f_{i+1})-\tilde E_\lambda(f_i)} = \tilde E_\lambda(f)-\tilde E_\lambda(f_0) \fstop
\]
Taking the supremum over all such chains, we conclude that~$E^\lambda(f)\leq \tilde E_\lambda(f)-\tilde E_\lambda(f_0)$ by Definition~\ref{d:CEnv}.
Conversely, for~$n\in\N$ choose~$f_i\eqdef f_0+\tfrac{i}{n+1}(f-f_0)$ for~$i\leq n+1$. Then,
\[
\sum_{i=0}^n \scalar{f_{i+1}-f_i}{A_\lambda(f_i)} = \frac{1}{n+1}\sum_{i=0}^n \tscalar{f-f_0}{\nabla\tilde E_\lambda\tparen{f_0+\tfrac{i}{n+1}(f-f_0)}} \fstop
\]
Since~$A_\lambda$ is Lipschitz continuous by~\cite[Prop.~2.6, p.~28]{Bre73}, the map~$s\mapsto\tscalar{f-f_0}{\nabla\tilde E_\lambda\tparen{f_0+s(f-f_0)}}$ is continuous on~$[0,1]$, and the right-hand side above is a Riemann sum for its integral.
Letting~$n\to\infty$ and applying the fundamental theorem of calculus to the $C^1$-function~$s\mapsto\tilde E_\lambda\tparen{f_0+s(f-f_0)}$, we obtain
\[
E^\lambda(f)\geq \int_0^1 \tscalar{f-f_0}{\nabla\tilde E_\lambda\tparen{f_0+s(f-f_0)}}\diff s = \tilde E_\lambda(f)-\tilde E_\lambda(f_0) \comma
\]
which concludes the proof of the claim.

\paragraph{Conclusion}
Since~$f_0\in\dom{A}\subset\dom{\tilde E}$, we have~$\lim_{\lambda\downarrow 0}\tilde E_\lambda(f_0)=\tilde E(f_0)\in\R$ by~\eqref{eq:t:FormResolvent:1}.
Thus, respectively by the Claim and by~\eqref{eq:t:FormResolvent:1}, for every~$f\in H$ the limit in~\eqref{eq:t:formResolvent:0} exists in~$(-\infty,\infty]$, and
\[
E(f)=\lim_{\lambda\downarrow 0}\tparen{\tilde E_\lambda(f)-\tilde E_\lambda(f_0)} = \tilde E(f)-\tilde E(f_0) \comma \qquad f\in H\comma
\]
where both sides equal~$+\infty$ if~$f\notin\dom{\tilde E}$.
That is,~$E$ differs from~$\tilde E$ by the real constant~$\tilde E(f_0)$.
In particular,~$E$ is proper convex and lower semicontinuous,~$E(f_0)=0$,~$\dom{E}=\dom{\tilde E}$, and~$\partial E=\partial\tilde E=A$, so that~$J_\lambda=(\car_H+\lambda\,\partial E)^{-1}$ for every~$\lambda>0$, i.e.~$J_\bullet$ is generated by~$\partial E$ in the sense of Theorem~\ref{t:ResolventGenerator}.

Finally,~$C=\overline{\dom{A}}$ by Theorem~\ref{t:ResolventGenerator}, and therefore~$\dom{E}=\dom{\tilde E}\subset\overline{\dom{\tilde E}}=\overline{\dom{A}}=C$ by~\cite[Prop.~2.11, p.~39]{Bre73}.
\end{proof}
\end{theorem}

In order to further describe the convergence in~\eqref{eq:t:formResolvent:0}, let us recall the definition of \emph{Mosco convergence} of functionals.

\begin{definition}[Mosco convergence]\label{d:Mosco}
Let~$F_n\colon H\to(-\infty,\infty]$, with~$n\in \N$, and~$F\colon H\to (-\infty,\infty]$ be functionals on~$H$.
We say that~$\seq{F_n}_n$ \emph{Mosco converges} to~$F$ if
\begin{itemize}
\item for each sequence~$\seq{h_n}_n\subset H$ weakly convergent to~$h\in H$ we have
\[
\liminf_n F_n(h_n)\geq F(h)\semicolon
\]
\item for each~$h\in H$ there exists a sequence~$\seq{h_n}_n\subset H$, strongly convergent to~$h$ and such that
\[
\limsup_n F_n(h_n)\leq F(h) \fstop
\]
\end{itemize}
\end{definition}

\begin{remark}[On the convergence in Theorem~\ref{t:FormResolvent}\ref{i:t:FormResolvent:2}]\label{r:FormResolvent}
\begin{enumerate*}[$(a)$]
\item\label{i:r:FormResolvent:1} The limit in~\eqref{eq:t:formResolvent:0} is in general \emph{not} monotone, since the normalization~$E^\lambda(f_0)=0$ shifts~$\tilde E_\lambda$ by the $\lambda$-dependent constant~$\tilde E_\lambda(f_0)$.
For instance, let~$H\eqdef\R$, let~$J_\lambda(t)\eqdef \tfrac{t}{1+\lambda}$, i.e.~$A=\id_\R$ and~$\tilde E(t)=\tfrac{1}{2}t^2$, and choose~$f_0\eqdef 1$.
Then~$\tilde E_\lambda(t)=\tfrac{t^2}{2(1+\lambda)}$, and thus~$E^\lambda(t)=\tfrac{t^2-1}{2(1+\lambda)}$ by the Claim in the proof of Theorem~\ref{t:FormResolvent}\ref{i:t:FormResolvent:2}.
In particular,~$\lambda\mapsto E^\lambda(0)$ is increasing, while~$\lambda\mapsto E^\lambda(2)$ is decreasing.

\item\label{i:r:FormResolvent:2} If~$f_0$ is a minimizer of~$\tilde E$, i.e.~$0\in A(f_0)$, then~$\tilde E_\lambda(f_0)=\tilde E(f_0)$ for every~$\lambda>0$, and therefore~$E^\lambda=\tilde E_\lambda-\tilde E(f_0)$ increases to~$E$ as~$\lambda\downarrow 0$ by~\eqref{eq:t:FormResolvent:1}.

\item\label{i:r:FormResolvent:3} In all cases, the convergence in~\eqref{eq:t:formResolvent:0} is a Mosco convergence (along every sequence~$\lambda_n\downarrow 0$).
Indeed,~$\tilde E_\lambda$ is monotone non-increasing in~$\lambda$ by~\eqref{eq:MonotonicityYosidaRegE} and converges to~$\tilde E$ as~$\lambda\downarrow 0$, hence it Mosco converges to~$\tilde E$ by e.g.~\cite[Thm.~3.20, p.~298]{Att84}, and~$E^\lambda$ differs from~$\tilde E_\lambda-\tilde E(f_0)$ by the constant~$\tilde E(f_0)-\tilde E_\lambda(f_0)$, which vanishes as~$\lambda\downarrow 0$.
\end{enumerate*}
\end{remark}

\subsubsection{Equivalence of energies and semigroups}
Let us now turn to the relation between energies and semigroups.
In the linear case, a quadratic functional~$E$ is recovered from its (self-adjoint) semigroup as the monotone limit~$2E(f)=\lim_{t\downarrow 0}\tfrac{1}{t}\scalar{f-T_tf}{f}$.
In the non-linear case, for every~$f\in\dom{A}$ we have instead~$\lim_{t\downarrow 0}\tfrac{1}{t}(f-T_t f)=A^\circ f$ by~\cite[Thm.~3.1(5), p.~54]{Bre73}, which only recovers the \emph{minimal section}~$A^\circ$ of~$\partial E$.
Recovering~$E$ itself rather requires an integration along chains, in the spirit of the cyclical envelope in Definition~\ref{d:CEnv}.
However, as shown by the following example, the operator~$A^\sym{t}=\tfrac{1}{t}(\car_H-T_t\circ\cproj_C)$ in~\eqref{eq:ApproxSemigroup} is in general \emph{not} cyclically monotone in which case~$\cenv(A^\sym{t})\equiv+\infty$ on~$H$ by Lemma~\ref{l:CEnv}\ref{i:l:CEnv:2}.

\begin{example}\label{ex:AtNotCyclical}
The operator~$A^\sym{t}$ is cyclically monotone if, e.g.,~$C=H$ and~$T_\bullet$ is a semigroup of linear self-adjoint operators, or if~$H=\R$.
However, this is \emph{not} implied by the cyclical monotonicity of~$A$, and it typically fails for non-linear gradient flows in dimension at least two.
Indeed, let~$H=\R^d$ with~$d\geq 2$, and let~$E\in C^3$ be convex and finite.
On the one hand, cyclical monotonicity of the $C^1$-map~$\car_H-T_t$ forces its Jacobian matrix to be symmetric.
On the other hand,~$M(t)\eqdef DT_t(x)$ solves~$\dot M(t)=-\nabla^2E\tparen{T_t(x)}\,M(t)$ with~$M(0)=\car$, and thus its antisymmetric part is~$\tfrac{t^3}{12}\big[\dot H_0,H_0\big]+O(t^4)$ for all sufficiently small~$t>0$, where~$H_0\eqdef\nabla^2E(x)$ and~$\dot H_0\eqdef\tfrac{\diff}{\diff s}\nabla^2E\tparen{T_s(x)}\big|_{s=0}$.
Thus, the antisymmetric part is non-zero whenever the two matrices~$H_0$ and~$\dot H_0$ do not commute, e.g.\ for~$E(x)\eqdef\tfrac{1}{4}x_1^4+\tfrac{1}{2}(x_1-x_2)^2$ on~$\R^2$ at~$x=(1,0)$.
In this case,~$\cenv(A^\sym{t})\equiv +\infty$ on~$H$ by Lemma~\ref{l:CEnv}\ref{i:l:CEnv:2}.
\end{example}

The correct object to be considered in place of~$\cenv(A^\sym{t})$ is the following specialization of the lower semicontinuous convex envelope discussed in~\cite{CorGarHan12,MarThi05}.

\begin{definition}[$t$-regularized cyclical envelope]\label{d:SemigroupEnvelope}
Let~$C$ be a non-empty closed convex subset of~$H$, let~$B\colon C\to H$, and fix~$t>0$.
For~$h\in C$ and~$g\in H$ set
\begin{equation}\label{eq:SigmaT}
\sigma^B_t(h,g)\eqdef \tfrac{1}{2t}\tparen{\norm{g-B(h)}^2-\norm{g-h}^2} \fstop
\end{equation}
For fixed~$f_0\in C$, we define the \emph{$t$-regularized cyclical envelope of~$B$} (\emph{based at~$f_0$}) by
\begin{equation}\label{eq:SemigroupEnvelope}
\cenv_t(B)(f)\eqdef \sup\set{\sum_{i=0}^n \sigma^B_t(h_i,g_{i+1}) : \begin{aligned} &h_0\eqdef f_0\comma\quad h_1,\dotsc, h_n\in C\comma \quad n\in\N_0\comma\\ &g_i\eqdef B(h_i)\ \text{for}\ 1\leq i\leq n\comma \quad g_{n+1}\eqdef f \end{aligned}} \comma \qquad f\in H \fstop
\end{equation}
\end{definition}

Since~$\norm{g-B(h)}^2-\norm{g-h}^2=2\scalar{h-B(h)}{g-\tfrac{1}{2}\tparen{h+B(h)}}$, for~$B=T_t$ and~$i<n$ the~$i^\text{th}$ summand in~\eqref{eq:SemigroupEnvelope} reads
\begin{align}
\nonumber
\sigma^{T_t}_t(h_i,g_{i+1}) &=\scalar{A^\sym{t}(h_i)}{T_t(h_{i+1})-\tfrac{1}{2}\tparen{h_i+T_t(h_i)}} 
\\
\label{eq:CorreaGarciaHantoute}
&=\scalar{A^\sym{t}(h_i)}{T_t(h_{i+1})-T_t(h_i)}-\tfrac{1}{2}\scalar{A^\sym{t}(h_i)}{h_i-T_t(h_i)} \comma
\end{align}
whereas the corresponding summand in Definition~\ref{d:CEnv} for~$A^\sym{t}$ is~$\scalar{A^\sym{t}(h_i)}{h_{i+1}-h_i}$.
In other words, the next point of the chain is evaluated after one step of the semigroup, and the current point is replaced by the midpoint of the step.
Furthermore, if~$B$ satisfies~$(\mathrm{VI}_t)$ with respect to some~$G$, then~$\sigma^B_t(h,g)\leq G(g)-G\tparen{B(h)}$ for every~$h\in C$ and~$g\in H$ by~\eqref{eq:VI}.

Replacing the expression for~$\sigma_t^{T_t}$ in~\eqref{eq:CorreaGarciaHantoute} into~\eqref{eq:SemigroupEnvelope} shows that~$\cenv_t(T_t)$ is a particular case of the lower semicontinuous convex envelope in~\cite[Thm.~13]{CorGarHan12} for the specific choice~$\eps_i\eqdef \tfrac{1}{2}\scalar{A^\sym{t}(h_i)}{h_i-T_t(h_i)}$ there.

\begin{lemma}[Cyclical non-expansiveness and~$(\mathrm{VI}_t)$]\label{l:OneStepEVI}
Let~$C$ be a non-empty closed convex subset of~$H$, let~$B\colon C\to H$, fix~$t>0$ and~$f_0\in C$, and let~$F\eqdef\cenv_t(B)$ be based at~$f_0$.
Then, the following are equivalent:
\begin{enumerate}[$(i)$]
\item\label{i:l:OneStepEVI:1} $B$ is cyclically non-expansive;
\item\label{i:l:OneStepEVI:2} $F$ is proper;
\item\label{i:l:OneStepEVI:3} there exists a proper convex lower semicontinuous~$G\colon H\to(-\infty,\infty]$ such that~$B$ satisfies~$(\mathrm{VI}_t)$ with respect to~$G$.
\end{enumerate}
If any of the above holds, then~$F$ is convex and lower semicontinuous,~$B$ satisfies~$(\mathrm{VI}_t)$ with respect to~$F$,
\begin{equation}\label{eq:l:OneStepEVI:1}
-\tfrac{1}{2t}\norm{f_0-B(f_0)}^2\leq F\tparen{B(f_0)}\leq 0\comma
\end{equation}
and~$F$ is \emph{minimal} in the following sense: for every~$G$ as in~\ref{i:l:OneStepEVI:3},
\begin{equation}\label{eq:l:OneStepEVI:2}
F\leq G-G\tparen{B(f_0)} \fstop
\end{equation}
\end{lemma}

\begin{proof}
We start with some preliminary observations.
Expanding the squares in~\eqref{eq:SigmaT}, for every~$h\in C$ and~$g\in H$,
\begin{equation}\label{eq:l:OneStepEVI:3}
\sigma^B_t(h,g)=\tfrac{1}{2t}\tparen{\norm{B(h)}^2-\norm{h}^2}+\tfrac{1}{t}\scalar{h-B(h)}{g} \comma
\end{equation}
so that~$g\mapsto\sigma^B_t(h,g)$ is affine and continuous.
For fixed~$n\in\N_0$ and~$h_1,\dotsc,h_n\in C$, the sum in~\eqref{eq:SemigroupEnvelope} depends on~$f$ only through its last summand~$\sigma^B_t(h_n,f)$.
Hence,~$F$ is the supremum of a non-empty family of affine continuous functions, and it is therefore convex and lower semicontinuous with values in~$(-\infty,\infty]$.
Choosing~$n=0$ and~$f=B(f_0)$ in~\eqref{eq:SemigroupEnvelope}, we further have
\[
F\tparen{B(f_0)}\geq \sigma^B_t\tparen{f_0,B(f_0)}=-\tfrac{1}{2t}\norm{f_0-B(f_0)}^2 \fstop
\]
Finally, let~$h\in C$ and~$g\in H$.
Appending~$h_{n+1}\eqdef h$ to any chain~$h_1,\dotsc,h_n$ with~$g_{n+1}\eqdef B(h)$ yields a chain~$h_1,\dotsc,h_{n+1}$ with~$g_{n+2}\eqdef  g$, and the corresponding sum in~\eqref{eq:SemigroupEnvelope} increases by~$\sigma^B_t(h,g)$.
Taking the supremum over all such chains,
\begin{equation}\label{eq:l:OneStepEVI:4}
F(g)\geq F\tparen{B(h)}+\sigma^B_t(h,g) \comma \qquad h\in C\comma \quad g\in H\comma
\end{equation}
where both sides equal~$\infty$ if~$F\tparen{B(h)}=\infty$.

\paragraph{\ref{i:l:OneStepEVI:1}$\implies$\ref{i:l:OneStepEVI:2}}
Fix~$n\in\N$ and~$h_1,\dotsc,h_n\in C$, and set~$h_{n+1}= h_0\eqdef f_0$.
For~$f=B(f_0)$, we have~$g_i=B(h_i)$ for every~$1\leq i\leq n+1$, and the sum in~\eqref{eq:SemigroupEnvelope} reads
\[
\sum_{i=0}^n \sigma^B_t\tparen{h_i,B(h_{i+1})}=\frac{1}{2t}\sum_{i=0}^n\tparen{\norm{B(h_{i+1})-B(h_i)}^2-\norm{B(h_{i+1})-h_i}^2} \fstop
\]
Set~$m\eqdef n+1$ and~$g_j\eqdef h_{m-j}$ for~$0\leq j\leq m$, so that~$g_0=g_m=f_0$.
Changing the summation index to~$j\eqdef m-i$, the right-hand side above equals
\[
\frac{1}{2t}\sum_{j=1}^m \tparen{\norm{B(g_j)-B(g_{j-1})}^2-\norm{g_j-B(g_{j-1})}^2} \comma
\]
which is non-positive by Definition~\ref{d:CyclicalNonExp}.
Taking the supremum over all chains, we conclude that~$F\tparen{B(f_0)}\leq 0$, and thus~$F$ is proper.

\paragraph{\ref{i:l:OneStepEVI:2}$\implies$\ref{i:l:OneStepEVI:3}}
Let~$g\in\dom{F}\neq \emp$.
By~\eqref{eq:l:OneStepEVI:4}, we have~$F\tparen{B(h)}\leq F(g)-\sigma^B_t(h,g)<\infty$ for every~$h\in C$, i.e.~$B(C)\subset\dom{F}$.
Multiplying~\eqref{eq:l:OneStepEVI:4} by~$t$, we see that~$B$ satisfies~$(\mathrm{VI}_t)$ with respect to~$F$, which is proper convex lower semicontinuous by the preliminary observations.

\paragraph{\ref{i:l:OneStepEVI:3}$\implies$\ref{i:l:OneStepEVI:1}}
This is Lemma~\ref{l:VI}\ref{i:l:VI:2}.

\paragraph{Conclusion}
Assume now that any of~\ref{i:l:OneStepEVI:1}--\ref{i:l:OneStepEVI:3} holds.
The first two assertions were shown in the preliminary observations and in the proof of~\ref{i:l:OneStepEVI:2}$\implies$\ref{i:l:OneStepEVI:3}, and~\eqref{eq:l:OneStepEVI:1} in the preliminary observations and in the proof of~\ref{i:l:OneStepEVI:1}$\implies$\ref{i:l:OneStepEVI:2}.
In order to show~\eqref{eq:l:OneStepEVI:2}, let~$G$ be as in~\ref{i:l:OneStepEVI:3}, and fix~$f\in H$,~$n\in\N$, and~$h_1,\dotsc,h_n\in C$.
Respectively: by~\eqref{eq:VI} with~$h_i$ in place of~$f$ and~$g_{i+1}$ in place of~$g$, and since~$B(h_i)=g_i$ and~$G(g_i)$ is finite for~$1\leq i\leq n$,
\[
\sum_{i=0}^n \sigma^B_t(h_i,g_{i+1})\leq \sum_{i=0}^n \tparen{G(g_{i+1})-G\tparen{B(h_i)}} = G(f)-G\tparen{B(f_0)} \fstop
\]
Taking the supremum over all chains proves~\eqref{eq:l:OneStepEVI:2}.
\end{proof}

\begin{remark}\label{r:OneStepEVI}
\begin{enumerate}[$(a)$, wide]
\item\label{i:r:OneStepEVI:1} In contrast with Example~\ref{ex:AtNotCyclical}, the functional~$\cenv_t(T_t)$ is always proper when~$T_t$ is cyclically non-expansive, by Lemma~\ref{l:OneStepEVI}.

\item\label{i:r:OneStepEVI:2} Lemma~\ref{l:OneStepEVI} is a fixed-time statement: a single map is cyclically non-expansive if and only if it satisfies~$(\mathrm{VI}_t)$ with respect to some convex lower semicontinuous potential.
Rearranging~\eqref{eq:VI}, the latter states that~$\tfrac{1}{t}\tparen{f-B(f)}$ is an~$\eps$-approximate subgradient of~$G$ at~$B(f)$ with~$\varepsilon=\tfrac{1}{2t}\norm{f-B(f)}^2$, viz.
\[
G(g)\geq G\tparen{B(f)}+\scalar{\tfrac{1}{t}\tparen{f-B(f)}}{g-B(f)}-\tfrac{1}{2t}\norm{f-B(f)}^2 \comma \qquad g\in H \fstop
\]
For resolvents, the analogous statement holds without error term, since~$\tfrac{1}{\lambda}\tparen{f-J_\lambda(f)}\in\partial E\tparen{J_\lambda(f)}$; this is why the cyclical envelopes in Theorem~\ref{t:FormResolvent} are those of Definition~\ref{d:CEnv}.
\end{enumerate}
\end{remark}

\begin{theorem}[Energies from semigroups]\label{t:FormSemigroup}
Let~$C$ be a non-empty closed convex subset of~$H$, and~$T_\bullet\colon C\to C$ be a strongly continuous non-expansive non-linear semigroup with maximal monotone generator~$A$.
Assume that~$T_{2^{-k}}$ is cyclically non-expansive for every~$k\in\N$.
Fix~$u_0\in C$ and set~$g_0\eqdef T_1(u_0)$.
For~$k\in\N$ further set
\[
\tau_k\eqdef 2^{-k}\comma\qquad h_k\eqdef T_{1-\tau_k}(u_0)\comma \qquad L_k\eqdef \tfrac{1}{2\tau_k}\norm{h_k-g_0}^2 \comma
\]
and let~$E^k\eqdef \cenv_{\tau_k}\tparen{T_{\tau_k}}$ be based at~$h_k$ (Dfn.~\ref{d:SemigroupEnvelope}).
Then:
\begin{enumerate}[$(i)$]
\item\label{i:t:FormSemigroup:1} $L_k\leq 2\tau_k\norm{u_0-g_0}^2$ for every~$k\in\N$; in particular,~$\sum_{k\in\N}L_k\leq 2\norm{u_0-g_0}^2$;
\item\label{i:t:FormSemigroup:2} $E^k\leq E^{k+1}+L_{k+1}$ on~$H$ for every~$k\in\N$; in particular, setting~$r_k\eqdef\sum_{j>k}L_j$, for every~$f\in H$ the sequence~$\seq{E^k(f)-r_k}_k$ is non-decreasing, and the limit
\begin{equation}\label{eq:t:FormSemigroup:0}
E^\infty(f)\eqdef \lim_{k\to\infty} E^k(f)\in(-\infty,\infty]
\end{equation}
exists;
\item\label{i:t:FormSemigroup:3} the functional~$E\eqdef E^\infty+\mbfI_C$ is the unique proper convex lower semicontinuous functional~$E\colon H\to(-\infty,\infty]$ such that~$\partial E=A$ and~$E(g_0)=0$, and~$T_\bullet$ satisfies~$(\mathrm{EVI}''_0)$ with respect to~$E$, i.e., by Proposition~\ref{p:EVIMurSav},~$T_\bullet$ is an~$\mathrm{EVI}_0$-gradient flow of~$E$ in~$C$; in particular,~$A$ is cyclically monotone,~$T_t$ is cyclically non-expansive for every~$t>0$, and
\[
E(f)=\lim_{k\to\infty} E^k(f) \comma \qquad f\in C \semicolon
\]
\item\label{i:t:FormSemigroup:4} for every~$t\in (0,1]$, the functional~$\cenv_t(T_t)$ based at~$T_{1-t}(u_0)$ satisfies~$\cenv_t(T_t)\leq E$ on~$H$.
\end{enumerate}
\end{theorem}

\begin{proof}
We start with some preliminary observations.
Fix~$k\in\N$.
By the semigroup property~\ref{i:d:Semigroup:2},~$T_{\tau_k}(h_k)=T_1(u_0)=g_0$.
Since~$T_{\tau_k}$ is cyclically non-expansive, by Lemma~\ref{l:OneStepEVI}, the functional~$E^k$ is proper convex lower semicontinuous,~$T_{\tau_k}$ satisfies~$(\mathrm{VI}_{\tau_k})$ with respect to~$E^k$, and, by~\eqref{eq:l:OneStepEVI:1} with~$T_{\tau_k}$ in place of~$B$ and~$h_k$ in place of~$f_0$,
\begin{equation}\label{eq:t:FormSemigroup:1}
-L_k\leq E^k(g_0)\leq 0 \fstop
\end{equation}
By Lemma~\ref{l:VI}\ref{i:l:VI:3},~$T_t$ satisfies~$(\mathrm{VI}_t)$ with respect to~$E^k$ for every~$t\in\tau_k\N$.
Furthermore, the curve~$s\mapsto T_s(f)$ is continuous on~$[0,\infty)$ for every~$f\in C$ by~\eqref{eq:OrbitContinuity}.

\paragraph{Proof of~\ref{i:t:FormSemigroup:1}}
Fix~$k\in\N$, and set~$\tau\eqdef\tau_k$,~$N\eqdef 2^k$, and~$M\eqdef N/2\geq 1$.
Further set~$w_j\eqdef T_{j\tau}(u_0)$ and $d_j\eqdef \norm{w_{j+1}-w_j}$ for~$j\in\N_0$, and note that~$w_N=g_0$ and~$w_{N-1}=h_k$, so that~$d_{N-1}^2=2\tau L_k$.
Since~$w_{j+1}=T_{j\tau}(w_1)$ and~$w_j=T_{j\tau}(w_0)$, the sequence~$\seq{d_j}_j$ is non-increasing by~\ref{i:d:Semigroup:2} and~\ref{i:d:Semigroup:4}.
By Lemma~\ref{l:VI}\ref{i:l:VI:1} for~$E^k$, we have~$\tfrac{1}{2\tau}d_j^2\leq E^k(w_j)-E^k(w_{j+1})$ for every~$j\geq 1$, where all terms are finite since~$w_j\in T_\tau(C)\subset\dom{E^k}$.
Respectively: since~$N-M=\tfrac{1}{2\tau}$, since~$\seq{d_j}_j$ is non-increasing, and by the latter inequality,
\[
\tfrac{1}{2\tau}\,L_k=\tfrac{N-M}{2\tau}\, d_{N-1}^2\leq \sum_{j=M}^{N-1}\tfrac{1}{2\tau}\,d_j^2\leq E^k(w_M)-E^k(g_0) \fstop
\]
Since~$M\tau=\tfrac{1}{2}\in\tau\N$, applying~\eqref{eq:VI} for~$B=T_{1/2}$ and~$t=\tfrac{1}{2}$ with respect to~$E^k$, with~$u_0$ in place of~$f$ and~$g_0$ in place of~$g$, we have
\[
\tfrac{1}{2}\norm{w_M-g_0}^2-\tfrac{1}{2}\norm{u_0-g_0}^2\leq \tfrac{1}{2}\tparen{E^k(g_0)-E^k(w_M)} \comma
\]
and therefore~$E^k(w_M)-E^k(g_0)\leq \norm{u_0-g_0}^2$.
Combining the above inequalities proves the first assertion.
The second one follows since~$\sum_{k\in\N}\tau_k=1$.

\paragraph{Proof of~\ref{i:t:FormSemigroup:2}}
Fix~$k\in\N$.
Since~$\tau_k=2\tau_{k+1}$, the map~$T_{\tau_k}$ satisfies~$(\mathrm{VI}_{\tau_k})$ with respect to~$E^{k+1}$ by the preliminary observations.
Since~$T_{\tau_k}(h_k)=g_0$, respectively by the minimality~\eqref{eq:l:OneStepEVI:2} of~$E^k$ and by~\eqref{eq:t:FormSemigroup:1} for~$E^{k+1}$,
\[
E^k\leq E^{k+1}-E^{k+1}(g_0)\leq E^{k+1}+L_{k+1} \fstop
\]
Note that~$r_k$ is finite by~\ref{i:t:FormSemigroup:1}, and set~$\hat E^k\eqdef E^k-r_k$.
Then,~$\hat E^k\leq E^{k+1}+L_{k+1}-r_k=\hat E^{k+1}$, hence~$\tseq{\hat E^k(f)}_k$ is non-decreasing for every~$f\in H$.
Since~$\lim_k r_k=0$, the limit in~\eqref{eq:t:FormSemigroup:0} exists, and
\begin{equation}\label{eq:t:FormSemigroup:2}
E^\infty=\sup_{k\in\N}\hat E^k \fstop
\end{equation}

\paragraph{Proof of~\ref{i:t:FormSemigroup:3}}
By~\eqref{eq:t:FormSemigroup:2},~$E^\infty$ is a supremum of convex lower semicontinuous functionals with values in~$(-\infty,\infty]$, and it is therefore convex and lower semicontinuous with values in~$(-\infty,\infty]$.
By~\eqref{eq:t:FormSemigroup:1} we have~$-r_{k-1}=-L_k-r_k\leq \hat E^k(g_0)\leq -r_k$ for every~$k$, thus~$E^\infty(g_0)=0$, and~$E^\infty$ is proper.

\paragraph{Claim: $T_\bullet$ satisfies~$(\mathrm{EVI}''_0)$ with respect to~$E^\infty$}
By Definition~\ref{d:EVI}, we need to show that~$T_t$ satisfies~$(\mathrm{VI}_t)$ with respect to~$E^\infty$ for every~$t>0$.
Fix~$f\in C$ and~$g\in\dom{E^\infty}$, and suppose first that~$t\in\tau_k\N$ for some~$k\in\N$.
For every~$j\geq k$ we have~$t\in\tau_j\N$, thus, by the preliminary observations,~$T_t$ satisfies~$(\mathrm{VI}_t)$ with respect to~$E^j$, hence with respect to~$\hat E^j$.
Respectively by~\eqref{eq:VI} and by~\eqref{eq:t:FormSemigroup:2},
\[
\hat E^j\tparen{T_t(f)}\leq \hat E^j(g)-\tfrac{1}{2t}\tparen{\norm{T_t(f)-g}^2-\norm{f-g}^2}\leq E^\infty(g)-\tfrac{1}{2t}\tparen{\norm{T_t(f)-g}^2-\norm{f-g}^2} \comma \qquad j\geq k \fstop
\]
Taking the supremum over~$j\geq k$ and using~\eqref{eq:t:FormSemigroup:2} once more, we conclude that~$E^\infty\tparen{T_t(f)}<\infty$ and that~\eqref{eq:VI} holds for~$B=T_t$ with respect to~$E^\infty$.
If otherwise~$t>0$ is arbitrary, set~$t_m\eqdef 2^{-m}\lceil 2^m t\rceil$, so that~$t_m\in\tau_m\N$ and~$\lim_m t_m=t$.
By the continuity of~$s\mapsto T_s(f)$ and by the lower semicontinuity of~$E^\infty$, we have~$\lim_m T_{t_m}(f)=T_t(f)$ and~$E^\infty\tparen{T_t(f)}\leq\liminf_m E^\infty\tparen{T_{t_m}(f)}$.
Applying~\eqref{eq:VI} for~$B=T_{t_m}$ at time~$t_m$ with respect to~$E^\infty$, and letting~$m\to\infty$, we conclude that~$E^\infty\tparen{T_t(f)}<\infty$ and that~\eqref{eq:VI} holds for~$B=T_t$ with respect to~$E^\infty$.
Since~\eqref{eq:VI} is trivially satisfied if~$g\notin\dom{E^\infty}$, this proves the claim.

\medskip

By the Claim and Lemma~\ref{l:EVIGradientFlow} with~$G=E^\infty$, the functional~$E=E^\infty+\mbfI_C$ is proper convex lower semicontinuous,~$T_\bullet$ satisfies~$(\mathrm{EVI}''_0)$ with respect to~$E$, and~$\partial E=A$; in particular,~$A$ is cyclically monotone maximal monotone.
Furthermore,~$E(g_0)=E^\infty(g_0)=0$ since~$g_0\in C$, and~$E(f)=E^\infty(f)=\lim_k E^k(f)$ for every~$f\in C$.
The uniqueness of~$E$ follows from Remark~\ref{r:AdditiveConstant}.
Finally, for every~$t>0$ the map~$T_t$ satisfies~$(\mathrm{VI}_t)$ with respect to~$E$, and it is therefore cyclically non-expansive by Lemma~\ref{l:VI}\ref{i:l:VI:2}.

\paragraph{\ref{i:t:FormSemigroup:4}}
Fix~$t\in(0,1]$.
By~\ref{i:t:FormSemigroup:3},~$T_t$ satisfies~$(\mathrm{VI}_t)$ with respect to~$E$, and~$T_t\tparen{T_{1-t}(u_0)}=g_0$ by~\ref{i:d:Semigroup:2}.
By the minimality~\eqref{eq:l:OneStepEVI:2} in Lemma~\ref{l:OneStepEVI}, we conclude that~$\cenv_t(T_t)\leq E-E(g_0)=E$.
\end{proof}

As a consequence, we may complete the characterization of cyclical monotonicity in Theorem~\ref{t:CyclicalTAJ} by conditions on the semigroup alone.

\begin{corollary}\label{c:CyclicalTAJ}
Let~$C$,~$T_\bullet$,~$A$, and~$J_\bullet$ be as in Theorem~\ref{t:CyclicalTAJ}.
Then, each of the conditions~\ref{i:t:CyclicalTAJ:1}--\ref{i:t:CyclicalTAJ:3} in Theorem~\ref{t:CyclicalTAJ} is further equivalent to each of the following:
\begin{enumerate}[$(i)$, start=1]
\item\label{i:c:CyclicalTAJ:4} $T_t$ is cyclically non-expansive for every~$t>0$;
\setcounter{enumi}{4}
\item\label{i:c:CyclicalTAJ:5} $T_{2^{-k}}$ is cyclically non-expansive for every~$k\in\N$.
\end{enumerate}
\end{corollary}

\begin{proof}
With the same numbering as in Theorem~\ref{t:CyclicalTAJ}, we show that~\ref{i:t:CyclicalTAJ:3}$\implies$\ref{i:c:CyclicalTAJ:4}$\implies$\ref{i:c:CyclicalTAJ:5}$\implies$\ref{i:t:CyclicalTAJ:1}.

\paragraph{\ref{i:t:CyclicalTAJ:3}$\implies$\ref{i:c:CyclicalTAJ:4}} If~\ref{i:t:CyclicalTAJ:3} holds, then~$T_\bullet$ satisfies~$(\mathrm{EVI}''_0)$ with respect to~$E$, i.e.~$T_\bullet$ is an~$\mathrm{EVI}_0$-gradient flow of~$E$ in~$C$ by Lemma~\ref{l:EVI}.
Furthermore, for every~$t>0$ the map~$T_t$ satisfies~$(\mathrm{VI}_t)$ with respect to~$E$ by Definition~\ref{d:EVI}, and it is therefore cyclically non-expansive by Lemma~\ref{l:VI}\ref{i:l:VI:2}.

\paragraph{Conclusion}
The implication~\ref{i:c:CyclicalTAJ:4}$\implies$\ref{i:c:CyclicalTAJ:5} is trivial.
Finally,~\ref{i:c:CyclicalTAJ:5} implies~\ref{i:t:CyclicalTAJ:1} by Theorem~\ref{t:FormSemigroup}\ref{i:t:FormSemigroup:3}.
\end{proof}

\begin{remark}\label{r:FormSemigroup}
\begin{enumerate*}[$(a)$]
\item\label{i:r:FormSemigroup:1} The indicator function in Theorem~\ref{t:FormSemigroup}\ref{i:t:FormSemigroup:3} may not be dispensed with.
For instance, if~$C=\set{0}$, then~$T_t=\id_C$, every~$E^k$ vanishes identically on~$H$, and so does~$E^\infty$, whereas~$A=\partial\mbfI_C$.
\item\label{i:r:FormSemigroup:2} The proof of Theorem~\ref{t:FormSemigroup} uses neither the resolvent nor the minimal section~$A^\circ$ of the generator of~$T_\bullet$; the generator only enters through Lemma~\ref{l:EVIGradientFlow}, which identifies~$T_\bullet$ with the semigroup generated by~$\partial E$ by means of the uniqueness of solutions to the EVI in Theorem~\ref{t:MurSav}\ref{i:t:MurSav:2}, and then~$A$ with~$\partial E$ by the uniqueness of the generator in Theorem~\ref{t:Barbu}.
Furthermore, the dyadic times may be replaced by any sequence~$t_k\eqdef 1/n_k$ with~$n_k\in 2\N$,~$n_k\mid n_{k+1}$, and~$n_k<n_{k+1}$ for every~$k$.
\item\label{i:r:FormSemigroup:3} Theorem~\ref{t:FormSemigroup} may be compared with Theorem~\ref{t:FormResolvent}\ref{i:t:FormResolvent:2}.
In the latter, the convergence is a Mosco convergence, which is in general not monotone; see Remark~\ref{r:FormResolvent}.
In the former, the convergence of~$E^k$ to~$E^\infty$ is monotone up to the summable corrections~$r_k$, by Theorem~\ref{t:FormSemigroup}\ref{i:t:FormSemigroup:1}--\ref{i:t:FormSemigroup:2}.
\end{enumerate*}
\end{remark}

\begin{remark}[On cyclical firm non-expansiveness]\label{r:CyclicalFirm}
Following~\cite{BarBauBorReiWan07,BauMofWan11}, an operator~$B\colon C\to H$ is \emph{cyclically firmly non-expansive} if, for every~$n\in\N$ and every~$h_1,\dotsc,h_n\in C$, with~$h_0\eqdef h_n$,
\begin{equation}\label{eq:r:CyclicalFirm:1}
\sum_{i=1}^n \scalar{h_i-B(h_i)}{B(h_i)-B(h_{i-1})}\geq 0 \fstop
\end{equation}
(In~\cite[Eqn.~(99)]{BarBauBorReiWan07} cycles run in the opposite orientation, which is immaterial.)
For~$n=2$, \eqref{eq:r:CyclicalFirm:1} is firm non-expansiveness. Furthermore, a map~$T\colon H\to H$ is cyclically firmly non-expansive if and only if~$T=(\car_H+\partial G)^{-1}$ for some proper convex lower semicontinuous~$G\colon H\to(-\infty,\infty]$, see~\cite[Thm.~6.6]{BarBauBorReiWan07} and~\cite[Thm.~2.1(xvi), Rmk.~2.2]{BauMofWan11}.
In particular, \ref{i:t:CyclicalTAJ:1} in Theorem~\ref{t:CyclicalTAJ} holds if and only if~$J_\lambda$ is cyclically firmly non-expansive for some, equivalently every, $\lambda>0$.

Writing~$h_i-B(h_{i-1})=\tparen{h_i-B(h_i)}+\tparen{B(h_i)-B(h_{i-1})}$ and expanding the square,
\[
\sum_{i=1}^n\norm{h_i-B(h_{i-1})}^2-\sum_{i=1}^n\norm{B(h_i)-B(h_{i-1})}^2
=\sum_{i=1}^n\norm{h_i-B(h_i)}^2+2\sum_{i=1}^n\scalar{h_i-B(h_i)}{B(h_i)-B(h_{i-1})} \fstop
\]
Thus, $B$ is cyclically firmly non-expansive if and only if
\[
\sum_{i=1}^n\norm{B(h_i)-B(h_{i-1})}^2+\sum_{i=1}^n\norm{h_i-B(h_i)}^2\leq\sum_{i=1}^n\norm{h_i-B(h_{i-1})}^2 \comma \qquad h_1,\dotsc,h_n\in C\comma \quad h_0\eqdef h_n\comma
\]
that is, cyclical non-expansiveness (Definition~\ref{d:CyclicalNonExp}) is cyclical firm non-expansiveness without the term $\sum_{i=1}^n\norm{h_i-B(h_i)}^2$.
In particular, every cyclically firmly non-expansive map is cyclically non-expansive.

Dividing by~$t>0$, cyclical non-expansiveness of~$B$ states that the pairs~$\tparen{B(h_i),\tfrac1t(h_i-B(h_i))}$ are cyclically monotone up to the errors~$\eps_i\eqdef\tfrac{1}{2t}\norm{h_i-B(h_i)}^2$, viz.
\[
\sum_{i=1}^n\scalar{B(h_i)-B(h_{i-1})}{\tfrac1t\tparen{h_i-B(h_i)}}\geq -\sum_{i=1}^n\eps_i \comma
\]
the cyclical counterpart of the approximate subgradients in Remark~\ref{r:OneStepEVI}\ref{i:r:OneStepEVI:2}.
Cyclical firm non-expansiveness is the same statement without errors, as for resolvents, where~$\tfrac1\lambda\tparen{f-J_\lambda(f)}\in\partial E\tparen{J_\lambda(f)}$ holds exactly.
%
\end{remark}

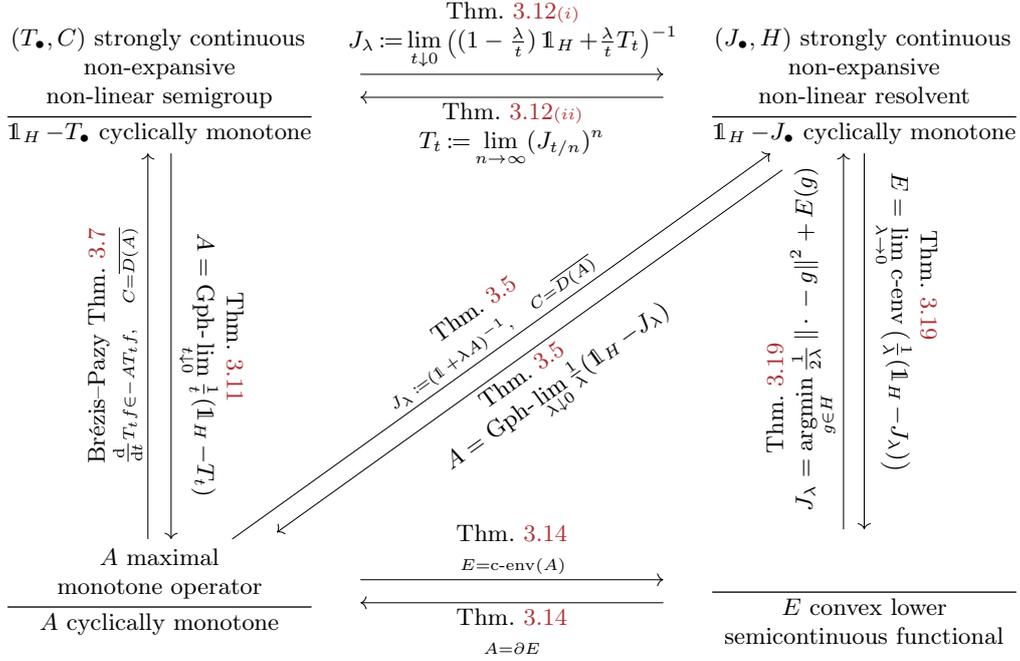
\begin{figure}[htb!]
\begin{adjustbox}{center}
{\small
\begin{tikzcd}[row sep=5cm, column sep=5cm]
\parbox{5cm}{\centering $(T_\bullet,C)$ strongly continuous\\non-expansive\\non-linear semigroup\\ \rule[1.5ex]{4cm}{.5pt}\vspace{-.3cm}\\ $T_\bullet$ cyclically non-expansive}
\arrow[r, shift left=2, "\parbox{5cm}{\centering Thm.~\ref{t:SemigroupResolvent}\ref{i:t:SemigroupResolvent:1}\\$\displaystyle{J_\lambda\eqdef \lim_{t\downarrow 0} \tparen{(1+\tfrac{\lambda}{t})\car_H - \tfrac{\lambda}{t} T_t\circ\cproj_C}^{-1}}$}"]
\arrow[d, shift left=2, "\parbox{3cm}{\centering Thm.~\ref{t:SemigroupGenerator}\\$\displaystyle{A=\Gph\text{-}\lim_{t \downarrow 0} \tfrac{1}{t}(\car_H-T_t\cproj_C)}$}",sloped]
& \parbox{5cm}{\centering $(J_\bullet,H)$ strongly continuous\\non-expansive\\non-linear resolvent\\ \rule[1.5ex]{4cm}{.5pt}\vspace{-.3cm}\\ $\car_H-J_\bullet$ cyclically monotone}
\arrow[l, shift left=2, "\parbox{5cm}{\centering Thm.~\ref{t:SemigroupResolvent}\ref{i:t:SemigroupResolvent:2}\\$\displaystyle{T_t\eqdef \lim_{n\to\infty} (J_{t/n})^n}$}"] \arrow[dl,shorten >=15, sloped, yshift=-2ex, labels=below, "\parbox{3cm}{\centering Thm.~\ref{t:ResolventGenerator}\\$\displaystyle{A=\Gph\text{-}\lim_{\lambda\downarrow 0} \tfrac{1}{\lambda}(\car_H-J_\lambda)}$}"]
\arrow{d}[anchor=center, rotate=-90, yshift=4ex, pos=.35]{\parbox{3cm}{\centering Thm.~\ref{t:FormResolvent}\\ $\displaystyle{E=\lim_{\lambda\to 0} \cenv\tparen{\tfrac{1}{\lambda}(\car_H-J_\lambda)}}$}}
\\
\parbox{5cm}{\centering $A$ maximal\\monotone operator\\ \rule[1.5ex]{4cm}{.5pt}\vspace{-.3cm}\\$A$ cyclically monotone}
\arrow[u,shift left=2,"\parbox{5cm}{\centering Br\'ezis--Pazy Thm.~\ref{t:BrezisPazy}\\$\tfrac{\diff}{\diff t}T_t f\in -AT_tf, \quad C=\overline{\dom{A}}$}",sloped]
\arrow{ur}[shift left=2, sloped]{\parbox{5cm}{\centering Thm.~\ref{t:ResolventGenerator}\\$J_\lambda\eqdef (\car +\lambda A)^{-1}, \quad C=\overline{\dom{A}}$}}
\arrow[r, shift left=2, "\parbox{3cm}{\centering Thm.~\ref{t:MaxMonSubDiff}\\$E=\cenv(A)$}"]
& \parbox{5cm}{\centering \vspace{.7cm}\rule[1.5ex]{4cm}{.5pt}\vspace{-.3cm}\\$E$ convex lower\\semicontinuous functional}
\arrow[l, shift left=2, "\parbox{3cm}{\centering Thm.~\ref{t:MaxMonSubDiff}\\$A=\partial E$}"]
\arrow[u,anchor=center, yshift=2ex, pos=.35, "\parbox{3cm}{\centering Thm.~\ref{t:FormResolvent}\\ $\displaystyle{J_\lambda=\argmin_{g\in H} \tfrac{1}{2\lambda}\norm{\emparg-g}^2+E(g)}$}", sloped]
\end{tikzcd}
}
\end{adjustbox}
\caption{Equivalences between semigroups, operators, resolvents, and functionals: \emph{cyclical monotonicity}.}\label{fig:DiagramMonotone}
\end{figure}

\subsubsection{Equivalent characterizations of evenness, preservation of convexity}
Let us collect here the following additional equivalence, which is not difficult to show.

\begin{proposition}\label{p:EvenTAJE}
Let~$A$ be a cyclically monotone maximal monotone operator on~$H$, with associated semigroup~$T_\bullet\colon \overline{\dom{A}}\to \overline{\dom{A}}$, associated resolvent~$J_\bullet\colon H\to H$, and associated proper convex lower semi-continuous functional~$E\colon H\to (-\infty,\infty]$.
Then, the following are equivalent:
\begin{enumerate}[$(i)$]
\item\label{i:p:EvenTAJE:1} $T_t$ is odd for every~$t\geq 0$;
\item\label{i:p:EvenTAJE:2} $A$ is odd;
\item\label{i:p:EvenTAJE:3} $J_\lambda$ is odd for every~$\lambda\geq 0$;
\item\label{i:p:EvenTAJE:4} $E$ is even.
\end{enumerate}
\end{proposition}

\begin{proof}
Set~$C\eqdef\overline{\dom{A}}$.

\paragraph{\ref{i:p:EvenTAJE:2}$\iff$\ref{i:p:EvenTAJE:3}}
Assume~\ref{i:p:EvenTAJE:2}, fix~$\lambda>0$ and~$f\in H$, and set~$h\eqdef J_\lambda(f)$.
Since~$f-h\in\lambda A(h)$ and~$A$ is odd, we have~$-f-(-h)\in\lambda A(-h)$, that is,~$J_\lambda(-f)=-h$.
Furthermore,~$\dom{A}$ is symmetric, hence so is~$C$, and~$J_0=\cproj_C$ is odd.
Vice versa, assume~\ref{i:p:EvenTAJE:3}, and let~$a\in A(h)$.
Then~$h=J_\lambda(h+\lambda a)$ for every~$\lambda>0$, hence~$-h=J_\lambda(-h-\lambda a)$, that is,~$-a\in A(-h)$.
Applying the latter with~$-h$ in place of~$h$ shows that~$A(-h)=-A(h)$.

\paragraph{\ref{i:p:EvenTAJE:2}$\implies$\ref{i:p:EvenTAJE:1}}
Fix~$f\in\dom{A}$, and set~$u(t)\eqdef -T_t(-f)$ for~$t\geq 0$.
By Theorem~\ref{t:BrezisPazy}, and since~$A$ is odd, the curve~$u$ is Lipschitz,~$u(0)=f$, and~$\tfrac{\diff}{\diff t}u(t)\in A\tparen{T_t(-f)}=-A\tparen{u(t)}$ for~$\Leb^1$-a.e.~$t>0$.
By the uniqueness in Theorem~\ref{t:BrezisPazy}, we have~$u(t)=T_t(f)$ for every~$t\geq 0$.
Since~$C$ is symmetric, the conclusion follows by density of~$\dom{A}$ in~$C$ and by continuity of~$T_t$.

\paragraph{\ref{i:p:EvenTAJE:1}$\implies$\ref{i:p:EvenTAJE:3}}
Since~$T_0=\id_C$ is odd,~$C$ is symmetric, and~$J_0=\cproj_C$ is odd.
Thus, for all~$\lambda,t>0$ the map~$(1+\tfrac{\lambda}{t})\car_H-\tfrac{\lambda}{t}\,T_t\circ\cproj_C$ is odd, and so is its inverse.
The conclusion follows from Theorem~\ref{t:SemigroupResolvent}\ref{i:t:SemigroupResolvent:1}, since pointwise limits of odd maps are odd.

\paragraph{\ref{i:p:EvenTAJE:2}$\iff$\ref{i:p:EvenTAJE:4}}
Set~$F(f)\eqdef E(-f)$ for~$f\in H$.
Then,~$F$ is proper convex lower semicontinuous and, by~\eqref{eq:ConvexSubDiff},~$\partial F(f)=-\partial E(-f)$ for every~$f\in H$.
If~$E$ is even, then~$F=E$, and therefore~$A=\partial E$ is odd.
Vice versa, if~$A$ is odd, then~$\partial F=A=\partial E$, and thus~$F=E+c$ for some~$c\in\R$ by Remark~\ref{r:AdditiveConstant}.
For~$f\in\dom{E}$ we then have~$E(f)=F(-f)=E(-f)+c=F(f)+c=E(f)+2c$, whence~$c=0$, i.e.~$E$ is even.
\end{proof}

\paragraph{Preservation of convex sets}
We summarize known equivalence results relating the preservation of convex sets by generators, semigroups, resolvents, and functionals.

\begin{proposition}\label{p:EquivConvex}
Let~$A$ be a cyclically monotone maximal monotone operator on~$H$, with associated semigroup~$T_\bullet\colon \overline{\dom{A}}\to \overline{\dom{A}}$, associated resolvent~$J_\bullet$, and associated proper convex lower semi-continuous functional~$E\colon H\to (-\infty,\infty]$.
Further let~$C\subset H$ be any non-empty closed convex set.
Then, the following are equivalent:
\begin{enumerate}[$(i)$]
\item\label{i:p:EquivConvex:1} $J_\lambda(C)\subset C$ for every~$\lambda>0$;
\item\label{i:p:EquivConvex:2} $E\tparen{\cproj_C(f)}\leq E(f)$ for every~$f\in H$;
\item\label{i:p:EquivConvex:7} $E_\lambda\tparen{\cproj_C(f)}\leq E_\lambda(f)$ for every~$f\in H$ and every~$\lambda>0$;
\item\label{i:p:EquivConvex:8} $\scalar{A_\lambda(f)}{f-\cproj_C(f)}\geq 0$ for every~$f\in H$ and every~$\lambda>0$.
\end{enumerate}

Furthermore, either of~\ref{i:p:EquivConvex:1}-\ref{i:p:EquivConvex:8} implies each of the following:
\begin{enumerate}[$(i)$]\setcounter{enumi}{4}
\item\label{i:p:EquivConvex:3} $\scalar{A^\circ f}{f-\cproj_C(f)}\geq 0$ for every~$f\in \dom{A}$;
\item\label{i:p:EquivConvex:4} $\dist\tparen{T_t(f), C}\leq \dist(f,C)$ for every~$f\in\overline{\dom{A}}$ and every~$t\geq 0$;
\item\label{i:p:EquivConvex:5} $A+\partial \mbfI_C$ is maximal monotone, $\overline{\dom{A}\cap C}=\overline{\dom{A}}\cap C$, and
\[
(A+\partial \mbfI_C)^\circ(f)= A^\circ(f)\comma \quad f\in\dom{A}\cap C \semicolon
\]
\item\label{i:p:EquivConvex:6} $T_t\tparen{\overline{\dom{A}}\cap C}\subset C$ for every~$t\geq 0$.
\end{enumerate}
Moreover,~\ref{i:p:EquivConvex:3}$\iff$\ref{i:p:EquivConvex:4}, and each of~\ref{i:p:EquivConvex:4} and~\ref{i:p:EquivConvex:5} implies~\ref{i:p:EquivConvex:6}.

Finally, either of~\ref{i:p:EquivConvex:1}-\ref{i:p:EquivConvex:8} implies
\begin{equation}\label{eq:p:EquivConvex:0}
\cproj_{\overline{\dom{A}}} C\subset C \comma
\end{equation}
and, if~$C$ is additionally such that~\eqref{eq:p:EquivConvex:0} holds, then~\ref{i:p:EquivConvex:6} implies~\ref{i:p:EquivConvex:1}.
In particular, if~\eqref{eq:p:EquivConvex:0} holds, then all of~\ref{i:p:EquivConvex:1}-\ref{i:p:EquivConvex:6} are equivalent to each other.
\end{proposition}

\begin{remark}
\begin{enumerate*}[$(a)$]
\item The equivalence for generators, semigroups, and resolvents under~\eqref{eq:p:EquivConvex:0} is proved in~\cite[Prop.~4.5, p.~131]{Bre73} and cyclical monotonicity is in fact not necessary to its validity.
The equivalence between the assertions for semigroups and functionals is an adaptation of~\cite{CipGri03}.
\item We state~\ref{i:p:EquivConvex:1} for~$\lambda>0$ in order to stress that~\eqref{eq:p:EquivConvex:0} is used in the implication~\ref{i:p:EquivConvex:6}$\implies$\ref{i:p:EquivConvex:1};
however,~\eqref{eq:p:EquivConvex:0} coincides with~\ref{i:p:EquivConvex:1} for~$\lambda=0$, and it is implied by~\ref{i:p:EquivConvex:1} for~$\lambda>0$.
\item Further note that~\eqref{eq:p:EquivConvex:0} can be equivalently stated by replacing~$\overline{\dom{A}}$ with either~$\overline{\dom{E}}$ or the domain of the semigroup~$T_\bullet$ since these three sets in fact coincide.
\item The condition~\eqref{eq:p:EquivConvex:0} cannot be dispensed with in the implication~\ref{i:p:EquivConvex:6}$\implies$\ref{i:p:EquivConvex:1}.
Indeed, let~$H\eqdef\R^2$, $K\eqdef\R\times\set{0}$, $E\eqdef\mbfI_K$, and~$C\eqdef\set{(s,s): s\in\R}$.
Then,~$A^\circ\equiv 0$ on~$K$ and~$T_t=\id_K$, so that~\ref{i:p:EquivConvex:3},~\ref{i:p:EquivConvex:4} and~\ref{i:p:EquivConvex:6} hold trivially, and~\ref{i:p:EquivConvex:5} holds since~$K\cap C=\set{0}$ and~$A+\partial\mbfI_C=\partial\mbfI_{\set{0}}$;
however,~$J_\lambda=\cproj_K$ and~$J_\lambda(1,1)=(1,0)\notin C$.
\end{enumerate*}
\end{remark}

\begin{proof}[Proof of Proposition~\ref{p:EquivConvex}]
Under the assumption of~\eqref{eq:p:EquivConvex:0}, the equivalences
\begin{equation}\label{eq:p:EquivConvex:10}
\text{\ref{i:p:EquivConvex:1}$\iff$\ref{i:p:EquivConvex:3}$\iff$\ref{i:p:EquivConvex:4}$\iff$\ref{i:p:EquivConvex:5}$\iff$\ref{i:p:EquivConvex:6} }
\end{equation}
can be found in~\cite[Prop.~4.5, p.~131]{Bre73}.
The equivalence of~\ref{i:p:EquivConvex:2}$\iff$\ref{i:p:EquivConvex:6} is shown in~\cite[Thm.~3.4, p.~210]{CipGri03}.
Whereas not explicitly stated, it is apparent from the proof in~\cite{CipGri03} ---in particular from the application there of the equivalence~\eqref{eq:p:EquivConvex:10} in~\cite{Bre73}--- that the authors assume~$\overline{\dom{A}}=H$, so that~\eqref{eq:p:EquivConvex:0} is immediately satisfied.

\paragraph{\ref{i:p:EquivConvex:1}$\iff$\ref{i:p:EquivConvex:2}$\iff$\ref{i:p:EquivConvex:7}$\iff$\ref{i:p:EquivConvex:8}}
The first part of the proof of~\cite[Thm.~3.4, p.~210]{CipGri03}, which shows the implication~\ref{i:p:EquivConvex:2}$\implies$\ref{i:p:EquivConvex:1}, goes through also in the case~$\overline{\dom{A}}\subsetneq H$ and without assuming~\eqref{eq:p:EquivConvex:0}.
We show the converse implication, again without assuming~\eqref{eq:p:EquivConvex:0}, through the chain of implications
\[
\text{\ref{i:p:EquivConvex:1}} \implies \text{\ref{i:p:EquivConvex:8}} \implies \text{\ref{i:p:EquivConvex:7}} \implies \text{\ref{i:p:EquivConvex:2}} \fstop
\]
Assume~\ref{i:p:EquivConvex:1}, fix~$f\in H$ and~$\lambda>0$, and set~$p\eqdef\cproj_C(f)$. Then,
\[
\lambda\scalar{A_\lambda(f)}{f-p} = \norm{f-p}^2 + \scalar{p-J_\lambda(p)}{f-p} + \scalar{J_\lambda(p)-J_\lambda(f)}{f-p} \fstop
\]
Since~$J_\lambda(p)\in C$ by~\ref{i:p:EquivConvex:1}, the second summand on the right-hand side is non-negative by the variational characterization of the convex projection~$p=\cproj_C(f)$, viz.~$\scalar{f-p}{g-p}\leq 0$ for every $g\in C$.
By the Cauchy--Schwarz inequality and by non-expansiveness of~$J_\lambda$, the third summand is bounded below by~$-\norm{f-p}^2$, which shows~\ref{i:p:EquivConvex:8}.
Assume now~\ref{i:p:EquivConvex:8}, fix~$f\in H$ and~$\lambda>0$, set~$p\eqdef\cproj_C(f)$, and~$g_s\eqdef p+s(f-p)$ for~$s\in[0,1]$.
By~\eqref{eq:YosidaForm}, and since~$A_\lambda$ is continuous, the function~$\varphi\colon s\mapsto E_\lambda(g_s)$ is continuously differentiable on~$[0,1]$ with~$\varphi'(s)=\scalar{A_\lambda(g_s)}{f-p}$.
For~$s\in(0,1]$ we have~$\cproj_C(g_s)=p$, since~$\scalar{g_s-p}{g-p}=s\scalar{f-p}{g-p}\leq 0$ for every~$g\in C$; hence~$g_s-\cproj_C(g_s)=s(f-p)$, and~\ref{i:p:EquivConvex:8} applied to~$g_s$ gives
\[
\varphi'(s)=s^{-1}\scalar{A_\lambda(g_s)}{g_s-\cproj_C(g_s)}\geq 0 \comma \qquad s\in(0,1] \fstop
\]
Therefore,~$E_\lambda(f)-E_\lambda\tparen{\cproj_C(f)}=\varphi(1)-\varphi(0)\geq 0$, which shows~\ref{i:p:EquivConvex:7}.
Finally,~\ref{i:p:EquivConvex:7} implies~\ref{i:p:EquivConvex:2} letting~$\lambda\downarrow 0$ and applying~\cite[Prop.~2.11]{Bre73}.

The implication~\ref{i:p:EquivConvex:2}$\implies$\ref{i:p:EquivConvex:7} may also be verified directly.
Indeed, assume~\ref{i:p:EquivConvex:2}.
Respectively by definition of the Yosida regularization~$E_\lambda\colon H\to (-\infty,\infty]$, by restriction of the minimization to the range of~$\cproj_C$, by non-expansiveness of the convex projection~$\cproj_C$ and by~\ref{i:p:EquivConvex:2},
\begin{align*}
E_\lambda\tparen{\cproj_C (f)} &= \min_{g\in H} \tfrac{1}{2\lambda} \norm{\cproj_C(f)-g}^2 + E(g)
\\
&\leq \min_{\substack{g= \cproj_C(h)\\ h\in H}} \tfrac{1}{2\lambda} \norm{\cproj_C(f)-\cproj_C(h)}^2 + E(\cproj_C(h))
\\
&\leq \min_{\substack{g= \cproj_C(h)\\ h\in H}} \tfrac{1}{2\lambda} \norm{f-h}^2 + E(h)
\\
&= \min_{\substack{h\in H}} \tfrac{1}{2\lambda} \norm{f-h}^2 + E(h) \defeq E_\lambda(f) \comma
\end{align*}
which shows~\ref{i:p:EquivConvex:7}.

\paragraph{\ref{i:p:EquivConvex:3}$\iff$\ref{i:p:EquivConvex:4}, and~\ref{i:p:EquivConvex:4}$\implies$\ref{i:p:EquivConvex:6}}
Fix~$f\in\dom{A}$, and set~$u(t)\eqdef T_t(f)$ and~$h(t)\eqdef \tfrac{1}{2}\dist\tparen{u(t),C}^2$ for~$t\geq 0$.
Since the resolvent of~$\partial\mbfI_C$ is~$\cproj_C$, the function~$\tfrac{1}{2}\dist(\emparg,C)^2$ is the Yosida regularization~$(\mbfI_C)_1$ in the sense of Definition~\ref{d:YosidaRegForm}, and it is Fr\'echet differentiable with gradient~$\car_H-\cproj_C$ by~\eqref{eq:YosidaForm} with~$\mbfI_C$ in place of~$E$.
By Theorem~\ref{t:BrezisPazy}, the curve~$u$ is Lipschitz, takes values in~$\dom{A}$, and has right derivative~$-A^\circ\tparen{u(t)}$ at every~$t\geq 0$.
As a consequence,~$h$ is locally Lipschitz and right-differentiable at every~$t\geq 0$, with
\begin{equation}\label{eq:p:EquivConvex:11}
\tfrac{\diff^+}{\diff t}h(t)=-\scalar{A^\circ\tparen{u(t)}}{u(t)-\cproj_C\tparen{u(t)}}\comma \qquad t\geq 0\fstop
\end{equation}
If~\ref{i:p:EquivConvex:3} holds, then the right-hand side of~\eqref{eq:p:EquivConvex:11} is non-positive, and~$h$ is non-increasing, since it is locally Lipschitz and its derivative coincides a.e.\ with its right derivative.
That is,~$\dist\tparen{T_t(f),C}\leq\dist(f,C)$ for every~$f\in\dom{A}$ and every~$t\geq 0$, and~\ref{i:p:EquivConvex:4} follows by non-expansiveness of~$T_t$ and $1$-Lipschitz continuity of~$\dist(\emparg,C)$.
Vice versa, if~\ref{i:p:EquivConvex:4} holds, then~$h(t)\leq h(0)$ for every~$t\geq 0$, hence~$\tfrac{\diff^+}{\diff t}h(0)\leq 0$, which is~\ref{i:p:EquivConvex:3} by~\eqref{eq:p:EquivConvex:11} with~$t=0$.
The implication~\ref{i:p:EquivConvex:4}$\implies$\ref{i:p:EquivConvex:6} is immediate, since~$\dist(f,C)=0$ for every~$f\in C$ and~$C$ is closed.

\paragraph{\ref{i:p:EquivConvex:5}$\implies$\ref{i:p:EquivConvex:6}}
Assume~\ref{i:p:EquivConvex:5}, set~$B\eqdef A+\partial\mbfI_C$, so that~$\dom{B}=\dom{A}\cap C$, and let~$S_\bullet$ be the semigroup of~$B$.
Fix~$f\in\dom{B}$, and set~$v(t)\eqdef S_t(f)$ for~$t\geq 0$.
By Theorem~\ref{t:BrezisPazy} with~$B$ in place of~$A$, the curve~$v$ is Lipschitz, takes values in~$\dom{B}\subset\dom{A}$, and has right derivative~$-B^\circ\tparen{v(t)}=-A^\circ\tparen{v(t)}$ at every~$t\geq 0$, where the equality holds by~\ref{i:p:EquivConvex:5}.
In particular,~$\tfrac{\diff}{\diff t}v(t)=-A^\circ\tparen{v(t)}\in -A\tparen{v(t)}$ for~$\Leb^1$-a.e.~$t>0$, and thus~$v(t)=T_t(f)$ for every~$t\geq 0$ by the uniqueness in Theorem~\ref{t:BrezisPazy}.
That is,~$T_t(f)=S_t(f)\in\dom{B}\subset C$ for every~$f\in\dom{A}\cap C$ and every~$t\geq 0$, and~\ref{i:p:EquivConvex:6} follows since~$\overline{\dom{A}\cap C}=\overline{\dom{A}}\cap C$ by~\ref{i:p:EquivConvex:5}, since~$T_t$ is continuous, and since~$C$ is closed.

\paragraph{\ref{i:p:EquivConvex:1}$\implies$\eqref{eq:p:EquivConvex:0}}
Assume~\ref{i:p:EquivConvex:1}, and fix~$f\in C$.
Then,~$J_\lambda(f)\in C$ for every~$\lambda>0$, and~$\lim_{\lambda\downarrow 0}J_\lambda(f)=\cproj_{\overline{\dom{A}}}(f)$ by~\cite[Thm.~2.2, p.~27]{Bre73}.
Since~$C$ is closed, we conclude that~$\cproj_{\overline{\dom{A}}}(f)\in C$, which shows~\eqref{eq:p:EquivConvex:0}.

\paragraph{Either of~\ref{i:p:EquivConvex:1}-\ref{i:p:EquivConvex:8} implies each of~\ref{i:p:EquivConvex:3}-\ref{i:p:EquivConvex:6}}
By the equivalence of~\ref{i:p:EquivConvex:1}-\ref{i:p:EquivConvex:8} shown above, it suffices to show that~\ref{i:p:EquivConvex:8} implies each of~\ref{i:p:EquivConvex:3},~\ref{i:p:EquivConvex:4},~\ref{i:p:EquivConvex:6}, and that~\ref{i:p:EquivConvex:1} implies~\ref{i:p:EquivConvex:5}.
The implication~\ref{i:p:EquivConvex:8}$\implies$\ref{i:p:EquivConvex:3} follows from the strong convergence of~$A_\lambda$ to~$A^\circ$ as~$\lambda\to 0$, see~\cite[Prop.~2.6(iii), p.~28]{Bre73}, and then~\ref{i:p:EquivConvex:4} and~\ref{i:p:EquivConvex:6} follow by the implications~\ref{i:p:EquivConvex:3}$\implies$\ref{i:p:EquivConvex:4}$\implies$\ref{i:p:EquivConvex:6} shown above.
If~\ref{i:p:EquivConvex:1} holds, then~\eqref{eq:p:EquivConvex:0} holds as shown above, and~\ref{i:p:EquivConvex:5} follows by~\eqref{eq:p:EquivConvex:10}.

\paragraph{\ref{i:p:EquivConvex:6}$\implies$\ref{i:p:EquivConvex:1} under~\eqref{eq:p:EquivConvex:0}}
This is part of~\eqref{eq:p:EquivConvex:10}.
\end{proof}

\section{Sub-Markovianity}\label{s:SubMarkov}
Let~$(X,\mfA,\mssm)$ be a measure space, and denote by~$L^p_\mssm$ the usual \emph{Lebesgue space} of~$\mssm$ with exponent~$p\in [1,\infty]$.
For ease of exposition, we will henceforth assume that~$\mssm$ be non-zero and $\sigma$-finite, and that~$\mfA$ be $\mssm$-essentially countably generated.
This assumption is equivalent to the \emph{separability} of~$L^p_\mssm$ for some (hence every)~$p\in [1,\infty)$.

We are interested in properties of a strongly continuous non-expansive non-linear semigroup~$T_\bullet\colon C\to L^2_\mssm$ with respect to the natural partial order on~$L^2_\mssm$.
In particular, given~$u,v\in C$ with~$u\leq v$ we need to compare~$T_t u$ with~$T_t v$.

\paragraph{Strips and projections}
In the following, we will apply the equivalence results in Proposition~\ref{p:EquivConvex} to the double of functionals, semigroups, etc., to convex sets of a particular form.

We shall often consider the following subsets of~$L^2_\mssm \oplus L^2_\mssm$: 
\begin{align*}
C_\tleq \eqdef \set{(u,v) \in L^2_\mssm \oplus L^2_\mssm : u \leq v \as{\mssm}}\comma
\end{align*}
and, for all~$\alpha \in [0,\infty]$,
\begin{align*}
C_\alpha^+ &\eqdef \set{(u,v) \in L^2_\mssm \oplus L^2_\mssm : u-v \leq \alpha \as{\mssm}}\comma
\\
C_\alpha^- &\eqdef \set{(u,v) \in L^2_\mssm \oplus L^2_\mssm : u-v \geq -\alpha \as{\mssm}}\comma
\\
C_\alpha &\eqdef \set{ (u,v) \in L^2_\mssm \oplus L^2_\mssm : \abs{u-v} \leq \alpha \as{\mssm} } \fstop
\end{align*}
A simple yet key observation is that
\[
C_\alpha= C_\alpha^+\cap C_\alpha^- \fstop
\]

Note that~$C_\tleq$,~$C_\alpha$, and~$C_\alpha^\pm$ are convex and closed in $L^2_\mssm \oplus L^2_\mssm$, for each~$\alpha \geq 0$.
Further note that, for every~$\alpha\geq 0$, for every~$(u,v)\in C_\alpha$, we have~$u-v\in L^\infty_\mssm$ and, if either~$u$ or~$v\in L^\infty_\mssm$, then~$(u,v)\in L^\infty_\mssm \oplus L^\infty_\mssm$.
As a consequence,
\begin{equation}\label{eq:CAlphaRmk}
\inf\set{\alpha\geq 0 : (u,v)\in C_\alpha} = \norm{u-v}_{L^\infty_\mssm} \fstop
\end{equation}

By~\cite[Lem.~3.3]{CipGri03}, the projection~$\cproj_\tleq$ onto~$C_\tleq$, resp.~$\cproj_\alpha$ onto~$C_\alpha$, have the following explicit \emph{pointwise} expression. For each~$(u,v)\in L^2_\mssm\oplus L^2_\mssm$,
\begin{align}\label{eq:StripProjections}
\begin{aligned}
\cproj_\tleq(u,v) =& \tparen{u-\tfrac{1}{2}(u-v)_+, v+\tfrac{1}{2}(u-v)_+}\comma
\\
\cproj_\alpha(u,v) =& \begin{cases}\tfrac{1}{2}(u+v-\alpha,u+v+\alpha) &\text{if } u-v<-\alpha\comma \\ (u,v) & \text{if } \abs{u-v}\leq \alpha\comma \\ \tfrac{1}{2}(u+v+\alpha,u+v-\alpha) &\text{if } u-v>\alpha\fstop \end{cases}
\end{aligned}
\end{align}
Analogously, the projections~$\cproj_\alpha^\pm$ onto~$C_\alpha^\pm$ have the explicit \emph{pointwise} expression
\begin{equation}\label{eq:StripProjections2}
\cproj_\alpha^\pm(u,v) = \begin{cases} (u,v) &\text{if } u-v \lesseqgtr \pm \alpha\comma \\ \tfrac{1}{2}(u+v\pm\alpha,u+v\mp\alpha) &\text{if } u-v \gtrless \pm \alpha\fstop \end{cases}
\end{equation}

\begin{lemma}\label{l:CompositionPM}
For every~$\alpha\in [0,\infty]$,
\[
\cproj_\alpha^+ \circ \cproj_\alpha^- = \cproj_\alpha =  \cproj_\alpha^- \circ  \cproj_\alpha^+ \fstop
\]
\begin{proof}
Straightforward from~\eqref{eq:StripProjections} and~\eqref{eq:StripProjections2}.
\end{proof}
\end{lemma}

\subsection{Sub-Markovian operators}
Let us now introduce the notion of \emph{sub-Markov\-ianity} for non-linear operators.

\begin{definition}[Sub-Markovianity]
An operator~$B\colon D\to L^2_\mssm$ is
\begin{enumerate}[$(M_1)$]
\item\label{i:d:Semigroups:5} \emph{order-preserving} if
\[
u\leq v \implies B(u) \leq B(v)\comma\qquad u,v\in D \fstop
\]
\item\label{i:d:Semigroups:6} \emph{$L^\infty$-non-expansive} if
\begin{equation}\label{eq:i:d:Semigroups:6}
\norm{B(u)-B(v)}_{L^\infty_\mssm}\leq \norm{u-v}_{L^\infty_\mssm} \comma \qquad u,v\in D\fstop
\end{equation}
\item\label{i:d:Semigroups:7} \emph{sub-Markovian} if it is both order-preserving and $L^\infty$-non-expansive.
\end{enumerate}

Further say that a family of operators~$\set{B_j}_{j\in J}$ is \emph{order-preserving}, resp.\ \emph{$L^\infty_\mssm$-non-expansive}, \emph{sub-Markovian}, if~$B_j$ is order-preserving, resp.\ $L^\infty_\mssm$-non-expansive, sub-Markovian, for every~$j\in J$.
\end{definition}

\begin{remark}
Our definition extends those in~\cite{CipGri03} and~\cite{BriHar22} in that we allow for the convex set~$D$ to be smaller than the whole space~$L^2_\mssm$.
This will allow us to consider the semigroup associated to the form in Example~\ref{e:FormBall} below as part of our framework.
\end{remark}

For any convex~$D\subset L^2_\mssm$ we write
\[
D^\opl{2}\eqdef D\oplus D \fstop
\]
In the following, it will be convenient to rephrase properties~\ref{i:d:Semigroups:5}-\ref{i:d:Semigroups:7} in terms of the action of the double operator~$(B^\opl{2},D^\opl{2})$ on some convex set~$C\subset L^2_\mssm \oplus L^2_\mssm$.

The next result is a minor extension of the analogous statement for everywhere-defined semigroups (i.e.~$T_\bullet\colon L^2_\mssm\to L^2_\mssm$) proved in~\cite[Lem.~2.6]{CipGri03}.

\begin{proposition}\label{p:SubMarkovianConvex}
For an operator~$B\colon D\to L^2_\mssm$,
\begin{enumerate}[$(i)$, wide]
\item\label{i:p:SubMarkovianConvex:1} $B$ is order-preserving~\ref{i:d:Semigroups:5} if and only if
\begin{equation}\tag*{$(M_1')$}\label{i:d:Semigroups:5'}
B^\opl{2}(D^\opl{2}\cap C_\tleq)\subset C_\tleq \semicolon
\end{equation}
\item\label{i:p:SubMarkovianConvex:2} $B$ is $L^\infty$-non-expansive~\ref{i:d:Semigroups:6} if and only if
\begin{equation}\tag*{$(M_2')$}\label{i:d:Semigroups:6'}
B^\opl{2}(D^\opl{2}\cap C_\alpha)\subset C_\alpha\comma \qquad \alpha \geq 0 \semicolon
\end{equation}
\item $B$ is sub-Markovian~\ref{i:d:Semigroups:7} if and only if
\begin{equation}\tag*{$(M_3')$}\label{i:d:Semigroups:7'}
\text{both~\ref{i:d:Semigroups:5'} and~\ref{i:d:Semigroups:6'} hold.}
\end{equation}
\end{enumerate}
\end{proposition}

\begin{proof}

Assume~\ref{i:d:Semigroups:5} holds, and let~$(u,v)\in D^\opl{2}\cap C_\tleq$.
Then,~$B(u)\leq B(v)$ by~\ref{i:d:Semigroups:5}, and therefore~$\tparen{B(u), B(v)}\in C_\tleq$, which shows~\ref{i:d:Semigroups:5'}.

Vice versa, assume~\ref{i:d:Semigroups:5'} holds, and let~$u,v\in D$ be so that~$u\leq v$, that is, $(u,v)\in D^\opl{2}\cap C_\tleq$.
Then,~$\tparen{B(u), B(v)}\in C_\tleq$ by~\ref{i:d:Semigroups:5'}, and therefore~$B(u)\leq B(v)$ by definition of~$C_\tleq$, which shows~\ref{i:d:Semigroups:5}.

Assume now~\ref{i:d:Semigroups:6} holds, and let~$(u,v)\in D^\opl{2}\cap C_\alpha$.
Then, by~\ref{i:d:Semigroups:6} and~\eqref{eq:CAlphaRmk},
\[
\norm{B(u)- B(v)}_{L^\infty_\mssm}\leq \norm{u-v}_{L^\infty_\mssm}\leq\alpha\comma
\]
and therefore~$\tparen{B(u), B (v)}\in C_\alpha$.

Vice versa, assume~\ref{i:d:Semigroups:6'} holds, let~$u,v\in D$.
If~$\norm{u-v}_{L^\infty_\mssm}=\infty$, then~\eqref{eq:i:d:Semigroups:6} is trivially satisfied.
Otherwise, set~$\alpha\eqdef \norm{u-v}_{L^\infty}$, whence $(u,v)\in D^\opl{2}\cap C_\alpha$.
Then,~$\tparen{B(u), B(v)}\in C_\alpha$ by~\ref{i:d:Semigroups:6'}, and therefore
\[
\norm{B(u)- B(v)}_{L^\infty_\mssm}\leq \alpha=\norm{u-v}_{L^\infty_\mssm}
\]
by definition of~$C_\alpha$, which shows~\ref{i:d:Semigroups:6}.

The equivalence~\ref{i:d:Semigroups:7}$\iff$\ref{i:d:Semigroups:7'} follows combining the previous ones.
\end{proof}

\begin{remark}[Linear case]\label{r:SubMarkovUnivariate}
The terminology `sub-Markovianity' is motivated by analogy with the linear case.
We recall that a \emph{linear} operator $B\colon L\to L^2_\mssm$, defined on a linear subspace~$L\subset L^2_\mssm$, is sub-Markovian if and only if (by definition)
\[
0\leq u\leq 1 \as{\mssm} \quad \implies \quad 0\leq B u \leq 1 \as{\mssm} \fstop
\]
In particular, in the linear case, sub-Markovianity is a \emph{uni}variate condition.
It is worth noting that, in the \emph{non-linear} case, this is no longer true. 
Indeed, in light of Proposition~\ref{p:SubMarkovianConvex}, the sub-Markovianity condition for non-linear operators is \emph{bi}variate.

For the same reason, the univariate sub-Markovianity condition considered in~\cite[\S3.2.1, p.~34]{vBe94} is weaker than the one ---namely~\ref{i:d:Semigroups:7}--- considered here.
Additionally, let us note that the condition in~\cite[\S3.2.1, p.~34]{vBe94} is given only on a positive strip~$0\leq u \leq k$ and is not suitable for functionals that are not even.
\end{remark}

\subsection{Dirichlet generators}
Let us now rephrase sub-Markovianity in terms of the generators.

\begin{definition}[Dirichlet generators]
A maximal monotone operator~$A$ on~$L^2_\mssm$ is called a \emph{Dirichlet operator} if it satisfies both
\begin{enumerate}[$(\protect{A}_{\arabic*}')$]\setcounter{enumi}{6}
\item\label{i:d:Generator:7'} 
$\scalar{{A^\circ}^\opl{2}(w)}{w-\cproj_\tleq(w)}_{L_\mssm^2\oplus L^2_\mssm}\geq 0$ for all~$w\in \dom{A^\opl{2}}$;

\item\label{i:d:Generator:8'} $\scalar{{A^\circ}^\opl{2}(w)}{w-\cproj_\alpha(w)}_{L_\mssm^2\oplus L_\mssm^2}\geq 0$ for all~$w\in \dom{A^\opl{2}}$ for every~$\alpha\geq 0$.
\end{enumerate}
\end{definition}

It will be clear from Theorem~\ref{t:Main} below that properties~\ref{i:d:Generator:7'} and~\ref{i:d:Generator:8'} ought to be understood as the analogues for generators of properties~\ref{i:d:Semigroups:5'} and~\ref{i:d:Semigroups:6'} for the corresponding semigroups (as opposed to:~\ref{i:d:Semigroups:5} and~\ref{i:d:Semigroups:6}).
In particular ---as hinted to by the notation--- this breaks the symmetry of our presentation in that we have not yet introduced properties~\ref{i:d:Generator:7} and~\ref{i:d:Generator:8}.

Our choice is motivated by the fact that these properties are seemingly somewhat less informative when compared with~\ref{i:d:Semigroups:5} and~\ref{i:d:Semigroups:6}.
For instance, by substituting the expression~\eqref{eq:StripProjections} for~$\cproj_\tleq$ in~\ref{i:d:Generator:7'}, it is not difficult to show that~\ref{i:d:Generator:7'} is equivalent to
\begin{equation}\tag*{$(A_7)$}\label{i:d:Generator:7}
\scalar{A^\circ(u)-A^\circ(v)}{(u-v)_+} \geq 0 \comma \qquad u,v\in\dom{A}
\fstop
\end{equation}

Analogously, by substituting the expression~\eqref{eq:StripProjections} for~$\cproj_\alpha$ in~\ref{i:d:Generator:8'}, property~\ref{i:d:Generator:8'} is equivalent to
\begin{equation}\tag*{$(A_8)$}\label{i:d:Generator:8}
\scalar{A^\circ(u)-A^\circ(v)}{(u-v-\alpha)_+ - (v-u-\alpha)_+} \geq 0 \comma \qquad u,v\in\dom{A} \comma \quad \alpha\geq 0
\fstop
\end{equation}

\begin{remark}[Linear case]
The terminology `Dirichlet operator' is motivated by analogy with the linear case.
We recall that a closed densely defined \emph{linear} operator $A$ on~$L^2_\mssm$ is a Dirichlet operator if (by definition)
\[
\ttscalar{Au}{(u-\car)_+}\geq 0 \comma \qquad u\in\dom{A} \fstop
\]
As already noted for sub-Markovianity in Remark~\ref{r:SubMarkovUnivariate}, being a Dirichlet operator is, in the linear case, a \emph{uni}variate condition.
In the \emph{non-linear} case, this is no longer true, and the Dirichlet condition for a (non-linear) maximal monotone operator is \emph{bi}variate.
\end{remark}

\subsection{Dirichlet functionals}
In this section, we introduce and study the Markov property for functionals.
For reasons that will be apparent in the following, it is convenient to address Markovianity separately from convexity and lower semicontinuity.

\subsubsection{Some special contractions}
In the following we will make use of different types of Lipschitz contractions.
Among all possible contractions, the following one-parameter families of bivariate Lipschitz contractions play a special role.

For every~$\alpha\geq 0$, let~$H_\alpha\colon \R^\tym{2}\to\R$ be defined by
\[
H_\alpha(s,t)\eqdef (t-\alpha) \vee s \wedge (t+\alpha) =\begin{cases} t-\alpha &\text{if } s-t <-\alpha \\ s &\text{if } -\alpha\leq s-t\leq \alpha \\
t+\alpha &\text{if } s-t>\alpha\end{cases}\fstop
\]

\begin{figure}[htb!]
\centering
\includegraphics[scale=.6]{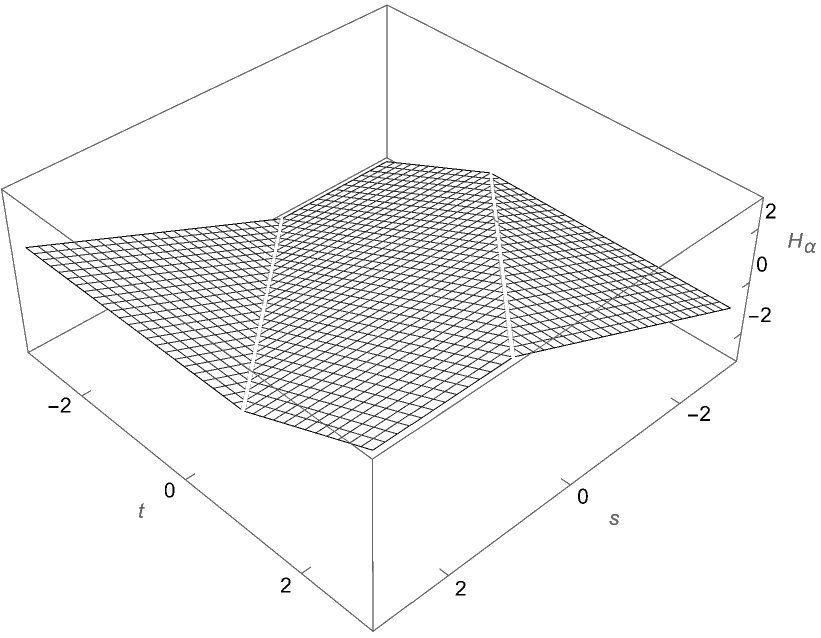}
\caption{The function~$H_\alpha$ for~$\alpha=2$.}
\end{figure}

For every~$\alpha\geq 0$, let~$h_\alpha^\pm, k_\alpha^\pm\colon \R^\tym{2}\to\R$ be defined by
\begin{align*}
h_\alpha^-(s,t)&\eqdef s \vee (t-\alpha) \comma & k_\alpha^-(s,t)&\eqdef t\wedge (s+\alpha)\comma
\\
h_\alpha^+(s,t)&\eqdef s \wedge (t+\alpha)\comma & k_\alpha^+(s,t)&\eqdef t \vee (s-\alpha)\comma
\end{align*}
so that
\begin{equation}\label{eq:HKpm}
h_\alpha^\pm(s,t)= \begin{cases} t\pm \alpha &\text{if } s-t \gtrless \pm \alpha \\ s &\text{if } s-t \lesseqgtr \pm\alpha\end{cases} \qquad \text{and}\qquad k_\alpha^\pm (s,t) = h_\alpha^\mp(t,s)\fstop
\end{equation}

\subsubsection{Dirichlet functionals}
We are now ready to introduce several notions pertaining to the one of \emph{Dirichlet functional}.
It is worth stating the definitions assuming \emph{neither} lower semicontinuity \emph{nor} convexity.
We start with properties of the double~$E^\opl{2}$ of a functional~$E\colon L^2_\mssm\to (-\infty,\infty]$.

\begin{definition}[Properties of doubles]\label{d:Dirichlet}
A functional~$E\colon L^2_\mssm\to (-\infty,\infty]$ is
\begin{enumerate}[$(E_{\arabic*})$]\setcounter{enumi}{6}
\item\label{i:d:Form:7} \emph{$\cproj_\tleq$-non-expansive} if its double~$E^\opl{2}\colon L^2_\mssm\oplus L^2_\mssm\to (-\infty,\infty]$ satisfies
\begin{equation}\label{eq:i:d:Form:10}
E^\opl{2}(\cproj_\tleq w) \leq E^\opl{2}(w)\comma \qquad w \in L^2_\mssm\oplus L^2_\mssm \semicolon
\end{equation}

\item\label{i:d:Form:8} \emph{$\cproj_\alpha$-non-expansive} if its double~$E^\opl{2}\colon L^2_\mssm\oplus L^2_\mssm\to (-\infty,\infty]$ satisfies
\begin{equation}\label{eq:i:d:Form:8}
E^\opl{2}(\cproj_\alpha w) \leq E^\opl{2}(w)\comma \qquad w \in L^2_\mssm\oplus L^2_\mssm \comma \qquad \alpha\geq 0 \semicolon
\end{equation}

\item\label{i:d:Form:9} \emph{Markovian} if it satisfies both~\ref{i:d:Form:7} and~\ref{i:d:Form:8};
\end{enumerate}
\end{definition}

The properties of the double functional listed above have the following counterparts on pairs.

\begin{definition}[Properties of pairs]
A functional~$E\colon L^2_\mssm\to (-\infty,\infty]$
\begin{enumerate}[$(E_{\arabic*}')$]\setcounter{enumi}{6}
\item\label{i:d:Form:7'} has the \emph{lattice contraction property} if for every~$u,v\in \dom{E}$,
\[
u\vee v\comma u\wedge v \in\dom{E} \qquad\text{and}\qquad E(u\vee v) + E(u\wedge v) \leq E(u)+E(v) \semicolon
\]
\item\label{i:d:Form:8'} has the \emph{strong unit-contraction property} if for every~$u,v\in \dom{E}$ and every~$\alpha\geq 0$,
\[
H_\alpha(u,v)\comma H_\alpha(v,u)\in\dom{E} \qquad \text{and}\qquad E\tparen{H_\alpha(u,v)}+E\tparen{H_\alpha(v,u)} \leq E(u)+E(v) \fstop
\]
\item\label{i:d:Form:9'} is a \emph{pre-Dirichlet} functional if it satisfies both~\ref{i:d:Form:7'} and~\ref{i:d:Form:8'}.
\end{enumerate}
\end{definition}

\begin{remark}[Terminology, cf.~{\cite[Rmk.~2.2(b)]{SchZim25}}]
The $\cproj_\tleq$-non-expansiveness property~\ref{i:d:Form:7} is considered in~\cite{CipGri03} with the name of `semi-Dirichlet property'.
We prefer not to use this terminology here, to avoid confusion with the unrelated concept of semi-Dirichlet form in the linear theory.
\end{remark}

\begin{remark}\label{r:CipGri}
We note the following:
\begin{enumerate}[$(a)$, wide]
\item\label{i:r:CipGri:1} A (proper) functional satisfying~\ref{i:d:Form:8} is \emph{midpoint convex}, i.e.
\[
E\tparen{\tfrac{1}{2}u+\tfrac{1}{2}v}\leq \tfrac{1}{2}E(u)+ \tfrac{1}{2}E(v) \comma \qquad u,v\in L^2_\mssm \fstop
\]
In particular, a proper lower semicontinuous functional satisfying~\ref{i:d:Form:8} is in fact automatically convex, see~\cite[Rmk.~3.2]{CipGri03}.
The same assertion is true when replacing~\ref{i:d:Form:8} with~\ref{i:d:Form:8'}, cf.~\cite[Lem.~8]{Puc25}.

\item The validity of~\eqref{eq:i:d:Form:10} implicitly implies that~$\cproj_\tleq \dom{E^\opl{2}}\subset \dom{E^\opl{2}}$ and, analogously, the validity of~\eqref{eq:i:d:Form:8} implicitly implies that~$\cproj_\alpha \dom{E^\opl{2}}\subset \dom{E^\opl{2}}$.
\end{enumerate}
\end{remark}

Under the assumption of convexity and lower semicontinuity, several equivalences among the above notions are available, together with further characterizations.
\begin{proposition}[Barth\'elemy]\label{p:Barthelemy}
Let~$E\colon L^2_\mssm\to (-\infty,\infty]$ be proper convex lower semicontinuous.
Then,
\begin{enumerate}[$(i)$]
\item $E$ satisfies~\ref{i:d:Form:7'} if and only if it satisfies~\ref{i:d:Form:7};
\item $E$ satisfies~\ref{i:d:Form:8'} if and only if it satisfies~\ref{i:d:Form:8}.
\end{enumerate}
In particular,
\begin{enumerate}[$(i)$]\setcounter{enumi}{2}
\item $E$ is pre-Dirichlet~\ref{i:d:Form:9'} if and only if it is Markovian~\ref{i:d:Form:9}.
\end{enumerate}

\begin{proof}
The equivalence \ref{i:d:Form:7'}$\iff$\ref{i:d:Form:7} is a consequence of the more general statement in~\cite[Prop.~2.5]{Bar96} with~$C=C_\tleq$.
Also cf.~\cite[Thm.~3.8]{CipGri03} for a simpler proof of the implication \ref{i:d:Form:7'}$\implies$\ref{i:d:Form:7}.

The equivalence \ref{i:d:Form:8'}$\iff$\ref{i:d:Form:8} is a consequence of the more general statement in~\cite[Prop.~2.5]{Bar96} with~$C=C_\alpha$ for~$\alpha\geq 0$, cf.~\cite[Thm.~2.3]{BriHar22}.
\end{proof}
\end{proposition}

Proposition~\ref{p:Barthelemy} motivates the following main definition.

\begin{definition}[Dirichlet functional]
A functional~$E\colon L^2_\mssm\to (-\infty,\infty]$ is called a \emph{Dirichlet} functional if
\begin{enumerate}[$(E_{\arabic*})$]\setcounter{enumi}{9}
\item\label{i:d:Form:10} $E$ is proper convex lower semicontinuous and Markovian.
\end{enumerate}
(Equivalently: if it is proper convex lower semicontinuous and pre-Dirichlet.)
\end{definition}

\begin{example}\label{e:FormBall}
Choose~$(X,\mfA,\mssm)$ as the standard unit interval endowed with (the restriction of) the Lebesgue measure~$\Leb^1$, let~$B_1\eqdef B_1^{L^\infty}$ be the closed unit ball in~$L^\infty$, and consider the functional~$E\colon L^2\to [0,\infty]$ defined as
\[
E(u)\eqdef \begin{cases} \displaystyle\int_0^1 \abs{u'}^2\diff\Leb^1 & \text{if } u\in H^1\cap B_1
\\
+\infty & \text{otherwise}
\end{cases} \fstop
\]

Note that~$H^1\cap B_1$ is convex and even, and its $L^2$-closure coincides with~$B_1$.
Clearly,~$E$ is convex, proper, and even.
In order to show that it is lower semicontinuous, let~$\seq{u_n}_n\subset H^1\cap B_1$ be any sequence $L^2$-converging to~$u\in B_1$.
If~$\liminf_n E(u_n)=\infty$, there is nothing to prove. If otherwise, suppose ---with no loss of generality, possibly up to passing to a non-relabeled subsequence--- that there exists~$\lim_n E(u_n)<\infty$.
Then,~$\seq{u_n}_n$ is $H^1$-bounded, and therefore ---again with no loss of generality, possibly up to passing to a non-relabeled subsequence---, it converges $H^1$-weakly to~$u\in H^1$.
By lower semicontinuity of the $H^1$-seminorm along $H^1$-weakly converging sequences,~$\liminf_n E(u_n)\geq E(u)$.

Furthermore,~$E$ is a Dirichlet functional.
Indeed, let~$u,v\in H^1\cap B_1$ and~$\alpha\geq 0$.
Each of the functions~$u\vee v$, $u\wedge v$, $H_\alpha(u,v)$, and~$H_\alpha(v,u)$ belongs to~$H^1$, and it belongs to~$B_1$, since it takes values between~$u$ and~$v$.
Since furthermore~$u'=v'$ $\Leb^1$-a.e.\ on every level set of~$u-v$, we have
\[
\abs{(u\vee v)'}^2+\abs{(u\wedge v)'}^2 = \abs{u'}^2+\abs{v'}^2 = \abs{H_\alpha(u,v)'}^2+\abs{H_\alpha(v,u)'}^2 \quad \Leb^1\text{-a.e.}\comma
\]
and integrating over~$[0,1]$ shows that~$E$ satisfies both~\ref{i:d:Form:7'} and~\ref{i:d:Form:8'}.
Thus,~$E$ is a Dirichlet functional by Proposition~\ref{p:Barthelemy}, with~$\overline{\dom{E}}=B_1\subsetneq L^2$.
\end{example}

\begin{proposition}[Reflection of Dirichlet functionals]\label{p:Reflection}
Let~$E\colon L^2_\mssm\to (-\infty,\infty]$ be a Dirichlet functional.
Then, its reflection~$\hat E$ (Dfn.~\ref{d:Reflection}) is a Dirichlet functional as well.

\begin{proof}
Since~$u\mapsto -u$ is a linear isometry of~$L^2_\mssm$, the functional~$\hat E$ is proper, convex, and lower semicontinuous, since so is~$E$.
Furthermore, for every~$u,v\in L^2_\mssm$ and every~$\alpha\geq 0$,
\[
(-u)\vee(-v)=-(u\wedge v)\comma \qquad (-u)\wedge(-v)=-(u\vee v)\comma \qquad H_\alpha(-u,-v)=-H_\alpha(u,v)\fstop
\]
Now, let~$u,v\in\dom{\hat E}$, i.e.~$-u,-v\in\dom{E}$, and fix~$\alpha\geq 0$.
By Proposition~\ref{p:Barthelemy},~$E$ satisfies~\ref{i:d:Form:7'} and~\ref{i:d:Form:8'}.
Applying~\ref{i:d:Form:7'} and~\ref{i:d:Form:8'} to the pair~$-u,-v$, we have that~$u\vee v$, $u\wedge v$, $H_\alpha(u,v)$, and~$H_\alpha(v,u)$ belong to~$\dom{\hat E}$, and that
\begin{align*}
\hat E(u\vee v)+\hat E(u\wedge v)&=E\tparen{(-u)\wedge(-v)}+E\tparen{(-u)\vee(-v)}\leq E(-u)+E(-v)=\hat E(u)+\hat E(v)\comma
\\
\hat E\tparen{H_\alpha(u,v)}+\hat E\tparen{H_\alpha(v,u)}&=E\tparen{H_\alpha(-u,-v)}+E\tparen{H_\alpha(-v,-u)}\leq E(-u)+E(-v)=\hat E(u)+\hat E(v)\fstop
\end{align*}
That is,~$\hat E$ satisfies~\ref{i:d:Form:7'} and~\ref{i:d:Form:8'}, and it is therefore a Dirichlet functional, again by Proposition~\ref{p:Barthelemy}.
\end{proof}
\end{proposition}

\subsubsection{General contractions}
Bilinear Dirichlet forms are conveniently characterized by their behaviour in relation with post-composition with contractions.
Similar characterizations in the non-linear case were understood in~\cite{BenPic79}, proved in~\cite{Cla21} and further elaborated in~\cite{Puc25}.

Define the following sets of \emph{Lipschitz contractions}, i.e.\ convex sets of real-valued non-expansive maps:
\begin{equation}\label{eq:PhiCirc}
\begin{aligned}
\Phi&\eqdef\set{\varphi\colon \R\to\R : \abs{\varphi(s)-\varphi(t)}\leq \abs{s-t} \text{ for all } s,t\in\R}\comma
\\
\ground{\Phi}&\eqdef\set{\varphi\in\Phi : \varphi(0)=0} \comma
\\
\ground{\Phi}^\tleq&\eqdef\set{\varphi\in\ground{\Phi} : \varphi \ \text{non-decreasing}} \fstop
\end{aligned}
\end{equation}

\begin{lemma}\label{l:LipschitzContractionsSplits}
Let~$\varphi,\psi\in \Phi$ satisfy~$\varphi\psi=0$.
Then,
\begin{enumerate}[$(i)$]
\item\label{i:l:LipschitzContractionsSplits:1} $\varphi+\psi\in\Phi$; in particular,~$\ground{\varphi}+\ground{\psi}=\ground{(\varphi+\psi)}\in\ground{\Phi}$;
\item\label{i:l:LipschitzContractionsSplits:2} if~$\Phi_*$ is either of~$\ground{\Phi}$ or~$\ground{\Phi}^\tleq$, and~$\varphi,\psi\in\Phi_*$, then~$\varphi+\psi\in\Phi_*$.
\end{enumerate}

\begin{proof}
\paragraph{Proof of~\ref{i:l:LipschitzContractionsSplits:1}}
Set~$U\eqdef\set{\varphi\neq 0}$ and~$V\eqdef\set{\psi\neq 0}$.
Since~$\varphi$ and~$\psi$ are continuous,~$U$ and~$V$ are open, and they are disjoint since~$\varphi\psi=0$.
Fix~$s,t\in\R$ with~$s<t$.
If~$s,t\notin V$, then~$\psi(s)=\psi(t)=0$, and therefore
\[
\abs{(\varphi+\psi)(s)-(\varphi+\psi)(t)}=\abs{\varphi(s)-\varphi(t)}\leq t-s \fstop
\]
The case~$s,t\notin U$ is analogous.
If otherwise, since~$U\cap V=\emp$, one of the points~$s,t$ lies in~$U$ and the other one lies in~$V$; up to exchanging~$\varphi$ and~$\psi$, we may and will assume that~$s\in U$ and~$t\in V$.
Then, the sets~$U\cap[s,t]$ and~$V\cap[s,t]$ are non-empty, disjoint, and relatively open in~$[s,t]$.
Since~$[s,t]$ is connected, they do not cover~$[s,t]$, i.e.\ there exists~$r\in[s,t]\setminus(U\cup V)$, and in fact~$r\in(s,t)$.
Since~$\varphi(r)=\psi(r)=0$, $\psi(s)=0$, and~$\varphi(t)=0$,
\[
\abs{(\varphi+\psi)(s)-(\varphi+\psi)(t)}=\abs{\tparen{\varphi(s)-\varphi(r)}+\tparen{\psi(r)-\psi(t)}}\leq (r-s)+(t-r)=t-s \fstop
\]
This shows that~$\varphi+\psi\in\Phi$.
Finally,~$\ground{\varphi}+\ground{\psi}=\varphi+\psi-(\varphi+\psi)(0)=\ground{(\varphi+\psi)}$, which is $1$-Lipschitz and grounded, i.e.\ it belongs to~$\ground{\Phi}$.

\paragraph{Proof of~\ref{i:l:LipschitzContractionsSplits:2}}
By~\ref{i:l:LipschitzContractionsSplits:1}, we have~$\varphi+\psi\in\Phi$.
The conclusion follows, since the sum of grounded functions is grounded, and the sum of non-decreasing functions is non-decreasing.
\end{proof}
\end{lemma}

\begin{proposition}[Claus, Puchert]\label{p:CP}
Let~$E\colon L^2_\mssm\to (-\infty,\infty]$ be proper convex lower semicontinuous.
Then, the following are equivalent:
\begin{enumerate}[$(i)$]
\item\label{i:p:CP:1} $E$ is Markovian;
\item\label{i:p:CP:2} \emph{(Claus,~\cite[Thm.~2.39]{Cla21})} $E$ satisfies, for every~$\psi\in \ground{\Phi}^\tleq$,
\begin{equation}\label{eq:Claus}
E\tparen{u-\psi\circ(u-v) } + E\tparen{v+\psi\circ(u-v)} \leq E(u) + E(v) \comma \qquad u,v\in L^2_\mssm\semicolon
\end{equation}

\item\label{i:p:CP:3} \emph{(Puchert,~\cite[Prop.~3]{Puc25})} $E$ satisfies, for every~$\varphi\in \ground{\Phi}$,
\begin{equation}\label{eq:Puchert}
E(u-\varphi \circ v) + E(u+\varphi\circ v) \leq E(u+v) + E(u-v) \comma \qquad u,v\in L^2_\mssm \fstop
\end{equation}
\end{enumerate}
\end{proposition}

As noted in~\cite{SchZim25}, which uses it as a definition, Equation~\eqref{eq:Puchert} is the correct non-linear generalization of the \emph{second Beurling--Deny criterion}, cf.~\cite[Dfn.~2.1]{SchZim25}.
It may thus be regarded as the correct bivariate generalization of the following univariate Lipschitz-contraction property, which is well-known for bilinear Dirichlet forms.

\begin{remark}[Comparison with~{\cite{BenPic79}}]
\eqref{eq:Claus} was originally considered by Bénilan and Picard in~\cite[Thm.~2.1(iii), p.~22]{BenPic79}.
The equivalence~\ref{i:p:CP:1}$\iff$\ref{i:p:CP:2} may be derived from~\cite[Thm.~2.1(i)$\iff$(iii), p.~22]{BenPic79} relating~\eqref{eq:Claus} to a property of the generator of~$E$.
For the definition of property~$\Pi(\mrmi)$ in~\cite{BenPic79}, see~\cite[Dfn.~1.1, p.~4]{BenPic79}.
Note that, in the notation of~\cite{BenPic79}, the operator~$\mrmi$ is the identity operator on~$L^2_\mssm$, and thus~$\Pi(\mrmi)$ reads 
\[\mssm\text{-}\esssup(f_1-f_2) = \mssm\text{-}\esssup \tparen{ (f_1-f_2)\car_{\set{g_1-g_2\geq 0}}} \comma \qquad f_i\in\dom{A}, g_i\in A(f_i)\comma \quad i=1,2 \fstop\]
\end{remark}

\begin{definition}
A functional~$E\colon L^2_\mssm\to (-\infty,\infty]$ has
\begin{enumerate}[$(E_{\arabic*})$]\setcounter{enumi}{10}
\item\label{i:d:Form:11} the \emph{Lipschitz-contraction property} if for every~$u\in\dom{E}$ and every~$\varphi\in \ground{\Phi}$,
\[\varphi\circ u \in \dom{E} \qquad\text{and}\qquad E(\varphi\circ u) \leq E(u) \fstop\]
\end{enumerate}
\end{definition}

Whereas the Lipschitz-contraction property equivalently characterizes bi-linear Dirichlet forms, this ought not to be expected for non-linear functionals ---in light of Remark~\ref{r:SubMarkovUnivariate}.
Indeed, here is an example of a (proper convex lower semicontinuous) functional satisfying the Lipschitz-contraction property yet not the Dirichlet property.

\begin{example}[Lipschitz contraction property vs.\ Dirichlet]\label{e:Chill}
Let~$E$ be any grounded even Dirichlet functional (as in the standing assumption of~\cite[\S3.1]{SchZim25}).
By~\cite[Lem.~3.2]{SchZim25} the associated \emph{Luxemburg norm}
\[
\norm{f}_{\mathsc{l}, E}\eqdef \inf\set{\lambda>0 : E(\lambda^{-1} f) \leq 1}
\]
satisfies the Lipschitz-contraction property~\ref{i:d:Form:11}.

Now, let~$E\eqdef\norm{\emparg}_2^2$ on~$L^2_\mssm([-1,1])$ where~$\mssm$ is the Lebesgue measure~$\Leb^1$ on~$[-1,1]$ and note that it is a (quadratic) Dirichlet form.
The associated Luxemburg norm is the $L^2$-norm~$\norm{\emparg}_{L^2}$.
Take~$f\eqdef \car_{[-1,0]}-\car_{(0,1]}$ and~$g=0$.
We have~$f\vee g= f_+$ and~$f\wedge g= -f_-$.
Thus,
\[
2= \norm{f_+}_{L^2_\mssm} + \norm{f_-}_{L^2_\mssm} \not\leq \norm{f}_{L^2_\mssm} + \norm{g}_{L^2_\mssm} = \sqrt{2} \comma
\]
which shows that~$\norm{\emparg}_{L^2_\mssm}$ does not satisfy the lattice contraction property~\ref{i:d:Form:7'}.
\end{example}

It is nonetheless desirable for Dirichlet functionals to satisfy the Lipschitz-contraction property. This turns out to be the case, yet only under the assumption that the functionals be additionally even.

\begin{proposition}[Brigati--Hartarsky, {\cite[Thm.~1.2]{BriHar22}}]\label{p:BH}
A Dirichlet functional has the Lipschitz-contraction property~\ref{i:d:Form:11} if and only if it is even~\ref{i:d:Form:4}.
\end{proposition}

\begin{example}[A Dirichlet functional which is not even]\label{ex:NotEven}
Let~$\mbbI$ be the standard unit interval endowed with the Lebesgue measure~$\Leb^1$. It is readily verified that the lower semicontinuous envelope~$\check E$ of the functional~$E\colon L^2(\mbbI)\to [0,\infty]$ defined by
\[
E\colon u\longmapsto \int_0^1 (u')_+ \diff\Leb^1\comma \qquad u\in H^1(\mbbI)\comma 
\]
is a Dirichlet functional; cf.~Prop.~\ref{p:LSCEnvelope} below. However, note that~$\check E$ is not even, since~$\check E(u)>0$ for every monotone strictly increasing~$u\in H^1(\mbbI)$, yet~$\check E(-u)=0$ for every such~$u$.
\end{example}

\subsubsection{Normal contractions of functions}
For bilinear Dirichlet forms, the Lipschitz-contraction property~\ref{i:d:Form:11} is equivalent to a stronger, \emph{intrinsic} property, in which the contracted function is not required to be of the form~$\varphi\circ u$.
We now extend this property to the non-linear setting.

\begin{definition}[Normal contraction of a function]\label{d:NormalContraction}
Let~$u,v\colon X\to\R$ be $\mfA$-measurable.
We say that~$u$ is a \emph{normal contraction of~$v$}, and write~$u\preccurlyeq v$, if
\begin{equation}\label{eq:d:NormalContraction}
\abs{u}\leq\abs{v} \as{\mssm} \qquad\text{and}\qquad \abs{u(x)-u(y)}\leq\abs{v(x)-v(y)} \quad \text{for $\mssm^\otym{2}$-a.e.~$(x,y)\in X^\tym{2}$}\fstop
\end{equation}
\end{definition}

Modifying~$u$ or~$v$ on an $\mssm$-negligible set~$N$ affects the second condition in~\eqref{eq:d:NormalContraction} only on the $\mssm^\otym{2}$-negligible set~$(N\times X)\cup(X\times N)$.
Hence,~$\preccurlyeq$ is well-defined on $\mssm$-classes of functions, and in particular on~$L^2_\mssm$.
Furthermore,~$u\preccurlyeq v$ implies~$c\, u\preccurlyeq c\, v$ for every~$c\in\R$.

\paragraph{Comparison with the bilinear theory}
In~\cite[p.~5]{FukOshTak11} and~\cite[Rmk.~I.4.14(i)]{MaRoe92}, $u$ is called a normal contraction of~$v$ if both inequalities in~\eqref{eq:d:NormalContraction} hold \emph{everywhere} on~$X$, resp.\ on~$X^\tym{2}$, for some versions of~$u$ and~$v$, and the operation of all such normal contractions characterizes bilinear Dirichlet forms among closed symmetric forms, see~\cite[Thm.~1.4.1]{FukOshTak11} and~\cite[Cor.~I.4.13]{MaRoe92} (with~$n=1$).
Both proofs rely on the linearity of the resolvent~$\alpha G_\alpha$: in~\cite{FukOshTak11} through its representation by a symmetric Radon measure on~$X^\tym{2}$, see~\cite[Lem.~1.4.1]{FukOshTak11}, under the topological assumption~\cite[(1.1.7)]{FukOshTak11}; in~\cite{MaRoe92} through the representation of~$\tparen{(1-\alpha G_\alpha)v,v}$, for simple~$v$, as a non-negative combination of squared differences and squares of the values of~$v$.
In either case, the inequality~$E(u)\leq E(v)$ is reduced to the inequalities in~\eqref{eq:d:NormalContraction} at pairs of points of~$X$.
Since the representing Radon measure may charge $\mssm^\otym{2}$-negligible sets (e.g.\ the graph of an $\mssm$-preserving involution of~$X$), the assumption that the inequalities hold \emph{everywhere} for some versions is essential to this argument.
Accordingly, $u$ is never written as a function of~$v$ defined on the whole real line: in~\cite[proof of Cor.~I.4.13]{MaRoe92}, $u=\varphi\circ v$ with~$\varphi$ defined only on the range of~$v$, which is well-defined by the everywhere assumption.
No such representation is available for non-linear resolvents, and we proceed instead in the opposite direction.
In the next lemma, we show that every normal contraction~$u$ of~$v$ in the $\mssm$-a.e.\ sense of Definition~\ref{d:NormalContraction} is of the form~$u=\varphi\circ v$ for some~$\varphi\in\ground{\Phi}$; here, the range of~$v$ is replaced by the support of the `joint law' of~$(v,u)$, and~$\varphi$ is obtained by McShane's extension~\cite{McS34}.
The intrinsic property for non-linear functionals will then follow from the characterization by post-composition with normal contractions in Proposition~\ref{p:CP}.

\begin{lemma}\label{l:NormalContraction}
Let~$u,v\colon X\to\R$ be $\mfA$-measurable. Then, the following assertions hold.
\begin{enumerate}[$(i)$]
\item\label{i:l:NormalContraction:1} If~$\abs{u(x)-u(y)}\leq\abs{v(x)-v(y)}$ for~$\mssm^\otym{2}$-a.e.~$(x,y)\in X^\tym{2}$, then there exists~$\varphi\in\Phi$ with~$u=\varphi\circ v$ $\mssm$-a.e.
If additionally~$\abs{u}\leq\abs{v}$ $\mssm$-a.e., then~$\varphi$ may be chosen in~$\ground{\Phi}$.

\item\label{i:l:NormalContraction:2} If~$u=\varphi\circ v$ $\mssm$-a.e.\ for some~$\varphi\in\ground{\Phi}$, then~$u\preccurlyeq v$.

\item\label{i:l:NormalContraction:3} If~$\mssm X=\infty$, if~$u,v\in L^2_\mssm$, and if~$\abs{u(x)-u(y)}\leq\abs{v(x)-v(y)}$ for~$\mssm^\otym{2}$-a.e.~$(x,y)\in X^\tym{2}$, then~$\abs{u}\leq\abs{v}$ $\mssm$-a.e.; that is,~$u\preccurlyeq v$.
\end{enumerate}
In particular,~$u\preccurlyeq v$ if and only if~$u=\varphi\circ v$ $\mssm$-a.e.\ for some~$\varphi\in\ground{\Phi}$.

\begin{proof}
\paragraph{Proof of~\ref{i:l:NormalContraction:1}}
Recall that~$\mssm X>0$ by assumption.
Since~$\mssm$ is $\sigma$-finite, there exists an $\mfA$-measurable~$\rho\colon X\to (0,\infty)$ with~$\int_X\rho\diff\mssm=1$.
The probability measure~$\mu\eqdef\rho\,\mssm$ has the same negligible sets as~$\mssm$, and~$\mu^\otym{2}$ has the same negligible sets as~$\mssm^\otym{2}$. 
Set~$\Gamma\eqdef (v,u)\colon X\to\R^2$, let~$\nu\eqdef\Gamma_\pfwd\mu$ be the image probability measure on~$\R^2$, and let~$S\eqdef\supp\nu$ be its support. 
We will show that~$S$ is the graph of a $1$-Lipschitz function.
Indeed, consider the open set
\[
U\eqdef\set{\tparen{(s,a),(t,b)}\in\R^2\times\R^2 : \abs{a-b}>\abs{s-t}} \fstop
\]
Since the \emph{finite} measures~$\nu^\otym{2}$ and~$(\Gamma\times\Gamma)_\pfwd\mu^\otym{2}$ coincide on measurable rectangles, they coincide.
Thus, by assumption,
\[
\nu^\otym{2}U =\mu^\otym{2}\tset{(x,y)\in X^\tym{2} : \abs{u(x)-u(y)}>\abs{v(x)-v(y)}}=0\fstop
\]
We claim that, \emph{everywhere} on~$S^\tym{2}$,
\begin{equation}\label{eq:l:NormalContraction:1}
\abs{a-b}\leq\abs{s-t}\comma \qquad (s,a),(t,b)\in S\fstop
\end{equation}
Indeed, suppose that~$(p,q)\in (S\times S)\cap U$.
Since~$U$ is open, there exist open neighborhoods~$V\ni p$ and~$W\ni q$ with~$V\times W\subset U$, and therefore~$\nu^\otym{2} U \geq\nu V \,\nu W>0$ since~$p,q\in S$ and by definition of support of a measure, a contradiction.

Let~$D_0\subset\R$ be the projection of~$S$ onto the first coordinate.
By~\eqref{eq:l:NormalContraction:1} with~$s=t$, for every~$s\in D_0$ there exists a unique~$a\in\R$ with~$(s,a)\in S$, and we set~$\varphi_0(s)\eqdef a$.
Again by~\eqref{eq:l:NormalContraction:1}, the function~$\varphi_0\colon D_0\to\R$ is $1$-Lipschitz.
Let~$\varphi$ be its (lower) McShane extension~\cite{McS34}
\begin{equation}\label{eq:l:NormalContraction:2}
\varphi(t)\eqdef\inf_{s\in D_0}\tparen{\varphi_0(s)+\abs{t-s}}\comma \qquad t\in\R\comma 
\end{equation}
and recall that~$\varphi\colon \R\to \R$ is real-valued, $1$-Lipschitz (since~$\varphi_0$ is so), and coincides with~$\varphi_0$ on~$D_0$.
Thus,~$\varphi\in\Phi$.
Since~$\nu S=1$, we have~$\Gamma(x)\in S$, and therefore~$u(x)=\varphi_0\tparen{v(x)}=\varphi\tparen{v(x)}$, for $\mu$-a.e., equivalently $\mssm$-a.e.,~$x\in X$.

Assume now additionally that~$\abs{u}\leq\abs{v}$ $\mssm$-a.e.
The set~$U_0\eqdef\set{(s,a)\in\R^2 : \abs{a}>\abs{s}}$ is open and~$\nu U_0 =\mu \set{\abs{u}>\abs{v}}=0$.
Thus~$S\cap U_0=\emp$ by definition of support, that is
\begin{equation}\label{eq:l:NormalContraction:3}
\abs{a}\leq\abs{s}\comma \qquad (s,a)\in S\fstop
\end{equation}
Set~$D\eqdef D_0\cup\set{0}$, and extend~$\varphi_0$ to~$D$ by~$\varphi_0(0)\eqdef 0$, which is consistent with the previous definition if~$0\in D_0$, by~\eqref{eq:l:NormalContraction:3}.
By~\eqref{eq:l:NormalContraction:1} and~\eqref{eq:l:NormalContraction:3}, the extended function~$\varphi_0$ is $1$-Lipschitz on~$D$.
Defining~$\varphi$ by~\eqref{eq:l:NormalContraction:2} with~$D$ in place of~$D_0$, the same argument as above shows that~$\varphi\in\Phi$, that~$\varphi(0)=\varphi_0(0)=0$, and that~$u=\varphi\circ v$ $\mssm$-a.e.
That is,~$\varphi\in\ground{\Phi}$ and~$u=\varphi\circ v$ $\mssm$-a.e.

\paragraph{Proof of~\ref{i:l:NormalContraction:2}}
Choose versions of~$u$ and~$v$ such that~$u=\varphi\circ v$ outside an $\mssm$-negligible set~$N$.
Since~$\varphi\in\ground{\Phi}$, for all~$x,y\in X\setminus N$,
\[
\abs{u(x)}=\tabs{(\varphi\circ v)(x)-\varphi(0)}\leq\abs{v(x)} \qquad \text{and}\qquad \abs{u(x)-u(y)}=\tabs{(\varphi \circ v)(x)-(\varphi\circ v)(y)}\leq\abs{v(x)-v(y)}\comma
\]
and the set~$(N\times X)\cup(X\times N)$ is $\mssm^\otym{2}$-negligible.

\paragraph{Proof of~\ref{i:l:NormalContraction:3}}
Fix versions of~$u$ and~$v$.
By Tonelli's Theorem, there exists an $\mssm$-negligible set~$N\subset X$ such that, for every~$x\in X\setminus N$, we have~$\abs{u(x)-u(y)}\leq\abs{v(x)-v(y)}$ for $\mssm$-a.e.~$y\in X$.
Fix~$x\in X\setminus N$ and suppose by contradiction that~$\abs{u(x)}>\abs{v(x)}$.
Set~$\delta\eqdef\tfrac12\tparen{\abs{u(x)}-\abs{v(x)}}>0$ and~$B\eqdef\set{\abs{v}<\delta}$.
By Chebyshev's inequality,~$\mssm(X\setminus B)\leq\delta^{-2}\norm{v}_{L^2_\mssm}^2<\infty$, and therefore~$\mssm B =\infty$, since~$\mssm X=\infty$.
Furthermore, for $\mssm$-a.e.~$y\in B$,
\[
\abs{u(y)}\geq\abs{u(x)}-\abs{u(x)-u(y)}\geq\abs{u(x)}-\abs{v(x)}-\abs{v(y)}>2\delta-\delta=\delta\comma
\]
whence~$\norm{u}_{L^2_\mssm}^2\geq\delta^2\,\mssm B=\infty$, contradicting~$u\in L^2_\mssm$.
Thus,~$\abs{u}\leq\abs{v}$ on~$X\setminus N$.

\medskip
The final assertion follows combining~\ref{i:l:NormalContraction:1} and~\ref{i:l:NormalContraction:2}.
\end{proof}
\end{lemma}

\begin{remark}\label{r:NormalContraction}
\begin{enumerate*}[$(a)$]
\item\label{i:r:NormalContraction:1} If some versions of~$u$ and~$v$ satisfy both inequalities in~\eqref{eq:d:NormalContraction} \emph{everywhere} on~$X$, resp.\ on~$X^\tym{2}$, then~$u\preccurlyeq v$.
Conversely, if~$u\preccurlyeq v$, then~$u=\varphi\circ v$ $\mssm$-a.e.\ for some~$\varphi\in\ground{\Phi}$ by Lemma~\ref{l:NormalContraction}, and, for every version~$\reptwo v$ of~$v$, the version~$\varphi\circ\reptwo v$ of~$u$ satisfies both inequalities everywhere.
Thus, Definition~\ref{d:NormalContraction} agrees with the notion of normal contraction of a function in~\cite[p.~5]{FukOshTak11} and~\cite[Rmk.~I.4.14(i)]{MaRoe92}.

\item\label{i:r:NormalContraction:2} By Lemma~\ref{l:NormalContraction}, a functional~$E$ has the Lipschitz-contraction property~\ref{i:d:Form:11} if and only if it satisfies its intrinsic form: for every~$v\in\dom{E}$ and every~$u\in L^2_\mssm$ with~$u\preccurlyeq v$, we have~$u\in\dom{E}$ and~$E(u)\leq E(v)$.
In light of Example~\ref{e:Chill}, this univariate condition does \emph{not} characterize Dirichlet functionals.
\end{enumerate*}
\end{remark}

The bivariate inequalities considered in the following, as well as all those considered above, are insensitive to affine perturbations of the functional, as we now note.

\begin{remark}[Affine perturbations]\label{r:AffinePerturbations}
Each of the inequalities defining~\ref{i:d:Form:7}, \ref{i:d:Form:8}, \ref{i:d:Form:7'}, \ref{i:d:Form:8'}, \eqref{eq:Claus}, and~\eqref{eq:Puchert} compares~$E(u')+E(v')$ with~$E(u)+E(v)$ for pairs satisfying~$u'+v'=u+v$; for~\ref{i:d:Form:7} and~\ref{i:d:Form:8} this follows from~\eqref{eq:StripProjections}.
Thus, for every continuous linear functional~$\ell\in (L^2_\mssm)'$ and every~$c\in\R$, a functional~$E$ satisfies any of them if and only if so does~$E+\ell+c$.
Since every proper convex lower semicontinuous functional has a continuous affine minorant, statements about these properties of such functionals reduce to the case of non-negative functionals, which is the setting of~\cite{Puc25}.
\end{remark}

We are now ready to state the announced non-linear counterpart of~\cite[(E.4)$''$]{FukOshTak11}.
Since any condition on normal contractions of functions is naturally bivariate, the non-linear generalization of the normal contraction property for functions (Eqn.~\eqref{eq:t:NormalContraction:1} below) is a recentering of the usual (linear) condition~\cite[(E.4)$''$]{FukOshTak11} around an arbitrary point~$z\in L^2_\mssm$.
The condition~\cite[(E.4)$''$]{FukOshTak11} is then recovered by choosing~$z\equiv 0$.

\begin{theorem}[Markovianity via normal contractions]\label{t:NormalContraction}
Let~$E\colon L^2_\mssm\to (-\infty,\infty]$ be lower semicontinuous.
Then, the following are equivalent:
\begin{enumerate}[$(i)$]
\item\label{i:t:NormalContraction:1} $E$ is Markovian~\ref{i:d:Form:9};

\item\label{i:t:NormalContraction:2} $E$ is pre-Dirichlet~\ref{i:d:Form:9'};

\item\label{i:t:NormalContraction:3} \emph{(bivariate normal-contraction property)} for every~$z\in L^2_\mssm$, and every~$u,v\in L^2_\mssm$ with~$u\preccurlyeq v$,
\begin{equation}\label{eq:t:NormalContraction:1}
E(z+u)+E(z-u)\leq E(z+v)+E(z-v)\semicolon
\end{equation}

\item\label{i:t:NormalContraction:4} for every~$u,v,u',v'\in L^2_\mssm$ with~$u'+v'=u+v$ and~$u'-v'\preccurlyeq u-v$,
\begin{equation}\label{eq:t:NormalContraction:2}
E(u')+E(v')\leq E(u)+E(v)\fstop
\end{equation}
\end{enumerate}
If~$E$ is additionally proper, then~\ref{i:t:NormalContraction:1}--\ref{i:t:NormalContraction:4} are further equivalent to
\begin{enumerate}[$(i)$]\setcounter{enumi}{4}
\item\label{i:t:NormalContraction:5} $E$ is a Dirichlet functional~\ref{i:d:Form:10}.
\end{enumerate}
Furthermore, for an \emph{arbitrary} functional~$E\colon L^2_\mssm\to (-\infty,\infty]$, not necessarily lower semicontinuous, \ref{i:t:NormalContraction:3} and~\ref{i:t:NormalContraction:4} are equivalent, and each of them implies both~\ref{i:t:NormalContraction:1} and~\ref{i:t:NormalContraction:2}.
\end{theorem}

Let us note that condition~\ref{i:t:NormalContraction:3} above is but Puchert's characterization~\eqref{eq:Puchert} for~$u\eqdef \varphi\circ v$.
The equivalence~\ref{i:t:NormalContraction:2}$\iff$\ref{i:t:NormalContraction:3} thus follows combining the results in~\cite{Puc25} with Lemma~\ref{l:NormalContraction}.

\begin{proof}[Proof of Theorem~\ref{t:NormalContraction}]
\paragraph{\ref{i:t:NormalContraction:3}$\iff$\ref{i:t:NormalContraction:4} for arbitrary~$E$}
Let~$z,u,v\in L^2_\mssm$ with~$u\preccurlyeq v$.
The quadruple~$(z+v,z-v,z+u,z-u)$ satisfies the assumptions in~\ref{i:t:NormalContraction:4} in place of~$(u,v,u',v')$, since~$(z+u)+(z-u)=2z=(z+v)+(z-v)$ and~$(z+u)-(z-u)=2u\preccurlyeq 2v=(z+v)-(z-v)$; the corresponding inequality~\eqref{eq:t:NormalContraction:2} is~\eqref{eq:t:NormalContraction:1}.
Vice versa, let~$u,v,u',v'$ be as in~\ref{i:t:NormalContraction:4}.
The triple~$\tparen{\tfrac12(u+v),\tfrac12(u'-v'),\tfrac12(u-v)}$ satisfies the assumptions in~\ref{i:t:NormalContraction:3} in place of~$(z,u,v)$, since~$\tfrac12(u'-v')\preccurlyeq\tfrac12(u-v)$.
Since~$\tfrac12(u+v)=\tfrac12(u'+v')$, the corresponding inequality~\eqref{eq:t:NormalContraction:1} is~\eqref{eq:t:NormalContraction:2}.

\paragraph{\ref{i:t:NormalContraction:3}$\implies$\ref{i:t:NormalContraction:1} and~\ref{i:t:NormalContraction:3}$\implies$\ref{i:t:NormalContraction:2} for arbitrary~$E$}
For~$\alpha\geq 0$ define~$\varphi_\alpha,\psi_\alpha\colon\R\to\R$ by
\[
\varphi_\alpha(t)\eqdef \tparen{-\tfrac{\alpha}{2}}\vee t\wedge \tfrac{\alpha}{2}\comma \qquad \psi_\alpha(t)\eqdef 2\varphi_\alpha(t)-t=\begin{cases} -\alpha-t &\text{if } t<-\tfrac{\alpha}{2}\comma\\ t &\text{if } \abs{t}\leq\tfrac{\alpha}{2}\comma\\ \alpha-t & \text{if } t>\tfrac{\alpha}{2}\comma\end{cases}
\]
and note that~$\varphi_\alpha,\psi_\alpha\in\ground{\Phi}$, and that the functions~$t\mapsto t\wedge 0$ and~$t\mapsto \abs{t}$ belong to~$\ground{\Phi}$ as well.
Every pair of elements of~$L^2_\mssm$ is of the form~$(z+v,z-v)$ for some~$z,v\in L^2_\mssm$, namely for~$z$ its midpoint and~$v$ its half-gap.
Respectively by~\eqref{eq:StripProjections}, by the identities~$s\vee t=\tfrac12(s+t)+\tfrac12\abs{s-t}$ and~$s\wedge t=\tfrac12(s+t)-\tfrac12\abs{s-t}$, and by the identity~$H_\alpha(s,t)=t+(-\alpha)\vee(s-t)\wedge\alpha=t+2\varphi_\alpha\tparen{\tfrac12(s-t)}$, pointwise computations show that, for every~$z,v\in L^2_\mssm$,
\begin{equation}\label{eq:t:NormalContraction:3}
\begin{aligned}
\cproj_\tleq(z+v,z-v)&=\tparen{z+(v\wedge 0),\, z-(v\wedge 0)}\comma
\\
\cproj_\alpha(z+v,z-v)&=\tparen{z+\varphi_\alpha\circ v,\, z-\varphi_\alpha\circ v}\comma
\\
\tparen{(z+v)\vee(z-v),\, (z+v)\wedge(z-v)}&=\tparen{z+\abs{v},\, z-\abs{v}}\comma
\\
\tparen{H_\alpha(z+v,z-v),\, H_\alpha(z-v,z+v)}&=\tparen{z+\psi_\alpha\circ v,\, z-\psi_\alpha\circ v}\fstop
\end{aligned}
\end{equation}
By Lemma~\ref{l:NormalContraction}\ref{i:l:NormalContraction:2}, each of~$v\wedge 0$, $\varphi_\alpha\circ v$, $\abs{v}$, and~$\psi_\alpha\circ v$ is a normal contraction of~$v$.
Thus, applying~\eqref{eq:t:NormalContraction:1} with each of them in place of~$u$, and in light of~\eqref{eq:t:NormalContraction:3}, we have, for every~$z,v\in L^2_\mssm$,
\begin{align*}
E^\opl{2}\tparen{\cproj_\tleq(z+v,z-v)}&\leq E^\opl{2}(z+v,z-v)\comma
\\
E^\opl{2}\tparen{\cproj_\alpha(z+v,z-v)}&\leq E^\opl{2}(z+v,z-v)\comma
\\
E\tparen{(z+v)\vee(z-v)}+E\tparen{(z+v)\wedge(z-v)}&\leq E(z+v)+E(z-v) \comma
\\
E\tparen{H_\alpha(z+v,z-v)}+E\tparen{H_\alpha(z-v,z+v)}&\leq E(z+v)+E(z-v) \fstop
\end{align*}
The first two inequalities show~\ref{i:d:Form:7} and~\ref{i:d:Form:8}, i.e.~\ref{i:t:NormalContraction:1}.
As for~\ref{i:d:Form:7'} and~\ref{i:d:Form:8'}, if~$z+v,z-v\in\dom{E}$, then the right-hand sides of the last two inequalities are finite, hence so is each summand on their left-hand sides, since~$E>-\infty$.
That is, the lattice operations and the operations~$H_\alpha$ map pairs of elements of~$\dom{E}$ to pairs of elements of~$\dom{E}$, which shows~\ref{i:t:NormalContraction:2}.

\paragraph{\ref{i:t:NormalContraction:1}$\implies$\ref{i:t:NormalContraction:3} and, if~$E$ is proper lower semicontinuous,~\ref{i:t:NormalContraction:1}$\implies$\ref{i:t:NormalContraction:5}}
If~$E\equiv+\infty$, then~\eqref{eq:t:NormalContraction:1} holds trivially.
Otherwise,~$E$ is proper lower semicontinuous and satisfies~\ref{i:d:Form:8}, and therefore it is convex by Remark~\ref{r:CipGri}\ref{i:r:CipGri:1}.
That is,~$E$ is a Dirichlet functional, which shows~\ref{i:t:NormalContraction:5}.
By Proposition~\ref{p:CP}\ref{i:p:CP:1}$\implies$\ref{i:p:CP:3}, the functional~$E$ satisfies~\eqref{eq:Puchert} for every~$\varphi\in\ground{\Phi}$.
Now, let~$z,u,v\in L^2_\mssm$ with~$u\preccurlyeq v$.
By Lemma~\ref{l:NormalContraction}\ref{i:l:NormalContraction:1}, there exists~$\varphi\in\ground{\Phi}$ with~$u=\varphi\circ v$.
Then,~\eqref{eq:Puchert} with~$z$ in place of~$u$ reads
\[
E(z-\varphi\circ v)+E(z+\varphi\circ v)\leq E(z+v)+E(z-v)\comma
\]
which is~\eqref{eq:t:NormalContraction:1} since~$u=\varphi\circ v$.

\paragraph{\ref{i:t:NormalContraction:2} implies~\ref{i:t:NormalContraction:1} for lower semicontinuous~$E$}
If~$E\equiv+\infty$, there is nothing to prove.
Otherwise,~$E$ is proper lower semicontinuous and satisfies~\ref{i:d:Form:8'}, and therefore it is convex by Remark~\ref{r:CipGri}\ref{i:r:CipGri:1}, and thus Markovian by Proposition~\ref{p:Barthelemy}.

\paragraph{Conclusion}
For lower semicontinuous~$E$, we have shown~\ref{i:t:NormalContraction:1}$\implies$\ref{i:t:NormalContraction:3}$\iff$\ref{i:t:NormalContraction:4}, \ref{i:t:NormalContraction:3}$\implies$\ref{i:t:NormalContraction:2}, and~\ref{i:t:NormalContraction:2}$\implies$\ref{i:t:NormalContraction:1}.
If~$E$ is additionally proper, then~\ref{i:t:NormalContraction:5}$\implies$\ref{i:t:NormalContraction:1} holds by definition, and~\ref{i:t:NormalContraction:1}$\implies$\ref{i:t:NormalContraction:5} was shown above.
\end{proof}

\begin{remark}\label{r:t:NormalContraction}
Since~$E>-\infty$, the inequality~\eqref{eq:t:NormalContraction:1} contains a domain statement: if~$z+v,z-v\in\dom{E}$, then~$z+u,z-u\in\dom{E}$.
When~$\dom{E}$ is a linear subspace of~$L^2_\mssm$, the assumption~$z+v,z-v\in\dom{E}$ is equivalent to~$z,v\in\dom{E}$.
\end{remark}

Lower semicontinuity cannot be dispensed with in the implication~\ref{i:t:NormalContraction:1}$\implies$\ref{i:t:NormalContraction:3} of Theorem~\ref{t:NormalContraction}, not even for convex functionals.

\begin{example}[A Markovian convex functional without the normal-contraction property]\label{e:NonLSCMarkov}
Let~$X\eqdef\set{1,2}$ and~$\mssm\eqdef\delta_1+\delta_2$, so that~$L^2_\mssm\cong\R^2$, and write~$u=(u_1,u_2)$ for~$u\in\R^2$.
Set~$K\eqdef[0,1]^2\setminus\set{(1,1)}$ and~$E\eqdef\mbfI_K$.
Since~$(1,1)$ is an extreme point of~$[0,1]^2$, the set~$K$ is convex; it is not closed.
Thus,~$E$ is proper and convex, but not lower semicontinuous.
We claim that~$E$ is Markovian, yet it satisfies neither the lattice contraction property~\ref{i:d:Form:7'} nor the equivalent conditions~\ref{i:t:NormalContraction:3} and~\ref{i:t:NormalContraction:4} in Theorem~\ref{t:NormalContraction}.

\paragraph{$E$ is Markovian}
Let~$w=(u,v)\in L^2_\mssm\oplus L^2_\mssm$, and let~$w'=(u',v')$ be either~$\cproj_\tleq(w)$ or~$\cproj_\alpha(w)$ for some~$\alpha\geq 0$.
If~$u\notin K$ or~$v\notin K$, then~$E^\opl{2}(w)=\infty$, and there is nothing to prove; thus, let~$u,v\in K$.
By~\eqref{eq:StripProjections}, for each~$i\in\set{1,2}$, either~$(u'_i,v'_i)=(u_i,v_i)$, or~$u_i\neq v_i$ and both~$u'_i$ and~$v'_i$ lie in the open interval with endpoints~$u_i$ and~$v_i$.
In either case~$u'_i,v'_i\in [0,1]$, and furthermore~$u'_i=1$ implies~$u_i=1$, and~$v'_i=1$ implies~$v_i=1$, since in the second case~$u'_i,v'_i<u_i\vee v_i\leq 1$.
Hence~$u',v'\in[0,1]^2$, and~$u'=(1,1)$, resp.~$v'=(1,1)$, would imply~$u=(1,1)$, resp.~$v=(1,1)$, contradicting~$u,v\in K$.
Thus,~$w'\in K\times K$, i.e.~$E^\opl{2}(w')=0=E^\opl{2}(w)$.

\paragraph{$E$ violates~\ref{i:d:Form:7'} and~\ref{i:t:NormalContraction:4}}
Let~$u\eqdef(1,0)$ and~$v\eqdef(0,1)$, both in~$K$.
Then~$u\vee v=(1,1)\notin K$, which violates~\ref{i:d:Form:7'}.
Equivalently, the pair~$u'\eqdef u\vee v$ and~$v'\eqdef u\wedge v$ satisfies~$u'+v'=u+v$ and~$u'-v'=\abs{u-v}\preccurlyeq u-v$ by Lemma~\ref{l:NormalContraction}\ref{i:l:NormalContraction:2}, while
\[
E(u')+E(v')=E\tparen{(1,1)}+E\tparen{(0,0)}=\infty>0=E(u)+E(v)\comma
\]
which violates~\ref{i:t:NormalContraction:4}, and therefore~\ref{i:t:NormalContraction:3} as well, by the last assertion in Theorem~\ref{t:NormalContraction}.
\end{example}

\begin{remark}\label{r:NonLSCMarkov}
\begin{enumerate*}[$(a)$]
\item Example~\ref{e:NonLSCMarkov} shows as well that, without lower semicontinuity,~\ref{i:d:Form:7} does not imply~\ref{i:d:Form:7'}, so that lower semicontinuity cannot be dispensed with in Proposition~\ref{p:Barthelemy} either.
Note that the lower semicontinuous envelope~$\mbfI_{[0,1]^2}$ of~$E$ is a Dirichlet functional.

\item Combining Theorem~\ref{t:NormalContraction} and Example~\ref{e:NonLSCMarkov}, the bivariate normal-contraction property implies Markovianity for every functional; among lower semicontinuous functionals, Markovianity alone ---with no \emph{a priori convexity}--- is equivalent to it; and without lower semicontinuity the two notions differ, even among convex functionals.
\end{enumerate*}
\end{remark}

Let us now specialize Theorem~\ref{t:NormalContraction} to even and to quadratic functionals.

\begin{corollary}[Even functionals]\label{c:NormalContractionEven}
Let~$E$ be an even Dirichlet functional.
Then, for every~$v\in\dom{E}$ and every~$u\in L^2_\mssm$ with~$u\preccurlyeq v$, we have~$u\in\dom{E}$ and~$E(u)\leq E(v)$.

\begin{proof}
Since~$E$ is even,~$-v\in\dom{E}$ and~$E(-v)=E(v)$.
By Theorem~\ref{t:NormalContraction}\ref{i:t:NormalContraction:5}$\implies$\ref{i:t:NormalContraction:3} with~$z=0$,
\[
E(u)+E(-u)\leq E(v)+E(-v)=2E(v)<\infty\fstop
\]
Since~$E>-\infty$, both~$u$ and~$-u$ belong to~$\dom{E}$, and thus~$E(-u)=E(u)$ since~$E$ is even, which concludes the proof.
\end{proof}
\end{corollary}

By Remark~\ref{r:NormalContraction}\ref{i:r:NormalContraction:2}, Corollary~\ref{c:NormalContractionEven} is equivalent to the `if' part of Proposition~\ref{p:BH}, of which it thus provides an alternative proof based on Proposition~\ref{p:CP}.
On the other hand, Example~\ref{e:Chill} shows that its conclusion does not characterize Dirichlet functionals, not even among even convex lower semicontinuous functionals.
For quadratic functionals, instead, the bivariate condition reduces to the univariate one.

\begin{corollary}[Quadratic functionals]\label{c:NormalContractionQuadratic}
Let~$E\colon L^2_\mssm\to [0,\infty]$ be proper, convex, lower semicontinuous, linearly defined, and satisfying the parallelogram identity~\ref{i:d:Form:6}.
Then, the following are equivalent:
\begin{enumerate}[$(i)$]
\item\label{i:c:NormalContractionQuadratic:1} $E$ is a Dirichlet functional;
\item\label{i:c:NormalContractionQuadratic:2} for every~$v\in\dom{E}$ and every~$u\in L^2_\mssm$ with~$u\preccurlyeq v$, we have~$u\in\dom{E}$ and~$E(u)\leq E(v)$.
\end{enumerate}

\begin{proof}
Since~$E$ is proper and~$\dom{E}$ is a linear subspace, we have~$0\in\dom{E}$.
Evaluating the parallelogram identity~\ref{i:d:Form:6} at the pair~$(0,0)$ shows that~$E(0)=0$, and then evaluating it at the pair~$(0,v)$ shows that~$E(-v)=E(v)$ for every~$v\in\dom{E}$; that is,~$E$ is even.
Thus,~\ref{i:c:NormalContractionQuadratic:1}$\implies$\ref{i:c:NormalContractionQuadratic:2} by Corollary~\ref{c:NormalContractionEven}.

Assume~\ref{i:c:NormalContractionQuadratic:2}, and let~$z,u,v\in L^2_\mssm$ with~$u\preccurlyeq v$.
If either~$z+v\notin\dom{E}$ or~$z-v\notin\dom{E}$, then~\eqref{eq:t:NormalContraction:1} holds trivially.
Otherwise,~$z,v\in\dom{E}$, since~$\dom{E}$ is a linear subspace.
Hence~$u\in\dom{E}$ by~\ref{i:c:NormalContractionQuadratic:2}, thus~$z\pm u\in\dom{E}$, and, respectively by~\ref{i:d:Form:6}, by~\ref{i:c:NormalContractionQuadratic:2}, and again by~\ref{i:d:Form:6},
\[
E(z+u)+E(z-u)=2E(z)+2E(u)\leq 2E(z)+2E(v)=E(z+v)+E(z-v)\fstop
\]
Thus,~$E$ satisfies~\ref{i:t:NormalContraction:3} in Theorem~\ref{t:NormalContraction}, and it is a Dirichlet functional by Theorem~\ref{t:NormalContraction}\ref{i:t:NormalContraction:3}$\implies$\ref{i:t:NormalContraction:5}.
\end{proof}
\end{corollary}

\begin{remark}\label{r:NormalContractionQuadratic}
Let~$E(u)=\mcE(u,u)$ for a closed symmetric bilinear form~$(\mcE,\dom{E})$.
By Remark~\ref{r:NormalContraction}\ref{i:r:NormalContraction:1}, condition~\ref{i:c:NormalContractionQuadratic:2} in Corollary~\ref{c:NormalContractionQuadratic} is precisely the property that `every normal contraction operates on~$\mcE$' in the sense of~\cite[(E.4)$''$, p.~5]{FukOshTak11}, which characterizes bilinear Dirichlet forms among closed symmetric forms by~\cite[Thm.~1.4.1]{FukOshTak11}.
Therefore, Corollary~\ref{c:NormalContractionQuadratic} shows that a quadratic functional is a Dirichlet functional if and only if the associated bilinear form is a Dirichlet form in the classical sense, and Theorem~\ref{t:NormalContraction} is the non-linear counterpart of~\cite[Thm.~1.4.1(e)]{FukOshTak11} and~\cite[Cor.~I.4.13]{MaRoe92} (with~$n=1$): for quadratic functionals, the parallelogram identity reduces~\eqref{eq:t:NormalContraction:1} to the univariate inequality~$E(u)\leq E(v)$.
\end{remark}

Finally, we discuss the pointwise bound~$\abs{u}\leq\abs{v}$ in~\eqref{eq:d:NormalContraction}, which is part of the classical definition of a normal contraction of a function.

\begin{corollary}[Normal contractions of increments]\label{c:NormalContractionIncrements}
Let~$E$ be a Dirichlet functional, and let~$u,v\in L^2_\mssm$ satisfy
\[
\abs{u(x)-u(y)}\leq\abs{v(x)-v(y)} \quad \text{for $\mssm^\otym{2}$-a.e.~$(x,y)\in X^\tym{2}$}\fstop
\]
Then~\eqref{eq:t:NormalContraction:1} holds for every~$z\in L^2_\mssm$, provided that either
\begin{enumerate}[$(i)$]
\item\label{i:c:NormalContractionIncrements:1} $\mssm X=\infty$; or
\item\label{i:c:NormalContractionIncrements:2} $\mssm X<\infty$, and~$\alpha\car\in\dom{E}$ and~$E(\alpha\car)=E(0)$ for every~$\alpha\in\R$.
\end{enumerate}

\begin{proof}
\paragraph{Proof of~\ref{i:c:NormalContractionIncrements:1}}
By Lemma~\ref{l:NormalContraction}\ref{i:l:NormalContraction:3}, we have~$u\preccurlyeq v$, and the conclusion follows from Theorem~\ref{t:NormalContraction}\ref{i:t:NormalContraction:5}$\implies$\ref{i:t:NormalContraction:3}.

\paragraph{Proof of~\ref{i:c:NormalContractionIncrements:2}}
By Lemma~\ref{l:NormalContraction}\ref{i:l:NormalContraction:1}, there exists~$\varphi\in\Phi$ with~$u=\varphi\circ v$ $\mssm$-a.e.
Set~$c\eqdef\varphi(0)$ and note that~$\ground{\varphi}=\varphi-c\in\ground{\Phi}$.
Since~$\mssm X<\infty$, we have~$c\car\in L^2_\mssm$, and~$u-c\car=\ground{\varphi}\circ v\preccurlyeq v$ by Lemma~\ref{l:NormalContraction}\ref{i:l:NormalContraction:2}.
Thus, by Theorem~\ref{t:NormalContraction}\ref{i:t:NormalContraction:5}$\implies$\ref{i:t:NormalContraction:3},
\begin{equation}\label{eq:c:NormalContractionIncrements:1}
E(z+u-c\car)+E(z-u+c\car)\leq E(z+v)+E(z-v)\fstop
\end{equation}
It remains to show that~$E$ is invariant under translations by constant functions, i.e.
\begin{equation}\label{eq:c:NormalContractionIncrements:2}
E(\emparg+\alpha\car)=E\comma \qquad \alpha\in\R\fstop
\end{equation}
The functional~$\ground{E}=E-E(0)$ is proper convex lower semicontinuous, and~$\ground{E}(\alpha\car)=0$ for every~$\alpha\in\R$.
By Lemma~\ref{l:ZeroLines} with~$\car$ in place of~$g$, we have~$\dom{E}+\alpha\car\subset\dom{E}$ and~$\ground{E}(\emparg+\alpha\car)=\ground{E}$ on~$\dom{E}$, for every~$\alpha\in\R$.
Applying the first inclusion with~$-\alpha$ in place of~$\alpha$ shows that~$\dom{E}+\alpha\car=\dom{E}$, and thus both sides of~\eqref{eq:c:NormalContractionIncrements:2} equal~$+\infty$ outside~$\dom{E}$.
This shows~\eqref{eq:c:NormalContractionIncrements:2}, and the conclusion follows from~\eqref{eq:c:NormalContractionIncrements:1}, since~$E(z+u)=E(z+u-c\car)$ and~$E(z-u)=E(z-u+c\car)$ by~\eqref{eq:c:NormalContractionIncrements:2}.
\end{proof}
\end{corollary}

\begin{remark}\label{r:NormalContractionIncrements}
\begin{enumerate*}[$(a)$]
\item For~$E(u)=\mcE(u,u)$ as in Remark~\ref{r:NormalContractionQuadratic}, the assumption in Corollary~\ref{c:NormalContractionIncrements}\ref{i:c:NormalContractionIncrements:2} reads $\car\in\dom{E}$ and~$\mcE(\car,\car)=0$, equivalently~$\mcE(u,\car)=0$ for every~$u\in\dom{E}$, by the Cauchy--Schwarz inequality.
Thus, Corollary~\ref{c:NormalContractionIncrements} extends~\cite[Rmk.~I.4.14(ii)]{MaRoe92} (with~$n=1$) to the non-linear setting, and adds the case~$\mssm X=\infty$, in which the pointwise bound~$\abs{u}\leq\abs{v}$ is automatic for elements of~$L^2_\mssm$.

\item Without either assumption in Corollary~\ref{c:NormalContractionIncrements}, the pointwise bound~$\abs{u}\leq\abs{v}$ may not be dispensed with, not even for quadratic functionals.
Indeed, let~$X\eqdef\set{x_0}$ and~$\mssm\eqdef\delta_{x_0}$, so that~$L^2_\mssm\cong\R$, and let~$E(s)\eqdef s^2$.
Here,~$u\preccurlyeq v$ if and only if~$\abs{u}\leq\abs{v}$, and therefore~$E$ is a Dirichlet functional by Corollary~\ref{c:NormalContractionQuadratic}.
The elements~$z=v\eqdef 0$ and~$u\eqdef\car$ satisfy the assumption on increments in Corollary~\ref{c:NormalContractionIncrements} trivially, while~$E(\car)+E(-\car)=2>0=E(0)+E(0)$.
\end{enumerate*}
\end{remark}

\subsubsection{Characterization via one-sided contractions}
Let us show here that $\cproj_\alpha$-non-expansiveness~\ref{i:d:Form:8} may be replaced by its one-sided counterparts.

\begin{definition}\label{d:OneSided}
A functional~$E\colon L^2_\mssm\to (-\infty,\infty]$ is
\begin{enumerate}[$(E_{\arabic*}^\pm)$]\setcounter{enumi}{7}
\item\label{i:d:Form:8pm} \emph{$\cproj_\alpha^\pm$-non-expansive} if
\[
E^\opl{2}(\cproj_\alpha^\pm w) \leq E^\opl{2}(w) \comma \qquad w \in L^2_\mssm\oplus L^2_\mssm \comma \qquad \alpha\geq 0\semicolon
\]
\end{enumerate}
and it has
\begin{enumerate}[$(E_{\arabic*}'^\pm)$]\setcounter{enumi}{7}
\item\label{i:d:Form:8'pm} the \emph{strong upper}, resp.\ \emph{lower}, \emph{unit-contraction property} if, for every~$u,v\in \dom{E}$ and every~$\alpha\geq 0$,
\[
h_\alpha^\pm(u,v)\comma k_\alpha^\pm(u,v)\in\dom{E} \quad \text{and}\quad E\tparen{h_\alpha^\pm(u,v)}+E\tparen{k_\alpha^\pm(u,v)} \leq E(u)+E(v) \fstop
\]
\end{enumerate}
\end{definition}

\begin{proposition}\label{p:HKpm}
Let~$E\colon L^2_\mssm\to (-\infty,\infty]$ be \emph{any} functional.
\begin{enumerate}[$(i)$, wide]
\item\label{i:p:HKpm:a} 
$(E_8^+)$ and~$(E_8^-)$ are equivalent, and so are~$(E_8'^+)$ and~$(E_8'^-)$.
Furthermore, if~$E$ satisfies~$(E_8^+)$, equivalently~$(E_8^-)$, then it is Markovian~\ref{i:d:Form:9}.

\item\label{i:p:HKpm:b} If~$E$ is additionally proper convex lower semicontinuous, then the following are equivalent:
\begin{enumerate}[$(a)$]
\item\label{i:p:HKpm:1} $E$ satisfies~$(E_8'^+)$, equivalently~$(E_8'^-)$;
\item\label{i:p:HKpm:2} $E$ satisfies~$(E_8^+)$, equivalently~$(E_8^-)$;
\item\label{i:p:HKpm:3} $E$ is a Dirichlet functional.
\end{enumerate}
\end{enumerate}

\begin{proof}
\paragraph{Proof of~\ref{i:p:HKpm:a}}
Let~$\sigma\colon (u,v)\mapsto (v,u)$ be the transposition of coordinates on~$L^2_\mssm\oplus L^2_\mssm$, and note that~$E^\opl{2}\circ\sigma=E^\opl{2}$.
By~\eqref{eq:StripProjections2} and~\eqref{eq:HKpm} respectively,
\[
\cproj_\alpha^-=\sigma\circ\cproj_\alpha^+\circ\sigma \qquad \text{and} \qquad \tparen{h_\alpha^-(u,v),k_\alpha^-(u,v)}=\sigma\tparen{h_\alpha^+(v,u),k_\alpha^+(v,u)} \comma
\]
which proves both equivalences.
Now assume~$(E_8^+)$, hence also~$(E_8^-)$.
Since~$C_0^+=C_\tleq$, we have~$\cproj_0^+=\cproj_\tleq$, and~$(E_8^+)$ with~$\alpha=0$ is~\ref{i:d:Form:7}.
Furthermore, by Lemma~\ref{l:CompositionPM}, for every~$\alpha\in[0,\infty]$ and every~$w$,
\[
E^\opl{2}(\cproj_\alpha w)= E^\opl{2}\tparen{\cproj_\alpha^+\tparen{\cproj_\alpha^-(w)}} \leq E^\opl{2}\tparen{\cproj_\alpha^-(w)} \leq E^\opl{2}(w) \comma
\]
which is~\ref{i:d:Form:8}.

\paragraph{Proof of~\ref{i:p:HKpm:b}}
By~\ref{i:p:HKpm:a}, it suffices to consider the `upper' conditions.

\paragraph{\ref{i:p:HKpm:3}$\implies$\ref{i:p:HKpm:1}}
Fix~$\alpha\geq 0$ and set~$\psi(t)\eqdef(t-\alpha)_+$, so that~$\psi\in\ground{\Phi}^\tleq$.
By~\eqref{eq:HKpm},~$u-\psi\circ(u-v)=h_\alpha^+(u,v)$ and~$v+\psi\circ(u-v)=k_\alpha^+(u,v)$ for all~$u,v\in L^2_\mssm$, and~$(E_8'^+)$ follows from Proposition~\ref{p:CP}\ref{i:p:CP:1}$\implies$\ref{i:p:CP:2}.

\paragraph{\ref{i:p:HKpm:1}$\implies$\ref{i:p:HKpm:2}}
Fix~$w=(u,v)$ with~$u,v\in\dom{E}$; otherwise there is nothing to prove.
For~$\alpha\in[0,\infty)$, comparing~\eqref{eq:StripProjections2} with~\eqref{eq:HKpm} shows that~$\tparen{h_\alpha^+(u,v),k_\alpha^+(u,v)}=2\cproj_\alpha^+(w)-w$ is the reflection of~$w$ through~$C_\alpha^+$.
Thus, by convexity of~$E^\opl{2}$ and by~$(E_8'^+)$,
\[
E^\opl{2}\tparen{\cproj_\alpha^+(w)}\leq \tfrac12 E^\opl{2}(w)+\tfrac12 E^\opl{2}\tparen{2\cproj_\alpha^+(w)-w}\leq E^\opl{2}(w)\fstop
\]
For~$\alpha=\infty$ there is nothing to prove, since~$\cproj_\infty^+$ is the identity.
For~$\alpha=0$, note that~$\abs{\cproj_\alpha^+(w)-\cproj_0^+(w)}\leq \tfrac12\tparen{\alpha\wedge(u-v)_+}$ componentwise; hence~$\cproj^+_\alpha(w)\to\cproj^+_0(w)$ in~$L^2_\mssm\oplus L^2_\mssm$ as~$\alpha\downarrow 0$ by Dominated Convergence, and the assertion follows from the case~$\alpha>0$ by lower semicontinuity of~$E^\opl{2}$.

\paragraph{\ref{i:p:HKpm:2}$\implies$\ref{i:p:HKpm:3}}
By~\ref{i:p:HKpm:a},~$E$ is Markovian, hence a Dirichlet functional.
\end{proof}
\end{proposition}

\subsubsection{Lower semicontinuous relaxations}
For a (proper) functional~$E\colon L^2_\mssm\to (-\infty,\infty]$, we denote by~$\check E\colon L^2_\mssm\to (-\infty,\infty]$ its lower semicontinuous envelope.

\begin{proposition}[Lower semicontinuous relaxation, cf.~{\cite[Cor.~4.9, 4.10]{CipGri03}}]\label{p:LSCEnvelope}
Let~$E\colon L^2_\mssm\to (-\infty,\infty]$ be proper, and assume that~$\check E>-\infty$ on~$L^2_\mssm$. If~$E$ satisfies~\ref{i:d:Form:7'}, resp.~\ref{i:d:Form:8'}, then so does~$\check E$.
In particular, if~$E$ is a pre-Dirichlet functional then~$\check E$ is a Dirichlet functional.
\end{proposition}
\begin{proof}
Fix~$\eps>0$ and~$u,v\in\dom{\check E}$. By definition of lower semicontinuous envelope there exist sequences~$\seq{u_n}_n,\seq{v_n}_n\subset L^2_\mssm$ such that
\begin{subequations}
\begin{align}
\label{eq:p:LSCEnvelope:1A}
L^2_\mssm\text{-}\lim_n u_n&=u\comma & L^2_\mssm\text{-}\lim_n v_n&=v
\\
\label{eq:p:LSCEnvelope:1B}
\exists \lim_n E(u_n) &\leq \check{E}(u)+\eps\comma & \exists \lim_n E(v_n) &\leq \check{E}(v)+\eps \fstop
\end{align}
By~\eqref{eq:p:LSCEnvelope:1A} we also have
\begin{align}\label{eq:p:LSCEnvelope:1C}
L^2_\mssm\text{-}\lim_n u_n\wedge v_n&=u\wedge v \comma & L^2_\mssm\text{-}\lim_n u_n\vee v_n&=u\vee v \fstop
\end{align}
\end{subequations}

Respectively: by definition of lower semicontinuous envelope and~\eqref{eq:p:LSCEnvelope:1C}, by superadditivity of the limit inferior, by~\ref{i:d:Form:7'}, and finally by~\eqref{eq:p:LSCEnvelope:1B},
\begin{align*}
\check E(u\wedge v) + \check E(u\vee v) &\leq \liminf_n E(u_n\wedge v_n) + \liminf_n E(u_n\vee v_n)
\\
&\leq \liminf_n \tparen{E(u_n\wedge v_n) +  E(u_n\vee v_n) }
\\
&\leq \liminf_n \tparen{E(u_n)+ E(v_n)}
\\
&\leq \check E(u) + \check E(v) + 2\eps\fstop
\end{align*}
Thus,~$\check E$ satisfies~\ref{i:d:Form:7'} by arbitrariness of~$\eps$.
A similar argument shows that~$\check E$ satisfies as well~\ref{i:d:Form:8'}.
Since~$\check E$ is lower semicontinuous, it is convex in light of Remark~\ref{r:CipGri}, thus a Dirichlet functional.
\end{proof}

\subsubsection{Continuity properties}
Even Dirichlet functionals display good continuity properties along converging sequences of Lipschitz contractions.
We collect here some of these properties, extending the results in~\cite[\S3.6]{Cla21} to the case of Dirichlet functionals which are possibly not linearly defined.

\begin{proposition}[Continuity along post-compositions]\label{p:Continuity2}
Let $E\colon L^2_\mssm\to(-\infty,\infty]$ be an even Dirichlet functional.
Further let~$\seq{\psi_n}_n\subset \ground{\Phi}^\tleq$ be a sequence of \emph{non-decreasing} functions satisfying~$\lim_{n\to\infty} \psi_n(t)=t$ for every~$t\in\R$.
Then, there exist
\[
\lim_{n\to\infty} E\paren{\psi_n\circ u}= E(u) \qquad \text{and} \qquad \lim_{n\to\infty} E\paren{u-\psi_n\circ u}= E(0) \comma \qquad u\in\dom{E} \fstop
\]

\begin{proof}
Fix~$u\in\dom{E}$.
Combining Propositions~\ref{p:Barthelemy} and~\ref{p:BH} (since~$E$ is even), and since~$\psi_n\in \ground{\Phi}$, we see that
\begin{equation}\label{eq:p:Continuity2:1}
E(\psi_n\circ u)\leq E(u) \comma \qquad n\in \N\fstop
\end{equation}

Furthermore, again since~$\psi_n\in\ground{\Phi}$, we have~$\abs{\psi_n(t)}\leq \abs{t}$ for every~$t\in\R$ and thus~$\abs{\psi_n\circ u}\leq \abs{u}$ for every~$n\in \N$. 
It follows that $L^2_\mssm$-$\lim_{n\to\infty} \psi_n\circ u=u$ by Dominated Convergence in~$L^2_\mssm$ with dominating function~$\abs{u}$. By lower semicontinuity of~$E$, and by~\eqref{eq:p:Continuity2:1},
\[
E(u)\leq \liminf_{n\to\infty} E(\psi_n\circ u) \leq \liminf_{n\to\infty} E(u)=E(u)\comma
\]
which proves the existence of the first limit.

Combining Propositions~\ref{p:Barthelemy} and~\ref{p:CP}\ref{i:p:CP:2}, and since~$\psi_n\in \ground{\Phi}^\tleq$, we have, choosing~$v=0$ in~\eqref{eq:Claus},
\begin{equation}\label{eq:p:Continuity:2}
E(u-\psi_n\circ u) \leq E(u)+E(0) - E(\psi_n\circ u) \fstop
\end{equation}
Since~$L^2_\mssm$-$\lim_{n\to\infty} u-\psi_n\circ u =0$ by Dominated Convergence in~$L^2_\mssm$ with dominating function~$2\abs{u}$, taking limit superiors on both sides of~\eqref{eq:p:Continuity:2}, and by lower semicontinuity of~$E$, we obtain
\begin{align*}
E(0)\leq& \liminf_{n\to\infty} E(u-\psi_n\circ u)\leq \limsup_{n\to\infty} E(u-\psi_n\circ u)
\\
\leq& \limsup_{n\to\infty} \tparen{E(u)+E(0) - E(\psi_n\circ u)}
\end{align*}
and the existence of the second limit follows from that of the first one.
\end{proof}
\end{proposition}

By the very same proof of Proposition~\ref{p:Continuity2} we also have:

\begin{corollary}[cf.~{\cite[Lem.~3.54]{Cla21}}]\label{c:Claus2}
Let $E\colon L^2_\mssm\to(-\infty,\infty]$ be an even Dirichlet functional.
Fix~$u\in\dom{E}$ and let~$\seq{\psi_n}_n\subset \ground{\Phi}^\tleq$ be a sequence of \emph{non-decreasing} functions satisfying~$L^2_\mssm\text{-}\lim_{n\to\infty} \psi_n\circ u=u$.
Then there exist
\[
\lim_{n\to\infty} E\paren{\psi_n\circ u}= E(u) \qquad \text{and} \qquad \lim_{n\to\infty} E\paren{u-\psi_n\circ u}= E(0) \fstop
\]
\end{corollary}

The next corollary extends the results in~\cite[Thm.~3.55(i)-(ii)]{Cla21} to Dirichlet functionals that are not linearly defined.

\begin{corollary}[Continuity along~$H_\alpha$]\label{c:Continuity}
Let~$E\colon L^2_\mssm\to(-\infty,\infty]$ be an even Dirichlet functional.
Then, for every~$u\in\dom{E}$ and every~$\alpha_0\geq 0$ there exist
\[
\lim_{\alpha\to \alpha_0} E\tparen{H_\alpha(u,0)}= E(H_{\alpha_0}(u,0))\qquad \text{and} \qquad \lim_{\alpha\to \alpha_0} E\tparen{H_\alpha(0,u)} = E(H_{\alpha_0}(0,u)) \fstop
\]

\begin{proof}
Fix~$u\in\dom{E}$ and~$\alpha_0\geq 0$.
For~$\alpha\geq 0$ set~$\kappa_\alpha\eqdef H_\alpha(\emparg,0)$ and~$\sigma_\alpha\eqdef H_\alpha(0,\emparg)$, so that~$H_\alpha(u,0)=\kappa_\alpha\circ u$ and~$H_\alpha(0,u)=\sigma_\alpha\circ u$.
Note that~$\kappa_\alpha$ and~$\sigma_\alpha$ are non-decreasing and in~$\ground{\Phi}$, i.e.~$\kappa_\alpha,\sigma_\alpha\in\ground{\Phi}^\tleq$, that~$\sigma_\alpha=\id_\R-\kappa_\alpha$, and that, for every~$0\leq\beta\leq\alpha$,
\begin{equation}\label{eq:c:Continuity:1}
\kappa_\beta=\kappa_\beta\circ\kappa_\alpha\comma \qquad \sigma_\alpha=\sigma_{\alpha-\beta}\circ\sigma_\beta\comma \qquad \sigma_\beta\circ\kappa_\alpha=\sigma_\beta-\sigma_\alpha \fstop
\end{equation}
Combining Propositions~\ref{p:Barthelemy} and~\ref{p:BH} (since~$E$ is even), we have~$\kappa_\alpha\circ u,\sigma_\alpha\circ u\in\dom{E}$ for every~$\alpha\geq 0$.
Furthermore, combining Propositions~\ref{p:Barthelemy} and~\ref{p:CP}\ref{i:p:CP:2}, the functional~$E$ satisfies~\eqref{eq:Claus}.

Now, let~$\seq{\alpha_n}_n\subset[0,\infty)$ be any sequence with~$\lim_n \alpha_n=\alpha_0$.
Since~$\abs{\kappa_{\alpha_n}(t)-\kappa_{\alpha_0}(t)}\leq \abs{\alpha_n-\alpha_0}\wedge\abs{t}$ for every~$t\in\R$, and since~$\sigma_\alpha=\id_\R-\kappa_\alpha$, by Dominated Convergence in~$L^2_\mssm$ with dominating function~$\abs{u}$ we have
\begin{equation}\label{eq:c:Continuity:2}
L^2_\mssm\text{-}\lim_n \kappa_{\alpha_n}\circ u=\kappa_{\alpha_0}\circ u \qquad \text{and} \qquad L^2_\mssm\text{-}\lim_n \sigma_{\alpha_n}\circ u=\sigma_{\alpha_0}\circ u \fstop
\end{equation}
Thus, by lower semicontinuity of~$E$,
\begin{equation}\label{eq:c:Continuity:3}
E(\kappa_{\alpha_0}\circ u)\leq \liminf_n E(\kappa_{\alpha_n}\circ u) \qquad \text{and} \qquad E(\sigma_{\alpha_0}\circ u)\leq \liminf_n E(\sigma_{\alpha_n}\circ u) \comma
\end{equation}
and it suffices to show the converse inequalities for the limits superior.
Up to splitting~$\seq{\alpha_n}_n$ into (at most) two subsequences, we may and will further assume that either~$\alpha_n\leq\alpha_0$ for every~$n$, or~$\alpha_n\geq\alpha_0$ for every~$n$.

\paragraph{Case~$\alpha_n\leq\alpha_0$ for every~$n$}
By~\eqref{eq:c:Continuity:1} with~$\beta=\alpha_n$ and~$\alpha=\alpha_0$, we have~$\kappa_{\alpha_n}\circ u=\kappa_{\alpha_n}\circ(\kappa_{\alpha_0}\circ u)$, and therefore, by the Lipschitz-contraction property~\ref{i:d:Form:11} applied to~$\kappa_{\alpha_0}\circ u\in\dom{E}$,
\[
E(\kappa_{\alpha_n}\circ u)\leq E(\kappa_{\alpha_0}\circ u) \comma \qquad n\in\N \fstop
\]
Furthermore, since~$u-\sigma_{\alpha_0}\circ u=\kappa_{\alpha_0}\circ u$, again by~\eqref{eq:c:Continuity:1} we have~$\sigma_{\alpha_n}\circ(u-\sigma_{\alpha_0}\circ u)=\sigma_{\alpha_n}\circ u-\sigma_{\alpha_0}\circ u$.
Thus, choosing~$\sigma_{\alpha_0}\circ u$ in place of~$v$ and~$\psi\eqdef\sigma_{\alpha_n}\in\ground{\Phi}^\tleq$ in~\eqref{eq:Claus},
\begin{equation}\label{eq:c:Continuity:4}
E(\kappa_{\alpha_n}\circ u+\sigma_{\alpha_0}\circ u)+E(\sigma_{\alpha_n}\circ u) \leq E(u)+E(\sigma_{\alpha_0}\circ u) \comma \qquad n\in\N \fstop
\end{equation}
By~\eqref{eq:c:Continuity:2}, we have~$L^2_\mssm$-$\lim_n (\kappa_{\alpha_n}\circ u+\sigma_{\alpha_0}\circ u)=\kappa_{\alpha_0}\circ u+\sigma_{\alpha_0}\circ u=u$, and thus~$E(u)\leq\liminf_n E(\kappa_{\alpha_n}\circ u+\sigma_{\alpha_0}\circ u)$ by lower semicontinuity of~$E$.
Since~$E(u)<\infty$, taking limits superior on both sides of~\eqref{eq:c:Continuity:4} we conclude that
\[
\limsup_n E(\sigma_{\alpha_n}\circ u)\leq E(u)+E(\sigma_{\alpha_0}\circ u)-\liminf_n E(\kappa_{\alpha_n}\circ u+\sigma_{\alpha_0}\circ u)\leq E(\sigma_{\alpha_0}\circ u) \fstop
\]

\paragraph{Case~$\alpha_n\geq\alpha_0$ for every~$n$}
By~\eqref{eq:c:Continuity:1} with~$\beta=\alpha_0$ and~$\alpha=\alpha_n$, we have~$\sigma_{\alpha_n}\circ u=\sigma_{\alpha_n-\alpha_0}\circ(\sigma_{\alpha_0}\circ u)$, and therefore, by~\ref{i:d:Form:11} applied to~$\sigma_{\alpha_0}\circ u\in\dom{E}$,
\[
E(\sigma_{\alpha_n}\circ u)\leq E(\sigma_{\alpha_0}\circ u) \comma \qquad n\in\N \fstop
\]
Furthermore, since~$u-\kappa_{\alpha_0}\circ u=\sigma_{\alpha_0}\circ u$, again by~\eqref{eq:c:Continuity:1} we have~$\sigma_{\alpha_n-\alpha_0}\circ(u-\kappa_{\alpha_0}\circ u)=\sigma_{\alpha_n}\circ u$.
Thus, choosing~$\kappa_{\alpha_0}\circ u$ in place of~$v$ and~$\psi\eqdef\sigma_{\alpha_n-\alpha_0}\in\ground{\Phi}^\tleq$ in~\eqref{eq:Claus},
\begin{equation}\label{eq:c:Continuity:5}
E(\kappa_{\alpha_n}\circ u)+E(\kappa_{\alpha_0}\circ u+\sigma_{\alpha_n}\circ u) \leq E(u)+E(\kappa_{\alpha_0}\circ u) \comma \qquad n\in\N \fstop
\end{equation}
By~\eqref{eq:c:Continuity:2}, we have~$L^2_\mssm$-$\lim_n (\kappa_{\alpha_0}\circ u+\sigma_{\alpha_n}\circ u)=u$, and, arguing as in the previous case, taking limits superior on both sides of~\eqref{eq:c:Continuity:5} we conclude that~$\limsup_n E(\kappa_{\alpha_n}\circ u)\leq E(\kappa_{\alpha_0}\circ u)$.

\paragraph{Conclusion}
In either case, combining the above upper bounds with~\eqref{eq:c:Continuity:3}, we have
\[
\lim_n E(\kappa_{\alpha_n}\circ u)=E(\kappa_{\alpha_0}\circ u) \qquad \text{and} \qquad \lim_n E(\sigma_{\alpha_n}\circ u)=E(\sigma_{\alpha_0}\circ u) \comma
\]
and the assertion follows by arbitrariness of the sequence~$\seq{\alpha_n}_n$.
\end{proof}
\end{corollary}

\begin{corollary}[Continuity at~$0$]\label{c:Continuity2}
Let~$E\colon L^2_\mssm\to(-\infty,\infty]$ be an even Dirichlet functional.
Then, there exists
\[
\lim_{\alpha\downarrow 0} E(u_+\wedge \alpha) = E(0) \comma \qquad u\in\dom{E} \fstop
\]
\begin{proof}
Since~$u_+=u\vee 0$ and~$0\in\dom{E}$ by Remark~\ref{r:Even}, we have~$u_+\in\dom{E}$ by the lattice contraction property~\ref{i:d:Form:7'} (cf.\ Prop.~\ref{p:Barthelemy}).
Since furthermore~$H_\alpha(u_+,0)=u_+\wedge\alpha$ for every~$\alpha\geq 0$, the conclusion follows from Corollary~\ref{c:Continuity} with~$u_+$ in place of~$u$ and~$\alpha_0=0$.
\end{proof}
\end{corollary}

\begin{remark}
\begin{enumerate*}[$(a)$]
\item In the linear case, Corollary~\ref{c:Continuity2} is a well-known continuity property of Dirichlet forms, see e.g.~\cite[Lem.~5.1.3.1, pp.~29-30]{BouHir91}.

\item Proposition~\ref{p:Continuity2} ought to be compared with~\cite[Lem.~3.3]{SchZim25}.
\end{enumerate*}
\end{remark}

\subsection{Equivalences}
Combining the various equivalence results stated above, we obtain the general equivalence theorem already summarized in Figure~\ref{fig:Intro}.
In the proof of the next result, and repeatedly in the following, we apply Proposition~\ref{p:EquivConvex} to doubles.
To this end, let~$E$ be proper convex lower semicontinuous with~$A\eqdef\partial E$, and note that~$\partial(E^\opl{2})=A^\opl{2}$; indeed, the inclusion~$\supset$ is immediate from~\eqref{eq:ConvexSubDiff}, and the converse one follows by testing~\eqref{eq:ConvexSubDiff} for~$E^\opl{2}$ at~$(f_1,f_2)$ with~$(g_1,f_2)$ and with~$(f_1,g_2)$.
Since all the involved constructions act componentwise, the resolvent, the Yosida regularization, and the minimal section of~$A^\opl{2}$ are respectively~$(J_\lambda)^\opl{2}$, $(A_\lambda)^\opl{2}$, and~$(A^\circ)^\opl{2}$, the Yosida regularization of~$E^\opl{2}$ is~$(E_\lambda)^\opl{2}$, and~$\overline{\dom{A^\opl{2}}}=\overline{\dom{A}}^\opl{2}$.
Finally, the semigroup generated by~$A^\opl{2}$ is~$(T_t)^\opl{2}$, $t\geq 0$, by the uniqueness in Theorem~\ref{t:BrezisPazy} and by density of~$\dom{A}^\opl{2}$ in~$\overline{\dom{A}}^\opl{2}$.

\begin{theorem}\label{t:Main}
Let~$E\colon L^2_\mssm \to (-\infty,\infty]$ be proper convex lower semicontinuous, with associated strongly continuous non-expansive non-linear semigroup $T_\bullet\colon \overline{\dom{E}}\to L^2_\mssm$, strongly continuous non-expansive non-linear resolvent~$J_\bullet\colon L^2_\mssm\to L^2_\mssm$, and generator~$A$.
Further assume that
\begin{equation}\label{eq:t:Main:0}
\cproj_{\overline{\dom{A^\opl{2}}}} C_\tleq \subset C_\tleq \qquad \text{and} \qquad \cproj_{\overline{\dom{A^\opl{2}}}} C_\alpha \subset C_\alpha \comma \quad \alpha\geq 0 \fstop
\end{equation}

Then, the following are equivalent:
\begin{itemize}
\item $T_\bullet$ satisfies~\ref{i:d:Semigroups:5}, equivalently~\ref{i:d:Semigroups:5'};

\item $A$ satisfies~\ref{i:d:Generator:7'};

\item $J_\bullet$ satisfies either~\ref{i:d:Semigroups:5} or~\ref{i:d:Semigroups:5'};

\item $E$ satisfies either~\ref{i:d:Form:7} or~\ref{i:d:Form:7'}.
\end{itemize}

Furthermore, the following are equivalent:
\begin{itemize}
\item $T_\bullet$ satisfies either~\ref{i:d:Semigroups:6} or~\ref{i:d:Semigroups:6'};

\item $A$ satisfies~\ref{i:d:Generator:8'};

\item $J_\bullet$ satisfies either~\ref{i:d:Semigroups:6} or~\ref{i:d:Semigroups:6'};

\item $E$ satisfies either~\ref{i:d:Form:8} or~\ref{i:d:Form:8'}.
\end{itemize}

In particular, the following are equivalent:
\begin{enumerate}[$(i)$]
\item $T_\bullet$ is sub-Markovian;

\item\label{i:t:Main:2} $A$ is a Dirichlet operator;

\item $J_\bullet$ is sub-Markovian;

\item\label{i:t:Main:4} $E$ is a Dirichlet functional.
\end{enumerate}
Without assuming~\eqref{eq:t:Main:0}, in each of the three groups above the assertion for~$J_\bullet$ (for~$J_\lambda$ with~$\lambda>0$) and the assertion for~$E$ are still equivalent to each other; furthermore, they imply the assertions for~$T_\bullet$ and for~$A$, as well as the corresponding inclusion(s) in~\eqref{eq:t:Main:0}.
\end{theorem}

\begin{proof}
Combine Propositions~\ref{p:Barthelemy} and~\ref{p:SubMarkovianConvex} with Proposition~\ref{p:EquivConvex} applied to the double of operators/function\-als with~$C$ either~$C_\tleq$, or~$C_\alpha$, $\alpha\geq 0$.

As for the last assertion, the equivalence~\ref{i:p:EquivConvex:1}$\iff$\ref{i:p:EquivConvex:2} in Proposition~\ref{p:EquivConvex} holds without~\eqref{eq:p:EquivConvex:0}, and so does the chain~\ref{i:p:EquivConvex:2}$\implies$\ref{i:p:EquivConvex:1}$\implies$\ref{i:p:EquivConvex:8}$\implies$\ref{i:p:EquivConvex:3} in its proof.
Furthermore, let~$C$ be closed convex with~$J_\lambda(C)\subset C$ for every~$\lambda>0$.
Then~$T_t\tparen{\overline{\dom{A}}\cap C}\subset C$ by the exponential formula~$T_t(f)=\lim_{n\to\infty}(J_{t/n})^n(f)$, $f\in\overline{\dom{A}}$, see~\cite[Cor.~4.4, p.~126]{Bre73}; and~$\cproj_{\overline{\dom{A}}}(C)\subset C$, since~$\cproj_{\overline{\dom{A}}}(f)=\lim_{\lambda\downarrow 0}J_\lambda(f)$ by~\cite[Thm.~2.2, p.~27]{Bre73}.
Applying this to the doubles with~$C=C_\tleq$ and~$C=C_\alpha$, $\alpha\geq 0$, proves the last assertion.
\end{proof}

\begin{remark}\label{r:Main0}
\begin{enumerate*}[$(a)$]
\item\label{i:r:Main0:1} By Proposition~\ref{p:SubMarkovianConvex} with~$B=\cproj_{\overline{\dom{E}}}$ and~$D=L^2_\mssm$, assumption~\eqref{eq:t:Main:0} holds if and only if the convex projection~$\cproj_{\overline{\dom{E}}}$ is sub-Markovian, i.e.\ if and only if~$\mbfI_{\overline{\dom{E}}}$ is a Dirichlet functional.
In particular, it holds if~$E$ is densely defined.

\item\label{i:r:Main0:2} Assumption~\eqref{eq:t:Main:0} may not be dispensed with in the implications starting from~$T_\bullet$ or from~$A$.
Indeed, let~$X\eqdef\set{1,2}$ and~$\mssm\eqdef\delta_1+\delta_2$, so that~$L^2_\mssm\cong\R^2$, and let~$E\eqdef\mbfI_C$ with~$C\eqdef\set{(s,-s) : s\in\R}$.
Then~$A=\partial\mbfI_C$ satisfies~$A^\circ\equiv 0$ on~$\dom{A}=C$, hence it is a Dirichlet operator, and~$T_t=\id_C$ is sub-Markovian, since~$u\leq v$ with~$u,v\in C$ forces~$u=v$.
However,~$J_\lambda=\cproj_C$ is not order-preserving, since~$(0,0)\leq(1,0)$ while~$\cproj_C(1,0)=(\tfrac12,-\tfrac12)\not\geq(0,0)=\cproj_C(0,0)$; accordingly,~$E$ violates~\ref{i:d:Form:7'}, since~$(0,0)\vee(1,-1)=(1,0)\notin C$.
\end{enumerate*}
\end{remark}

\begin{remark}
Let us note in passing that, for~$A=\partial E$ with~$E$ as in Theorem~\ref{t:Main} satisfying~\eqref{eq:t:Main:0}, \ref{i:d:Generator:7'} and~\ref{i:d:Generator:8'} together are equivalent to the \emph{$T$-accretivity} of~$A$ on~$L^1_\mssm$;
see~\cite[Prop.~1.2(ii), p.~17]{BenPic79} for the definition of $T$-accretivity.
For a proof of this, combine~\cite[Thm.~2.1(iii)$\iff$(ii), p.~22]{BenPic79} with Theorem~\ref{t:Main}\ref{i:t:Main:4}$\iff$\ref{i:t:Main:2} and Proposition~\ref{p:CP}\ref{i:p:CP:1}$\iff$\ref{i:p:CP:2}.
Assumption~\eqref{eq:t:Main:0} cannot be dispensed with: the Dirichlet operator~$A=\partial\mbfI_C$ in Remark~\ref{r:Main0}\ref{i:r:Main0:2} is not $T$-accretive, as shown by choosing~$u\eqdef(1,-1)$, $v\eqdef(0,0)$, and noting that~$(-1,-1)\in A(u)$, $(0,0)\in A(v)$, and~$\lambda=1$ in the definition of $T$-accretivity from~\cite[p.~7]{BenPic79}.
\end{remark}

\section{Locality}\label{s:Locality}
The goal of this section is to extend to Dirichlet functionals the well-known characterization of locality of bilinear Dirichlet forms in~\cite[\S{I.1.5}]{BouHir91}.
In the non-linear case, the \emph{uni}variate conditions characterizing (weak and strong) locality of bilinear forms need to be replaced by \emph{bi}variate ones, in the spirit of Remark~\ref{r:SubMarkovUnivariate}.
The univariate conditions are recovered under a mild assumption on the domain, namely \emph{the zero-product-cone condition} (Dfn.~\ref{d:ZPC}).
We first address weak locality, and then strong locality.

Throughout, we only assume that~$0\in\dom{E}$, and we assume~$E$ to be \emph{even} only where this is needed.
Namely, the characterizations by means of the lattice operations, resp.\ of the one-sided contractions~$h^\pm_\alpha$ and~$k^\pm_\alpha$, hold for every Dirichlet functional~$E$ with~$0\in\dom{E}$, whereas the characterizations involving the absolute value, or more general non-monotone Lipschitz contractions, require evenness, as we show by examples below.

Let~$E\colon L^2_\mssm\to(-\infty,\infty]$ be a functional with~$0\in\dom{E}$.
In the next results it will occasionally be convenient to switch between~$E$ and its grounding~$\ground{E}$.
It is clear from Definition~\ref{d:Dirichlet} that~$\ground{E}$ is a Dirichlet functional if and only if~$E$ is.

\subsection{Weak locality}\label{ss:WLocality}
We start by recalling the standard definition of a `local' functional; cf., e.g.~\cite[Dfn.~6.1, p.~87]{Cla21}.
We prefer the terminology of \emph{weakly local} functional, in order to better distinguish it from the stronger notion of \emph{strongly local} functional.

\begin{definition}[Weak locality]
A functional~$E\colon L^2_\mssm\to (-\infty,\infty]$ with~$0\in\dom{E}$ is called \emph{weakly local} if
for every~$u,v\in\dom{\ground{E}}$ with~$uv= 0$ and~$u+v\in \dom{\ground{E}}$,
we have
\begin{equation}\label{eq:p:WLocality:1}
\ground{E}(u+v)=\ground{E}(u)+\ground{E}(v) \fstop
\end{equation}
\end{definition}

\begin{remark}\label{r:Grounding}
\begin{enumerate}[$(a)$, wide]
\item Suppose~$E$ is a proper convex functional with~$0\in\dom{E}$ and such that, for every~$u,v\in\dom{E}$ with~$uv= 0$ and~$u+v\in \dom{E}$, we have
\begin{equation}\label{eq:r:WLocality:1}
E(u+v)=E(u)+E(v)\fstop
\end{equation}
Computing~\eqref{eq:r:WLocality:1} at~$u=v\equiv 0$ shows that~$E(0)=0$.
In particular, in the definition of weak locality we may not replace~\eqref{eq:p:WLocality:1} with~\eqref{eq:r:WLocality:1}.

\item\label{i:r:Grounding:1} Some conditions considered in this section will be either expressed in terms of~$\ground{E}$, or as identities between two sums of values of~$E$ with the same number of summands on both sides.
In both cases, these conditions will be invariant under replacing~$E$ with~$E+c$ for any constant~$c\in\R$.
Since~$\ground{E}$ is a Dirichlet functional if and only if~$E$ is, in the \emph{proofs} of this section we may assume with no loss of generality that~$E$ is grounded.
\end{enumerate}
\end{remark}

Let~$E\colon L^2_\mssm\to(-\infty,\infty]$ be a Dirichlet functional with~$0\in\dom{E}$.
Then,~$E$ satisfies the lattice contraction property~\ref{i:d:Form:7'} by Proposition~\ref{p:Barthelemy}; in particular,~$u_+=u\vee 0$ and~$-u_-=u\wedge 0$ belong to~$\dom{E}$ for every~$u\in\dom{E}$.
Furthermore, choosing~$v=0$ in~\eqref{eq:Claus}, by Proposition~\ref{p:CP}\ref{i:p:CP:1}$\implies$\ref{i:p:CP:2} we have
\begin{equation}\label{eq:s:Locality:Monotone}
E(\psi\circ u)+E(u-\psi\circ u)\leq E(u)+E(0)\comma \qquad u\in\dom{E}\comma \quad \psi\in\ground{\Phi}^\tleq\comma
\end{equation}
and therefore, since~$E>-\infty$, we have~$\psi\circ u\in\dom{E}$ for every~$u\in\dom{E}$ and every \emph{non-decreasing} grounded Lipschitz contraction~$\psi\in\ground{\Phi}^\tleq$.
If~$E$ is additionally even, then, combining Propositions~\ref{p:Barthelemy} and~\ref{p:BH},~$E$ further satisfies the Lipschitz-contraction property~\ref{i:d:Form:11} for \emph{all}~$\varphi\in\ground{\Phi}$; by Proposition~\ref{p:BH}, the latter fails whenever~$E$ is not even, e.g.\ for the functional in Example~\ref{ex:NotEven}.
(Recall that every even Dirichlet functional~$E$ satisfies~$0\in\dom{E}$ by Remark~\ref{r:Even}.)
Furthermore, for every~$u,v\in L^2_\mssm$ with~$uv=0$,
\begin{equation}\label{eq:p:WLocality:4}
\abs{u}\wedge \abs{v}=0 \quad \text{and} \quad \abs{u}\vee \abs{v}=\abs{u}+\abs{v} = \abs{u+v}\comma
\end{equation}
and
\begin{equation}\label{eq:p:WLocality:4b}
u\vee v=(u+v)\vee 0 \quad \text{and} \quad u\wedge v=(u+v)\wedge 0 \fstop
\end{equation}
In particular, if~$E$ is even and~$u,v\in\dom{E}$ satisfy~$uv=0$, then~$\abs{u}+\abs{v}=\abs{u}\vee\abs{v}\in\dom{E}$ by~\ref{i:d:Form:11} and~\ref{i:d:Form:7'}.
We will make use of these facts throughout this section.

\medskip

We will prove the following characterization of weak locality in terms of \emph{bivariate} conditions.
Their numbering and labelling is modelled after the numbering of the analogous univariate conditions in the textbook presentation of~\cite[\S{I.5}, pp.~28ff.]{BouHir91}.

\begin{proposition}[Characterization of weak locality]\label{p:WLocality}
Let~$E\colon L^2_\mssm\to (-\infty,\infty]$ be a Dirichlet functional with~$0\in\dom{E}$, and consider the following conditions:
\begin{enumerate}[$(\ell_1)$]\setcounter{enumi}{-1}
\item\label{i:p:WLocality:0} for every~$u,v\in\dom{E}$ with~$uv=0$,
\begin{equation}\label{eq:p:WLocality:0}
\ground{E}\tparen{\abs{u}+\abs{v}}=\ground{E}(u)+\ground{E}(v) \semicolon
\end{equation}

\item\label{i:p:WLocality:1} $E$ is weakly local;

\item\label{i:p:WLocality:2} for every~$u,v\in\dom{E}$ with~$uv= 0$,
we have
\[
E(u\vee v)+E(u\wedge v) = E(u) + E(v) \semicolon
\]

\item\label{i:p:WLocality:3} for every~$u,v\in\dom{E}$ with~$uv=0$ and every~$\varphi,\psi\in\ground{\Phi}$ with~$\varphi\psi= 0$,
\begin{equation}\label{eq:p:WLocality:3}
\begin{aligned}
\ground{E}\tparen{\abs{(\varphi+\psi)\circ u}+\abs{(\varphi+\psi)\circ v}} ={}& \ground{E}(\varphi\circ u) + \ground{E}(\psi\circ u)
+\ground{E}(\varphi\circ v) + \ground{E}(\psi\circ v) \fstop
\end{aligned}
\end{equation}
\end{enumerate}
Then,~\ref{i:p:WLocality:3}$\implies$\ref{i:p:WLocality:0}$\implies$\ref{i:p:WLocality:1}$\iff$\ref{i:p:WLocality:2}.
If~$E$ is additionally even, then~\ref{i:p:WLocality:0}, \ref{i:p:WLocality:1}, \ref{i:p:WLocality:2}, and~\ref{i:p:WLocality:3} are all equivalent.
\end{proposition}

Before proving Proposition~\ref{p:WLocality}, we shall first consider two additional conditions~\ref{i:p:WLocality:0d} and~\ref{i:p:WLocality:3d}, which are the univariate versions of~\ref{i:p:WLocality:0} and~\ref{i:p:WLocality:3} respectively.
While these conditions characterize weak locality in the bilinear case, in the non-linear case they are not sufficient, as we shall comment about below.
Here, and elsewhere below for other conditions, a superscript dot indicates that the condition is \emph{uni}variate.

\begin{proposition}[Univariate conditions for weak locality]\label{p:WLocalityUni}
Let~$E\colon L^2_\mssm\to (-\infty,\infty]$ be a Dirichlet functional with~$0\in\dom{E}$, and consider the following conditions:
\begin{enumerate}[$(\dot\ell_1)$]\setcounter{enumi}{-1}
\item\label{i:p:WLocality:0d} for every~$u\in\dom{E}$,
\begin{equation}\label{eq:p:WLocality:0d}
E(\abs{u})=E(u) \semicolon
\end{equation}

\setcounter{enumi}{2}
\item\label{i:p:WLocality:3d} for every~$\varphi,\psi\in\ground{\Phi}$ with~$\varphi\psi= 0$,
we have, for every~$u\in\dom{E}$,
\begin{equation}\label{eq:p:WLocality:3d}
\ground{E}\tparen{(\varphi+\psi)\circ u} = \ground{E}(\varphi\circ u) + \ground{E}(\psi\circ u) \fstop
\end{equation}
\end{enumerate}
Then,
\begin{enumerate}[$(i)$]
\item\label{i:p:WLocalityUni:1} if~\ref{i:p:WLocality:0d} holds, then, for every~$u,v\in\dom{E}$ with~$uv=0$ and~$\abs{u}-\abs{v}\in\dom{E}$,
\begin{equation}\label{eq:p:WLocalityUni:1}
\ground{E}\tparen{\abs{u}+\abs{v}}=\ground{E}(u)+\ground{E}(v) \semicolon
\end{equation}
\item\label{i:p:WLocalityUni:2} if~$E$ is additionally even, then~\ref{i:p:WLocality:0d} and~\ref{i:p:WLocality:3d} are equivalent.
\end{enumerate}
\end{proposition}

Let us note that the condition in~\ref{i:p:WLocalityUni:1} is \emph{strictly} weaker than~\ref{i:p:WLocality:0} in that one further \emph{assumes} that~$\abs{u}-\abs{v}\in\dom{E}$, which may fail in full generality for~$E$; cf.\ Example~\ref{e:WLocality} below.

\begin{proof}[Proof of Proposition~\ref{p:WLocalityUni}]
In light of Remark~\ref{r:Grounding}\ref{i:r:Grounding:1}, we may and will assume with no loss of generality that~$E$ is grounded.

\paragraph{Proof of~\ref{i:p:WLocalityUni:1}}
Assume~\ref{i:p:WLocality:0d}, let~$u,v\in\dom{E}$ be as in~\ref{i:p:WLocalityUni:1}, and note that~$\abs{u},\abs{v}\in\dom{E}$ by~\ref{i:p:WLocality:0d} applied to~$u$ and to~$v$, since~$E(u),E(v)<\infty$.
Respectively: since~$E(0)=0$, by~\eqref{eq:p:WLocality:4}, and by the lattice contraction property~\ref{i:d:Form:7'},
\begin{align}
\nonumber
E\tparen{\abs{u}+\abs{v}} & = E(0)+E\tparen{\abs{u}+\abs{v}} = E(\abs{u}\wedge \abs{v}) + E(\abs{u}\vee \abs{v})
\\
\label{eq:p:WLocality:5}
&\leq E(\abs{u})+ E(\abs{v})\fstop
\end{align}

Furthermore, again since~$uv=0$, we have
\begin{equation}\label{eq:p:WLocality:6}
\abs{u}=(\abs{u}-\abs{v})_+ \quad \text{and} \quad -\abs{v}=-(\abs{u}-\abs{v})_- \fstop
\end{equation}
In particular,~$-\abs{v}=(\abs{u}-\abs{v})\wedge 0\in\dom{E}$ by~\ref{i:d:Form:7'}, and therefore~$E(-\abs{v})=E\tparen{\tabs{-\abs{v}}}=E(\abs{v})$ by~\ref{i:p:WLocality:0d} applied to~$-\abs{v}$.
Respectively: by the latter identity, by~\eqref{eq:p:WLocality:6}, by the lattice contraction property~\ref{i:d:Form:7'} applied to the pair~$\abs{u}-\abs{v}, 0$, and since~$E(0)=0$, by the assumption~\ref{i:p:WLocality:0d} applied to~$\abs{u}-\abs{v}$, and since~$\tabs{\abs{u}-\abs{v}}=\abs{u}+\abs{v}$ because~$uv=0$,
\begin{align}
\nonumber
E(\abs{u})+E(\abs{v}) &= E(\abs{u})+E(-\abs{v}) = E\tparen{(\abs{u}-\abs{v})_+} + E\tparen{-(\abs{u}-\abs{v})_-}
\\
\nonumber
&\leq E\tparen{\abs{u}-\abs{v}} + E(0) = E\tparen{\abs{u}-\abs{v}} = E\tparen{\tabs{\abs{u}-\abs{v}}}
\\
\label{eq:p:WLocality:7}
&= E\tparen{\abs{u}+\abs{v}}\fstop
\end{align}

Finally, combining~\eqref{eq:p:WLocality:5} and~\eqref{eq:p:WLocality:7} and applying the assumption~\ref{i:p:WLocality:0d} to~$u$ and to~$v$,
\[
E\tparen{\abs{u}+\abs{v}}= E(\abs{u})+ E(\abs{v})= E(u)+E(v) \comma
\]
which is~\eqref{eq:p:WLocalityUni:1}.

\paragraph{Proof of~\ref{i:p:WLocalityUni:2}}
Assume now that~$E$ is even.

\paragraph{\ref{i:p:WLocality:3d}$\implies$\ref{i:p:WLocality:0d}}
Fix~$u\in\dom{E}$ and note that~$u_\pm\in\dom{E}$ by~\ref{i:d:Form:11}.
Respectively: since~$\abs{t}=t_++t_-$, by~\ref{i:p:WLocality:3d} with~$\varphi(t)\eqdef t_+$ and~$\psi(t)\eqdef t_-$, since~$E$ is even, by~\ref{i:p:WLocality:3d} with~$\varphi(t)\eqdef t_+$ and~$\psi(t)\eqdef -t_-$, and finally since~$t=t_+-t_-$,
\[
E(\abs{u})=E(u_++ u_-)=E(u_+)+E(u_-)= E(u_+)+E(-u_-) = E(u_+-u_-)=E(u)\comma
\]
which proves the assertion.

\paragraph{\ref{i:p:WLocality:0d}$\implies$\ref{i:p:WLocality:3d}}
Let~$\varphi,\psi\in\ground{\Phi}$ with~$\varphi\psi=0$, fix~$u\in\dom{E}$, and set~$f\eqdef\varphi\circ u$ and~$g\eqdef\psi\circ u$.
Then,~$\varphi+\psi\in\ground{\Phi}$ by Lemma~\ref{l:LipschitzContractionsSplits}\ref{i:l:LipschitzContractionsSplits:2}, hence~$f,g,f+g\in\dom{E}$ by~\ref{i:d:Form:11}, and~$fg=(\varphi\psi)\circ u=0$.
Furthermore,~$\abs{\varphi}$ and~$-\abs{\psi}$ belong to~$\ground{\Phi}$ and satisfy~$\abs{\varphi}\cdot(-\abs{\psi})=-\abs{\varphi\psi}=0$.
Hence,~$\abs{\varphi}-\abs{\psi}\in\ground{\Phi}$ by Lemma~\ref{l:LipschitzContractionsSplits}\ref{i:l:LipschitzContractionsSplits:2}, and therefore~$\abs{f}-\abs{g}=\tparen{\abs{\varphi}-\abs{\psi}}\circ u\in\dom{E}$ by~\ref{i:d:Form:11}.
Respectively: by~\ref{i:p:WLocality:0d} applied to~$f+g$, by~\eqref{eq:p:WLocality:4}, and by~\ref{i:p:WLocalityUni:1},
\[
E(f+g)=E(\abs{f+g})=E\tparen{\abs{f}+\abs{g}}=E(f)+E(g)\comma
\]
which is~\eqref{eq:p:WLocality:3d}.
\end{proof}

\begin{proof}[Proof of Proposition~\ref{p:WLocality}]
In light of Remark~\ref{r:Grounding}\ref{i:r:Grounding:1}, we may and will assume with no loss of generality that~$E$ be grounded.
Note that, since~$E(0)=0$, choosing~$v=0$ in~\ref{i:p:WLocality:0} shows that
\begin{equation}\label{eq:p:WLocality:9}
\text{\ref{i:p:WLocality:0}$\implies$\ref{i:p:WLocality:0d}} \fstop
\end{equation}
In the proofs of the implications~\ref{i:p:WLocality:2}$\implies$\ref{i:p:WLocality:0} and~\ref{i:p:WLocality:0}$\implies$\ref{i:p:WLocality:3} we assume that~$E$ is even; the proofs of all other implications do not require evenness.

\paragraph{\ref{i:p:WLocality:0}$\implies$\ref{i:p:WLocality:1}}
By~\eqref{eq:p:WLocality:9},~\ref{i:p:WLocality:0d} holds.
Let~$u,v\in\dom{E}$ with~$uv=0$ and~$u+v\in\dom{E}$.
Respectively: by~\ref{i:p:WLocality:0d} applied to~$u+v$, by~\eqref{eq:p:WLocality:4}, and by~\ref{i:p:WLocality:0},
\[
E(u+v)=E(\abs{u+v})=E\tparen{\abs{u}+\abs{v}}=E(u)+E(v)\fstop
\]

\paragraph{\ref{i:p:WLocality:1}$\implies$\ref{i:p:WLocality:2}}
Let~$u,v\in\dom{E}$ with~$uv=0$.
Since~$uv=0$, we have
\begin{equation}\label{eq:p:WLocality:8}
u\wedge v=-u_- - v_-\quad \text{and} \quad u\vee v=u_+ + v_+ \fstop
\end{equation}
Furthermore,~$u_-v_-=u_+v_+=u_-u_+=v_-v_+=0$.
Note that~$u_+=u\vee 0$,~$-u_-=u\wedge 0$,~$v_+$,~$-v_-$,~$u\vee v$, and~$u\wedge v$ belong to~$\dom{E}$ by~\ref{i:d:Form:7'}, since~$0\in\dom{E}$, so that each of the pairs below consists of elements of~$\dom{E}$ with sum in~$\dom{E}$.
Respectively: by~\eqref{eq:p:WLocality:8}, by the assumption~\ref{i:p:WLocality:1} applied to the pairs~$-u_-,-v_-$ and~$u_+,v_+$, and by the assumption~\ref{i:p:WLocality:1} applied to the pairs~$-u_-,u_+$ and~$-v_-,v_+$,
\begin{align*}
E(u\wedge v)+ E(u\vee v)&=E(-u_- - v_-)+ E(u_+ + v_+)
\\
&= E(-u_-)+ E(-v_-) + E(u_+) + E(v_+)
\\
&= E(-u_-+u_+)+ E(-v_-+v_+)=E(u)+E(v) \fstop
\end{align*}

\paragraph{\ref{i:p:WLocality:2}$\implies$\ref{i:p:WLocality:1}}
Let~$u,v\in\dom{E}$ with~$uv=0$ and~$u+v\in\dom{E}$, and note that~$(u+v)\cdot 0=0$.
Respectively: by the assumption~\ref{i:p:WLocality:2} applied to the pair~$u,v$, by~\eqref{eq:p:WLocality:4b}, by the assumption~\ref{i:p:WLocality:2} applied to the pair~$u+v,0$, and since~$E(0)=0$,
\[
E(u)+E(v)=E(u\vee v)+E(u\wedge v)=E\tparen{(u+v)\vee 0}+E\tparen{(u+v)\wedge 0}=E(u+v)+E(0)=E(u+v)\fstop
\]

\paragraph{\ref{i:p:WLocality:2}$\implies$\ref{i:p:WLocality:0}}
Assume that~$E$ is even.
We first show that~\ref{i:p:WLocality:2} implies~\ref{i:p:WLocality:0d}.
Fix~$u\in\dom{E}$, and note that~$u_\pm,\abs{u}\in\dom{E}$ by~\ref{i:d:Form:11}.
Since~$u_+u_-=0$, we have~$\abs{u}=u_+\vee u_-$ and~$u_+\wedge u_-=0$.
Thus, respectively: since~$E(0)=0$, by the assumption~\ref{i:p:WLocality:2} applied to the pair~$u_+,u_-$, since~$E$ is even, by the assumption~\ref{i:p:WLocality:2} applied to the pair~$u,0$, and since~$E(0)=0$,
\begin{align*}
E(\abs{u})&=E(u_+\vee u_-)+E(u_+\wedge u_-) = E(u_+)+E(u_-)
\\
&= E(u\vee 0)+E(u\wedge 0)= E(u)+E(0)=E(u) \fstop
\end{align*}
Now, let~$u,v\in\dom{E}$ with~$uv=0$, and note that~$\abs{u},\abs{v}\in\dom{E}$ by~\ref{i:d:Form:11}.
Respectively: by~\eqref{eq:p:WLocality:4} and since~$E(0)=0$, by the assumption~\ref{i:p:WLocality:2} applied to the pair~$\abs{u},\abs{v}$, and by~\ref{i:p:WLocality:0d},
\[
E\tparen{\abs{u}+\abs{v}}=E\tparen{\abs{u}\vee\abs{v}}+E\tparen{\abs{u}\wedge\abs{v}} = E(\abs{u})+E(\abs{v})=E(u)+E(v)\comma
\]
which proves~\ref{i:p:WLocality:0}.

\paragraph{\ref{i:p:WLocality:0}$\implies$\ref{i:p:WLocality:3}}
Assume that~$E$ is even.
By~\eqref{eq:p:WLocality:9},~\ref{i:p:WLocality:0d} holds, and therefore so does~\ref{i:p:WLocality:3d} by Proposition~\ref{p:WLocalityUni}\ref{i:p:WLocalityUni:2}.
Let~$\varphi,\psi\in\ground{\Phi}$ with~$\varphi\psi=0$, and set~$\theta\eqdef\varphi+\psi$.
Then,~$\theta\in\ground{\Phi}$ by Lemma~\ref{l:LipschitzContractionsSplits}\ref{i:l:LipschitzContractionsSplits:2}.
Now, let~$u,v\in\dom{E}$ with~$uv=0$.
Then,~$\theta\circ u,\theta\circ v\in\dom{E}$ by~\ref{i:d:Form:11}, and~$(\theta\circ u)\cdot(\theta\circ v)=0$ since~$\theta(0)=0$.
Respectively: by~\ref{i:p:WLocality:0} applied to the pair~$\theta\circ u,\theta\circ v$, and by~\ref{i:p:WLocality:3d} applied to~$u$ and to~$v$,
\[
E\tparen{\abs{\theta\circ u}+\abs{\theta\circ v}}=E(\theta\circ u)+E(\theta\circ v)=E(\varphi\circ u)+E(\psi\circ u)+E(\varphi\circ v)+E(\psi\circ v)\comma
\]
which is~\eqref{eq:p:WLocality:3}.

\paragraph{\ref{i:p:WLocality:3}$\implies$\ref{i:p:WLocality:0}}
Since~$E(0)=0$, it suffices to choose~$\varphi\eqdef\id_\R$ and~$\psi\eqdef 0$ in~\eqref{eq:p:WLocality:3}.
\end{proof}

In order to recover the univariate conditions from the bivariate ones, we introduce the announced assumption on the domain.

\begin{definition}[Zero-product cone condition]\label{d:ZPC}
A subset~$D\subset L^2_\mssm$ is \emph{zero-product cone} if~$u+v\in D$ for every~$u,v\in D$ with~$uv=0$.
\end{definition}

\begin{remark}\label{r:ZPC}
\begin{enumerate*}[$(a)$]
\item The domain of every linearly defined functional (in particular, of every everywhere finite functional) is zero-product cone.
The domain~$H^1\cap B_1$ of the functional in Example~\ref{e:FormBall} is zero-product cone as well, since~$\abs{u+v}=\abs{u}\vee\abs{v}$ whenever~$uv=0$.
\item If~$\dom{E}$ is a zero-product cone, then the requirement~$u+v\in\dom{\ground{E}}$ in the definition of weak locality is automatically satisfied.
\end{enumerate*}
\end{remark}

Combining Propositions~\ref{p:WLocalityUni} and~\ref{p:WLocality}, we obtain the following complete picture.

\begin{corollary}[Weak locality and univariate conditions]\label{c:WLocality}
Let~$E\colon L^2_\mssm\to (-\infty,\infty]$ be an even Dirichlet functional.
Then,
\begin{enumerate}[$(i)$]
\item\label{i:c:WLocality:1} each of~\ref{i:p:WLocality:0}, \ref{i:p:WLocality:1}, \ref{i:p:WLocality:2}, and~\ref{i:p:WLocality:3} implies both~\ref{i:p:WLocality:0d} and~\ref{i:p:WLocality:3d};
\item\label{i:c:WLocality:2} if~$\dom{E}$ is zero-product cone (e.g., if~$E$ is linearly defined), then the conditions~\ref{i:p:WLocality:0d}, \ref{i:p:WLocality:0}, \ref{i:p:WLocality:1}, \ref{i:p:WLocality:2}, \ref{i:p:WLocality:3d}, and~\ref{i:p:WLocality:3} are all equivalent:
\[
\begin{tikzcd}[column sep=2.2em, row sep=1.8em]
\text{\ref{i:p:WLocality:0}} \arrow[r, Leftrightarrow] \arrow[d, Leftrightarrow] & \text{\ref{i:p:WLocality:1}} \arrow[r, Leftrightarrow] & \text{\ref{i:p:WLocality:2}} \arrow[r, Leftrightarrow] & \text{\ref{i:p:WLocality:3}} \arrow[d, Leftrightarrow]
\\
\text{\ref{i:p:WLocality:0d}} \arrow[rrr, Leftrightarrow] & & & \text{\ref{i:p:WLocality:3d}}
\end{tikzcd}
\]
\end{enumerate}
\end{corollary}

\begin{proof}
\ref{i:c:WLocality:1}
By Proposition~\ref{p:WLocality}, it suffices to consider~\ref{i:p:WLocality:0}.
Choosing~$v=0$ in~\ref{i:p:WLocality:0} shows that~\ref{i:p:WLocality:0d} holds, and~\ref{i:p:WLocality:3d} follows by Proposition~\ref{p:WLocalityUni}\ref{i:p:WLocalityUni:2}.

\ref{i:c:WLocality:2}
The equivalences in the first row are Proposition~\ref{p:WLocality}, the one in the second row is Proposition~\ref{p:WLocalityUni}\ref{i:p:WLocalityUni:2}, and the downward implications are~\ref{i:c:WLocality:1}.
It thus suffices to show that~\ref{i:p:WLocality:0d} implies~\ref{i:p:WLocality:0}.
Assume~\ref{i:p:WLocality:0d}, and let~$u,v\in\dom{E}$ with~$uv=0$.
Then,~$\abs{u},-\abs{v}\in\dom{E}$ by~\ref{i:d:Form:11}, and~$\abs{u}\cdot(-\abs{v})=0$, thus~$\abs{u}-\abs{v}\in\dom{E}$ since~$\dom{E}$ is zero-product cone.
By Proposition~\ref{p:WLocalityUni}\ref{i:p:WLocalityUni:1}, we conclude that~$\ground{E}\tparen{\abs{u}+\abs{v}}=\ground{E}(u)+\ground{E}(v)$, which is~\ref{i:p:WLocality:0}.
\end{proof}

Under the zero-product-cone property, Corollary~\ref{c:WLocality}\ref{i:c:WLocality:2} recovers the well-known characterization of weak locality for bilinear Dirichlet forms.
The bivariate conditions are obtained from the univariate ones as we now show.

\begin{remark}[Univariate vs.\ bivariate conditions]\label{r:Bivariate}
The bivariate conditions~\ref{i:p:WLocality:0} and~\ref{i:p:WLocality:3} are obtained from the univariate conditions~\ref{i:p:WLocality:0d} and~\ref{i:p:WLocality:3d} in the following way: an identity~$\ground{E}\tparen{\Theta(u)}=\Lambda(u)$,~$u\in\dom{E}$, is replaced by
\[
\ground{E}\tparen{\abs{\Theta(u)}+\abs{\Theta(v)}}=\Lambda(u)+\Lambda(v)\comma \qquad u,v\in\dom{E}\comma \quad uv=0\fstop
\]
(For~\ref{i:p:WLocality:0d}, $\Theta(u)=\abs{u}$ and~$\Lambda(u)=\ground{E}(u)$; for~\ref{i:p:WLocality:3d}, $\Theta(u)=(\varphi+\psi)\circ u$ and~$\Lambda(u)=\ground{E}(\varphi\circ u)+\ground{E}(\psi\circ u)$.)
Since~$\Theta(0)=0$ in both cases, we have~$\Theta(u)\Theta(v)=0$, and therefore, if~$E$ is even,~$\abs{\Theta(u)}+\abs{\Theta(v)}=\abs{\Theta(u)}\vee\abs{\Theta(v)}\in\dom{E}$ by~\ref{i:d:Form:7'} and~\ref{i:d:Form:11}, with no assumption on~$\dom{E}$.
For non-even~$E$, the analogous role is played by the pair~$(u\vee v,u\wedge v)=\tparen{(u+v)\vee 0,(u+v)\wedge 0}$, cf.~\eqref{eq:p:WLocality:4b}, which belongs to~$\dom{E}^\opl{2}$ by~\ref{i:d:Form:7'}; this leads to~\ref{i:p:WLocality:2}.
\end{remark}

Without the zero-product-cone property, the implications in Corollary~\ref{c:WLocality}\ref{i:c:WLocality:1} still hold, while their converses may fail, as shown by the next example.

\begin{example}[On the necessity of the zero-product-cone property]\label{e:WLocality}
Let~$X\eqdef\set{1,2}$ and~$\mssm\eqdef\delta_1+\delta_2$, so that~$L^2_\mssm\cong\R^2$ with the Euclidean norm.
Write~$u=(u_1,u_2)$ for~$u\in\R^2$, let~$\rho\colon\R\to[0,\infty)$ be the \emph{Huber loss function}, defined by~$\rho(r)\eqdef \tfrac12 r^2$ if~$\abs{r}\leq 1$ and~$\rho(r)\eqdef \abs{r}-\tfrac12$ if~$\abs{r}> 1$, and set
\[
K\eqdef\set{u\in\R^2 : \abs{u_1-u_2}\leq 1}\comma\qquad E(u)\eqdef \tfrac12(u_1-u_2)^2+\rho(u_1+u_2)+\mbfI_K(u)\fstop
\]
Explicitly, for~$u\in K$,
\begin{equation}\label{eq:e:WLocality:1}
E(u)=\begin{cases} u_1^2+u_2^2 &\text{if } \abs{u_1+u_2}\leq 1\comma \\ \tfrac12(u_1-u_2)^2+\abs{u_1+u_2}-\tfrac12 &\text{if } \abs{u_1+u_2}\geq 1\fstop\end{cases}
\end{equation}
Then,~$E$ is a grounded even Dirichlet functional satisfying~\ref{i:p:WLocality:0d} and~\ref{i:p:WLocality:3d}, but none of~\ref{i:p:WLocality:0}, \ref{i:p:WLocality:1}, \ref{i:p:WLocality:2}, and~\ref{i:p:WLocality:3}.

\paragraph{$E$ is an even Dirichlet functional}
For~$\Psi\colon\R^2\to\R$ and~$u,v\in\R^2$ write~$\Psi(u,v)\eqdef\tparen{\Psi(u_1,v_1),\Psi(u_2,v_2)}$, and let~$(\Psi,\Psi')$ be either~$(\vee,\wedge)$ or~$\tparen{H_\alpha(s,t),H_\alpha(t,s)}$ for some~$\alpha>0$.
Since~$\Psi$ and~$\Psi'$ are obtained from~$s$,~$t$, and~$t\pm\alpha$, resp.~$s\pm\alpha$, by taking maxima and minima, they are $1$-Lipschitz w.r.t.\ the $\ell^\infty$-norm on~$\R^2$.
Thus,~$\Psi(u,v),\Psi'(u,v)\in K$ for every~$u,v\in K$, and~$\Psi(u,v),\Psi'(u,v)\in\Delta\eqdef\set{u\in\R^2: u_1=u_2}$ for every~$u,v\in\Delta$.
Furthermore, for~$i=1,2$, the numbers~$\Psi(u_i,v_i)$ and~$\Psi'(u_i,v_i)$ lie between~$u_i$ and~$v_i$ and have sum~$u_i+v_i$; thus~$f\tparen{\Psi(u_i,v_i)}+f\tparen{\Psi'(u_i,v_i)}\leq f(u_i)+f(v_i)$ for every convex~$f\colon \R\to\R$.
As a consequence, the functionals~$\mbfI_K$ and~$G(u)\eqdef \abs{u_1}+\abs{u_2}+\mbfI_\Delta(u)$ satisfy~\ref{i:d:Form:7'} and~\ref{i:d:Form:8'}, and they are Markovian by Proposition~\ref{p:Barthelemy}, being proper convex lower semicontinuous.

Now, set~$F(u)\eqdef \tfrac12(u_1-u_2)^2+\rho(u_1+u_2)$, and note that~$F$ is the Yosida regularization~$G_{1/2}$ of~$G$ (Dfn.~\ref{d:YosidaRegForm}).
Indeed, every~$g\in\dom{G}$ is of the form~$g=(s,s)$ with~$s\in\R$, and, setting~$m\eqdef\tfrac12(u_1+u_2)$,
\[
\norm{u-g}^2+G(g)=(u_1-s)^2+(u_2-s)^2+2\abs{s}=\tfrac12(u_1-u_2)^2+2\tparen{(m-s)^2+\abs{s}}\fstop
\]
Since~$\min_{s\in\R} \tparen{(m-s)^2+\abs{s}}$ equals~$m^2$ if~$\abs{m}\leq\tfrac12$ and~$\abs{m}-\tfrac14$ otherwise, we conclude that~$G_{1/2}=F$.
Applying the computation in the proof of the implication~\ref{i:p:EquivConvex:2}$\implies$\ref{i:p:EquivConvex:7} in Proposition~\ref{p:EquivConvex} to the double~$G^\opl{2}$ with~$C$ either~$C_\tleq$ or~$C_\alpha$, and noting that the double of~$G_\lambda$ is the Yosida regularization of the double of~$G$, we conclude that~$F=G_{1/2}$ is Markovian.
Since~$F$ is finite and convex, it is a Dirichlet functional.
Since Markovianity is preserved under sums,~$E=F+\mbfI_K$ is a Dirichlet functional; it is even since~$F$ and~$K$ are, and grounded since~$E\geq 0=E(0)$.

\paragraph{$E$ satisfies~\ref{i:p:WLocality:0d} and~\ref{i:p:WLocality:3d}}
Let~$u\in K$.
If~$u_1u_2\geq 0$, then~$\abs{u}=\pm u$, and~$E(\abs{u})=E(u)$ since~$E$ is even.
If~$u_1u_2<0$, then~$\abs{u_1+u_2}\leq \abs{u_1}+\abs{u_2}=\abs{u_1-u_2}\leq 1$, and~$\abs{u}\in K$ satisfies~$\abs{u_1}+\abs{u_2}\leq 1$; thus,~$E(u)=u_1^2+u_2^2=E(\abs{u})$ by~\eqref{eq:e:WLocality:1}.
This proves~\ref{i:p:WLocality:0d}, and~\ref{i:p:WLocality:3d} follows by Proposition~\ref{p:WLocalityUni}\ref{i:p:WLocalityUni:2}.

\paragraph{$E$ does not satisfy~\ref{i:p:WLocality:1} or~\ref{i:p:WLocality:2}}
Let~$u\eqdef(1,0)$ and~$v\eqdef(0,1)$.
Then,~$uv=0$, the functions~$u$,~$v$, and~$u+v=u\vee v=(1,1)$ belong to~$K$, and~$u\wedge v=0$.
By~\eqref{eq:e:WLocality:1},
\[
E(u+v)=E(u\vee v)+E(u\wedge v)=\tfrac32 \neq 2 = E(u)+E(v)\comma
\]
so that neither~\ref{i:p:WLocality:1} nor~\ref{i:p:WLocality:2} holds, and neither do~\ref{i:p:WLocality:0} and~\ref{i:p:WLocality:3} by Proposition~\ref{p:WLocality}.
Accordingly,~$K$ is not zero-product cone: $(1,0),(0,-1)\in K$, while~$(1,-1)\notin K$.
\end{example}

\subsection{Strong locality}\label{ss:SLocality}
Let us now turn to strong locality.

\begin{definition}[Strong locality]
A functional~$E\colon L^2_\mssm\to (-\infty,\infty]$ with~$0\in\dom{E}$ is called \emph{strongly local} if
for every~$u,v\in\dom{\ground{E}}$ with~$(u-c)v=0$ for some~$c\in\R$ and~$u+v\in\dom{\ground{E}}$, we have
\begin{equation}\label{eq:t:Locality:1}
\ground{E}(u+v)=\ground{E}(u)+\ground{E}(v) \fstop
\end{equation}
\end{definition}

Since the condition~$uv=0$ is the case~$c=0$ of the condition~$(u-c)v=0$, every strongly local functional is weakly local.
In the study of strong locality, the following families of Lipschitz contractions play a special role.

\begin{lemma}\label{l:LocalityFolds}
For~$\alpha\geq 0$ define functions~$\varphi_\alpha,\psi_\alpha,\varphi_\alpha^+,\varphi_\alpha^-\colon \R\to\R$ by
\[
\varphi_\alpha\colon t\mapsto \abs{t+\alpha}-\alpha\comma \qquad \psi_\alpha\colon t\mapsto \abs{t-\alpha}-\alpha\comma \qquad \varphi_\alpha^+\colon t\mapsto (t+\alpha)_+-\alpha\comma \qquad \varphi_\alpha^- \colon t\mapsto (t+\alpha)_- \fstop
\]
Then,~$\varphi_\alpha,\psi_\alpha,\varphi_\alpha^\pm\in\ground{\Phi}$,~$\varphi_0=\psi_0=\abs{\emparg}$, and~$\psi_\alpha(t)=\varphi_\alpha(-t)$ for every~$t\in\R$.
Furthermore,
\begin{enumerate}[$(i)$]
\item\label{i:l:LocalityFolds:1} the functions~$\varphi_\alpha^\pm$ satisfy
\begin{equation}\label{eq:l:LocalityFolds:1}
\varphi_\alpha^+-\varphi_\alpha^- = \id_{\R} \comma \qquad \varphi_\alpha^++\varphi_\alpha^- = \varphi_\alpha\comma \qquad (\varphi_\alpha^+ +\alpha)\,\varphi_\alpha^- = 0 \semicolon
\end{equation}
\item\label{i:l:LocalityFolds:2} for every~$s,r\in\R$ with~$(s+2\alpha)r=0$,
\begin{equation}\label{eq:l:LocalityFolds:2}
\varphi_\alpha(s+r) = \varphi_\alpha(s) + \psi_\alpha(r) \qquad \text{and}\qquad \varphi_\alpha(s)\,\psi_\alpha(r) = 0\fstop
\end{equation}
\end{enumerate}
\end{lemma}

\begin{proof}
All assertions but~\ref{i:l:LocalityFolds:2} are readily verified.
As for~\ref{i:l:LocalityFolds:2}, if~$r=0$, then both identities in~\eqref{eq:l:LocalityFolds:2} hold since~$\psi_\alpha(0)=0$.
If otherwise~$s=-2\alpha$, then~$\varphi_\alpha(s)=0$ and~$\varphi_\alpha(s+r)=\abs{r-\alpha}-\alpha=\psi_\alpha(r)$.
\end{proof}

As for weak locality, we first consider the univariate conditions characterizing strong locality in the bilinear case.

\begin{proposition}[Univariate conditions for strong locality]\label{p:LocalityUni}
Let~$E\colon L^2_\mssm\to (-\infty,\infty]$ be an even Dirichlet functional, and consider the following conditions:
\begin{enumerate}[$(\dot{L}_1)$]\setcounter{enumi}{-1}
\item\label{i:t:Locality:0d} for every~$u\in\dom{E}$,
\begin{equation}\label{eq:t:Locality:0d}
E\tparen{\abs{u+\alpha}-\alpha}=E(u) \comma \qquad \alpha\geq 0 \semicolon
\end{equation}

\setcounter{enumi}{2}
\item\label{i:t:Locality:3d} for every~$u\in\dom{\ground{E}}=\dom{E}$, for every~$\varphi,\psi\in\Phi$ with~$\varphi\psi=0$, we have~$(\ground{\varphi}+\ground{\psi})\circ u\in\dom{\ground{E}}$ and
\begin{equation}\label{eq:t:Locality:3d}
\ground{E}\tparen{(\ground{\varphi}+\ground{\psi})\circ u} = \ground{E}(\ground{\varphi}\circ u) + \ground{E}(\ground{\psi}\circ u) \fstop
\end{equation}
\end{enumerate}
Then,
\begin{enumerate}[$(a)$]
\item\label{i:p:LocalityUni:1} \ref{i:t:Locality:0d} holds if and only if, for every~$u\in\dom{E}$,
\begin{equation}\label{eq:t:Locality:5}
E\tparen{\abs{u-\alpha}-\alpha}=E(u)\comma \qquad \alpha\geq 0 \semicolon
\end{equation}
\item\label{i:p:LocalityUni:2} \ref{i:t:Locality:0d} and~\ref{i:t:Locality:3d} are equivalent.
\end{enumerate}
\end{proposition}

In the notation of Lemma~\ref{l:LocalityFolds}, condition~\ref{i:t:Locality:0d} reads~$E(\varphi_\alpha\circ u)=E(u)$, and~\eqref{eq:t:Locality:5} reads~$E(\psi_\alpha\circ u)=E(u)$, for every~$u\in\dom{E}$ and every~$\alpha\geq 0$.

\begin{proof}[Proof of Proposition~\ref{p:LocalityUni}]
In light of Remark~\ref{r:Grounding}\ref{i:r:Grounding:1}, we may and will assume with no loss of generality that~$E$ is grounded.
Throughout the proof, let~$\varphi_\alpha$,~$\psi_\alpha$, and~$\varphi_\alpha^\pm$ be as in Lemma~\ref{l:LocalityFolds}, and recall that~$\varphi_\alpha\circ u,\psi_\alpha\circ u,\varphi_\alpha^\pm\circ u\in\dom{E}$ for every~$u\in\dom{E}$ by~\ref{i:d:Form:11}.

\paragraph{Proof of~\ref{i:p:LocalityUni:1}}
Fix~$u\in\dom{E}$ and~$\alpha\geq 0$, and recall that~$-u\in\dom{E}$ since~$E$ is even.
Assume~\ref{i:t:Locality:0d}.
Respectively: since~$\psi_\alpha(t)=\varphi_\alpha(-t)$, by~\ref{i:t:Locality:0d} applied to~$-u$, and since~$E$ is even,
\[
E(\psi_\alpha\circ u) = E\tparen{\varphi_\alpha\circ(-u)} = E(-u)= E(u) \comma
\]
which proves~\eqref{eq:t:Locality:5}.
Vice versa, assume~\eqref{eq:t:Locality:5}.
Respectively: since~$\varphi_\alpha(t)=\psi_\alpha(-t)$, by~\eqref{eq:t:Locality:5} applied to~$-u$, and since~$E$ is even,
\[
E\tparen{\varphi_\alpha \circ u} = E\tparen{\psi_\alpha\circ(-u)} = E(-u)=E(u) \comma
\]
which proves~\ref{i:t:Locality:0d}.

\paragraph{\ref{i:t:Locality:3d}$\implies$\ref{i:t:Locality:0d}}
Fix~$u\in\dom{E}$ and~$\alpha\geq 0$, and set~$\varphi\eqdef\varphi_\alpha^+ +\alpha$ and~$\chi_\pm\eqdef \pm\varphi_\alpha^-$.
Then,~$\varphi,\chi_\pm\in\Phi$, $\varphi\chi_\pm=0$ by the last identity in~\eqref{eq:l:LocalityFolds:1}, $\ground{\varphi}=\varphi_\alpha^+$, and~$\ground{(\chi_\pm)}=\chi_\pm$.
Respectively: by the second identity in~\eqref{eq:l:LocalityFolds:1}, by~\eqref{eq:t:Locality:3d} applied to the pair~$\varphi,\chi_+$, since~$E$ is even, by~\eqref{eq:t:Locality:3d} applied to the pair~$\varphi,\chi_-$, and finally by the first identity in~\eqref{eq:l:LocalityFolds:1},
\begin{align}
\nonumber
E(\varphi_\alpha\circ u) &=  E(\varphi_\alpha^+\circ u + \varphi_\alpha^-\circ u) = E(\varphi_\alpha^+\circ u) + E(\varphi_\alpha^-\circ u)
\\
\nonumber
&= E(\varphi_\alpha^+\circ u) + E(-\varphi_\alpha^-\circ u) = E(\varphi_\alpha^+\circ u -\varphi_\alpha^-\circ u)
\\
\label{eq:t:Locality:9}
&= E(u)\comma
\end{align}
which proves~\ref{i:t:Locality:0d}.

\paragraph{\ref{i:t:Locality:0d}$\implies$\ref{i:t:Locality:3d}}
By~\ref{i:p:LocalityUni:1},~\eqref{eq:t:Locality:5} holds.
Furthermore, since~\ref{i:p:WLocality:0d} is~\ref{i:t:Locality:0d} for~$\alpha=0$, \ref{i:p:WLocality:3d} holds by Proposition~\ref{p:WLocalityUni}\ref{i:p:WLocalityUni:2}.
Let~$\varphi,\psi\in\Phi$ with~$\varphi\psi=0$, and fix~$u\in\dom{E}$.
Clearly,~$\ground{\varphi},\ground{\psi}\in\ground{\Phi}$, and~$\ground{\varphi}+\ground{\psi}\in\ground{\Phi}$ by Lemma~\ref{l:LipschitzContractionsSplits}\ref{i:l:LipschitzContractionsSplits:1}; thus,~$(\ground{\varphi}+\ground{\psi})\circ u\in\dom{E}$ by~\ref{i:d:Form:11}.
Since~$\varphi(0)\psi(0)=0$, we have either~$\varphi(0)=0$ or~$\psi(0)=0$.
Thus, without loss of generality, up to exchanging~$\varphi$ and~$\psi$, we may and will assume that~$\psi(0)=0$, so that~$\psi=\ground{\psi}$.
Set~$c\eqdef -\varphi(0)$,~$f\eqdef\ground{\varphi}\circ u$, and~$g\eqdef\psi\circ u$, and note that~$f,g\in\dom{E}$ by~\ref{i:d:Form:11}, that~$f+g=(\ground{\varphi}+\ground{\psi})\circ u$, and that~$(f-c)g=(\varphi\circ u)\cdot(\psi\circ u)=0$.

Suppose first that~$c\leq 0$, and set~$\beta\eqdef -c/2\geq 0$.
For every~$t\in\R$, the numbers~$s\eqdef\ground{\varphi}(t)$ and~$r\eqdef\psi(t)$ satisfy~$(s+2\beta)r=(s-c)r=\varphi(t)\psi(t)=0$.
Thus, by Lemma~\ref{l:LocalityFolds}\ref{i:l:LocalityFolds:2}, the functions~$\vartheta_1\eqdef\varphi_\beta\circ\ground{\varphi}$ and~$\vartheta_2\eqdef\psi_\beta\circ\psi$, which belong to~$\ground{\Phi}$, satisfy
\[
\vartheta_1\vartheta_2=0 \qquad\text{and}\qquad \vartheta_1+\vartheta_2=\varphi_\beta\circ(\ground{\varphi}+\psi)\fstop
\]
Since~$\vartheta_1\circ u=\varphi_\beta\circ f$ and~$\vartheta_2\circ u=\psi_\beta\circ g$, respectively: by~\ref{i:t:Locality:0d}, by the latter identity and~\ref{i:p:WLocality:3d} applied to~$\vartheta_1,\vartheta_2$, and by~\ref{i:t:Locality:0d} and~\eqref{eq:t:Locality:5},
\[
E(f+g)=E\tparen{\varphi_\beta\circ(f+g)}=E(\varphi_\beta\circ f)+E(\psi_\beta\circ g)=E(f)+E(g)\comma
\]
which is~\eqref{eq:t:Locality:3d}.

If otherwise~$c>0$, the functions~$-\varphi,-\psi\in\Phi$ satisfy~$(-\varphi)(-\psi)=0$,~$-\psi(0)=0$, and~$-(-\varphi)(0)=-c<0$.
Applying the previous case to~$-\varphi,-\psi$ in place of~$\varphi,\psi$ (that is, to~$-f,-g$ in place of~$f,g$), and since~$E$ is even, we conclude that~$E(f+g)=E(-f-g)=E(-f)+E(-g)=E(f)+E(g)$.
\end{proof}

Recall that, for every~$\varphi\colon\R\to\R$, we define~$\ground{\varphi}\colon\R\to\R$ as~$\ground{\varphi}\eqdef \varphi-\varphi(0)$.
Further recall the definition~\eqref{eq:HKpm} of the functions~$h^\pm_\alpha$ and~$k^\pm_\alpha$, and let~$\varphi_\alpha$ be as in Lemma~\ref{l:LocalityFolds}.
We are now ready to state the main result of this section, a characterization of strong locality for Dirichlet functionals in terms of bivariate conditions, which is complete for even Dirichlet functionals.

\begin{theorem}[Characterization of strong locality]\label{t:Locality}
Let~$E\colon L^2_\mssm\to (-\infty,\infty]$ be a Dirichlet functional with~$0\in\dom{E}$, and consider the following conditions:
\begin{enumerate}[$({L}_1)$]\setcounter{enumi}{-1}
\item\label{i:t:Locality:0} for every~$u,v\in\dom{E}$ with~$uv=0$,
\begin{equation}\label{eq:t:Locality:0}
\ground{E}\tparen{\abs{\varphi_\alpha\circ u}+\abs{\varphi_\alpha\circ v}}=\ground{E}(u)+\ground{E}(v) \comma \qquad \alpha\geq 0 \semicolon
\end{equation}
\item\label{i:t:Locality:1} $E$ is strongly local;

\item\label{i:t:Locality:2} for every~$u,v\in\dom{E}$, with~$(u-c)v=0$ for some~$c\in\R$,
\begin{equation}\label{eq:t:Locality:2}
E\tparen{h_{\abs{c}}^\pm(u,v)} + E\tparen{k_{\abs{c}}^\pm(u,v)} = E(u)+E(v)\semicolon
\end{equation}

\item\label{i:t:Locality:3} for every~$u,v\in\dom{E}$ with~$uv=0$ and every~$\varphi,\psi\in\Phi$ with~$\varphi\psi=0$,
\begin{equation}\label{eq:t:Locality:3}
\begin{aligned}
\ground{E}\tparen{\abs{(\ground{\varphi}+\ground{\psi})\circ u}+\abs{(\ground{\varphi}+\ground{\psi})\circ v}} = \ground{E}(\ground{\varphi}\circ u) + \ground{E}(\ground{\psi}\circ u)
+ \ground{E}(\ground{\varphi}\circ v) + \ground{E}(\ground{\psi}\circ v) \fstop
\end{aligned}
\end{equation}
\end{enumerate}
Then,~\ref{i:t:Locality:1} and~\ref{i:t:Locality:2} are equivalent.
If~$E$ is additionally \emph{even}, then~\ref{i:t:Locality:0}, \ref{i:t:Locality:1}, \ref{i:t:Locality:2}, and~\ref{i:t:Locality:3} are all equivalent.
\end{theorem}

\begin{remark}\label{r:HKpmLocality}
\begin{enumerate}[$(a)$, wide]
\item\label{i:r:HKpmLocality:1} Let~$E$ be a Dirichlet functional with~$0\in\dom{E}$, and let~$\hat E$ be its reflection (Dfn.~\ref{d:Reflection}), which is a Dirichlet functional by Proposition~\ref{p:Reflection}.
Furthermore,~$0\in\dom{\hat E}$ and~$\ground{\hat E}(u)=\ground{E}(-u)$ for every~$u\in L^2_\mssm$; in particular,~$E$ is strongly local if and only if so is~$\hat E$, as it is readily seen by replacing~$u$,~$v$, and~$c$ with~$-u$,~$-v$, and~$-c$ in the definition of strong locality.
Furthermore, for every~$u,v\in\dom{E}$ and~$c\in\R$ with~$(u-c)v=0$, we have~$-u,-v\in\dom{\hat E}$ and~$\tparen{-u-(-c)}(-v)=(u-c)v=0$, and, for every~$\alpha\geq 0$,
\[
h^+_\alpha(-u,-v)=-h^-_\alpha(u,v) \qquad\text{and}\qquad k^+_\alpha(-u,-v)=-k^-_\alpha(u,v) \fstop
\]
Thus, the `upper' identity (with~$h^+,k^+$) in~\eqref{eq:t:Locality:2} for~$\hat E$ and~$(-u,-v,-c)$ coincides with the `lower' identity (with~$h^-,k^-$) in~\eqref{eq:t:Locality:2} for~$E$ and~$(u,v,c)$.
If~$E$ is even, then~$\hat E=E$, and therefore the upper and the lower identity in~\eqref{eq:t:Locality:2} are mutually equivalent, when required for all~$u,v\in\dom{E}$ and all~$c\in\R$ with~$(u-c)v=0$.
Note that this equivalence cannot be concluded by exchanging~$u$ with~$v$ as in Proposition~\ref{p:HKpm}\ref{i:p:HKpm:a}, since the condition~$(u-c)v=0$ is not symmetric w.r.t.\ exchanging~$u$ and~$v$.
In particular, for even~$E$, in the definition of~\ref{i:t:Locality:2} we may equivalently consider either identity in place of both.

\item\label{i:r:HKpmLocality:2} We note that~\ref{i:p:WLocality:2} is the particular case~$c=0$ in~\ref{i:t:Locality:2}, since~$h^+_0(u,v)=k^-_0(u,v)=u\wedge v$ and~$k^+_0(u,v)=h^-_0(u,v)=u\vee v$.
\end{enumerate}
\end{remark}

\begin{proof}[Proof of Theorem~\ref{t:Locality}]
In light of Remark~\ref{r:Grounding}\ref{i:r:Grounding:1}, we may and will assume with no loss of generality that~$E$ is grounded.
Throughout the proof, let~$\varphi_\alpha$ and~$\psi_\alpha$ be as in Lemma~\ref{l:LocalityFolds}.
We prove the equivalence~\ref{i:t:Locality:1}$\iff$\ref{i:t:Locality:2} with no further assumption on~$E$.
Assuming that~$E$ be additionally even, we prove the equivalences of~\ref{i:t:Locality:0} with each of~\ref{i:t:Locality:3} and~\ref{i:t:Locality:1}.

In applying~\ref{i:t:Locality:1} to a pair~$u,v\in\dom{E}$ with~$u+v\in\dom{E}$ and~$(u-c)v=0$, we will say that we apply~\ref{i:t:Locality:1} to the pair~$u,v$ \emph{with constant~$c$}.

\paragraph{Step 1: \ref{i:t:Locality:0}$\iff$\ref{i:t:Locality:0d}\&\ref{i:p:WLocality:0}}
Since~$\varphi_0=\abs{\emparg}$, condition~\ref{i:p:WLocality:0} is~\ref{i:t:Locality:0} for~$\alpha=0$.
Assume~\ref{i:t:Locality:0}.
Then~\ref{i:p:WLocality:0} holds, and therefore so does~\ref{i:p:WLocality:0d}, choosing~$v=0$ in~\ref{i:p:WLocality:0}.
Fix~$u\in\dom{E}$ and~$\alpha\geq 0$, and note that~$\varphi_\alpha\circ u\in\dom{E}$ by~\ref{i:d:Form:11}.
Respectively: by~\ref{i:p:WLocality:0d} applied to~$\varphi_\alpha\circ u$, and by~\eqref{eq:t:Locality:0} with~$v=0$, since~$\varphi_\alpha(0)=0$,
\[
E(\varphi_\alpha\circ u)=E(\abs{\varphi_\alpha\circ u})=E(u)\comma
\]
which proves~\ref{i:t:Locality:0d}.
Vice versa, assume~\ref{i:t:Locality:0d} and~\ref{i:p:WLocality:0}, and let~$u,v\in\dom{E}$ with~$uv=0$, and~$\alpha\geq 0$.
Then,~$\varphi_\alpha\circ u,\varphi_\alpha\circ v\in\dom{E}$ by~\ref{i:d:Form:11}, and~$(\varphi_\alpha\circ u)\cdot(\varphi_\alpha\circ v)=0$ since~$\varphi_\alpha(0)=0$.
Respectively: by~\ref{i:p:WLocality:0} applied to the pair~$\varphi_\alpha\circ u,\varphi_\alpha\circ v$, and by~\ref{i:t:Locality:0d} applied to~$u$ and to~$v$,
\[
E\tparen{\abs{\varphi_\alpha\circ u}+\abs{\varphi_\alpha\circ v}}=E(\varphi_\alpha\circ u)+E(\varphi_\alpha\circ v)=E(u)+E(v)\comma
\]
which proves~\ref{i:t:Locality:0}.

\paragraph{Step 2: \ref{i:t:Locality:3}$\iff$\ref{i:t:Locality:3d}\&\ref{i:p:WLocality:0}}
Let~$\varphi,\psi\in\Phi$ with~$\varphi\psi=0$, and set~$\theta\eqdef\ground{\varphi}+\ground{\psi}$.
Then,~$\theta\in\ground{\Phi}$ by Lemma~\ref{l:LipschitzContractionsSplits}\ref{i:l:LipschitzContractionsSplits:1}, and thus~$\theta\circ u\in\dom{E}$ for every~$u\in\dom{E}$ by~\ref{i:d:Form:11}.
Assume~\ref{i:t:Locality:3}.
Choosing~$\varphi\eqdef\id_\R$ and~$\psi\eqdef 0$ in~\eqref{eq:t:Locality:3} yields~\ref{i:p:WLocality:0}, and thus~\ref{i:p:WLocality:0d} holds, choosing~$v=0$ in~\ref{i:p:WLocality:0}.
Respectively: by~\ref{i:p:WLocality:0d} applied to~$\theta\circ u$, and by~\eqref{eq:t:Locality:3} with~$v=0$, since~$\theta(0)=\ground{\varphi}(0)=\ground{\psi}(0)=0$,
\[
E(\theta\circ u)=E(\abs{\theta\circ u})=E(\ground{\varphi}\circ u)+E(\ground{\psi}\circ u)\comma \qquad u\in\dom{E}\comma
\]
which proves~\ref{i:t:Locality:3d}.
Vice versa, assume~\ref{i:t:Locality:3d} and~\ref{i:p:WLocality:0}, and let~$u,v\in\dom{E}$ with~$uv=0$.
Since~$\theta(0)=0$, we have~$(\theta\circ u)\cdot(\theta\circ v)=0$.
Respectively: by~\ref{i:p:WLocality:0} applied to the pair~$\theta\circ u,\theta\circ v$, and by~\ref{i:t:Locality:3d} applied to~$u$ and to~$v$,
\begin{align*}
E\tparen{\abs{\theta\circ u}+\abs{\theta\circ v}}&=E(\theta\circ u)+E(\theta\circ v)
\\
&=E(\ground{\varphi}\circ u)+E(\ground{\psi}\circ u)+E(\ground{\varphi}\circ v)+E(\ground{\psi}\circ v)\comma
\end{align*}
which proves~\ref{i:t:Locality:3}.

\paragraph{\ref{i:t:Locality:0}$\iff$\ref{i:t:Locality:3}}
By Steps~1 and~2, it suffices to note that~\ref{i:t:Locality:0d} and~\ref{i:t:Locality:3d} are equivalent by Proposition~\ref{p:LocalityUni}\ref{i:p:LocalityUni:2}.

\paragraph{\ref{i:t:Locality:0}$\implies$\ref{i:t:Locality:1}}
By Step~1, both~\ref{i:t:Locality:0d} and~\ref{i:p:WLocality:0} hold.
In particular,~\ref{i:p:WLocality:1} holds by Proposition~\ref{p:WLocality}, and~\eqref{eq:t:Locality:5} holds by Proposition~\ref{p:LocalityUni}\ref{i:p:LocalityUni:1}.
Let~$u,v\in\dom{E}$ with~$u+v\in\dom{E}$ and~$(u-c)v=0$ for some~$c\in\R$.

\paragraph{Claim: Without loss of generality we may assume~$c\leq 0$}
Indeed suppose we have shown~\eqref{eq:t:Locality:1} whenever~$u,v\in\dom{E}$ satisfy~$u+v\in\dom{E}$ and~$(u-c)v=0$ for some~$c\leq 0$.
Let~$u',v'\in\dom{E}$ with~$u'+v'\in\dom{E}$ and~$(u'-c')v'=0$ for some~$c'>0$.
Since~$E$ is even,~$u\eqdef -u'$ and~$v\eqdef -v'$ satisfy~$u,v\in\dom{E}$,~$u+v=-(u'+v')\in\dom{E}$ and, setting~$c\eqdef -c'<0$, we have
\[
(u-c)v=-(u'-c')(-v')=(u'-c')v'=0 \fstop
\]
Thus, applying~\eqref{eq:t:Locality:1} and since~$E$ is even,
\begin{align*}
E(u'+v') &= E\tparen{-(u+v)}=E(u+v)=E(u)+E(v)=E(-u)+E(-v)
\\
&=E(u')+E(v')\comma
\end{align*}
which is the desired conclusion also in the case~$c'>0$, and thus proves the claim.

\medskip

In light of the claim, in the following we will assume that~$u,v\in\dom{E}$ satisfy~$u+v\in\dom{E}$ and~$(u-c)v=0$ for some~$c\leq 0$.
Fix~$\alpha\eqdef -c/2\geq 0$, so that~$(u+2\alpha)v=0$.
Applying Lemma~\ref{l:LocalityFolds}\ref{i:l:LocalityFolds:2} pointwise, we have
\begin{subequations}\label{eq:t:Locality:6}
\begin{gather}
\label{eq:t:Locality:6A}
\varphi_\alpha \circ (u+v) = \varphi_\alpha\circ u + \psi_\alpha\circ v\comma
\\
\label{eq:t:Locality:6B}
(\varphi_\alpha\circ u) \cdot (\psi_\alpha\circ v) = 0\fstop
\end{gather}
\end{subequations}
Furthermore,~$\varphi_\alpha,\psi_\alpha\in\ground{\Phi}$ and thus~$\varphi_\alpha\circ u,\psi_\alpha\circ v$, and~$\varphi_\alpha\circ(u+v)$ belong to~$\dom{E}$ by~\ref{i:d:Form:11}.

Respectively: by~\ref{i:t:Locality:0d}, by~\eqref{eq:t:Locality:6A}, by~\eqref{eq:t:Locality:6B} and~\ref{i:p:WLocality:1}, by~\ref{i:t:Locality:0d} and~\eqref{eq:t:Locality:5},
\begin{align*}
E(u+v) &= E\tparen{\varphi_\alpha\circ (u+v)} = E(\varphi_\alpha\circ u + \psi_\alpha\circ v) = E(\varphi_\alpha\circ u) + E(\psi_\alpha\circ v)
\\
&=E(u)+E(v) \comma
\end{align*}
which concludes~\ref{i:t:Locality:1}.

\paragraph{\ref{i:t:Locality:1}$\implies$\ref{i:t:Locality:0}}
By Step~1, it suffices to show both~\ref{i:p:WLocality:0} and~\ref{i:t:Locality:0d}.
Since the condition~$uv=0$ is the case~$c=0$ of the condition~$(u-c)v=0$, \ref{i:t:Locality:1}~implies~\ref{i:p:WLocality:1}, and thus~\ref{i:p:WLocality:0} holds by Proposition~\ref{p:WLocality}, since~$E$ is even.
As for~\ref{i:t:Locality:0d}, by Proposition~\ref{p:LocalityUni}\ref{i:p:LocalityUni:2} it suffices to show~\ref{i:t:Locality:3d}.
To this end, let~$\varphi,\psi\in \Phi$ with~$\varphi\psi=0$, and fix~$u\in\dom{E}$.
Clearly,~$\ground{\varphi},\ground{\psi}\in\ground{\Phi}$, and~$\ground{\varphi}+\ground{\psi}\in\ground{\Phi}$ by Lemma~\ref{l:LipschitzContractionsSplits}\ref{i:l:LipschitzContractionsSplits:1}.
Thus,~$\ground{\varphi}\circ u$, $\ground{\psi}\circ u$, and~$(\ground{\varphi}+\ground{\psi})\circ u$ belong to~$\dom{E}$ by~\ref{i:d:Form:11}.
Since~$\varphi(0)\psi(0)=0$, we have either~$\varphi(0)=0$ or~$\psi(0)=0$.
Thus, without loss of generality, up to exchanging~$\varphi$ and~$\psi$, we may and will assume that~$\psi(0)=0$, so that~$\psi=\ground{\psi}$.
Then, setting~$c\eqdef -\varphi(0)$,
\[
\tparen{\ground{\varphi}\circ u-c}\cdot (\ground{\psi}\circ u) = (\varphi\circ u)\cdot(\psi\circ u)= (\varphi \psi)\circ u=0 \comma
\]
and therefore~\eqref{eq:t:Locality:3d} follows applying~\ref{i:t:Locality:1} to the pair~$\ground{\varphi}\circ u,\ground{\psi}\circ u$ with constant~$c$.

\paragraph{\ref{i:t:Locality:1}$\implies$\ref{i:t:Locality:2}}
We first show~\eqref{eq:t:Locality:2} with~$h^+$ and~$k^+$.
Let~$u,v\in\dom{E}$ with~$(u-c)v=0$ for some~$c\in\R$, and set~$\alpha\eqdef\abs{c}$.
Note that~$h^+_\alpha(u,v),k^+_\alpha(u,v)\in\dom{E}$, by Proposition~\ref{p:HKpm}\ref{i:p:HKpm:b} if~$\alpha>0$, and by the lattice contraction property~\ref{i:d:Form:7'} if~$\alpha=0$, since~$h^+_0(u,v)=u\wedge v$ and~$k^+_0(u,v)=u\vee v$.
Define functions~$\theta_1,\theta_2,\chi_1,\chi_2\colon\R\to\R$ by
\begin{align*}
\theta_1(s)&\eqdef h^+_\alpha(s,0)=s\wedge\alpha \comma & \chi_1(r)&\eqdef h^+_\alpha(c,r)-h^+_\alpha(c,0)=c\wedge(r+\alpha)-c\wedge\alpha\comma
\\
\theta_2(s)&\eqdef k^+_\alpha(s,0)=(s-\alpha)_+ \comma & \chi_2(r)&\eqdef k^+_\alpha(c,r)-k^+_\alpha(c,0)=r\vee(c-\alpha)-0\vee(c-\alpha)\comma
\end{align*}
and note that~$\theta_1,\theta_2,\chi_1,\chi_2\in\ground{\Phi}^\tleq$, since~$\alpha\geq 0$ and~$h^+_\alpha$ and~$k^+_\alpha$ are $1$-Lipschitz and non-decreasing in each variable.
For every~$s,r\in\R$ with~$(s-c)r=0$ we have
\begin{equation}\label{eq:t:Locality:12}
\begin{aligned}
h^+_\alpha(s,r)&=\theta_1(s)+\chi_1(r)\comma & \tparen{\theta_1(s)-\theta_1(c)}\,\chi_1(r)&=0\comma
\\
k^+_\alpha(s,r)&=\theta_2(s)+\chi_2(r)\comma & \tparen{\theta_2(s)-\theta_2(c)}\,\chi_2(r)&=0\comma
\end{aligned}
\end{equation}
as it is readily verified by distinguishing the cases~$r=0$, where~$\chi_1(r)=\chi_2(r)=0$, and~$s=c$, where~$\theta_1(s)-\theta_1(c)=\theta_2(s)-\theta_2(c)=0$.
Furthermore, for every~$s,r\in\R$,
\begin{equation}\label{eq:t:Locality:13}
\begin{aligned}
\theta_1(s)+\theta_2(s)&=s\comma & \tparen{\theta_1(s)-\alpha}\,\theta_2(s)&=0\comma
\\
\chi_1(r)+\chi_2(r)&=r\comma & \tparen{\chi_2(r)-(c-\alpha)}\,\chi_1(r)&=0\fstop
\end{aligned}
\end{equation}
The first line in~\eqref{eq:t:Locality:13} is readily verified.
As for the second line, if~$c\geq 0$, then~$\alpha=c$, and thus~$\chi_1(r)=-r_-$ and~$\chi_2(r)=r_+$; if otherwise~$c<0$, then~$\alpha=-c$, and thus~$\chi_1(r)=-(r+2\alpha)_-$ and~$\chi_2(r)=r\vee(-2\alpha)$; in either case, the second line in~\eqref{eq:t:Locality:13} is readily verified.

Now, note that~$\theta_1\circ u, \theta_2\circ u, \chi_1\circ v, \chi_2\circ v\in\dom{E}$ by~\eqref{eq:s:Locality:Monotone}, since~$\theta_1,\theta_2,\chi_1,\chi_2$ are non-decreasing.
Since~$(u-c)v=0$, by~\eqref{eq:t:Locality:12} we may apply~\ref{i:t:Locality:1} to the pairs~$\theta_1\circ u,\chi_1\circ v$ with constant~$\theta_1(c)$, and~$\theta_2\circ u,\chi_2\circ v$ with constant~$\theta_2(c)$, whose sums are respectively~$h^+_\alpha(u,v)$ and~$k^+_\alpha(u,v)$.
By~\eqref{eq:t:Locality:13}, we may further apply~\ref{i:t:Locality:1} to the pairs~$\theta_1\circ u,\theta_2\circ u$ with constant~$\alpha$, and~$\chi_2\circ v,\chi_1\circ v$ with constant~$c-\alpha$, whose sums are respectively~$u$ and~$v$.
As a consequence,
\begin{align*}
E\tparen{h^+_\alpha(u,v)}+E\tparen{k^+_\alpha(u,v)} &= E(\theta_1\circ u)+E(\chi_1\circ v)+E(\theta_2\circ u)+E(\chi_2\circ v)
\\
&= E(u)+E(v)\comma
\end{align*}
which is~\eqref{eq:t:Locality:2} with~$h^+$ and~$k^+$.
Finally, let~$\hat E$ be the reflection of~$E$ (Dfn.~\ref{d:Reflection}).
Since~$\hat E$ is a grounded strongly local Dirichlet functional with~$0\in\dom{\hat E}$ by Proposition~\ref{p:Reflection} and Remark~\ref{r:HKpmLocality}\ref{i:r:HKpmLocality:1}, the above argument applies to~$\hat E$ in place of~$E$, and shows that~$\hat E$ satisfies~\eqref{eq:t:Locality:2} with~$h^+$ and~$k^+$.
By Remark~\ref{r:HKpmLocality}\ref{i:r:HKpmLocality:1}, this is~\eqref{eq:t:Locality:2} with~$h^-$ and~$k^-$ for~$E$.

\paragraph{\ref{i:t:Locality:2}$\implies$\ref{i:t:Locality:1}}
Let~$u,v\in\dom{E}$ with~$u+v\in\dom{E}$ and~$(u-c)v=0$ for some~$c\in\R$, and set~$w\eqdef u+v$ and~$\alpha\eqdef\abs{c}$.
If~$c\geq 0$, then~$\alpha=c$, and
\begin{equation}\label{eq:t:Locality:14}
h^+_\alpha(u,v)=h^+_\alpha(w,0) \qquad\text{and}\qquad k^+_\alpha(u,v)=k^+_\alpha(w,0)\comma
\end{equation}
as it is readily verified by distinguishing the cases~$v=0$, where~$w=u$, and~$u=c$, where both sides of the first, resp.\ second, identity in~\eqref{eq:t:Locality:14} equal~$c+(v\wedge 0)$, resp.~$v\vee 0$.
If otherwise~$c<0$, then~$\alpha=-c$, and~\eqref{eq:t:Locality:14} holds with~$h^-,k^-$ in place of~$h^+,k^+$; indeed, if~$u=c$, then both sides of the first, resp.\ second, identity equal~$c+(v\vee 0)$, resp.~$v\wedge 0$.
Now, note that~$w,0\in\dom{E}$ and~$(w-c)\cdot 0=0$.
Respectively: since~$E$ is grounded; by~\eqref{eq:t:Locality:2} applied to the pair~$w,0$; by~\eqref{eq:t:Locality:14}; and by~\eqref{eq:t:Locality:2} applied to the pair~$u,v$, where we use the identity with~$h^+,k^+$ if~$c\geq 0$, resp.\ with~$h^-,k^-$ if~$c<0$,
\[
E(u+v)=E(w)+E(0)=E\tparen{h^\pm_\alpha(w,0)}+E\tparen{k^\pm_\alpha(w,0)} = E\tparen{h^\pm_\alpha(u,v)}+E\tparen{k^\pm_\alpha(u,v)}=E(u)+E(v)\comma
\]
which is~\eqref{eq:t:Locality:1}.
\end{proof}

Combining Theorem~\ref{t:Locality} with Propositions~\ref{p:LocalityUni} and~\ref{p:WLocality}, we obtain the following complete picture.

\begin{corollary}[Strong locality and univariate conditions]\label{c:SLocality}
Let~$E\colon L^2_\mssm\to (-\infty,\infty]$ be an even Dirichlet functional.
Then,
\begin{enumerate}[$(i)$]
\item\label{i:c:SLocality:1} each of the mutually equivalent conditions~\ref{i:t:Locality:0}, \ref{i:t:Locality:1}, \ref{i:t:Locality:2}, and~\ref{i:t:Locality:3} implies both~\ref{i:t:Locality:0d} and~\ref{i:t:Locality:3d};
\item\label{i:c:SLocality:2} if~$\dom{E}$ is zero-product cone (e.g., if~$E$ is linearly defined), then the conditions~\ref{i:t:Locality:0d}, \ref{i:t:Locality:0}, \ref{i:t:Locality:1}, \ref{i:t:Locality:2}, \ref{i:t:Locality:3d}, and~\ref{i:t:Locality:3} are all equivalent:
\[
\begin{tikzcd}[column sep=2.2em, row sep=1.8em]
\text{\ref{i:t:Locality:0}} \arrow[r, Leftrightarrow] \arrow[d, Leftrightarrow] & \text{\ref{i:t:Locality:1}} \arrow[r, Leftrightarrow] & \text{\ref{i:t:Locality:2}} \arrow[r, Leftrightarrow] & \text{\ref{i:t:Locality:3}} \arrow[d, Leftrightarrow]
\\
\text{\ref{i:t:Locality:0d}} \arrow[rrr, Leftrightarrow] & & & \text{\ref{i:t:Locality:3d}}
\end{tikzcd}
\]
\end{enumerate}
\end{corollary}

\begin{proof}
\ref{i:c:SLocality:1}
By Theorem~\ref{t:Locality}, it suffices to consider~\ref{i:t:Locality:0}, which implies~\ref{i:t:Locality:0d} by Step~1 in the proof of Theorem~\ref{t:Locality}.
Then,~\ref{i:t:Locality:3d} holds by Proposition~\ref{p:LocalityUni}\ref{i:p:LocalityUni:2}.

\ref{i:c:SLocality:2}
The equivalences in the first row are Theorem~\ref{t:Locality}, the one in the second row is Proposition~\ref{p:LocalityUni}\ref{i:p:LocalityUni:2}, and the downward implications are~\ref{i:c:SLocality:1}.
It thus suffices to show that~\ref{i:t:Locality:0d} implies~\ref{i:t:Locality:0}.
Assume~\ref{i:t:Locality:0d}.
Then,~\ref{i:p:WLocality:0d} holds, being~\ref{i:t:Locality:0d} for~$\alpha=0$, and thus~\ref{i:p:WLocality:0} holds by Corollary~\ref{c:WLocality}\ref{i:c:WLocality:2}, since~$\dom{E}$ is zero-product cone.
Therefore,~\ref{i:t:Locality:0} holds by Step~1 in the proof of Theorem~\ref{t:Locality}.
\end{proof}

Under the zero-product-cone property, Corollary~\ref{c:SLocality}\ref{i:c:SLocality:2} recovers the well-known characterization of strong locality for bilinear Dirichlet forms, see~\cite[\S{I.5}]{BouHir91}.

\begin{remark}\label{r:SLocality}
\begin{enumerate*}[$(a)$]
\item As for weak locality, the bivariate conditions~\ref{i:t:Locality:0} and~\ref{i:t:Locality:3} are obtained from the univariate conditions~\ref{i:t:Locality:0d} and~\ref{i:t:Locality:3d} as in Remark~\ref{r:Bivariate}, with~$\Theta(u)=\varphi_\alpha\circ u$ and~$\Lambda(u)=\ground{E}(u)$ for~\ref{i:t:Locality:0d}, and with~$\Theta(u)=(\ground{\varphi}+\ground{\psi})\circ u$ and~$\Lambda(u)=\ground{E}(\ground{\varphi}\circ u)+\ground{E}(\ground{\psi}\circ u)$ for~\ref{i:t:Locality:3d}.

\item Combining Step~1 in the proof of Theorem~\ref{t:Locality} with Proposition~\ref{p:WLocality}, an even Dirichlet functional is strongly local if and only if it is weakly local and satisfies~\ref{i:t:Locality:0d}.

\item We do not know whether the zero-product-cone property may be dispensed with in Corollary~\ref{c:SLocality}\ref{i:c:SLocality:2}, i.e.\ whether~\ref{i:t:Locality:0d} implies weak locality for every even Dirichlet functional.
Note that the functional~$E$ in Example~\ref{e:WLocality} does not satisfy~\ref{i:t:Locality:0d}: for~$u\eqdef(-1,0)$ and~$\alpha\eqdef\tfrac12$ we have~$\varphi_\alpha\circ u=0$, while~$E(u)=1$.

\item The proof of the implication~\ref{i:t:Locality:2}$\implies$\ref{i:t:Locality:1} in Theorem~\ref{t:Locality} only makes use of the upper identity in~\eqref{eq:t:Locality:2} for~$c\geq 0$, and of the lower identity in~\eqref{eq:t:Locality:2} for~$c<0$.
We do not know whether, for Dirichlet functionals which are not even, either identity alone in~\eqref{eq:t:Locality:2} implies strong locality; cf.\ Remark~\ref{r:HKpmLocality}\ref{i:r:HKpmLocality:1}.
\end{enumerate*}
\end{remark}

The conditions involving non-monotone Lipschitz contractions are closely related to evenness, as we now show.

\begin{remark}[Conditions forcing evenness]\label{r:ForcedEvenness}
Let~$E\colon L^2_\mssm\to (-\infty,\infty]$ be a Dirichlet functional with~$0\in\dom{E}$.
If~$E$ satisfies either~\ref{i:p:WLocality:3} or~\ref{i:t:Locality:3}, then~$E$ is even.
Indeed, choosing~$v=0$,~$\psi=0$, and either~$\varphi=\id_\R$ or~$\varphi=-\id_\R$ in~\eqref{eq:p:WLocality:3}, resp.~\eqref{eq:t:Locality:3}, shows that~$\ground{E}(-u)=\ground{E}(\abs{u})=\ground{E}(u)<\infty$ for every~$u\in\dom{E}$.
In particular,~$-\dom{E}\subset\dom{E}$, hence~$-\dom{E}=\dom{E}$, and therefore~$E(-u)=E(u)$ for every~$u\in L^2_\mssm$.
As a consequence, by Proposition~\ref{p:WLocality} and Theorem~\ref{t:Locality},~$E$ satisfies~\ref{i:p:WLocality:3}, resp.~\ref{i:t:Locality:3}, if and only if it is even and weakly, resp.\ strongly, local.

Analogously, if~$E$ satisfies~\eqref{eq:t:Locality:5} for every~$u\in\dom{E}$, then~$E$ is even.
Indeed, let~$\psi_\alpha$ be as in Lemma~\ref{l:LocalityFolds}, and fix~$u\in\dom{E}$.
Since~$\psi_\alpha(t)=-t$ for every~$t\leq\alpha$ and~$\abs{\psi_\alpha(t)}\leq\abs{t}$ for every~$t\in\R$, we have~$L^2_\mssm$-$\lim_{\alpha\to\infty}\psi_\alpha\circ u=-u$ by Dominated Convergence with dominating function~$2\abs{u}$.
By lower semicontinuity of~$E$ and by~\eqref{eq:t:Locality:5}, we conclude that~$E(-u)\leq\liminf_{\alpha\to\infty}E(\psi_\alpha\circ u)=E(u)<\infty$.
Thus,~$-u\in\dom{E}$, and applying the same argument to~$-u$ in place of~$u$ yields the converse inequality~$E(u)\leq E(-u)$.
Hence,~$-\dom{E}=\dom{E}$ and~$E(-u)=E(u)$ for every~$u\in L^2_\mssm$, that is,~$E$ is even.
\end{remark}

By contrast, the remaining conditions involving the absolute value, namely~\ref{i:p:WLocality:0},~\ref{i:p:WLocality:0d},~\ref{i:t:Locality:0}, and~\ref{i:t:Locality:0d}, do not imply evenness, and the corresponding characterizations of locality fail for Dirichlet functionals which are not even, as shown by the next example.

\begin{example}[On the necessity of evenness]\label{e:Evenness}
Firstly, let~$X\eqdef\set{x_0}$ and~$\mssm\eqdef\delta_{x_0}$, so that~$L^2_\mssm\cong\R$.
Then, \emph{every} proper convex lower semicontinuous~$E\colon\R\to(-\infty,\infty]$ is a Dirichlet functional.
Indeed, let~$w=(s,t)\in\R^2$, and let~$w'=(s',t')$ be either~$\cproj_\tleq(w)$ or~$\cproj_\alpha(w)$ for some~$\alpha\geq 0$.
By~\eqref{eq:StripProjections}, we have~$s'+t'=s+t$, and~$s'$ lies between~$s$ and~$t$.
Hence,~$w'=\lambda(s,t)+(1-\lambda)(t,s)$ for some~$\lambda\in[0,1]$, and~$E^\opl{2}(w')\leq E^\opl{2}(w)$ by convexity of~$E^\opl{2}$ and since~$E^\opl{2}(t,s)=E^\opl{2}(s,t)$.

\paragraph{$(a)$ A strongly local functional satisfying none of the conditions involving~$\abs{\emparg}$}
Let~$X\eqdef\set{x_0}$ as above, and~$E(s)\eqdef s$ for~$s\in\R$.
Then,~$E$ is an everywhere finite grounded Dirichlet functional, and it is strongly local, since it is additive.
Thus,~$E$ satisfies~\ref{i:p:WLocality:1},~\ref{i:p:WLocality:2},~\ref{i:t:Locality:1}, and~\ref{i:t:Locality:2} by Proposition~\ref{p:WLocality} and Theorem~\ref{t:Locality}; again since~$E$ is additive, it satisfies~\ref{i:p:WLocality:3d} and~\ref{i:t:Locality:3d} as well.
However,~$E(\abs{-1})=1\neq -1=E(-1)$, so that~$E$ satisfies neither~\ref{i:p:WLocality:0d} nor~\ref{i:t:Locality:0d}, and thus, choosing~$v=0$ and~$\alpha=0$, it satisfies neither~\ref{i:p:WLocality:0} nor~\ref{i:t:Locality:0}.
In particular, without evenness, the implications~\ref{i:p:WLocality:1}$\implies$\ref{i:p:WLocality:0} in Proposition~\ref{p:WLocality} and~\ref{i:t:Locality:1}$\implies$\ref{i:t:Locality:0} in Theorem~\ref{t:Locality} fail, and so do the implications~\ref{i:p:WLocality:3d}$\implies$\ref{i:p:WLocality:0d} in Proposition~\ref{p:WLocalityUni}\ref{i:p:WLocalityUni:2} and~\ref{i:t:Locality:3d}$\implies$\ref{i:t:Locality:0d} in Proposition~\ref{p:LocalityUni}\ref{i:p:LocalityUni:2}, and Corollaries~\ref{c:WLocality} and~\ref{c:SLocality}, even for everywhere finite functionals.

\paragraph{$(b)$ A functional on a zero-product cone satisfying the univariate conditions involving~$\abs{\emparg}$, yet not weakly local}
Let~$X\eqdef\set{1,2}$ and~$\mssm\eqdef\delta_1+\delta_2$, so that~$L^2_\mssm\cong\R^2$, write~$u=(u_1,u_2)$ for~$u\in\R^2$, and set
\[
K\eqdef[0,\infty)^\tym{2}\comma \qquad E(u)\eqdef (u_1-u_2)^2+\mbfI_K(u)\fstop
\]
The quadratic functional~$u\mapsto (u_1-u_2)^2$ is a Dirichlet functional by Corollary~\ref{c:NormalContractionQuadratic}, since~$\abs{u_1-u_2}\leq\abs{v_1-v_2}$ whenever~$u\preccurlyeq v$.
The functional~$\mbfI_K$ is a Dirichlet functional as well, since~$\cproj_\tleq$ and~$\cproj_\alpha$ act by pointwise averaging by~\eqref{eq:StripProjections}, and thus preserve~$K\oplus K$.
Since Markovianity is preserved under sums,~$E$ is a non-negative grounded Dirichlet functional, and~$\dom{E}=K$ is zero-product cone.
Since~$\abs{u}=u$ and~$\varphi_\alpha\circ u=u$ for every~$u\in K$ and every~$\alpha\geq 0$, the functional~$E$ satisfies~\ref{i:p:WLocality:0d} and~\ref{i:t:Locality:0d}.
However,~$E$ is not weakly local, since~$u\eqdef(1,0)$ and~$v\eqdef(0,1)$ satisfy~$uv=0$ and~$E(u+v)=0\neq 2=E(u)+E(v)$.
Furthermore, set~$\varphi(t)\eqdef(1-\abs{t-1})_+$, $\psi(t)\eqdef 1\wedge(t-2)_+$, and~$f\eqdef(1,3)\in K$.
Then,~$\varphi,\psi\in\ground{\Phi}$,~$\varphi\psi=0$,~$\varphi\circ f=u$,~$\psi\circ f=v$, and~$(\varphi+\psi)\circ f=u+v$, so that~$E$ satisfies neither~\ref{i:p:WLocality:3d} nor~\ref{i:t:Locality:3d}.
In particular, without evenness, the implications~\ref{i:p:WLocality:0d}$\implies$\ref{i:p:WLocality:3d} in Proposition~\ref{p:WLocalityUni}\ref{i:p:WLocalityUni:2} and~\ref{i:t:Locality:0d}$\implies$\ref{i:t:Locality:3d} in Proposition~\ref{p:LocalityUni}\ref{i:p:LocalityUni:2} fail, and so does Corollary~\ref{c:WLocality}\ref{i:c:WLocality:2}, even on zero-product cones.

\paragraph{$(c)$ A functional satisfying~\ref{i:t:Locality:0}, yet not strongly local}
Let~$X\eqdef\set{x_0}$ as above, and~$E(s)\eqdef s^2+\mbfI_{[0,\infty)}(s)$ for~$s\in\R$, a grounded Dirichlet functional with zero-product-cone domain~$[0,\infty)$.
If~$u,v\in\dom{E}$ satisfy~$uv=0$, then one of them vanishes; furthermore,~$\varphi_\alpha(t)=t$ for every~$t\geq 0$ and every~$\alpha\geq 0$.
Thus, both sides of~\eqref{eq:t:Locality:0} equal~$E(u+v)$, and~$E$ satisfies~\ref{i:t:Locality:0} and, trivially,~\ref{i:t:Locality:0d}.
However,~$E$ is not strongly local, since~$u=v\eqdef 1$ satisfy~$(u-1)v=0$ and~$E(u+v)=4\neq 2=E(u)+E(v)$.
Furthermore,~\eqref{eq:t:Locality:5} fails for~$u=\alpha\eqdef 1$, since~$\abs{u-\alpha}-\alpha=-1\notin\dom{E}$.
In particular, without evenness, the implication~\ref{i:t:Locality:0}$\implies$\ref{i:t:Locality:1} in Theorem~\ref{t:Locality} fails, and so do Proposition~\ref{p:LocalityUni}\ref{i:p:LocalityUni:1} and Corollary~\ref{c:SLocality}\ref{i:c:SLocality:2}.
\end{example}

In the case of finite measure spaces, strong locality may be further characterized in terms of the behavior of~$E$ on constant functions. (Cf.~\cite[Cor.~I.5.1.4, p.~31]{BouHir91} for the case of (bilinear) Dirichlet forms.)
To this end, we first note the following immediate consequence of Lemma~\ref{l:ZeroLines} with~$g\equiv 1$.

\begin{corollary}\label{c:ZeroLines}
Let~$E\colon L^2_\mssm\to (-\infty,\infty]$ be a proper convex lower semicontinuous functional.
Further assume that~$\mssm$ is a finite measure and that~$\alpha\in\dom{E}$ and~$E(\alpha)=0$ for every constant~$\alpha\in\R$.
Then,
\begin{equation}\label{eq:c:ZeroLines:0}
u+\alpha\in\dom{E} \qquad \text{and}\qquad E(u+\alpha)= E(u)\comma \qquad u\in\dom{E}\comma \alpha\in\R\fstop
\end{equation}
\end{corollary}

\begin{corollary}\label{c:Locality}
Let~$E\colon L^2_\mssm\to (-\infty,\infty]$ be a Dirichlet functional.
Further assume that~$\mssm$ is a finite measure and that~$\alpha\in\dom{\ground{E}}$ for every constant~$\alpha\in\R$.
Then,~$E$ is strongly local if and only if it is weakly local and the map~$\alpha\mapsto\ground{E}(\alpha)$ is additive on~$\R$.

If~$E$ is additionally even, consider the following bivariate conditions:
\begin{enumerate}[$({L}_1')$]\setcounter{enumi}{-1}
\item\label{i:c:Locality:0} $\ground{E}(\alpha)=0$ for every~$\alpha\in\R$ and~\ref{i:p:WLocality:0} holds;
\item\label{i:c:Locality:1} $\ground{E}(\alpha)=0$ for every~$\alpha\in\R$ and~\ref{i:p:WLocality:1} holds;
\item\label{i:c:Locality:2} $\ground{E}(\alpha)=0$ for every~$\alpha\in\R$ and~\ref{i:p:WLocality:2} holds;
\item\label{i:c:Locality:3} $\ground{E}(\alpha)=0$ for every~$\alpha\in\R$ and~\ref{i:p:WLocality:3} holds;
\end{enumerate}
and the following univariate conditions:
\begin{enumerate}[$(\dot{L}_1')$]\setcounter{enumi}{-1}
\item\label{i:c:Locality:0d} $\ground{E}(\alpha)=0$ for every~$\alpha\in\R$ and~\ref{i:p:WLocality:0d} holds;
\setcounter{enumi}{2}
\item\label{i:c:Locality:3d} $\ground{E}(\alpha)=0$ for every~$\alpha\in\R$ and~\ref{i:p:WLocality:3d} holds.
\end{enumerate}
Then,
\begin{enumerate}[$(i)$]
\item\label{i:c:Locality:i} $E$~is strongly local if and only if either of~\ref{i:c:Locality:0}, \ref{i:c:Locality:1}, \ref{i:c:Locality:2}, or~\ref{i:c:Locality:3} holds;
\item\label{i:c:Locality:ii} $E$ satisfies~\ref{i:t:Locality:0d} (equivalently:~\ref{i:t:Locality:3d}) if and only if either of~\ref{i:c:Locality:0d} or~\ref{i:c:Locality:3d} holds;
\item\label{i:c:Locality:iii} if~$\dom{E}$ is a zero-product cone, then each of the above conditions is equivalent to the strong locality.
\end{enumerate}
\end{corollary}

\begin{proof}[Proof of Corollary~\ref{c:Locality}]
In light of Remark~\ref{r:Grounding}\ref{i:r:Grounding:1}, we may and will assume with no loss of generality that~$E$ is grounded.
Let~$\varphi_\alpha$ be as in Lemma~\ref{l:LocalityFolds}.

\paragraph{Proof of the first assertion}
Assume first that~$E$ is strongly local.
Since every strongly local functional is weakly local, it suffices to show that~$\alpha\mapsto E(\alpha)$ is additive.
For every~$\alpha,\beta\in\R$, the constant functions~$u\equiv\alpha$ and~$v\equiv\beta$ satisfy~$u,v,u+v\in\dom{E}$ and~$(u-\alpha)v=0$, and therefore~$E(\alpha+\beta)=E(\alpha)+E(\beta)$ by~\eqref{eq:t:Locality:1} with~$c=\alpha$.

Vice versa, assume that~$E$ is weakly local and that~$\alpha\mapsto E(\alpha)$ is additive.
Then,~$E(-\alpha)=-E(\alpha)$ for every~$\alpha\in\R$, hence the map~$\alpha\mapsto E(\alpha)=-E(-\alpha)$ is both convex and concave, i.e.\ affine, and therefore linear, since~$E(0)=0$.
That is,~$E(\alpha)=\kappa\alpha$ for every~$\alpha\in\R$, where~$\kappa\eqdef E(\car)$.
Let~$\ell\colon L^2_\mssm\to\R$ be the continuous linear functional~$\ell(u)\eqdef\kappa\,(\mssm X)^{-1}\int_X u\diff\mssm$, and set~$F\eqdef E-\ell$.
By Remark~\ref{r:AffinePerturbations},~$F$ is a Dirichlet functional.
Furthermore,~$\dom{F}=\dom{E}$, the functional~$F$ is grounded and weakly local, since so is~$E$ and since~$\ell$ is linear, and~$F(\alpha)=E(\alpha)-\kappa\alpha=0$ for every~$\alpha\in\R$.
Now, let~$u,v\in\dom{E}$ with~$u+v\in\dom{E}$ and~$(u-c)v=0$ for some~$c\in\R$.
By Corollary~\ref{c:ZeroLines} applied to~$F$, we have~$u-c,u+v-c\in\dom{F}$, as well as~$F(u-c)=F(u)$ and~$F(u+v-c)=F(u+v)$.
Thus, applying the weak locality of~$F$ to the pair~$u-c,v$,
\[
F(u+v)=F\tparen{(u-c)+v}=F(u-c)+F(v)=F(u)+F(v)\fstop
\]
Since~$\ell$ is linear, this is~\eqref{eq:t:Locality:1}, and~$E$ is strongly local.

\paragraph{Claim: \ref{i:c:Locality:0d} implies~\ref{i:t:Locality:0d}}
Assume~\ref{i:c:Locality:0d} and fix~$\alpha\geq 0$ and~$u\in\dom{E}$.
Respectively: by Corollary~\ref{c:ZeroLines}, by~\ref{i:p:WLocality:0d}, and by Corollary~\ref{c:ZeroLines} with~$\abs{u+\alpha}$ in place of~$u$ and~$-\alpha$ in place of~$\alpha$, we have
\[
E(u)=E(u+\alpha)=E(\abs{u+\alpha}) = E (\varphi_\alpha\circ u)\comma
\]
which is~\ref{i:t:Locality:0d}.

\paragraph{Proof of~\ref{i:c:Locality:i}}
Assume now that~$E$ is even.
The equivalence of~\ref{i:c:Locality:0}, \ref{i:c:Locality:1}, \ref{i:c:Locality:2}, and~\ref{i:c:Locality:3} follows immediately from Proposition~\ref{p:WLocality}.
Since~$E$ is even, the map~$\alpha\mapsto E(\alpha)$ is additive if and only if it vanishes identically; indeed, if it is additive, then~$E(\alpha)=E(-\alpha)=-E(\alpha)$ for every~$\alpha\in\R$.
Thus,~$E$ is strongly local if and only if~\ref{i:c:Locality:1} holds by the first assertion, which concludes the proof of~\ref{i:c:Locality:i}.

\paragraph{Proof of~\ref{i:c:Locality:ii}}
The conditions~\ref{i:c:Locality:0d} and~\ref{i:c:Locality:3d} are equivalent by Proposition~\ref{p:WLocalityUni}\ref{i:p:WLocalityUni:2}, and~\ref{i:t:Locality:0d} and~\ref{i:t:Locality:3d} are equivalent by Proposition~\ref{p:LocalityUni}\ref{i:p:LocalityUni:2}.
By the claim,~\ref{i:c:Locality:0d} implies~\ref{i:t:Locality:0d}.
Vice versa, assume~\ref{i:t:Locality:0d}.
Then,~\ref{i:p:WLocality:0d} holds, being~\ref{i:t:Locality:0d} for~$\alpha=0$.
Furthermore, for every~$\alpha\geq 0$, applying~\ref{i:t:Locality:0d} to the constant function~$u\equiv -2\alpha\in\dom{E}$, and since~$\varphi_\alpha(-2\alpha)=0$, we have~$E(-2\alpha)=E(0)=0$.
Since~$E$ is even, we conclude that~$E(\alpha)=0$ for every~$\alpha\in\R$, and therefore~\ref{i:c:Locality:0d} holds.

\paragraph{Proof of~\ref{i:c:Locality:iii}}
If~$\dom{E}$ is a zero-product cone, then~\ref{i:c:Locality:0d} and~\ref{i:c:Locality:0} are equivalent by Corollary~\ref{c:WLocality}, and the conclusion follows from~\ref{i:c:Locality:i} and~\ref{i:c:Locality:ii}.
\end{proof}

Without evenness, the additivity of~$\alpha\mapsto\ground{E}(\alpha)$ in Corollary~\ref{c:Locality} may not be replaced by its vanishing: the functional in Example~\ref{e:Evenness}(a), defined on a one-point space, is strongly local, while~$E(\alpha)=\alpha$ for every~$\alpha\in\R$.

\section{Invariance}\label{s:Invariance}
Let~$H$ be a Hilbert space and~$T_\bullet$ be a strongly continuous non-expansive non-linear semigroup.
In this section we discuss several notions of \emph{$T_\bullet$-invariant} subsets of~$H$.
In the case of bi-linear Dirichlet forms, all these notions reduce to the classical notion of invariance of a set under the semigroup action.

Let~$B\colon D\to H$ be an operator, and recall that a closed convex subset~$K\subset H$ is called $B$-\emph{invariant} if
\[
B(D\cap K)\subset K \fstop
\]
(This is the same as~$B(K)\subset K$ since~$\emp\subset K$ for every~$K$ and~$B(f)\eqdef \emp$ for every~$f\notin D$.)

\begin{definition}[Invariant subsets, subspaces]
Let~$H$ be a Hilbert space and~$T_\bullet$ be a strongly continuous non-expansive non-linear semigroup on a closed convex set~$C$.
A closed convex subset~$K$ of~$H$ is called
\begin{itemize}
\item \emph{$T_\bullet$-invariant} if it is $T_t$-invariant for every~$t>0$, see, e.g.~\cite[\S{IV.4}, p.~130]{Bre73};
\end{itemize}
A closed sub\emph{space}~$V$ of $H$ is called
\begin{itemize}
\item 
\emph{strongly $T_\bullet$-invariant} if
\begin{equation}\label{eq:d:Invariance:1.5}
\cproj_V\circ\cproj_C=\cproj_C\circ\cproj_V \quad \text{on} \quad H\comma
\end{equation}
and
\begin{equation}\label{eq:d:Invariance:2}
\cproj_V\circ T_t= T_t\circ \cproj_V \quad \text{on}\quad C\comma \qquad t\geq 0 \semicolon
\end{equation}
\item \emph{doubly $T_\bullet$-invariant} if both~$V$ and~$V^\tperp$ are strongly $T_\bullet$-invariant.
\end{itemize}
\end{definition}

Note that the commutation~\eqref{eq:d:Invariance:1.5} implies~$\cproj_V C\subset C$, so that~\eqref{eq:d:Invariance:2} is well-posed.
Further note that~\eqref{eq:d:Invariance:1.5} and~\eqref{eq:d:Invariance:2} together imply that~$V$ is $T_\bullet$-invariant, since~$T_t(f)=T_t(\cproj_V f)=\cproj_V T_t(f)\in V$ for every~$f\in C\cap V$.
Finally,~\eqref{eq:d:Invariance:1.5} is trivially satisfied if~$C=H$.

We refer the reader to e.g.,~\cite[Prop.~4.5, p.~131]{Bre73} for a characterization of $T_\bullet$-invariance via the associated generator and resolvent; and to~\cite[Thm.~1.9]{Bar96} and~\cite[Thm.~3.3]{ChiHauKen16} for a characterization of $T_\bullet$-invariance via the associated convex lower semicontinuous functional.

\medskip

In the following, it will be more convenient to express (double) $T_\bullet$-invariance in terms of the orthogonal projection~$\cproj_V$ onto~$V$, rather than in terms of~$V$ itself.
Thus, for an orthogonal projection~$\cproj\colon H\to H$, we set
\[
\cproj^\tperp\eqdef \car_H-\cproj\comma \qquad H_\cproj\eqdef \im \cproj  \quad \text{and}\quad H_{\cproj^\tperp}= \im \cproj^\tperp \fstop
\]
Further say that an orthogonal projection is \emph{strongly} (resp.\ \emph{doubly}) \emph{$T_\bullet$-invariant} if so is its range.

\begin{remark}
In the case of (self-adjoint) linear semigroups, strong invariance and double invariance are equivalent, since for every orthogonal projection~$\cproj$ we have~$\cproj^\tperp= \car_H-\cproj$, and therefore
\[
T_t \cproj^\tperp = T_t(\car_H- \cproj) = \car_H T_t - \cproj T_t = \cproj^\tperp T_t \fstop
\]
In the non-linear case this is no longer true in general, and ---as shown by the next theorem--- one ought to directly work with double invariance.
\end{remark}

\begin{theorem}\label{t:Invariance}
Let~$A$ be a maximal monotone operator on~$H$, with associated semigroup~$T_\bullet$ on~$C\eqdef\overline{\dom{A}}$ and resolvent~$J_\bullet$.
Further let~$V$ be a closed subspace of~$H$, and~$\cproj$ be the orthogonal projection onto~$V$.
Then, the following are equivalent:
\begin{enumerate}[$(i)$]
\item\label{i:t:Invariance:1} 
$V$ is doubly $T_\bullet$-invariant, i.e.~$\cproj\circ\cproj_C=\cproj_C\circ\cproj$ and~$\cproj^\tperp\circ\cproj_C=\cproj_C\circ\cproj^\tperp$ on~$H$, and, for every~$t\geq 0$ and every~$f\in C$,
\[
\cproj T_t(f)=T_t(\cproj f)\qquad \text{and} \qquad \cproj^\tperp T_t(f)=T_t(\cproj^\tperp f) \semicolon
\]

\item\label{i:t:Invariance:2} for every~$f\in\dom{A}$ we have~$\cproj f, \cproj^\tperp f\in\dom{A}$ and
\[
\cproj A(f)\subset A(\cproj f) \qquad \text{and} \qquad \cproj^\tperp A(f)\subset A(\cproj^\tperp f) \semicolon
\]

\item\label{i:t:Invariance:3} for every~$\lambda>0$ and every~$f\in H$,
\[
\cproj J_\lambda(f) = J_\lambda(\cproj f)\qquad \text{and}\qquad \cproj^\tperp  J_\lambda(f) = J_\lambda(\cproj^\tperp f) \fstop
\]
\end{enumerate}

Furthermore, if~$A$ is additionally cyclically monotone with associated proper convex lower semicontinuous functional~$E\colon H\to [0,\infty]$ non-negative and grounded, then each of~\ref{i:t:Invariance:1},~\ref{i:t:Invariance:2},~\ref{i:t:Invariance:3} is equivalent to 
\begin{enumerate}[$(i)$]\setcounter{enumi}{3}
\item\label{i:t:Invariance:4} for every~$f\in H$,
\[
E(f)=E(\cproj f) + E(\cproj^\tperp f) \comma
\]
in particular,~$\cproj f, \cproj^\tperp f\in\dom{E}$ for every~$f\in\dom{E}$.
\end{enumerate}
\end{theorem}

\begin{proof}
Throughout the proof,~$\aproj$ is a placeholder for either~$\cproj$ or~$\cproj^\tperp$, and~$\aproj^\tperp\eqdef \car_H-\aproj$.

\paragraph{Preliminaries: two consequences of~\ref{i:t:Invariance:3}}
Assume~\ref{i:t:Invariance:3}.
Since~$\lim_{\lambda\downarrow 0}J_\lambda(f)=\cproj_C(f)$ for every~$f\in H$ by~\cite[Thm.~2.2, p.~27]{Bre73}, and since~$\aproj$ is a bounded (linear) operator, letting~$\lambda\downarrow 0$ in~\ref{i:t:Invariance:3} shows that
\begin{equation}\label{eq:t:Invariance:0a}
\aproj\circ\cproj_C=\cproj_C\circ\aproj \quad \text{on } H\comma \qquad \text{in particular} \qquad \aproj C\subset C \comma
\end{equation}
since~$\aproj f=\aproj\cproj_C(f)=\cproj_C(\aproj f)$ for every~$f\in C$.
Furthermore, applying~\ref{i:t:Invariance:3} with~$f=0$ both for~$\aproj=\cproj$ and for~$\aproj=\cproj^\tperp$ yields~$J_\lambda(0)=\cproj J_\lambda(0)=\cproj^\tperp J_\lambda(0)\in V\cap V^\tperp=\set{0}$.
By definition~\eqref{eq:YosidaRegularization} of the resolvent, this reads
\begin{equation}\label{eq:t:Invariance:0b}
J_\lambda(0)=0\comma \quad \lambda>0\comma \qquad \text{equivalently} \qquad 0\in A(0) \fstop
\end{equation}

\medskip

We now prove the implications~\ref{i:t:Invariance:2}$\iff$\ref{i:t:Invariance:3} and~\ref{i:t:Invariance:1}$\iff$\ref{i:t:Invariance:3}.

\paragraph{\ref{i:t:Invariance:3}$\implies$\ref{i:t:Invariance:2}}
Fix~$f\in\dom{A}$,~$g\in A(f)$, and~$\lambda>0$.
Since~$f+\lambda g\in f+\lambda A(f)$, we have~$f=J_\lambda(f+\lambda g)$ by definition~\eqref{eq:YosidaRegularization} of the resolvent.
Applying~$\aproj$ on both sides and using the assumption,
\[
\aproj f=\aproj J_\lambda(f+\lambda g)=J_\lambda(\aproj f+\lambda\aproj g) \comma
\]
that is, again by~\eqref{eq:YosidaRegularization},~$\aproj f+\lambda\aproj g\in\aproj f+\lambda A(\aproj f)$.
Thus~$\aproj f\in\dom{A}$ and~$\aproj g\in A(\aproj f)$, which proves~\ref{i:t:Invariance:2} by arbitrariness of~$f$ and~$g$.

For later use, let us further note that, by the assumption,
\[
A_\lambda \circ \aproj = \tfrac{1}{\lambda}(\car_H-J_\lambda) \circ \aproj = \aproj \circ \tfrac{1}{\lambda}(\car_H-J_\lambda) = \aproj \circ A_\lambda \comma \qquad \lambda>0 \fstop
\]
Since~$\aproj$ is a bounded (linear) operator and~$\aproj\dom{A}\subset\dom{A}$, letting~$\lambda\downarrow 0$ above and applying~\cite[Prop.~2.6(iii), p.~28]{Bre73} we conclude that
\begin{equation}\label{eq:t:Invariance:5}
A_\lambda\circ\aproj=\aproj\circ A_\lambda \quad \text{on } H\comma\quad \lambda>0\comma \qquad \text{and} \qquad A^\circ(\aproj f)=\aproj A^\circ(f)\comma \quad f\in\dom{A}\fstop
\end{equation}

\paragraph{\ref{i:t:Invariance:2}$\implies$\ref{i:t:Invariance:3}}
Fix~$\lambda>0$ and~$g\in H$, and set~$f\eqdef J_\lambda(g)\in\dom{A}$.
By definition~\eqref{eq:YosidaRegularization} of the resolvent,~$a\eqdef \tfrac{1}{\lambda}(g-f)\in A(f)$.
By the assumption,~$\aproj a\in A(\aproj f)$, and therefore
\[
\aproj g=\aproj f+\lambda\aproj a\in \aproj f+\lambda A(\aproj f) \fstop
\]
Again by~\eqref{eq:YosidaRegularization}, this means that~$J_\lambda(\aproj g)=\aproj f=\aproj J_\lambda(g)$, which proves~\ref{i:t:Invariance:3} by arbitrariness of~$\lambda$ and~$g$.

\paragraph{\ref{i:t:Invariance:1}$\implies$\ref{i:t:Invariance:3}}
Let~$A^\sym{t}\colon H\to H$ be defined as in~\eqref{eq:ApproxSemigroup}.
Recall that it is a maximal monotone operator by Proposition~\ref{p:At}\ref{i:p:At:1}, and denote its resolvent by~$J^\sym{t}_\lambda$.
By the assumption, for every~$f\in H$ we have~$\cproj_C(f)\in C$ and~$\aproj\cproj_C(f)=\cproj_C(\aproj f)$, thus
\[
\aproj T_t\tparen{\cproj_C(f)}=T_t\tparen{\aproj\cproj_C(f)}=T_t\tparen{\cproj_C(\aproj f)} \comma \qquad t>0\comma
\]
and therefore~$\aproj\circ A^\sym{t}=A^\sym{t}\circ\aproj$ on~$H$, since~$\aproj$ is linear.
In particular,~\ref{i:t:Invariance:2} holds with~$A^\sym{t}$ in place of~$A$ for every~$t>0$.
By the implication~\ref{i:t:Invariance:2}$\implies$\ref{i:t:Invariance:3} with~$A^\sym{t}$ in place of~$A$ (hence~$J^\sym{t}_\lambda$ in place of~$J_\lambda$) we conclude that
\[
\aproj J^\sym{t}_\lambda(f) = J^\sym{t}_\lambda(\aproj f)\comma \qquad \lambda,t>0\comma \qquad f\in H \fstop
\]
Letting~$t\downarrow 0$, by Proposition~\ref{p:At}\ref{i:p:At:4} and since~$\aproj$ is a bounded linear operator,
\[
\aproj J_\lambda(f) = J_\lambda(\aproj f)\comma \qquad \lambda>0\comma \qquad f\in H \comma
\]
which concludes the proof of~\ref{i:t:Invariance:3}.

\paragraph{\ref{i:t:Invariance:3}$\implies$\ref{i:t:Invariance:1}}
The commutation of~$\aproj$ with~$\cproj_C$ is~\eqref{eq:t:Invariance:0a}, which also shows that~$\aproj f\in C$ for every~$f\in C$.
Fix~$f\in C$ and~$t>0$; for~$t=0$ there is nothing to prove.
Respectively: by Theorem~\ref{t:SemigroupResolvent}\ref{i:t:SemigroupResolvent:2} applied to~$\aproj f\in C$; since~$\aproj$ is idempotent and commutes with~$J_{t/n}$ for every~$n$ by the assumption; and since~$\aproj$ is continuous,
\begin{equation}\label{eq:t:Invariance:7}
\begin{aligned}
T_t(\aproj  f) &=\lim_{n\to\infty} (J_{t/n})^n(\aproj  f) = \lim_{n\to\infty} (J_{t/n}\circ \aproj )^n (f) =  \lim_{n\to\infty} (\aproj \circ J_{t/n})^n (f)
\\
&= \lim_{n\to\infty} \aproj^n (J_{t/n})^n (f) = \aproj  \lim_{n\to\infty} (J_{t/n})^n (f) = \aproj  T_t(f) \comma
\end{aligned}
\end{equation}
which is the sought conclusion.

\bigskip

\paragraph{Equivalences with functionals}
We now prove the implications~\ref{i:t:Invariance:4}$\implies$\ref{i:t:Invariance:3} and~\ref{i:t:Invariance:3}$\implies$\ref{i:t:Invariance:4} under the additional assumption of cyclical monotonicity.
Recall that~$E\geq 0=E(0)$ by assumption.

\paragraph{\ref{i:t:Invariance:4}$\implies$\ref{i:t:Invariance:3}}
Since~$E\geq 0$, the assumption gives~$E(\aproj f)\leq E(\aproj f)+E(\aproj^\tperp f)=E(f)$ for every~$f\in H$.
Thus, by the equivalence between~\ref{i:p:EquivConvex:1} and~\ref{i:p:EquivConvex:2} in Proposition~\ref{p:EquivConvex} with~$H_\aproj$ in place of~$C$, we have~$J_\lambda(H_\aproj)\subset H_\aproj$, that is
\begin{equation}\label{eq:t:Invariance:4}
\aproj  J_\lambda(g) = J_\lambda(g)\comma \qquad g\in H_\aproj\comma \qquad \lambda>0 \fstop
\end{equation}

Now, fix~$f\in H$ and~$\lambda>0$.
Respectively: by Theorem~\ref{t:FormResolvent}\ref{i:t:FormResolvent:1}, and by orthogonality of the decomposition~$H=H_\cproj\oplus^\tperp H_{\cproj^\tperp}$ together with the assumption~\ref{i:t:Invariance:4} applied to every~$g\in H$,
\begin{align*}
J_\lambda(f)&=\argmin_{g\in H} \tfrac{1}{2\lambda}\norm{f-g}^2+E(g)
\\
&= \argmin_{g\in H} \paren{\tfrac{1}{2\lambda}\norm{\cproj f-\cproj g}^2+E(\cproj g)}+\paren{\tfrac{1}{2\lambda}\norm{\cproj^\tperp f-\cproj^\tperp g}^2 + E(\cproj^\tperp g)} \fstop
\end{align*}
Since the two summands depend on~$\cproj g$ and on~$\cproj^\tperp g$ only, the unique minimizer~$g_*$ is characterized by
\[
\aproj g_*= \argmin_{h\in H_\aproj} \tfrac{1}{2\lambda}\norm{\aproj f-h}^2+E(h) \comma \qquad \aproj\in\set{\cproj,\cproj^\tperp} \fstop
\]
By~\eqref{eq:t:Invariance:4}, the global minimizer~$J_\lambda(\aproj f)$ on~$H$ of the functional~$h\mapsto \tfrac{1}{2\lambda}\norm{\aproj f-h}^2+E(h)$ lies in~$H_\aproj$, and therefore it coincides with its minimizer on~$H_\aproj$.
Hence,
\begin{equation}\label{eq:t:Invariance:3}
J_\lambda(f)= J_\lambda(\cproj f)+J_\lambda(\cproj^\tperp  f) \comma
\end{equation}
with~$J_\lambda(\cproj f)\in H_\cproj$ and~$J_\lambda(\cproj^\tperp f)\in H_{\cproj^\tperp}$ by~\eqref{eq:t:Invariance:4}.
Applying~$\aproj$ on both sides of~\eqref{eq:t:Invariance:3} we thus obtain~$\aproj J_\lambda(f)=J_\lambda(\aproj f)$, as desired.

\paragraph{\ref{i:t:Invariance:3}$\implies$\ref{i:t:Invariance:4}}
Let~$E_\lambda$ be the Yosida regularization of~$E$, Definition~\ref{d:YosidaRegForm}.
Recall that~$E_\lambda\colon H\to\R$ is Fr\'echet differentiable with~$\nabla E_\lambda=A_\lambda$, see~\eqref{eq:YosidaForm}, and that~$E_\lambda(f)$ is non-decreasing as~$\lambda\downarrow 0$ and converges to~$E(f)$ for every~$f\in H$, see~\eqref{eq:MonotonicityYosidaRegE} and~\cite[Prop.~2.11, p.~39]{Bre73}.
Since~$E\geq 0=E(0)$, we further have~$E_\lambda\geq 0$ and~$0\leq E_\lambda(0)\leq E(0)=0$.

By the assumption, for every~$\lambda>0$ and every~$f\in H$,
\begin{align*}
A_\lambda(f)=\tfrac{1}{\lambda} \tparen{f-J_\lambda(f)} &= \tfrac{1}{\lambda} \tparen{\cproj f-\cproj J_\lambda(f)} + \tfrac{1}{\lambda} \tparen{\cproj^\tperp f-\cproj^\tperp J_\lambda(f)}
\\
&= \tfrac{1}{\lambda} \tparen{\cproj f- J_\lambda(\cproj f)} + \tfrac{1}{\lambda} \tparen{\cproj^\tperp f- J_\lambda(\cproj^\tperp f)} 
\\
&= A_\lambda(\cproj f) + A_\lambda(\cproj^\tperp f) \fstop
\end{align*}
Now, fix~$\lambda>0$ and define~$\phi_\lambda\colon H\to\R$ by~$\phi_\lambda(f)\eqdef E_\lambda(\cproj f)+E_\lambda(\cproj^\tperp f)$.
Respectively: by the chain rule, since~$\cproj$ and~$\cproj^\tperp$ are self-adjoint; by~\eqref{eq:t:Invariance:5}, since~$\cproj$ and~$\cproj^\tperp$ are idempotent; and by the above equality,
\[
\nabla\phi_\lambda(f)=\cproj A_\lambda(\cproj f)+\cproj^\tperp A_\lambda(\cproj^\tperp f) = A_\lambda(\cproj f)+A_\lambda(\cproj^\tperp f)=A_\lambda(f)=\nabla E_\lambda(f) \comma \qquad f\in H\fstop
\]
Hence, by the fundamental theorem of calculus along the segment from~$0$ to~$f$, and since~$\phi_\lambda(0)=2E_\lambda(0)=0=E_\lambda(0)$,
\begin{equation}\label{eq:t:Invariance:6}
E_\lambda(f)=E_\lambda(\cproj f)+E_\lambda(\cproj^\tperp f)\comma \qquad f\in H\comma \quad \lambda>0 \fstop
\end{equation}
Finally, since all terms in~\eqref{eq:t:Invariance:6} are non-negative and non-decreasing as~$\lambda\downarrow 0$, letting~$\lambda\downarrow 0$ we conclude that~$E(f)=E(\cproj f)+E(\cproj^\tperp f)$ in~$[0,\infty]$ for every~$f\in H$, which is~\ref{i:t:Invariance:4}.

\paragraph{Consequences}
If any of~\ref{i:t:Invariance:1}--\ref{i:t:Invariance:3} holds (or~\ref{i:t:Invariance:4}, in the cyclically monotone case), then~\ref{i:t:Invariance:3} holds, and therefore so do~\eqref{eq:t:Invariance:0a} and~\eqref{eq:t:Invariance:0b}.
In particular, for every~$f\in H_\aproj$ we have~$\cproj_C(f)=\cproj_C(\aproj f)=\aproj\cproj_C(f)\in H_\aproj$, that is,~$\cproj_{\overline{\dom{A}}}H_\aproj\subset H_\aproj$, which is~\eqref{eq:p:EquivConvex:0} with~$H_\aproj$ in place of~$C$.
\end{proof}

\begin{remark}[On the assumptions in Theorem~\ref{t:Invariance}]\label{r:Invariance}
\begin{enumerate*}[$(a)$]
\item\label{i:r:Invariance:1} The commutation of~$\cproj$ and~$\cproj^\tperp$ with~$\cproj_C$ in~\ref{i:t:Invariance:1} is equivalent to~$C=\set{a+b : a\in C\cap V\comma b\in C\cap V^\tperp}$.
Indeed, if the commutation holds, then\linebreak $\cproj f=\cproj_C(\cproj f)\in C\cap V$ and~$\cproj^\tperp f\in C\cap V^\tperp$ for every~$f\in C$, and, for~$a\in C\cap V$ and~$b\in C\cap V^\tperp$, we have~$\cproj\cproj_C(a+b)=\cproj_C(a)=a$ and~$\cproj^\tperp\cproj_C(a+b)=\cproj_C(b)=b$, whence~$a+b=\cproj_C(a+b)\in C$.
Vice versa, if~$C=(C\cap V)+(C\cap V^\tperp)$, then~$0\in C\cap V$ and~$0\in C\cap V^\tperp$ (write any~$c\in C\cap V$ as~$c=a+b$ with~$a\in C\cap V$ and~$b\in C\cap V^\tperp$, so that~$b=\cproj^\tperp c=0$, and analogously), and, by orthogonality,~$\cproj_C(f)=\cproj_{C\cap V}(\cproj f)+\cproj_{C\cap V^\tperp}(\cproj^\tperp f)$ for every~$f\in H$; applying this identity to~$\cproj f$ and to~$\cproj^\tperp f$ yields the commutation.
In particular, in this case~$\cproj\circ\cproj_C=\cproj_{C\cap V}$, since~$\cproj_{C\cap V}(f)=\cproj_{C\cap V}(\cproj f)$ for every~$f\in H$.

\item The commutation of~$\cproj$ and~$\cproj^\tperp$ with~$\cproj_C$ may not be replaced by the weaker condition~$\cproj C,\cproj^\tperp C\subset C$.
For instance, let~$H\eqdef\R^2$, $C\eqdef\set{h\in\R^2 : h_1,h_2\geq 0\comma h_1+h_2\leq 1}$, $E\eqdef\mbfI_C$, and~$\cproj$ be the orthogonal projection onto the first coordinate axis.
Then~$T_t=\id_C$ for every~$t\geq 0$ and~$J_\lambda=\cproj_C$ for every~$\lambda>0$.
Thus,~$\cproj C,\cproj^\tperp C\subset C$ and~$T_t$ commutes with~$\cproj$ and~$\cproj^\tperp$ on~$C$, yet~$\cproj J_\lambda(1,1)=(\tfrac12,0)\neq (1,0)=J_\lambda\tparen{\cproj(1,1)}$.
The same example shows that the identity in~\ref{i:t:Invariance:4} must be required for every~$f\in H$, and not only for~$f\in\dom{E}$.

\item\label{i:r:Invariance:2} The inclusions in~\ref{i:t:Invariance:2} may not be replaced by equalities.
For instance, let~$H\eqdef\R^2$, $K\eqdef[0,\infty)^\tym{2}$, $E\eqdef\mbfI_K$, and~$\cproj$ be as above.
Then~$J_\lambda(h)=(h_1\vee 0,h_2\vee 0)$ commutes with~$\cproj$ and~$\cproj^\tperp$, while~$\cproj A(0)=(-\infty,0]\times\set{0}\subsetneq(-\infty,0]^\tym{2}=A(\cproj 0)$.
Identifying~$\R^2$ with~$L^2_\mssm$ for~$\mssm\eqdef\delta_1+\delta_2$ on~$X\eqdef\set{1,2}$, we have~$\cproj=\car_{\set{1}}$, and~$E$ is a Dirichlet functional, since~$\cproj_\tleq$ and~$\cproj_\alpha$ act by pointwise averaging and thus preserve~$K\oplus K$.

\item\label{i:r:Invariance:3} By~\eqref{eq:t:Invariance:0b}, each of~\ref{i:t:Invariance:1}--\ref{i:t:Invariance:3} implies that~$0\in A(0)$, i.e.\ that~$0$ is a global minimizer of~$E$ in the cyclically monotone case.
Thus, in light of Remark~\ref{r:AdditiveConstant}, the assumption that~$E$ be non-negative and grounded is merely a normalization, which may always be achieved by replacing~$E$ with~$E-E(0)$.
It may however not be dispensed with: for~$E\equiv 1$ we have~$A\equiv 0$, and~\ref{i:t:Invariance:1}--\ref{i:t:Invariance:3} hold for every closed subspace~$V$, while~\ref{i:t:Invariance:4} holds for none.
\end{enumerate*}
\end{remark}

The following examples show that
\begin{itemize}
\item $T_\bullet$-invariant subspaces are generally not strongly $T_\bullet$-invariant, thus neither are they doubly $T_\bullet$-invariant (Examples~\ref{ex:StrongDoubleInv} and~\ref{ex:InvarianceTime});
\item a subspace~$V$ such that both~$V$ and~$V^\tperp$ are $T_\bullet$-invariant is generally not doubly $T_\bullet$-invariant (Examples~\ref{ex:StrongDoubleInv} and~\ref{ex:InvarianceTime});
\item the commutation~\eqref{eq:d:Invariance:2} may fail for every~$t>0$, even if the semigroup becomes stationary in finite time, i.e.\ even if~$T_t=T_{t_0}$ for every~$t\geq t_0$ and some~$t_0>0$ (Example~\ref{ex:InvarianceTime});
\item a set which is eventually $T_\bullet$-invariant (i.e.\ $T_t$-invariant for all~$t\geq c$ for some~$c>0$) is generally \emph{not} $T_\bullet$-invariant (Example~\ref{ex:EventualInvariance}).
\end{itemize}

\begin{example}\label{ex:StrongDoubleInv}
For~$p\in [1,\infty]$ denote by~$\abs{h}_p$ the~$\ell^p$-norm of~$h\in \R^2$, and set
\[
\sgn(h)\eqdef \begin{bmatrix} \sgn(h_1)\\ \sgn(h_2)\end{bmatrix}\comma \qquad K_\infty\eqdef\overline{B^\infty_1(0)}\comma \qquad K_1\eqdef\overline{B^1_1(0)} \fstop
\]

On~$H\eqdef (\R^2,\abs{\emparg}_2)$ consider the functional
\[
E\colon h \longmapsto \begin{cases}0 &\text{if } \abs{h}_1\leq 1 \\ \tparen{\abs{h}_1-1}^2 &\text{if } \abs{h}_1> 1 \comma \abs{h}_\infty\leq 1 \\ +\infty &\text{if } \abs{h}_\infty> 1 \end{cases} \comma
\]
that is,~$E=\phi\circ\abs{\emparg}_1+\mbfI_{K_\infty}$ with~$\phi(r)\eqdef \tparen{(r-1)_+}^2$ for~$r\in\R$.
For~$i=1,2$ further let~$\cproj_i$ be the orthogonal projection onto the $i^\text{th}$ axis, and note that~$\cproj_2^\tperp=\cproj_1$.

\paragraph{Basic properties}
Since~$\phi$ is convex and non-decreasing,~$E$ is proper convex lower semicontinuous, even, non-negative, and grounded, and~$\dom{E}=K_\infty$.
Furthermore,~$\partial E(u)\neq\emp$ for every~$u\in K_\infty$.
Indeed, if~$\abs{u}_1\leq 1$, then~$0\in\partial E(u)$, since~$u$ is a global minimizer of~$E$.
If otherwise~$\abs{u}_1>1$, then both coordinates of~$u$ are non-zero, since~$\abs{u}_\infty\leq 1$, and thus~$\abs{u}_1=\scalar{\sgn(u)}{u}$.
Since~$\phi$ is non-decreasing, convex, and differentiable, and since~$\abs{g}_1\geq\scalar{\sgn(u)}{g}$, for every~$g\in\R^2$ we have
\begin{align*}
E(g)&\geq \phi\tparen{\abs{g}_1}\geq\phi\tparen{\scalar{\sgn(u)}{g}}\geq \phi\tparen{\abs{u}_1}+\phi'\tparen{\abs{u}_1}\scalar{\sgn(u)}{g-u} 
\\
&= E(u)+\scalar{2\tparen{\abs{u}_1-1}\sgn(u)}{g-u}\comma
\end{align*}
whence~$2\tparen{\abs{u}_1-1}\sgn(u)\in\partial E(u)$ by~\eqref{eq:ConvexSubDiff}.
As a consequence,~$\dom{A}=\dom{E}=K_\infty$, and the semigroup~$T_\bullet$ is defined on~$C\eqdef K_\infty$.
Since~$K_\infty=[-1,1]^2$ is a product set, we have~$\cproj_i\circ\cproj_{K_\infty}=\cproj_{K_\infty}\circ\cproj_i$ for~$i=1,2$; in particular,~\eqref{eq:d:Invariance:1.5} holds with~$V=H_{\cproj_i}$, and~\eqref{eq:p:EquivConvex:0} holds with~$C=H_{\cproj_i}$, for~$i=1,2$.

\paragraph{Semigroup}
For~$h\in K_\infty$ set~$t_h\eqdef\tfrac{1}{2}\tparen{\abs{h}_1-1}_+\in[0,\tfrac12]$.
We claim that
\begin{equation}\label{eq:ex:StrongDoubleInv:3}
T_t(h)= h-\tparen{1-e^{-4t}}\,t_h\sgn(h)\comma \qquad h\in K_\infty\comma\quad t\geq 0\fstop
\end{equation}
If~$\abs{h}_1\leq 1$, then~$t_h=0$, and~\eqref{eq:ex:StrongDoubleInv:3} holds since~$0\in\partial E(h)$.
If otherwise~$\abs{h}_1>1$, denote by~$u(s)$ the right-hand side of~\eqref{eq:ex:StrongDoubleInv:3} with~$s$ in place of~$t$.
For~$i=1,2$ and~$j\neq i$ we have
\[
\abs{h_i}-t_h=\tfrac12\tparen{1+\abs{h_i}-\abs{h_j}}\geq\tfrac12\abs{h_i}>0\comma
\]
since~$\abs{h_j}\leq 1$.
Thus, for every~$s\geq 0$ we have~$\sgn\tparen{u(s)}=\sgn(h)$,~$\abs{u(s)}_\infty\leq\abs{h}_\infty\leq 1$, and
\[
\abs{u(s)}_1=\abs{h}_1-2\tparen{1-e^{-4s}}t_h=1+e^{-4s}\tparen{\abs{h}_1-1}>1\comma
\]
and therefore, by the previous paragraph,
\[
-\tfrac{\diff}{\diff s}u(s)=2e^{-4s}\tparen{\abs{h}_1-1}\sgn(h)=2\tparen{\abs{u(s)}_1-1}\sgn\tparen{u(s)}\in\partial E\tparen{u(s)}\fstop
\]
Since~$u$ is Lipschitz and~$u(0)=h$, the claim follows from the uniqueness assertion in Theorem~\ref{t:BrezisPazy}.

Furthermore,~$h-t_h\sgn(h)=\cproj_{K_1}(h)$ for every~$h\in K_\infty$.
Indeed, this is trivial if~$\abs{h}_1\leq 1$.
If otherwise~$\abs{h}_1>1$, then~$h_*\eqdef h-t_h\sgn(h)$ satisfies~$h_*\in K_1$ and~$\scalar{\sgn(h)}{h_*}=\abs{h}_1-2t_h=1$ by the above computation, whereas~$\scalar{\sgn(h)}{g}\leq\abs{g}_1\leq 1$ for every~$g\in K_1$.
Hence,~$\scalar{h-h_*}{g-h_*}=t_h\scalar{\sgn(h)}{g-h_*}\leq 0$ for every~$g\in K_1$, which characterizes~$h_*$ as~$\cproj_{K_1}(h)$.
As a consequence,~\eqref{eq:ex:StrongDoubleInv:3} may be rewritten as
\begin{equation}\label{eq:ex:StrongDoubleInv:4}
T_t(h)=\cproj_{K_1}(h)+e^{-4t}\tparen{h-\cproj_{K_1}(h)}\comma\qquad h\in K_\infty\comma\quad t\geq 0\comma
\end{equation}
so that~$T_t(h)$ converges to~$\cproj_{K_1}(h)$ exponentially fast as~$t\to\infty$, while~$T_t(h)\notin K_1$ for every~$t\geq 0$ whenever~$\abs{h}_1>1$.

\paragraph{Invariance}
Since~$\abs{\cproj_i h}_1=\abs{h_i}\leq 1$ for every~$h\in K_\infty$, we have~$E(\cproj_i h)=0$ for every~$h\in K_\infty$, and therefore
\begin{align}
\label{eq:ex:StrongDoubleInv:1}
E(\cproj_i h)&\leq E(h)\comma & h&\in H \comma \quad i=1,2\comma
\\
\label{eq:ex:StrongDoubleInv:2}
E(\cproj_1 h)+ E(\cproj_2 h) &< E(h)\comma & h&\in K_\infty\comma \quad \abs{h}_1> 1 \fstop
\end{align}
By Proposition~\ref{p:EquivConvex} and~\eqref{eq:ex:StrongDoubleInv:1},~$H_{\cproj_i}$ is $T_\bullet$-invariant for~$i=1,2$.
(This may also be seen directly from~\eqref{eq:ex:StrongDoubleInv:3}, which gives~$T_t(\cproj_i h)=\cproj_i h$ for every~$h\in K_\infty$ and every~$t\geq 0$, since~$t_{\cproj_i h}=0$.)
Yet, neither~$H_{\cproj_1}$ nor~$H_{\cproj_2}=H_{\cproj_1}^\tperp$ is doubly $T_\bullet$-invariant, in light of~\eqref{eq:ex:StrongDoubleInv:2} and Theorem~\ref{t:Invariance}\ref{i:t:Invariance:4}.
In fact, neither is strongly $T_\bullet$-invariant: if~$h\in K_\infty$ and~$\abs{h}_1>1$, then~$\cproj_i\sgn(h)\neq 0$, and~\eqref{eq:ex:StrongDoubleInv:3} gives
\[
\cproj_i T_t(h)=\cproj_i h-\tparen{1-e^{-4t}}\,t_h\,\cproj_i\sgn(h)\neq\cproj_i h=T_t(\cproj_i h)\comma\qquad t>0\fstop
\]
\end{example}

\begin{example}\label{ex:InvarianceTime}
With the same notation as in Example~\ref{ex:StrongDoubleInv} consider the functional
\[
E\colon h \longmapsto \begin{cases}0 &\text{if } \abs{h}_1\leq 1 \\ \abs{h}_1-1 &\text{if } \abs{h}_1> 1 \comma \abs{h}_\infty\leq 1 \\ +\infty &\text{if } \abs{h}_\infty> 1 \end{cases} \comma
\]
that is,~$E=\psi\circ\abs{\emparg}_1+\mbfI_{K_\infty}$ with~$\psi(r)\eqdef(r-1)_+$ for~$r\in\R$.

\paragraph{Basic properties}
As in Example~\ref{ex:StrongDoubleInv},~$E$ is proper convex lower semicontinuous, even, non-negative, and grounded, and~$\dom{E}=K_\infty$.
Since~$\psi$ is convex, non-decreasing, and differentiable at every~$r>1$ with~$\psi'(r)=1$, the argument in Example~\ref{ex:StrongDoubleInv}, with~$\psi$ in place of~$\phi$, shows that~$\sgn(u)\in\partial E(u)$ for every~$u\in K_\infty$ with~$\abs{u}_1>1$, while~$0\in\partial E(u)$ whenever~$\abs{u}_1\leq 1$.
Thus,~$\dom{A}=\dom{E}=K_\infty$, the semigroup~$T_\bullet$ is defined on~$C\eqdef K_\infty$, and, again as in Example~\ref{ex:StrongDoubleInv},~\eqref{eq:d:Invariance:1.5} holds with~$V=H_{\cproj_i}$, and~\eqref{eq:p:EquivConvex:0} holds with~$C=H_{\cproj_i}$, for~$i=1,2$.

\paragraph{Semigroup}
With~$t_h$ as in Example~\ref{ex:StrongDoubleInv}, we claim that
\begin{equation}\label{eq:ex:InvarianceTime:1}
T_t(h)= h-(t\wedge t_h)\,\sgn(h)\comma \qquad h\in K_\infty\comma\quad t\geq 0\fstop
\end{equation}
If~$\abs{h}_1\leq 1$, then~$t_h=0$, and~\eqref{eq:ex:InvarianceTime:1} holds since~$0\in\partial E(h)$.
If otherwise~$\abs{h}_1>1$, denote by~$u(s)$ the right-hand side of~\eqref{eq:ex:InvarianceTime:1} with~$s$ in place of~$t$.
Since~$\abs{h_i}-t_h>0$ for~$i=1,2$, as shown in Example~\ref{ex:StrongDoubleInv}, for every~$s\in[0,t_h)$ we have~$\sgn\tparen{u(s)}=\sgn(h)$,~$\abs{u(s)}_1=\abs{h}_1-2s>1$, and~$\abs{u(s)}_\infty\leq\abs{h}_\infty\leq 1$, and therefore
\[
-\tfrac{\diff}{\diff s}u(s)=\sgn(h)=\sgn\tparen{u(s)}\in\partial E\tparen{u(s)}
\]
by the previous paragraph.
For every~$s>t_h$, instead,~$u(s)=u(t_h)$ satisfies~$\abs{u(t_h)}_1=\abs{h}_1-2t_h=1$, hence~$-\tfrac{\diff}{\diff s}u(s)=0\in\partial E\tparen{u(s)}$.
Since~$u$ is Lipschitz and~$u(0)=h$, the claim follows from the uniqueness assertion in Theorem~\ref{t:BrezisPazy}.

Since~$h-t_h\sgn(h)=\cproj_{K_1}(h)$ for every~$h\in K_\infty$, see Example~\ref{ex:StrongDoubleInv}, and since~$t_h\leq\tfrac12$ for every~$h\in K_\infty$, we conclude that
\begin{equation}\label{eq:ex:InvarianceTime:2}
T_t=\cproj_{K_1} \quad \text{on}\quad K_\infty\comma\qquad t\geq\tfrac12\comma
\end{equation}
that is, the semigroup becomes stationary in finite time, uniformly on its domain.
This is in contrast with Example~\ref{ex:StrongDoubleInv}, where~$T_t(h)$ converges to~$\cproj_{K_1}(h)$ only exponentially fast as~$t\to\infty$, see~\eqref{eq:ex:StrongDoubleInv:4}, and never reaches it whenever~$\abs{h}_1>1$.

\paragraph{Invariance}
Exactly as in Example~\ref{ex:StrongDoubleInv}, we have~$E(\cproj_i h)=0$ for every~$h\in K_\infty$ and~$i=1,2$, so that~\eqref{eq:ex:StrongDoubleInv:1} and~\eqref{eq:ex:StrongDoubleInv:2} hold for the present functional as well.
By~\eqref{eq:ex:InvarianceTime:1} we have~$T_t(\cproj_i h)=\cproj_i h$ for every~$h\in K_\infty$ and every~$t\geq 0$, and in particular~$H_{\cproj_1}$ and~$H_{\cproj_2}=H_{\cproj_1}^\tperp$ are both $T_\bullet$-invariant.
On the other hand, if~$h\in K_\infty$ and~$\abs{h}_1>1$, then~$\cproj_i\sgn(h)\neq 0$, and~\eqref{eq:ex:InvarianceTime:1} gives
\[
\cproj_i T_t(h)=\cproj_i h-(t\wedge t_h)\,\cproj_i\sgn(h)\neq\cproj_i h=T_t(\cproj_i h)\comma \qquad t>0\fstop
\]
Thus,~\eqref{eq:d:Invariance:2} fails for \emph{every}~$t>0$, and neither~$H_{\cproj_1}$ nor~$H_{\cproj_2}$ is strongly (hence doubly) $T_\bullet$-invariant, consistently with~\eqref{eq:ex:StrongDoubleInv:2} and Theorem~\ref{t:Invariance}\ref{i:t:Invariance:4}.
In light of~\eqref{eq:ex:InvarianceTime:2}, the failure of the commutation persists after the semigroup has become stationary: for every~$t\geq\tfrac12$ we have~$\cproj_i\circ T_t=\cproj_i\circ\cproj_{K_1}$ and~$T_t\circ\cproj_i=\cproj_i$ on~$K_\infty$, and the two differ on~$K_\infty\setminus K_1$.

\paragraph{Eventual invariance}
Finally, for this functional eventual invariance implies invariance.
Namely, for a closed convex set~$K\subset\R^2$ the following are equivalent:
\begin{enumerate*}[$(a)$]
\item $K$ is $T_\bullet$-invariant;
\item $K$ is $T_t$-invariant for \emph{some}~$t\geq\tfrac12$;
\item $\cproj_{K_1}(K\cap K_\infty)\subset K$.
\end{enumerate*}
Indeed,~$(a)$ trivially implies~$(b)$, and~$(b)$ is equivalent to~$(c)$ by~\eqref{eq:ex:InvarianceTime:2}.
If~$(c)$ holds, then, for every~$h\in K\cap K_\infty$ and every~$t\geq 0$, the point~$T_t(h)$ lies on the segment with endpoints~$h$ and~$\cproj_{K_1}(h)$ by~\eqref{eq:ex:InvarianceTime:1}, and thus in~$K$ by convexity of~$K$, which shows~$(a)$.
\end{example}

\begin{figure}[htb!]
\centering
\begin{subfigure}[t]{.3\textwidth}
\includegraphics[scale=.2]{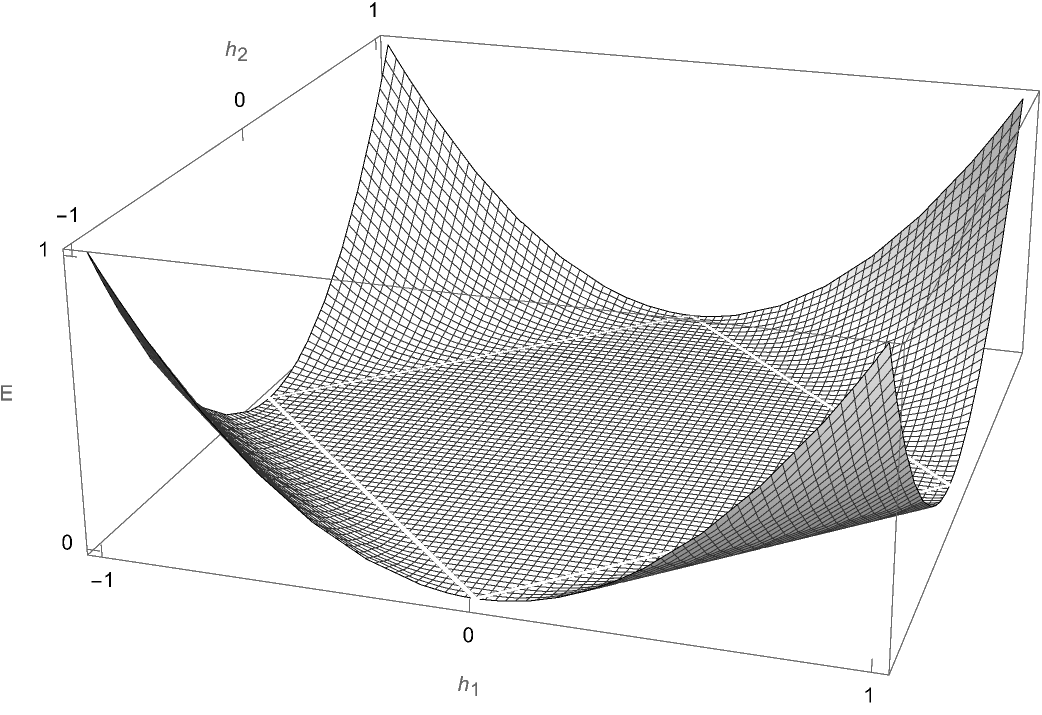}
\caption{Example~\ref{ex:StrongDoubleInv}}
\end{subfigure}
\begin{subfigure}[t]{.3\textwidth}
\includegraphics[scale=.2]{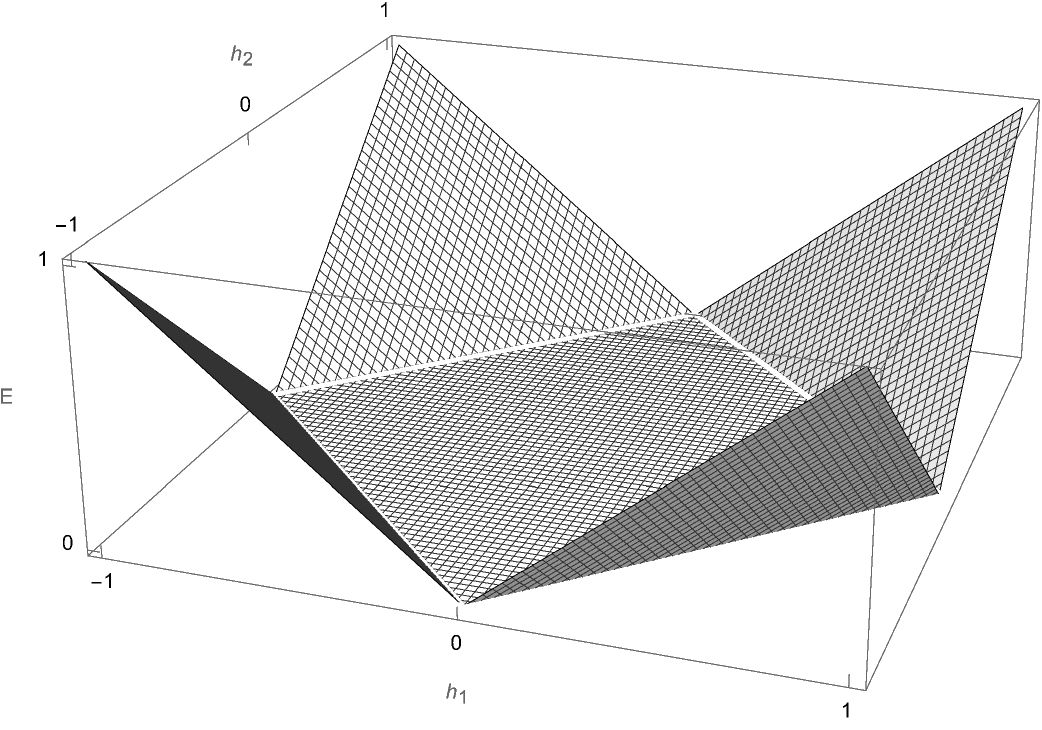}
\caption{Example~\ref{ex:InvarianceTime}}
\end{subfigure}
\begin{subfigure}[t]{.3\textwidth}
\includegraphics[scale=.2]{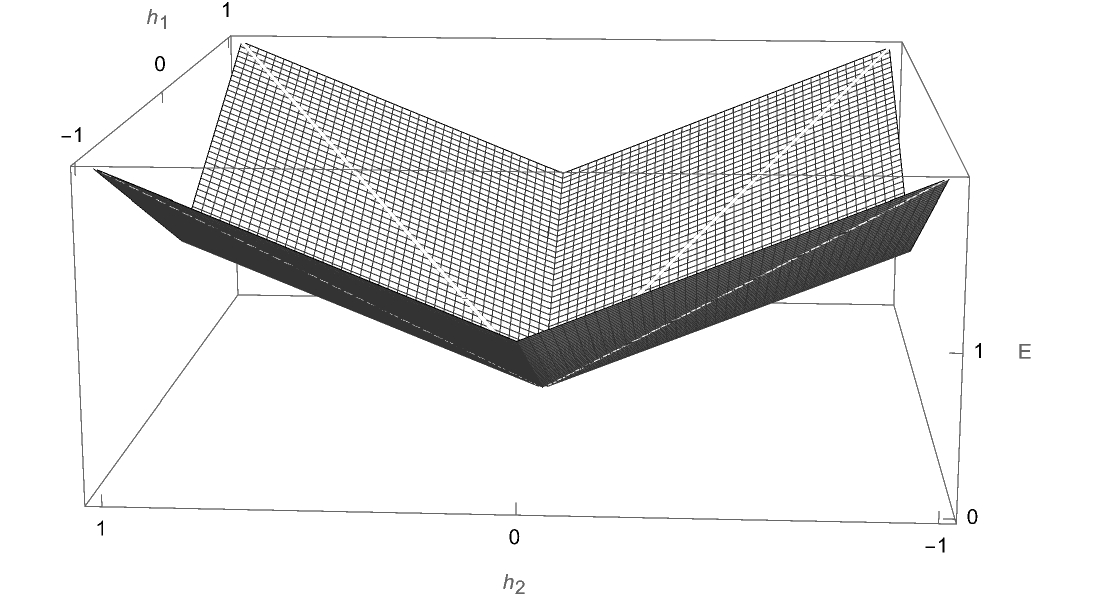}
\caption{Example~\ref{ex:EventualInvariance}}
\end{subfigure}
\caption{The functionals in Examples~\ref{ex:StrongDoubleInv},~\ref{ex:InvarianceTime}, and~\ref{ex:EventualInvariance}.}
\end{figure}

\begin{example}\label{ex:EventualInvariance}
With the same notation as in Example~\ref{ex:StrongDoubleInv} consider the functional
\[
E\colon h \longmapsto \begin{cases}\abs{h}_1 &\text{if } \abs{h}_\infty\leq 1 \\ +\infty &\text{if } \abs{h}_\infty> 1 \end{cases} \comma
\]
that is,~$E=\abs{\emparg}_1+\mbfI_{K_\infty}$.
Equivalently,~$E(h)=e(h_1)+e(h_2)$ for every~$h\in\R^2$, where~$e\colon\R\to[0,\infty]$ is defined by~$e(r)\eqdef\abs{r}+\mbfI_{[-1,1]}(r)$.

\paragraph{Basic properties}
Clearly,~$E$ is proper convex lower semicontinuous, even, non-negative, and grounded,~$\dom{E}=K_\infty$, and~$0$ is the unique minimizer of~$E$.
Since~$E$ is the sum of two convex functions, each depending on a single coordinate, it readily follows from~\eqref{eq:ConvexSubDiff} that~$\partial E(u)=\partial e(u_1)\times\partial e(u_2)$ for every~$u\in K_\infty$, where
\[
\partial e(r)=\begin{cases} [-1,1] &\text{if } r=0\comma\\ \set{\sgn(r)} &\text{if } 0<\abs{r}<1\comma \\ \sgn(r)\,[1,\infty) &\text{if } \abs{r}=1\comma\end{cases} \qquad r\in[-1,1]\fstop
\]
In particular,~$\dom{A}=\dom{E}=K_\infty$, and the semigroup~$T_\bullet$ is defined on~$C\eqdef K_\infty$.

\paragraph{Semigroup}
We claim that~$T_t$ acts as
\begin{equation}\label{eq:ex:EventualInvariance:1}
T_t(h)=\begin{bmatrix} \sgn(h_1)\tparen{\abs{h_1}-t}_+\\ \sgn(h_2)\tparen{\abs{h_2}-t}_+\end{bmatrix}\comma \qquad h\in K_\infty\comma\quad t\geq 0\fstop
\end{equation}
Indeed, fix~$h\in K_\infty$, denote by~$u(s)$ the right-hand side of~\eqref{eq:ex:EventualInvariance:1} with~$s$ in place of~$t$, and fix~$i=1,2$.
For every~$s\in[0,\abs{h_i})$ we have~$0<\abs{u_i(s)}\leq 1$ and~$\sgn\tparen{u_i(s)}=\sgn(h_i)$, hence~$-\tfrac{\diff}{\diff s}u_i(s)=\sgn(h_i)\in\partial e\tparen{u_i(s)}$.
For every~$s>\abs{h_i}$, instead,~$u_i(s)=0$, hence~$-\tfrac{\diff}{\diff s}u_i(s)=0\in\partial e(0)$.
Thus,~$-\tfrac{\diff}{\diff s}u(s)\in\partial E\tparen{u(s)}$ for every~$s\geq 0$ with~$s\neq\abs{h_1},\abs{h_2}$.
Since~$u$ is Lipschitz and~$u(0)=h$, the claim follows from the uniqueness assertion in Theorem~\ref{t:BrezisPazy}.
In particular, every trajectory reaches the minimizer~$0$ at time~$\abs{h}_\infty\leq 1$, and~$T_t\equiv 0$ on~$K_\infty$ for every~$t\geq 1$.
However, in contrast with Example~\ref{ex:InvarianceTime}, the trajectories are in general not segments: if~$0<\abs{h_2}<\abs{h_1}$, the trajectory of~$h$ moves in the direction~$-\sgn(h)$ up to time~$\abs{h_2}$, and in the direction~$-\cproj_1\sgn(h)$ afterwards.

Let us also note that~$E(h)=E(\cproj_1 h)+E(\cproj_2 h)$ for every~$h\in\R^2$, since~$e(0)=0$.
Thus, in contrast with Examples~\ref{ex:StrongDoubleInv} and~\ref{ex:InvarianceTime}, both~$H_{\cproj_1}$ and~$H_{\cproj_2}$ are doubly $T_\bullet$-invariant by Theorem~\ref{t:Invariance}, as is also apparent from~\eqref{eq:ex:EventualInvariance:1}.

\paragraph{Eventual invariance}
Let
\[
K\eqdef\set{h\in\R^2 : 0\leq h_1\leq 1\comma\ \tfrac12 h_1\leq h_2\leq 1}
\]
be the closed quadrilateral with vertices~$(0,0)$,~$(1,\tfrac12)$,~$(1,1)$, and~$(0,1)$, and note that~$K$ is closed and convex, and that~$K\subset K_\infty$; see Fig.~\ref{fig:EventualInvariance}.
We claim that~$K$ is $T_t$-invariant if and only if either~$t=0$ or~$t\geq 1$.
In particular,~$K$ is eventually $T_\bullet$-invariant, with~$c=1$, yet it is not $T_\bullet$-invariant.

Indeed, if~$t\geq 1$, then~$T_t(K)=\set{0}\subset K$ by~\eqref{eq:ex:EventualInvariance:1}, since~$K\subset K_\infty$.
If otherwise~$t\in(0,1)$, let~$k\eqdef(1,\tfrac12)\in K$.
By~\eqref{eq:ex:EventualInvariance:1},~$T_t(k)=\tparen{1-t,(\tfrac12-t)_+}$, and
\[
\tparen{\tfrac12-t}_+<\tfrac12(1-t)\comma\qquad t\in(0,1)\comma
\]
so that~$T_t(k)\notin K$.
That is, the trajectory of~$k$ leaves~$K$ immediately, moving in the direction~$-(1,1)$, reaches~$(\tfrac12,0)$ at time~$\tfrac12$, then moves along the first axis, and re-enters~$K$ only at time~$1$, when it reaches the minimizer~$0$; see Figure~\ref{fig:EventualInvariance}.
This is only possible because the trajectory of~$k$ is not a segment, cf.\ the last paragraph of Example~\ref{ex:InvarianceTime}.
\end{example}

\begin{figure}[htb!]
\centering
\begin{tikzpicture}[scale=2,>=stealth]
\draw[gray!60] (-1,-1) rectangle (1,1);
\draw[gray!60,->] (-1.15,0) -- (1.25,0) node[right,black] {\small $h_1$};
\draw[gray!60,->] (0,-1.15) -- (0,1.25) node[above,black] {\small $h_2$};
\fill[gray!25] (0,0) -- (1,0.5) -- (1,1) -- (0,1) -- cycle;
\draw (0,0) -- (1,0.5) -- (1,1) -- (0,1) -- cycle;
\node at (0.42,0.7) {$K$};
\node[gray] at (-0.72,-0.82) {\small $K_\infty$};
\draw[thick,->] (1,0.5) -- (0.52,0.02);
\draw[thick,->] (0.5,0) -- (0.03,0);
\fill (1,0.5) circle (0.8pt) node[right] {\small $k$};
\fill (0.5,0) circle (0.8pt) node[below] {\small $T_{1/2}(k)$};
\fill (0,0) circle (0.8pt) node[below left] {\small $T_1(k)=0$};
\end{tikzpicture}
\caption{The set~$K$ and the trajectory~$t\mapsto T_t(k)$ in Example~\ref{ex:EventualInvariance}.}\label{fig:EventualInvariance}
\end{figure}

In the next section we will make use of the following lemma, which we state in full generality as it is of independent interest.

\begin{lemma}\label{l:SOT}
Let~$T_\bullet\colon C\to C$ be a strongly continuous non-expansive non-linear semigroup.
The family of all strongly, resp.\ doubly, $T_\bullet$-invariant orthogonal projections is closed in the strong operator topology.
\begin{proof}
Let~$\seq{\cproj_\gamma}_\gamma$ be a net of strongly $T_\bullet$-invariant orthogonal projection operators converging to~$\cproj$ in the strong operator topology.
Since~$\norm{\cproj_\gamma}\leq 1$ for every~$\gamma$, it is readily verified that~$\cproj$ is itself an orthogonal projection operator.
Furthermore, since~$\cproj_\gamma\cproj_C=\cproj_C\cproj_\gamma$ for every~$\gamma$ and~$\cproj_C$ is continuous, for every~$f\in H$,
\[
\cproj\cproj_C(f)=\lim_\gamma\cproj_\gamma\cproj_C(f)=\lim_\gamma\cproj_C(\cproj_\gamma f)=\cproj_C(\cproj f) \fstop
\]
Since~$\cproj_\gamma T_t(f) = T_t(\cproj _\gamma f)$ for every~$f\in C$ and every~$t>0$, and since~$T_t$ is continuous (being non-expansive), for every~$f\in C$ and every~$t>0$,
\[
\cproj T_t(f) = \lim_\gamma \cproj_\gamma T_t(f) = \lim_\gamma T_t(\cproj _\gamma f)= T_t(\lim_\gamma \cproj_\gamma f ) = T_t(\cproj f)\fstop
\]
If each~$\cproj_\gamma$ is doubly $T_\bullet$-invariant, then~$\cproj_\gamma^\tperp\to\cproj^\tperp$ in the strong operator topology, and the assertion above holds as well with~$\cproj^\tperp$ in place of~$\cproj$, which concludes the proof in the doubly $T_\bullet$-invariant case.
\end{proof}
\end{lemma}

\subsection{Invariant subsets}
Let us now turn to the notion(s) of $E$-invariant set for a proper convex lower semicontinuous functional on~$L^2_\mssm$.
For each $\mssm$-measurable subset~$Y\subset X$ let~$\car_Y\colon L^2_\mssm\to L^2_\mssm$ be the multiplication operator induced by the characteristic function~$\car_Y$ of~$Y$, and set~$H_Y\eqdef L^2_\mssm(Y)$.
Note that~$\car_Y$ is the orthogonal projector onto~$H_Y$.

\begin{definition}[Invariant set]\label{d:InvariantSet}
Let~$E\colon L^2_\mssm\to (-\infty,\infty]$ be a proper convex lower semicontinuous functional, with strongly continuous non-expansive non-linear semigroup~$T_\bullet$ on~$C\eqdef\overline{\dom{E}}$, and denote by~$\cproj_C$ the convex projection onto~$C$.
A set~$Y\in\mfA$ is called
\begin{itemize}
\item \emph{$E$-invariant} if~$E(\car_Y u)\leq E(u)$ for every~$u\in L^2_\mssm$, cf.~\cite[\S3.4]{SchZim25};
\item \emph{strongly $E$-invariant} if
\begin{equation}\label{eq:d:InvariantSet:1}
\begin{aligned}
\car_Y \cproj_C(f)&=\cproj_C(\car_Y f)\comma \quad && f\in L^2_\mssm\comma
\\
\car_Y T_t(f)&=T_t(\car_Y f)\comma \quad && f\in C\comma \quad t\geq 0\fstop
\end{aligned}
\end{equation}
i.e.\ if~$\car_Y$ is a strongly $T_\bullet$-invariant orthogonal projection;

\item \emph{doubly $E$-invariant} if both~$Y$ and~$Y^\complement$ are strongly $E$-invariant.
\end{itemize}
\end{definition}

The first identity in~\eqref{eq:d:InvariantSet:1} implies that~$\car_Y C\subset C$, so that the second one is well-posed.
Since~$\car_Y$ is the orthogonal projection onto~$H_Y$ and~$\car_Y^\tperp=\car_{Y^\complement}$, a set~$Y$ is doubly $E$-invariant if and only if~$\car_Y$ is a doubly $T_\bullet$-invariant orthogonal projection, i.e.\ if and only if~$H_Y$ is a doubly $T_\bullet$-invariant subspace of~$L^2_\mssm$.

\begin{corollary}[Characterization of doubly $E$-invariant sets]\label{c:DoubleInvariance}
Let~$E\colon L^2_\mssm\to (-\infty,\infty]$ be a proper convex lower semicontinuous functional, with strongly continuous non-expansive non-linear semigroup~$T_\bullet$ on~$C\eqdef\overline{\dom{E}}$, cyclically monotone maximal monotone generator~$A\eqdef\partial E$ and strongly continuous non-expansive non-linear resolvent~$J_\bullet$.
Then, for every~$Y\in\mfA$, the following are equivalent:
\begin{enumerate}[$(i)$]
\item\label{i:c:DoubleInvariance:1} $Y$ is doubly $E$-invariant in the sense of Definition~\ref{d:InvariantSet};

\item\label{i:c:DoubleInvariance:2} for both~$Z=Y, Y^\complement$
\[
\car_Z f\in\dom{A}\quad \text{and} \quad \car_Z A(f)\subset A(\car_Z f)\comma \qquad f\in\dom{A}\semicolon
\]

\item\label{i:c:DoubleInvariance:3} for both~$Z=Y,Y^\complement$,
\[
\car_Z J_\lambda(f)=J_\lambda(\car_Z f)\comma \qquad f\in L^2_\mssm\comma \quad \lambda>0\fstop
\]
\end{enumerate}
If any (hence every) of~\ref{i:c:DoubleInvariance:1}--\ref{i:c:DoubleInvariance:3} holds for some~$Y\in\mfA$, then~$0$ is a global minimizer of~$E$; in particular,~$E(0)<\infty$.
If~$0$ is a global minimizer of~$E$, then each of~\ref{i:c:DoubleInvariance:1}--\ref{i:c:DoubleInvariance:3} is further equivalent to
\begin{enumerate}[$(i)$]\setcounter{enumi}{3}
\item\label{i:c:DoubleInvariance:4} for every~$f\in L^2_\mssm$,
\begin{equation}\label{eq:c:DoubleInvariance:1}
E(f)+E(0)=E(\car_Y f)+E(\car_{Y^\complement} f) \fstop
\end{equation}
\end{enumerate}
In particular, if~$E$ is non-negative and grounded, then~$Y$ is doubly $E$-invariant if and only if
\[
E(f)=E(\car_Y f)+E(\car_{Y^\complement} f)\comma \qquad f\in L^2_\mssm\fstop
\]
\end{corollary}

\begin{proof}
Recall that~$\car_Y$ is the orthogonal projection onto~$H_Y$, with~$\car_Y^\tperp=\car_{Y^\complement}$, and that~$C=\overline{\dom{E}}=\overline{\dom{A}}$ by~\cite[Prop.~2.11, p.~39]{Bre73}.
Thus,~\ref{i:c:DoubleInvariance:1},~\ref{i:c:DoubleInvariance:2}, and~\ref{i:c:DoubleInvariance:3} are respectively Theorem~\ref{t:Invariance}\ref{i:t:Invariance:1},~\ref{i:t:Invariance:2}, and~\ref{i:t:Invariance:3} with~$H=L^2_\mssm$ and~$\cproj=\car_Y$, and their equivalence is the first part of Theorem~\ref{t:Invariance}.

If~\ref{i:c:DoubleInvariance:3} holds for some~$Y\in\mfA$, then~$J_\lambda(0)=0$ for every~$\lambda>0$ by~\eqref{eq:t:Invariance:0b}, that is~$0\in A(0)=\partial E(0)$, and therefore~$0$ is a global minimizer of~$E$ by~\eqref{eq:ConvexSubDiff}.
Since~$E$ is proper, we have~$E(0)\leq E(g)<\infty$ for every~$g\in\dom{E}$.

Finally, assume that~$0$ is a global minimizer of~$E$, so that~$E(0)\in\R$.
Then, the grounding~$\ground{E}=E-E(0)\colon L^2_\mssm\to[0,\infty]$ is a non-negative grounded proper convex lower semicontinuous functional with~$\partial\ground{E}=\partial E=A$, and thus with the same semigroup and resolvent as~$E$.
Applying Theorem~\ref{t:Invariance}\ref{i:t:Invariance:3}$\iff$\ref{i:t:Invariance:4} to~$\ground{E}$ shows that~\ref{i:c:DoubleInvariance:3} holds if and only if~$\ground{E}(f)=\ground{E}(\car_Y f)+\ground{E}(\car_{Y^\complement} f)$ for every~$f\in L^2_\mssm$, which is~\eqref{eq:c:DoubleInvariance:1} since~$E(0)<\infty$.
\end{proof}

\begin{remark}[On double invariance]\label{r:DoubleInvariance}
\begin{enumerate*}[$(a)$]
\item\label{i:r:DoubleInvariance:1} The minimality assumption in the second part of Corollary~\ref{c:DoubleInvariance} may not be dispensed with.
On the one hand,~\eqref{eq:c:DoubleInvariance:1} with~$Y=\emp$ is trivially satisfied by every~$E$, whereas~$\emp$ is doubly $E$-invariant only if~$0$ is a global minimizer of~$E$.
On the other hand, the identity $E(f)=E(\car_Y f)+E(\car_{Y^\complement} f)$ may not replace~\eqref{eq:c:DoubleInvariance:1} unless~$E$ is grounded, even if~$0$ is a global minimizer of~$E$.
For instance, if~$E\equiv 1$, then~$\partial E\equiv 0$, so that~$T_t=J_\lambda=\car_H$ for every~$t\geq 0$ and every~$\lambda>0$, and every~$Y\in\mfA$ is doubly $E$-invariant, whereas the above identity fails for every~$Y\in\mfA$; cf.\ Remark~\ref{r:Invariance}\ref{i:r:Invariance:3}.

\item\label{i:r:DoubleInvariance:2} Every doubly $E$-invariant set~$Y$ is $E$-invariant.
Indeed,~$0$ is a global minimizer of~$E$ by Corollary~\ref{c:DoubleInvariance}, thus~$E(\car_{Y^\complement}u)\geq E(0)$ and~$E(0)<\infty$, and~\eqref{eq:c:DoubleInvariance:1} yields~$E(\car_Y u)\leq E(u)$ for every~$u\in L^2_\mssm$.

\item Double $E$-invariance is strictly stronger than $E$-invariance.
Indeed, let~$X\eqdef\set{1,2}$ be the discrete space with two points, set~$\mssm\eqdef\delta_1+\delta_2$, so that~$L^2_\mssm\cong\R^2$ with~$\car_{\set{i}}=\cproj_i$ for~$i=1,2$, and let~$E$ be the functional in Example~\ref{ex:StrongDoubleInv}, which is non-negative and grounded.
By~\eqref{eq:ex:StrongDoubleInv:1}, both~$\set{1}$ and~$\set{2}$ are $E$-invariant.
By~\eqref{eq:ex:StrongDoubleInv:2} and Corollary~\ref{c:DoubleInvariance}\ref{i:c:DoubleInvariance:1}$\iff$\ref{i:c:DoubleInvariance:4}, neither of them is doubly $E$-invariant.
\end{enumerate*}
\end{remark}

\begin{remark}[Comparison with~\cite{SchZim25}]
The definition of $E$-invariant set has been recently considered in~\cite[\S3.4]{SchZim25}.
In light of Proposition~\ref{p:EquivConvex}, if a set~$Y$ is $E$-invariant, then~$H_Y$ is $T_\bullet$-invariant.
Analogously, by Definition~\ref{d:InvariantSet}, a set~$Y$ is doubly $E$-invariant if and only if~$\car_Y$ is doubly $T_\bullet$-invariant.
Corollary~\ref{c:DoubleInvariance}, that is, Theorem~\ref{t:Invariance} specialized to the setting of~$H=L^2_\mssm$, extends the analogous characterization for (bi-linear) Dirichlet forms, see, e.g.,~\cite[\S1.6, pp.~53f.]{FukOshTak11}.
In particular, the stronger notion of double $E$-invariance (rather than the mere $E$-invariance) provides a much better analogy with the classical results for bilinear forms, where the natural notion of invariance is precisely the identity~\eqref{eq:c:DoubleInvariance:1} with~$E(0)=0$.
\end{remark}

\subsubsection{Restriction of functionals to invariant subsets}
Again as in the bi-linear case, it is possible to restrict a proper convex lower semicontinuous functional to an invariant set.

\begin{proposition}[Restrictions]\label{p:Restrictions}
Let~$A$ be a maximal monotone operator on~$H$, with associated semigroup~$T_\bullet$ and resolvent~$J_\bullet$, and let~$\cproj\colon H\to H$ be an orthogonal projection operator.
Suppose that some (hence every) of the equivalent conditions~\ref{i:t:Invariance:1}--\ref{i:t:Invariance:3} in Theorem~\ref{t:Invariance} is satisfied.
Then,
\begin{enumerate}[$(i)$]
\item\label{i:p:Restrictions:1} $A+\partial\mbfI_{H_\cproj}$ is a maximal monotone operator with~$\overline{\dom{A+\partial\mbfI_{H_\cproj}}}=C\cap H_\cproj$, where~$C\eqdef\overline{\dom{A}}$.
Its resolvent is~$\cproj\circ J_\bullet$, including~$\lambda=0$, since~$\cproj\circ\cproj_C=\cproj_{C\cap H_\cproj}$ by Remark~\ref{r:Invariance}\ref{i:r:Invariance:1}.
Its semigroup is the restriction of~$T_\bullet$ (equivalently, of~$\cproj\circ T_\bullet$) to~$C\cap H_\cproj$.

\item\label{i:p:Restrictions:2} if~$A$ is additionally cyclically monotone with~$A=\partial E$ for some proper convex lower semicontinuous functional~$E\colon H\to (-\infty,\infty]$, then~$A+\partial\mbfI_{H_\cproj}=\partial \paren{E+\mbfI_{H_\cproj}}$ for the proper convex lower semicontinuous functional~$E+\mbfI_{H_\cproj}\colon H\to (-\infty,\infty]$.
\end{enumerate}

Assume now that~$H=L^2_\mssm$.
Then,
\begin{enumerate}[$(i)$]\setcounter{enumi}{2}
\item\label{i:p:Restrictions:3} if~$E$ is additionally $\cproj_\tleq$-non-expansive and~$\cproj\colon L^2_\mssm \to L^2_\mssm$ is additionally order-preserving, then~$E+\mbfI_{H_\cproj}$ is $\cproj_\tleq$-non-expansive;

\item\label{i:p:Restrictions:4} if~$E$ is additionally $\cproj_\alpha$-non-expansive and~$\cproj\colon L^2_\mssm \to L^2_\mssm$ is additionally $L^\infty$-non-expansive, then~$E+\mbfI_{H_\cproj}$ is $\cproj_\alpha$-non-expansive.
\end{enumerate}

In particular, if~$E$ is a Dirichlet functional and~$\cproj$ is both order-preserving and $L^\infty$-non-expansive, then~$E+\mbfI_{H_\cproj}$ is a Dirichlet functional.
Finally, if both~$\cproj$ and~$\cproj^\tperp$ are order-preserving, 
and if~$E$ is a Dirichlet functional, then so are both~$E+\mbfI_{H_\cproj}$ and~$E+\mbfI_{H_{\cproj^\tperp}}$.
\end{proposition}

Before proving Proposition~\ref{p:Restrictions}, let us note that the $L^2_\mssm$-orthogonal projections~$\cproj$ such that \emph{both}~$\cproj$ and~$\cproj^\tperp$ are order-preserving are in one-to-one correspondence with the measurable subsets of~$X$ regarded up to $\mssm$-equivalence, via~$Y\mapsto\car_Y$.
Requiring that~$\cproj$ be only order-preserving is not sufficient, see Remark~\ref{r:Weis} below.

\begin{proposition}\label{p:Weis}
Let~$\cproj\colon L^2_\mssm\to L^2_\mssm$ be an $L^2_\mssm$-orthogonal projection, and set~$\cproj^\tperp\eqdef\car_H-\cproj$ with~$H\eqdef L^2_\mssm$.
Then, the following are equivalent:
\begin{enumerate}[$(i)$]
\item\label{i:p:Weis:1} $\cproj=\car_Y$ for some~$Y\in\mfA$;
\item\label{i:p:Weis:2} both~$\cproj$ and~$\cproj^\tperp$ are order-preserving;
\item\label{i:p:Weis:3} $0\leq \cproj u\leq u$ for every~$u\in L^2_\mssm\cap L^\infty_\mssm$ with~$u\geq 0$.
\end{enumerate}
If this is the case, then~$Y$ is unique up to $\mssm$-negligible sets,~$\cproj^\tperp=\car_{Y^\complement}$, and both~$\cproj$ and~$\cproj^\tperp$ are sub-Markovian; in particular, they are $L^\infty$-non-expansive.
\end{proposition}

\begin{proof}
\paragraph{\ref{i:p:Weis:1}$\implies$\ref{i:p:Weis:2}}
If~$\cproj=\car_Y$, then~$\cproj^\tperp=\car_{Y^\complement}$, and both are multiplication operators by non-negative functions, hence order-preserving.

\paragraph{\ref{i:p:Weis:2}$\implies$\ref{i:p:Weis:3}}
Since~$\cproj$ and~$\cproj^\tperp$ are linear, we have~$\cproj 0=\cproj^\tperp 0=0$.
Thus, for every~$u\geq 0$, we have~$\cproj u\geq \cproj 0=0$ and~$u-\cproj u=\cproj^\tperp u\geq\cproj^\tperp 0=0$.

\paragraph{\ref{i:p:Weis:3}$\implies$\ref{i:p:Weis:1}}
We divide the proof into several steps.

\paragraph{Step~1: \ref{i:p:Weis:3} holds for every non-negative~$u\in L^2_\mssm$}
Fix~$u\in L^2_\mssm$ with~$u\geq 0$, and set~$u_n\eqdef u\wedge n\in L^2_\mssm\cap L^\infty_\mssm$ for~$n\in\N$.
By Dominated Convergence with dominating function~$u$, we have~$L^2_\mssm$-$\nlim u_n=u$, and therefore~$L^2_\mssm$-$\nlim \cproj u_n=\cproj u$ by continuity of~$\cproj$.
Passing to an $\mssm$-a.e.\ converging subsequence in the inequalities~$0\leq\cproj u_n\leq u_n$ yields~$0\leq\cproj u\leq u$.

\paragraph{Step~2: $\cproj$ commutes with~$\car_A$ for every~$A\in\mfA$}
Fix~$A\in\mfA$, and let~$u\in L^2_\mssm$ satisfy~$u\geq 0$ and~$u=\car_A u$.
By Step~1 we have~$0\leq\cproj u\leq u$, and therefore~$\car_{A^\complement}\cproj u=0$.
Applying this to the positive and negative parts of~$\car_A u$ for an arbitrary~$u\in L^2_\mssm$, and since~$\cproj$ is linear, we conclude that~$\car_{A^\complement}\cproj\car_A=0$.
Exchanging the roles of~$A$ and~$A^\complement$ further gives~$\car_A\cproj\car_{A^\complement}=0$.
As a consequence,
\[
\car_A\cproj=\car_A\cproj\car_A+\car_A\cproj\car_{A^\complement}=\car_A\cproj\car_A=\car_A\cproj\car_A+\car_{A^\complement}\cproj\car_A=\cproj\car_A\fstop
\]

\paragraph{Step~3: $\cproj$ is a multiplication operator}
Since~$\mssm$ is $\sigma$-finite, there exists a partition~$\seq{X_n}_n\subset\mfA$ of~$X$ with~$\mssm X_n <\infty$ for every~$n$.
Set~$h_n\eqdef\cproj\car_{X_n}\in L^2_\mssm$.
By Step~1,~$0\leq h_n\leq\car_{X_n}$, and we may and will choose a $[0,1]$-valued representative of~$h_n$ vanishing outside~$X_n$.
Define~$h\eqdef\sum_n h_n$ pointwise, and note that~$h$ is measurable,~$[0,1]$-valued, and satisfies~$h\car_{X_n}=h_n$ for every~$n$.
For every~$n$ and every~$A\in\mfA$ with~$A\subset X_n$, by Step~2,
\[
\cproj\car_A=\cproj(\car_A\car_{X_n})=\car_A\,\cproj\car_{X_n}=\car_A h_n=h\car_A\fstop
\]
By linearity,~$\cproj u=hu$ for every~$u$ in the linear span~$\mcS$ of~$\set{\car_A : A\in\mfA\comma A\subset X_n \text{ for some } n\in\N}$.
The space~$\mcS$ is dense in~$L^2_\mssm$, since simple functions supported on sets of finite measure are dense in~$L^2_\mssm$, and since, for every~$B\in\mfA$ with~$\mssm(B)<\infty$, we have~$L^2_\mssm$-$\lim_N\sum_{n\leq N}\car_{B\cap X_n}=\car_B$ by Dominated Convergence.
Since both~$\cproj$ and the multiplication by~$h$ are bounded operators on~$L^2_\mssm$, we conclude that~$\cproj u=hu$ for every~$u\in L^2_\mssm$.

\paragraph{Step~4: conclusion}
Since~$\cproj$ is idempotent, for every~$n$,
\[
h\car_{X_n}=\cproj\car_{X_n}=\cproj^2\car_{X_n}=\cproj(h\car_{X_n})=h^2\car_{X_n}\comma
\]
hence~$h^2=h$ $\mssm$-a.e., that is~$h\in\set{0,1}$ $\mssm$-a.e.
Setting~$Y\eqdef\set{h=1}\in\mfA$, we have~$h=\car_Y$ $\mssm$-a.e., and thus~$\cproj=\car_Y$.

\paragraph{Uniqueness and sub-Markovianity}
If~$\car_Y=\car_{Y'}$ as operators on~$L^2_\mssm$, then~$\car_Y\car_{X_n}=\car_{Y'}\car_{X_n}$ in~$L^2_\mssm$, hence~$\mssm\tparen{(Y\triangle Y')\cap X_n}=0$ for every~$n$, and so~$Y$ and~$Y'$ are $\mssm$-equivalent.
Finally,~$\car_Y$ and~$\car_{Y^\complement}$ are order-preserving, and~$\abs{\car_Y u-\car_Y v}=\car_Y\abs{u-v}\leq\abs{u-v}$ $\mssm$-a.e.\ for every~$u,v\in L^2_\mssm$, and analogously for~$\car_{Y^\complement}$; hence both are sub-Markovian.
\end{proof}

\begin{remark}\label{r:Weis}
\begin{enumerate*}[$(a)$]
\item\label{i:r:Weis:1} The orthogonality of~$\cproj$ is not used in the proof of the implication~\ref{i:p:Weis:3}$\implies$\ref{i:p:Weis:1}, which thus holds for every bounded linear idempotent operator~$\cproj$ on~$L^2_\mssm$.

\item\label{i:r:Weis:2} The assumption on~$\cproj^\tperp$ in~\ref{i:p:Weis:2} may not be dispensed with, not even if~$\cproj$ is additionally $L^\infty$-non-expansive.
For instance, if~$\mssm$ is a finite measure and~$\dim L^2_\mssm\geq 2$, then the orthogonal projection~$\cproj u\eqdef (\mssm X)^{-1}\int u\diff\mssm \car$ onto the constant functions is order-preserving and $L^\infty$-non-expansive.
However, it is not of the form~$\car_Y$, since~$\cproj\car=\car$ would force~$Y=X$ up to $\mssm$-negligible sets, and thus~$\cproj=\car_H$, which has rank~$\dim L^2_\mssm\geq 2$, while~$\cproj$ has rank one.
Furthermore, an order-preserving orthogonal projection need not be $L^\infty$-non-expansive.
For instance, on~$X\eqdef\set{1,2}$ with~$\mssm\eqdef\delta_1+\delta_2$, the orthogonal projection~$\cproj u\eqdef\tfrac{1}{5}\tparen{2u(1)+u(2)}\,(2,1)$ onto the span of~$(2,1)$ is order-preserving, but~$\cproj\car=(\tfrac{6}{5},\tfrac{3}{5})$.

\item\label{i:r:Weis:3} In particular, order-preserving doubly $T_\bullet$-invariant orthogonal projections need not be of the form~$\car_Y$.
Indeed, for~$E\eqdef\norm{\emparg}^2_{L^2_\mssm}$, every orthogonal projection~$\cproj$ satisfies~$E(f)=E(\cproj f)+E(\cproj^\tperp f)$ for every~$f\in L^2_\mssm$, and is therefore doubly $T_\bullet$-invariant by Theorem~\ref{t:Invariance}; this applies to the projection onto the constant functions in~\ref{i:r:Weis:2}.
\end{enumerate*}
\end{remark}

\begin{proof}[Proof of Proposition~\ref{p:Restrictions}]
Throughout the proof~$\aproj$ is a placeholder for either~$\cproj$ or~$\cproj^\tperp$.

\paragraph{Proof of~\ref{i:p:Restrictions:1}}
Since~$H_\aproj$ is a closed subspace, $\partial\mbfI_{H_\aproj}(h)=H_{\aproj^\tperp}$ for~$h\in H_\aproj$, and~$\partial\mbfI_{H_\aproj}(h)=\emp$ otherwise.
Thus~$A+\partial\mbfI_{H_\aproj}$ is monotone with domain~$\dom{A}\cap H_\aproj$.
Fix~$\lambda>0$ and~$g\in H$, and set~$f\eqdef J_\lambda(g)$ and~$a\eqdef\lambda^{-1}(g-f)\in A(f)$.
By Theorem~\ref{t:Invariance}\ref{i:t:Invariance:2}, $\aproj f\in\dom{A}\cap H_\aproj$ and~$\aproj a\in A(\aproj f)$, and
\[
g-\aproj f-\lambda\aproj a=\aproj(g-f-\lambda a)+\aproj^\tperp g=\aproj^\tperp g\in H_{\aproj^\tperp}=\lambda\,\partial\mbfI_{H_\aproj}(\aproj f)\fstop
\]
Hence~$g\in\aproj f+\lambda(A+\partial\mbfI_{H_\aproj})(\aproj f)$.
So~$\car_H+\lambda(A+\partial\mbfI_{H_\aproj})$ is surjective, $A+\partial\mbfI_{H_\aproj}$ is maximal monotone by Minty's Theorem, and its resolvent is~$\aproj\circ J_\lambda$.
Since~$\aproj$ is continuous and~$\aproj\dom{A}\subset\dom{A}\cap H_\aproj$, we have~$\overline{\dom{A}\cap H_\aproj}=C\cap H_\aproj$.
Finally, for~$f\in C\cap H_\aproj$, $(\aproj J_{t/n})^n(f)=\aproj (J_{t/n})^n(f)\to\aproj T_t(f)=T_t(\aproj f)=T_t(f)$ by Theorem~\ref{t:SemigroupResolvent}\ref{i:t:SemigroupResolvent:2}.

\paragraph{Proof of~\ref{i:p:Restrictions:2}}
The functional~$E+\mbfI_{H_\aproj}$ is convex and lower semicontinuous, and it is proper, since~$\emp\neq\dom{A+\partial\mbfI_{H_\aproj}}\subset\dom{E}\cap H_\aproj$ by maximal monotonicity of~$A+\partial\mbfI_{H_\aproj}$; see~\ref{i:p:Restrictions:1}.
The inclusion~$\partial E+\partial\mbfI_{H_\aproj}\subset\partial\tparen{E+\mbfI_{H_\aproj}}$ is immediate from~\eqref{eq:ConvexSubDiff}.
Since the left-hand side is maximal monotone by~\ref{i:p:Restrictions:1}, and the right-hand side is monotone by Theorem~\ref{t:MaxMonSubDiff}\ref{i:t:MaxMonSubDiff:1}, the two coincide.

\bigskip

Throughout the rest of the proof, assume~$H=L^2_\mssm$.

\paragraph{Proof of~\ref{i:p:Restrictions:3}}
By~\ref{i:p:Restrictions:1} and~\ref{i:p:Restrictions:2}, the resolvent of the proper convex lower semicontinuous functional~$E+\mbfI_{H_\aproj}$ is~$\aproj\circ J_\bullet$.
Since~$E$ is $\cproj_\tleq$-non-expansive, Proposition~\ref{p:EquivConvex}\ref{i:p:EquivConvex:2}$\implies$\ref{i:p:EquivConvex:1}, applied to the double~$E^\opl{2}$ with~$C_\tleq$ in place of~$C$, gives~$J_\lambda^\opl{2}(C_\tleq)\subset C_\tleq$ for every~$\lambda>0$; that is,~$J_\lambda$ is order-preserving by Proposition~\ref{p:SubMarkovianConvex}\ref{i:p:SubMarkovianConvex:1}.
Since~$\aproj$ is order-preserving by assumption, so is~$\aproj\circ J_\lambda$, i.e.~$(\aproj\circ J_\lambda)^\opl{2}(C_\tleq)\subset C_\tleq$, for every~$\lambda>0$.
Applying Proposition~\ref{p:EquivConvex}\ref{i:p:EquivConvex:1}$\implies$\ref{i:p:EquivConvex:2}, to the double of~$E+\mbfI_{H_\aproj}$ with~$C_\tleq$ in place of~$C$, we conclude that~$E+\mbfI_{H_\aproj}$ is $\cproj_\tleq$-non-expansive.
(Recall that the equivalence of~\ref{i:p:EquivConvex:1} and~\ref{i:p:EquivConvex:2} in Proposition~\ref{p:EquivConvex} does not require~\eqref{eq:p:EquivConvex:0}.)

\paragraph{Proof of~\ref{i:p:Restrictions:4}}
Arguing as in the proof of~\ref{i:p:Restrictions:3}, with~$C_\alpha$ in place of~$C_\tleq$ for every~$\alpha\geq 0$ and Proposition~\ref{p:SubMarkovianConvex}\ref{i:p:SubMarkovianConvex:2} in place of Proposition~\ref{p:SubMarkovianConvex}\ref{i:p:SubMarkovianConvex:1}, we see that~$J_\lambda$ is $L^\infty$-non-expansive for every~$\lambda>0$.
Since~$\aproj$ is linear and $L^\infty$-non-expansive, we have~$\norm{\aproj J_\lambda(u)-\aproj J_\lambda(v)}_{L^\infty_\mssm}\leq\norm{J_\lambda(u)-J_\lambda(v)}_{L^\infty_\mssm}\leq\norm{u-v}_{L^\infty_\mssm}$ for every~$u,v\in L^2_\mssm$, i.e.~$\aproj\circ J_\lambda$ is $L^\infty$-non-expansive, and the conclusion follows as in the proof of~\ref{i:p:Restrictions:3}.

\paragraph{Conclusion} In order to prove the last assertion, note that all assertions hold analogously with~$\cproj^\tperp$ in place of~$\cproj$.
If both~$\cproj$ and~$\cproj^\tperp$ are order-preserving, then~$\cproj=\car_Y$ for some~$Y\in\mfA$ by Proposition~\ref{p:Weis}.
Thus, both~$\cproj$ and~$\cproj^\tperp$ are additionally $L^\infty$-non-expansive, and therefore, if~$E$ is a Dirichlet functional, so are both~$E+\mbfI_{H_\cproj}$ and~$E+\mbfI_{H_{\cproj^\tperp}}$.
\end{proof}

\subsubsection{Irreducibility and families of invariant subsets}\label{sss:InvAlgebra}
We consider the notion of \emph{irreducibility} for proper convex lower semicontinuous functionals on~$L^2_\mssm$.
Say that $Y\subset X$ is \emph{$\mssm$-trivial} if it is either $\mssm$-negligible or $\mssm$-conegligible.

\begin{definition}[Irreducibility]
A proper convex lower semicontinuous functional~$E\colon L^2_\mssm\to (-\infty,\infty]$ (the corresponding semigroup, generator, resolvent) is called
\begin{itemize}
\item \emph{irreducible} if every $E$-invariant set is $\mssm$-trivial, see~\cite[\S3.4]{SchZim25};
\item \emph{strongly irreducible} if every strongly $E$-invariant set is $\mssm$-trivial;
\item \emph{doubly irreducible} if every doubly $E$-invariant set is $\mssm$-trivial.
\end{itemize}
\end{definition}

By Remark~\ref{r:DoubleInvariance}\ref{i:r:DoubleInvariance:2}, every doubly $E$-invariant set is $E$-invariant, and therefore every irreducible~$E$ is doubly irreducible.

\medskip

We conclude our treatment of invariance by considering the family of all invariant sets of a proper convex lower semicontinuous functional~$E\colon L^2_\mssm\to (-\infty,\infty]$, regarded up to $\mssm$-equivalence.
Following the notation in~\cite{SchZim25}, we denote by~$\mfB_\inv$ the family of all $E$-invariant subsets of~$X$. 
Since~$E(\car_{Y\cap Z}u)=E(\car_Y\car_Zu)\leq E(\car_Zu)\leq E(u)$, the family~$\mfB_\inv$ is closed under finite intersections.
We call~$\mfA_\inv\eqdef\sigma(\mfB_\inv)$ the \emph{$E$-invariant $\sigma$-algebra}.
Under the standing assumptions of~\cite[\S3.1]{SchZim25}, it coincides with
\[
\set{Y\in\mfA : Y\in\mfB_\inv\ \text{or}\ Y^\complement\in\mfB_\inv}
\]
by~\cite[Lem.~3.14]{SchZim25}.

Analogously, we denote by~$\mfB_\dinv$ the family of all doubly $E$-invariant subsets of~$X$.

\begin{theorem}\label{t:BDinvSAlg}
Let~$E\colon L^2_\mssm \to (-\infty,\infty]$ be a proper convex lower semicontinuous functional, and let~$\mfB_\dinv$ be the family of all doubly $E$-invariant sets in the sense of Definition~\ref{d:InvariantSet}.
Then,~$\mfB_\dinv$ is non-empty if and only if~$0$ is a global minimizer of~$E$, in which case~$\emp,X\in\mfB_\dinv$.
When this is the case,~$\mfB_\dinv$ is a $\sigma$-algebra, and, for every~$f\in L^2_\mssm$, the set function
\[
\nu_f\colon \mfB_\dinv\longrightarrow[0,\infty]\comma \qquad \nu_f Y\eqdef E(\car_Y f)-E(0)\comma
\]
is a measure on~$(X,\mfB_\dinv)$.
In particular, if~$E$ is grounded, then~$Y\mapsto E(\car_Y f)$ is countably additive on~$\mfB_\dinv$ for every~$f\in L^2_\mssm$.
\end{theorem}

\begin{proof}
Let~$A\eqdef\partial E$, and let~$J_\bullet$ be the resolvent of~$E$.

\paragraph{Step~1: non-triviality of~$\mfB_\dinv$}
If~$\mfB_\dinv\neq\emp$, then~$0$ is a global minimizer of~$E$ by Corollary~\ref{c:DoubleInvariance}.

Conversely, assume that~$0$ is a global minimizer of~$E$.
Then~$E(0)<\infty$, since~$E$ is proper, and~$0\in\partial E(0)=A(0)$ by~\eqref{eq:ConvexSubDiff}, whence~$J_\lambda(0)=0$ for every~$\lambda>0$ by~\eqref{eq:YosidaRegularization}.
For~$Y=\emp$, condition~\ref{i:c:DoubleInvariance:3} in Corollary~\ref{c:DoubleInvariance} reads~$0=J_\lambda(0)$ for~$Z=\emp$, and~$J_\lambda(f)=J_\lambda(f)$ for~$Z=X$, and it is thus satisfied.
Hence,~$\emp\in\mfB_\dinv$ by Corollary~\ref{c:DoubleInvariance}\ref{i:c:DoubleInvariance:3}$\implies$\ref{i:c:DoubleInvariance:1}, and~$X\in\mfB_\dinv$ since Definition~\ref{d:InvariantSet} is symmetric in~$Y$ and~$Y^\complement$.

\paragraph{Step~2: $\sigma$-algebra property of~$\mfB_\dinv$}
Assume now that~$0$ is a global minimizer of~$E$, so that~$J_\lambda(0)=0$ for every~$\lambda>0$ by Step~1.
Since Definition~\ref{d:InvariantSet} is symmetric in~$Y$ and~$Y^\complement$, the family~$\mfB_\dinv$ is closed under complementation.
Set
\[
\mfS\eqdef \set{Y\in\mfA : J_\lambda(\car_Y f) = \car_Y J_\lambda(f) \quad \text{for all $\lambda>0$ and all $f\in L^2_\mssm$} } \comma
\]
so that~$\mfB_\dinv= \set{Y\in\mfA: Y,Y^\complement\in\mfS}$ by Corollary~\ref{c:DoubleInvariance}\ref{i:c:DoubleInvariance:1}$\iff$\ref{i:c:DoubleInvariance:3}.
In particular, for every~$Y\in\mfB_\dinv$, every~$\lambda>0$, and every~$f\in L^2_\mssm$, we have~$J_\lambda(f)=\car_YJ_\lambda(f)+\car_{Y^\complement}J_\lambda(f)=J_\lambda(\car_Y f)+ J_\lambda(\car_{Y^\complement}f)$, that is,
\begin{equation}\label{eq:t:BDinvSAlg:1}
J_\lambda=J_\lambda\circ \car_Y + J_\lambda\circ \car_{Y^\complement} \quad \text{on } L^2_\mssm\comma \qquad \lambda>0\comma \qquad Y\in\mfB_\dinv\fstop
\end{equation}

\paragraph{Claim~1: if~$Y\in\mfS$, then~$\car_Y J_\lambda(\car_{Y^\complement} f)=0$ for every~$f\in L^2_\mssm$ and every~$\lambda>0$}
Since~$J_\lambda(0)=0$, we have~$\car_Y J_\lambda(\car_{Y^\complement} f) = J_\lambda(\car_Y \car_{Y^\complement} f) = J_\lambda(0)=0$.

\paragraph{Claim~2: $\mfS$ is closed under finite intersections}
For~$Y_1,Y_2\in\mfS$, every~$\lambda>0$, and every~$f\in L^2_\mssm$,
\[
\car_{Y_1\cap Y_2} J_\lambda(f) = \car_{Y_1}\car_{Y_2} J_\lambda(f) = \car_{Y_1} J_\lambda(\car_{Y_2} f) = J_\lambda(\car_{Y_1} \car_{Y_2} f)= J_\lambda(\car_{Y_1 \cap Y_2} f) \fstop
\]

\paragraph{Claim~3: $\mfB_\dinv$ is closed under finite unions}
Let~$Y_1,Y_2\in\mfB_\dinv$ and set~$Z\eqdef Y_1\cup Y_2$ for simplicity of notation.
Since~$Z^\complement = Y_1^\complement\cap Y_2^\complement$, we have~$Z^\complement\in\mfS$ by Claim~2.
Thus, it remains to show that~$Z\in\mfS$.
Fix~$f\in L^2_\mssm$ and~$\lambda>0$.
Applying~\eqref{eq:t:BDinvSAlg:1} with~$Y_1$ in place of~$Y$ and~$\car_Z f$ in place of~$f$, and using that~$Y_1\cap Z = Y_1$ and~$Y_1^\complement\cap Z = Y_2\setminus Y_1$,
\begin{equation}\label{eq:t:BDinvSAlg:4}
J_\lambda(\car_Z f) = J_\lambda(\car_{Y_1}f) + J_\lambda(\car_{Y_2\setminus Y_1} f) \fstop
\end{equation}
Applying~\eqref{eq:t:BDinvSAlg:1} with~$Y_1$ in place of~$Y$ and multiplying by~$\car_Z$ on both sides,
\begin{equation}\label{eq:t:BDinvSAlg:5}
\car_Z J_\lambda(f) = \car_Z J_\lambda(\car_{Y_1} f) + \car_Z J_\lambda(\car_{Y_1^\complement} f) \fstop
\end{equation}
By Claim~1 applied with~$Y=Y_1^\complement\in\mfS$, the function~$J_\lambda(\car_{Y_1} f)$ in the first summand on the right-hand side of~\eqref{eq:t:BDinvSAlg:5} is supported on~$Y_1$, thus~$\car_Z$ acts on it as the identity.
By Claim~1 applied with~$Y=Y_1\in\mfS$, the function~$J_\lambda(\car_{Y_1^\complement} f)$ in the second summand on the right-hand side of~\eqref{eq:t:BDinvSAlg:5} is supported on~$Y_1^\complement$, thus~$\car_Z=\car_{Y_1}+\car_{Y_2\setminus Y_1}$ acts on it as~$\car_{Y_2}$, and therefore, since~$Y_2\in\mfS$,
\[
\car_Z J_\lambda(\car_{Y_1^\complement} f) = \car_{Y_2} J_\lambda(\car_{Y_1^\complement} f) =  J_\lambda(\car_{Y_2} \car_{Y_1^\complement} f) = J_\lambda(\car_{Y_2\setminus Y_1} f) \fstop
\]
Substituting both the last assertions into~\eqref{eq:t:BDinvSAlg:5} gives
\[
\car_Z J_\lambda(f) = J_\lambda(\car_{Y_1}f) + J_\lambda(\car_{Y_2\setminus Y_1} f) \comma
\]
and comparing the latter with~\eqref{eq:t:BDinvSAlg:4} proves that~$\car_Z J_\lambda(f) = J_\lambda(\car_Z f)$.
This shows that~$Z\in\mfS$, which proves the claim.

\paragraph{Claim~4: $\mfB_\dinv$ is an algebra}
Since~$\emp,X\in\mfB_\dinv$ by Step~1, and since~$\mfB_\dinv$ is closed under complementation and finite unions,~$\mfB_\dinv$ is an algebra of sets.

\paragraph{Conclusion: $\mfB_\dinv$ is a $\sigma$-algebra}
An algebra of sets which is closed under countable decreasing intersections is, by complementation, closed under countable increasing unions, and every countable union is an increasing union of finite unions.
Thus, it suffices to show that~$\mfB_\dinv$ is closed under countable \emph{decreasing} intersections.
To this end, let~$\seq{Y_n}_n\subset \mfB_\dinv$ with~$Y_{n+1}\subset Y_n$, and set~$Y\eqdef\bigcap_n Y_n$.
For every~$f\in L^2_\mssm$, by Dominated Convergence with dominating function~$\car_{Y_1} \abs{f}\in L^2_\mssm$, we have~$L^2_\mssm$-$\nlim \car_{Y_n}f=\car_Y f$; that is,~$\car_{Y_n}\colon L^2_\mssm\to L^2_\mssm$ converges to~$\car_{Y}\colon L^2_\mssm\to L^2_\mssm$ in the strong operator topology as~$n\to\infty$.
Since each~$\car_{Y_n}$ is a doubly $T_\bullet$-invariant orthogonal projection by Definition~\ref{d:InvariantSet}, so is~$\car_Y$ by Lemma~\ref{l:SOT}, that is,~$Y\in\mfB_\dinv$.

\paragraph{Step~3: countable additivity}
Assume that~$0$ is a global minimizer of~$E$.
Then~$E(0)<\infty$ by Step~1, and thus~$\nu_f$ is well-defined and $[0,\infty]$-valued for every~$f\in L^2_\mssm$, with~$\nu_f(\emp)=0$.
Fix~$f\in L^2_\mssm$, and let~$B,Y\in\mfB_\dinv$, so that~$B\cap Y,B\setminus Y\in\mfB_\dinv$ by Step~2.
Applying Corollary~\ref{c:DoubleInvariance}\ref{i:c:DoubleInvariance:1}$\implies$\ref{i:c:DoubleInvariance:4} to~$Y$, with~$\car_B f$ in place of~$f$, and subtracting~$2E(0)<\infty$ from both sides of~\eqref{eq:c:DoubleInvariance:1}, we obtain
\begin{equation}\label{eq:t:BDinvSAlg:6}
\nu_f B=\nu_f(B\cap Y)+\nu_f(B\setminus Y)\comma \qquad B,Y\in\mfB_\dinv\fstop
\end{equation}

Now, let~$\seq{Y_k}_k\subset \mfB_\dinv$ consist of pairwise disjoint sets, and set~$Z_n\eqdef \bigcup_{k\leq n} Y_k$ and~$Z\eqdef \bigcup_kY_k$, so that~$Z_n,Z\in\mfB_\dinv$ by Step~2.
Applying~\eqref{eq:t:BDinvSAlg:6} with~$B=Z_n$ and~$Y=Y_n$, and arguing by induction on~$n$, we have~$\nu_f Z_n=\sum_{k\leq n}\nu_f Y_k$.
Applying~\eqref{eq:t:BDinvSAlg:6} with~$B=Z$ and~$Y=Z_n$, and since~$\nu_f\geq 0$, we further have
\[
\sum_{k\leq n}\nu_f Y_k =\nu_f Z_n \leq\nu_f Z_n +\nu_f(Z\setminus Z_n)=\nu_f Z \comma\qquad n\in\N\comma
\]
whence~$\sum_k\nu_f Y_k \leq\nu_f Z$.
Conversely,~$L^2_\mssm$-$\nlim\car_{Z_n}f=\car_Z f$ by Dominated Convergence with dominating function~$\abs{f}$, and thus, by lower semicontinuity of~$E$,
\begin{equation}\label{eq:t:BDinvSAlg:7}
\nu_f Z =E(\car_Z f)-E(0)\leq\liminf_n\tparen{E(\car_{Z_n}f)-E(0)}=\liminf_n\sum_{k\leq n}\nu_f Y_k =\sum_k\nu_f Y_k \fstop
\end{equation}
Therefore,~$\nu_f Z =\sum_k\nu_f Y_k$ in~$[0,\infty]$, i.e.~$\nu_f$ is countably additive, hence a measure on~$(X,\mfB_\dinv)$.
If~$E$ is grounded, then~$\nu_f(Y)=E(\car_Y f)$ for every~$Y\in\mfB_\dinv$, which proves the last assertion.
\end{proof}

\section*{Competing interest}
The authors have no competing interest.

\section*{Declaration of generative AI and AI-assisted technologies in the manuscript preparation process}
During the preparation of a revised version of this work, the authors used the large language model \emph{Claude Opus~5.5} in order to review the content of the manuscript for factual errors.
Several errors were discovered which were subsequently carefully addressed by the authors.
In the process of reviewing the manuscript, the model produced a version of Example~\ref{ex:AtNotCyclical} as a counterexample.
We found the example instructive and decided to include a revised version into the final manuscript.
After using the model, we reviewed and edited the content as needed and take full responsibility for the content of the published article.

\end{document}